\documentclass{article}

\usepackage{amsmath,amssymb,amsthm}
\usepackage{enumerate}
\usepackage{hyperref}
\usepackage{mathtools,pifont}
\usepackage{tikz-cd}
\usepackage{tikz}
\usepackage{mathrsfs}
\usepackage{xcolor}

\usepackage[margin=1.2in]{geometry}

\newtheorem{theorem}{Theorem}[section]
\newtheorem{observation}[theorem]{Observation}
\newtheorem{lemma}[theorem]{Lemma}
\newtheorem{proposition}[theorem]{Proposition}
\newtheorem{corollary}[theorem]{Corollary}

\newtheorem*{theorem*}{Theorem}

\numberwithin{equation}{section}

\theoremstyle{definition}
\newtheorem{definition}[theorem]{Definition}

\newtheorem{assumption}[theorem]{Assumption}

\theoremstyle{remark}
\newtheorem{remark}[theorem]{Remark}
\newtheorem{example}[theorem]{Example}

\newcommand\R{\mathbb{R}}
\newcommand\C{\mathbb{C}}

\newcommand\Z{\mathbb{Z}}
\newcommand\N{\mathbb{N}}

\newcommand\E{\mathbb{E}}

\newcommand{\cA}{\mathcal{A}}
\newcommand{\cB}{\mathcal{B}}

\newcommand\varstar{\text{\ding{86}}}

\DeclareMathOperator{\id}{id}
\DeclareMathOperator{\Id}{Id}

\DeclareMathOperator{\Span}{Span}

\DeclareMathOperator{\re}{Re}
\DeclareMathOperator{\im}{Im}
\DeclareMathOperator{\Tr}{Tr}
\DeclareMathOperator{\tr}{tr}

\DeclareMathOperator{\TrP}{TrP}

\DeclareMathOperator{\sa}{sa}
\DeclareMathOperator{\ev}{ev}
\DeclareMathOperator{\Perm}{Perm}

\DeclareMathOperator{\supp}{supp}
\DeclareMathOperator{\Diff}{Diff}
\DeclareMathOperator{\BDiff}{BDiff}
\DeclareMathOperator{\app}{app}
\DeclareMathOperator{\Div}{div}

\DeclareMathOperator{\Vect}{Vect}
\DeclareMathOperator{\grad}{grad}

\DeclarePairedDelimiter{\norm}{\lVert}{\rVert}
\DeclarePairedDelimiter{\ip}{\langle}{\rangle}

\begin{document}
	
	%\keywords{free transport, free entropy, free Gibbs law, functional calculus, information geometry, Wasserstein distance}
	
	%\mathclass{46L54, 46L52, 35Q49, 94A17, 58D99}

	\title{Tracial smooth functions of non-commuting variables and the free Wasserstein manifold}
	
	\author{David Jekel, Wuchen Li, Dimitri Shlyakhtenko}
	
	\maketitle

	%	\address{Department of Mathematics, \\
	%		University of California, San Diego, \\
	%		9500 Gilman Drive \# 0112, \\
	%		La Jolla, CA  92093-0112, USA \\
	%		E-mail: djekel@ucsd.edu}
	
	%	\author{Wuchen Li}
	%	\address{Department of Mathematics, \\
	%		University of South Carolina, \\
	%		1523 Greene Street, \\
	%		Columbia, SC 29208-4014, USA \\
	%		E-mail: wuchen@mailbox.sc.edu}
	
	%	\author{Dimitri Shlyakhtenko}
	%	\address{Department of Mathematics, \\
	%		University of California, Los Angeles, \\
	%		520 Portola Plaza, \\ 
	%		Box 951555, \\
	%		Los Angeles, CA 90095-1555, USA, \\
	%		E-mail: shlyakht@math.ucla.edu}

	\begin{abstract}
		Using new spaces of tracial non-commutative smooth functions, we formulate a free probabilistic analog of the Wasserstein manifold on $\R^d$ (the formal Riemannian manifold of smooth probability densities on $\R^d$), and we use it to study smooth non-commutative transport of measure.  The points of the free Wasserstein manifold $\mathscr{W}(\R^{*d})$ are smooth tracial non-commutative functions $V$ with quadratic growth at $\infty$, which correspond to minus the log-density in the classical setting.  The space of non-commutative diffeomorphisms $\mathscr{D}(\R^{*d})$ acts on $\mathscr{W}(\R^{*d})$ by transport, and the basic relationship between tangent vectors for $\mathscr{D}(\R^{*d})$ and tangent vectors for $\mathscr{W}(\R^{*d})$ is described using the Laplacian $L_V$ associated to $V$ and its pseudo-inverse $\Psi_V$ (when defined).
		
		Following similar arguments to \cite{GS2014,DGS2016,JekelExpectation}, we prove the existence of smooth transport along any path $t \mapsto V_t$ when $V_t$ is sufficiently close $(1/2) \sum_j \tr(x_j^2)$, as well as smooth triangular transport.  The two main ingredients are (1) the construction of $\Psi_V$ through the heat semigroup and (2) the theory of free Gibbs laws, that is, non-commutative laws maximizing the free entropy minus the expectation with respect to $V$.  We conclude with a mostly heuristic discussion of the smooth structure on $\mathscr{W}(\R^{*d})$ and hence of the free heat equation, optimal transport equations, incompressible Euler equation, and inviscid Burgers' equation.
	\end{abstract}
	
	\subsection*{Acknowledgements}
	
	We thank Alice Guionnet, Yoann Dabrowski, and Wilfrid Gangbo for various useful discussions.  In particular, we have used many ideas of the joint work of Dabrowski, Guionnet, and Shlyakhtenko \cite{DGS2016}.  Moreover, Jekel would like to thank Guionnet and Dabrowski for enlightening discussions about free Gibbs laws and non-commutative smooth functions at the \'{E}cole Normale Superieure Lyon in March 2020, as well as the Mathematische Forschungsinstitut Oberwolfach for travel support for that visit.  Jekel was supported by a Dissertation Year Fellowship from the UCLA Graduate Division and by the NSF postdoctoral grant DMS-2002826. Li was supported by start-up funding from the University of South Carolina.  Shlyakhenko was partially supported by NSF grant DMS-1762360.
	
	\newpage
	
	\tableofcontents
	
	\newpage
	
	\section{Introduction}
	
	\subsection{Motivation}
	
	Voiculescu's free probability theory treats tracial von Neumann algebras as a non-commutative analog of probability spaces, and studies an analog of probabilistic independence, called free independence, which relates to free products of these von Neumann algebras.  Free probability also describes the large $N$ behavior of certain probability distributions on $N \times N$ matrices, and more generally $d$-tuples of $N \times N$ matrices.  Free probability uses both complex-analytic and combinatorial tools, and relates to the large $N$ representation theory of unitary, orthogonal, and symmetric groups.  For background, see e.g.\ \cite{Voiculescu1991,VDN1992,AGZ2009}.
	
		Voiculescu's theory of free entropy \cite{VoiculescuFE1,VoiculescuFE2,VoiculescuFE5,VoiculescuFE6} is the beginning of free information theory.  As in classical information theory, there are versions of entropy and Fisher's information, which satisfy inequalities similar to the classical entropy and Fisher information.  Voiculescu actually initiated two approaches to free entropy theory.  The first approach uses matricial microstates, or $d$-tuples of matrices that approximate the behavior of the $d$-tuple of operators we want to study; the microstates free entropy describes the $\limsup$ exponential growth rate of the volume of the microstate spaces \cite{VoiculescuFE2}.  Thus, free entropy is the rate function for a (still partially conjectural) large deviation principle in random matrix theory; see \cite{BCG2003}.  The second ``infinitesimal approach'' defines free entropy via the free Fisher information and perturbation by freely independent semicircular families (the free version of Gaussian random variables) \cite{VoiculescuFE5}.
	
	Our main motivation is to find a free version of the Wasserstein manifold.  The classical Wasserstein manifold $\mathscr{P}(\R^d)$ is a formal infinite-dimensional Riemannian manifold whose points are smooth probability densities $\rho$, which has many natural properties \cite{Lafferty1988, LiGeometry, Villani2008}.  By taking the infimum of the lengths of smooth curves in the manifold, the Riemannian metric gives rise to the ($L^2$) Wasserstein distance of two probability measures $\mu$ and $\nu$, which describes the $L^2(\mu)$ distance between an optimal transport map $f$ from $\mu$ to $\nu$ and the identity function \cite{Villani2008}.  The gradient structure of $\mathscr{P}(\R^d)$ describes the differentiation with respect to $\rho$ of certain functionals on the space of probability measures \cite{Otto2001}, and the evolution of a measure under Brownian diffusion turns out to be the gradient flow of the entropy functional \cite{JKO1998} \cite{OV2000}.  Furthermore, the tangent manifold of $\mathscr{P}(\R^d)$ has a symplectic structure \cite{Lafferty1988}, which relates to the geodesic equations on this space.  With suitable modifications, one can connect these results to hydrodynamic equations, including the compressible Euler equation, Schr{\"o}dinger equation, Schr{\"o}dinger bridge problem, and mean field games \cite{CLZ2019,LiGeometry}. The field of transport information geometry is active, and the Hessian operators on the Wasserstein manifold are useful in studying fluid dynamics and formulating functional inequalities \cite{Li2019_diffusionb, Li2020_hessian, Villani2008}.
	
	Although a Wasserstein manifold has never been systematically described for multivariable free probability, some of the key ideas of information geometry have been present as motivation throughout the development of free information theory.  This includes the relationship between entropy and Fisher information \cite{VoiculescuFE1,VoiculescuFE5}, Talagrand inequalities \cite{BV2001,HPU2004,HU2006}, and the relationship between entropy and transport of measure \cite[\S 3]{VoiculescuFE2}.  Seeking a free analog of optimal transport, the third author and Alice Guionnet solved a free Monge-Amp\`ere equation to obtain free monotone transport \cite{GS2014}.  The third author and Yoann Dabrowski and Alice Guionnet used constructed transport along a path of potentials using the relationship between infinitesimal transport and perturbations of the potential, which is the approach we will follow here in \S \ref{sec:pseudoinverse} and \ref{subsec:constructtransport}.  Moreover, the first author used ideas from transport theory (as in \cite{Lafferty1988,OV2000,Otto2001}) to construct free (non-optimal) transport as a large $N$ limit of transport of measure on the space of $N \times N$ matrices \cite{JekelExpectation,JekelThesis}.  Non-commutative transport ideas have been generalized beyond the setting of tracial von Neumann algebras \cite{Shlyakhtenko2003,Nelson2015a,Nelson2015b}.
	
	For a single variable, free entropy has been studied as a functional on the Wasserstein manifold of $\R$, and the relationship between optimal transport for probability measures on $\R$ and optimal transport for random matrix models is better understood \cite{BS2001,HPU2004,MMS2014,LLX2020}.  The setting of several non-commuting variables is significantly more challenging, as is apparent for instance from the open problems about free entropy (see \cite{Voiculescu2002}).  We also point out that several other non-commutative variants of the Wasserstein manifold in quantum information theory.  Carlen and Mass \cite{CM2014} studied the Wasserstein distance related to Gross's Fermionic Fokker Planck equation, which pertains to states on the (finite-dimensional) Clifford algebra.  These states are represented by positive operators of trace $1$, which are a substitute for densities in quantum information theory. Several recent papers have also described Wasserstein manifolds whose points are matrix-valued densities on $\R^d$ or another classical manifold $M$ \cite{NGT2015,CGT2018,BV2020}, positive elements of $L^\infty(M;M_n(\C))$ that integrate to $1$.  But rather than studying matrix-valued densities on $\R^d$, this paper concerns (scalar-valued) densities on the space of $d$-tuples of self-adjoint $N \times N$ matrices and their free probabilistic large-$N$ limit.  As we will see, there is not a direct analog of density in our setting, only of log-density.
	
	We define the free Wasserstein manifold as a space of certain ``smooth (minus) log-densities,'' which are smooth scalar-valued functions of several non-commuting self-adjoint operators (see \S \ref{sec:NCfunc1}).  We define the tangent space at a log-density $V$ in terms of perturbations of $V$, and we describe the relationship between tangent vectors and infinitesimal transport maps through a Laplacian operator $L_V$ associated to $V$ and its pseudo-inverse.  Following the same strategy as \cite{DGS2016} (but in a different technical framework), we give a rigorous treatment in the case of log-densities $V$ that are sufficiently close to the quadratic $V(x_1,\dots,x_d) = (1/2) \sum_{j=1}^d \tr(x_j^2)$, which leads to a free transport result similar to \cite{GS2014,DGS2016} as well as a new $\mathrm{C}^*$ version of the triangular transport results of \cite{JekelExpectation,JekelThesis}.  We conclude by stating versions of the heat equation, Wasserstein geodesic equation, incompressible Euler equation, and inviscid Burgers' equation in our tracial non-commutative framework.
	
	The results in this paper, even though they are technically new, have a large overlap with previous work such as \cite{GS2014,DGS2016,JekelExpectation}, and this is because our goals are largely expository.  The free Wasserstein manifold has been treated in prior work only as motivation or as interpretation a posteriori of analytically rigorous results.  We want to bring it to center stage as a unifying framework that simultaneously provides a heuristic and a proof strategy for rigorous results, playing a similar role to that of the classical Wasserstein manifold in \cite{OV2000}.  With the benefit of hindsight, we strive to organize and present the proofs in the most natural way possible.
	
	The end goals of defining the Wasserstein manifold and constructing transport for potentials close to $(1/2) \sum_j \tr(x_j^2)$ seem modest compared to wealth of knowledge that exists about the classical Wasserstein manifold.  However, as in \cite{GS2014,DGS2016,JekelExpectation,JekelThesis}, even results that are basic in the classical setting require a lot of technical preparation in the free setting.  When developing the classical Wasserstein manifold, people already had a clear understanding of smooth functions, measure and probability theory, and partial differential equations.  By contrast, there is not even a well-established definition of smooth functions for several non-commuting real variables.  Thus, in \S \ref{sec:NCfunc1} and \S \ref{sec:NCfunc2}, we define new spaces of tracial non-commutative smooth functions of several self-adjoint operators in a tracial von Neumann algebra.  Like \cite{DGS2016}, the functions are based on trace polynomials, but the approach to defining the norms is completely different.
	
	Another technical difficulty that arises in the free setting is that there is no direct analog of density in the free setting.  We only know how to pass from a log-density $V$ to a non-commutative law $\mu_V$ through free entropy/random matrix theory or through the heat semigroup associated to $V$ (and the related stochastic differential equations), and in fact we will combine both of these approaches in this paper (see \S \ref{sec:freeGibbslaws} and \S \ref{sec:pseudoinverse} respectively).  In particular, in \S \ref{sec:freeGibbslaws}, we define free Gibbs laws for $V$ as the maximizers of free entropy minus the expectation of $V$, giving for the first time a proof of their existence and properties directly from the definition of free entropy, as motivated by \cite[\S 3.7]{Voiculescu2002} and \cite{Hiai2005}.
	
	We hope that the framework of tracial non-commutative functions in the first part of this paper will be a starting point for future work on the free Wasserstein manifold, non-commutative SDE and PDE theory, and non-commutative optimal transport, and thus that the detailed discussion of the properties of these smooth functions will save time for later work.  In particular, in \S \ref{sec:applications}, we formulate several differential equations of interest for free transport information geometry and operator algebras, including the geodesic equation and gradient flow on the Wasserstein manifold and the compressible Euler equation. Our framework allows for a closer resemblance of these equations with their classical analogs than previously understood, because it includes a natural description of scalar-valued smooth functions of several operators.  Of course, the rigorous study of these equations will be another undertaking, and we do not expect all the results from the classical setting to carry over in the same level of generality.  Nonetheless, it is a crucial first step to clarify the connection between the classical and free versions of an equation and what it would mean for a smooth function to solve the equation.
	
	In the remainder of the introduction, \S \ref{subsec:summary} gives an executive summary of key constructions and results, \S \ref{subsec:heuristics} describes the random matrix heuristics for our technical framework as well as the challenges that arise in the non-commutative setting, and \S \ref{subsec:outline} describes the organization of the paper.  We will give brief explanations of terminology we use in the introduction when possible, but the reader may also refer as needed to \S \ref{sec:preliminaries} for background on operator algebras and on the classical Wasserstein manifold.
	
	\subsection{Summary of constructions and results} \label{subsec:summary}
	
	We will set up the free Wasserstein manifold as follows:
	\begin{itemize}
		\item We define a space $\tr(C_{\tr}^\infty(\R^{*d}))$ of scalar-valued smooth functions of several self-adjoint operators in a tracial von Neumann algebra.  Another space $C_{\tr}^\infty(\R^{*d})_{\sa}^d$ provides the analog of smooth functions $\R^d \to \R^d$ (a.k.a.\ vector fields on $\R^d$).
		\item The free Wasserstein manifold $\mathscr{W}(\R^{*d})$ is defined as the space of $V \in \tr(C_{\tr}^\infty(\R^{*d}))$ such that $V$ is bounded above and below by a quadratic function, that is, $a + b V_0 \leq V \leq a' + b' V_0$ for some constants with $b, b' > 0$, where $V_0(\mathbf{x}) = (1/2) \sum_{j=1}^d \tr(x_j^2)$.
		\item The tangent space to $\mathscr{W}(\R^{*d})$ consists of $\tr(C_{\tr}^\infty(\R^{*d}))$ functions with some bounds on the first and second derivatives.
		\item For $V \in \tr(C_{\tr}^\infty(\R^{*d}))$, we define the associated free Gibbs laws as non-commutative laws that maximize a certain entropy functional.  A free Gibbs law $\nu$ must satisfy the integration-by-parts relation $\nu(\nabla_V^* \mathbf{h}) = 0$ for any vector field $\mathbf{h}$, where $\nabla_V^*$ is the free analog of the divergence operator associated to $V$.  If there is a unique law satisfying this equation, we denote it by $\mu_V$.
		\item The Riemannian metric at $V$ for two tangent vectors $W_1$ and $W_2$ is given by $\nu_V(\ip{\nabla L_V^{-1}W_1, \nabla L_V^{-1}W_2})$, where $L_V = -\nabla_V^* \nabla$ is a Laplacian operator associated to $V$, whenever the above expression makes sense.
		\item We show rigorously that the definition makes sense for $V$ sufficiently close to the quadratic $V_0$.
	\end{itemize}
	
	We have the following definitions and results relating to non-commutative transport of measure:
	\begin{itemize}
		\item We define an analog of diffeomorphisms of $\R^d$, as well as a construction of certain diffeomorphisms as flows along vector fields.  A Lie bracket on vector fields is defined analogous to the classical case.
		\item For a diffeomorphism $\mathbf{f}$ and a potential $V$, there is a push-forward defined by $\mathbf{f}_* V = V \circ \mathbf{f}^{-1} - \log \Delta_\# (\partial \mathbf{f}^{-1})$, where $\log \Delta_\#$ is an analog of the log-determinant.  The push-forward defines an action of the diffeomorphism group on the Wasserstein manifold.
		\item With certain assumptions on $V$, if there is a unique free Gibbs law $\mu_V$, then $\mathbf{f}_* \mu_V$ is the unique free Gibbs law for $\mathbf{f}_*V$ (see Proposition \ref{prop:changeofvariables}).
		\item Given a one-parameter group of diffeomorphisms $\mathbf{f}_t$ generated by a vector field $\mathbf{h}$, the tangent vector $(d/dt)|_{t=0} (\mathbf{f}_t)_* V$ is given by $\nabla_V^* \mathbf{h}$.
		\item Conversely, for a tangent vector $W$, a possible vector field $\mathbf{h}$ for producing transport is given by $\nabla (-L_V)^{-1} W$, provided that the latter makes sense.
		\item When $V$ is sufficiently close to the quadratic, we can make this relationship between tangent vectors and infinitesimal tranport rigorous.  Thus, for any continuously differentiable path $t \mapsto V_t$ of potentials close to the quadratic, we can naturally produce a family of transport maps $\mathbf{f}_t$ with $(\mathbf{f}_t)_* V_0 = V_t$ (see Theorem \ref{thm:transport}).
		\item We can also arrange that the transport maps $\mathbf{f}_t$ are lower-triangular functions in the sense that for $j = 1$, \dots, $d$, the $j$th coordinate of $\mathbf{f}_t(\mathbf{x}_1,\dots,\mathbf{x}_d)$ depends only on $\mathbf{x}_1$, \dots, $\mathbf{x}_j$ (see Theorem \ref{thm:triangulartransport}).
	\end{itemize}
	
	The last result on triangular transport is a partial analog of classical triangular transport of measure studied in \cite{BKM2005}.  It has the following consequence for operator algebras, which is given in further detail in Corollary \ref{cor:triangulartransport}.
	
	\begin{theorem*}
		Let $V \in \tr(C_{\tr}^\infty(\R^{*d}))_{\sa}$ be sufficiently close to $V_0(\mathbf{x}) = (1/2) \sum_j \tr(x_j^2)$ (more precisely, assume that the first and second derivatives are sufficiently close and third derivative is uniformly bounded).  Let $\mu_V$ be the associated free Gibbs law, and let $(\cA,\tau)$ be the tracial $\mathrm{W}^*$-algebra associated to $\mu_V$, with the canonical generators $\mathbf{X} = (X_1,\dots,X_d)$. Let $(\cB,\sigma)$ be the tracial $\mathrm{W}^*$-algebra generated by a standard free semicircular family $\mathbf{S} = (S_1,\dots,S_d)$.  Then there exists an isomorphism of tracial von Neumann algebras $\phi: (\cA,\tau) \to (\cB,\sigma)$ such that for each $j = 1, \dots, d$, we have
		\[
		\phi(\mathrm{C}^*(X_1,\dots,X_j)) = \mathrm{C}^*(S_1,\dots,S_j).
		\]
	\end{theorem*}
	
	This is in some sense an improvement of the triangular transport results from \cite{JekelExpectation,JekelThesis}; it asserts an isomorphism of $\mathrm{C}^*$-algebras not only of $\mathrm{W}^*$-algebras, but it also has stronger smoothness hypotheses on $V$.  Of course, the existence of transport that was not necessarily triangular was already known from \cite{GS2014,DGS2016}.
	
	In the final section, we present several differential equations related to the free Wasserstein manifold for future study, including the following:
	\begin{itemize}
		\item We differentiate the functional $V \mapsto \mu_V(f)$ for $V \in \mathscr{W}(\R^{*d})$.
		\item We explain how the non-commutative heat equation $\dot{V}_t = L_{V_t} V_t$ represents the gradient flow of free entropy, similar to the classical case \cite{Otto2001}.
		\item We state the free version of the geodesic equations on $\mathscr{W}(\R^{*d})$, which are $\dot{V}_t = L_{V_t} \phi_t$ and $\dot{\phi}_t = -(1/2) \ip{\nabla \phi_t, \nabla \phi_t}_{\tr}$.  We show that smooth solutions satisfy $V_t = (\id + t \nabla \dot{\phi}_0)_* V_0$.  We also show that the path $t \mapsto \mu_{V_t}$ is a minimal curve in the $L^2$-coupling distance on the space of non-commutative laws.
		\item We state a non-commutative incompressible Euler equation with respect to a potential $V$ in a similar spirit to \cite{VoiculescuEuler}.  Similar to the classical case \cite{Arnold1966}, this represents the geodesic equation on the group of non-commutative diffeomorphisms that preserve $V$.  Similarly, the geodesic equation on the entire non-commutative diffeomorphism group is the non-commutative inviscid Burgers' equation.
	\end{itemize}
	
	\subsection{Random matrix heuristics} \label{subsec:heuristics}
	
	Our formulation of the free Wasserstein manifold is closely linked with random matrix theory and free Gibbs laws.  One branch of random matrix theory studies probability measures $\mu^{(N)}$ on $M_N(\C)_{\sa}^d$ (the space of $d$-tuples of self-adjoint $N \times N$ matrices) of the form
	\[
	d\mu_f^{(N)}(\mathbf{X}) = \text{constant} \, e^{-N^2 \tr_N(f(\mathbf{X}))}\,d\mathbf{X}.
	\]
	Here $\mathbf{X} = (X_1,\dots,X_d) \in M_N(\C)_{\sa}^d$; $\tr_N$ denotes the normalized trace $(1/N) \Tr$ on $M_N(\C)$, and $d\mathbf{X}$ is Lebesgue measure on $M_N(\C)_{\sa}^d$, which we view as a real inner product space of dimension $dN^2$ with the inner product $\ip{\mathbf{X},\mathbf{Y}} = \sum_{j=1}^d \tr_N(X_j Y_j)$; and $f$ is a non-commutative polynomial in $d$-variables such that $\tr_N(f(\mathbf{X}))$ is real for $\mathbf{X} \in M_N(\C)_{\sa}^d$.  More generally, we can consider
	\[
	d\mu_V^{(N)}(\mathbf{X}) = \text{constant} \, e^{-N^2 V(\mathbf{X})}\,d\mathbf{X},
	\]
	where $V$ is a \emph{trace polynomial}, that is, a formal linear combination of terms of the form $\tr(f_1) \dots \tr(f_k)$ for some $k \in \N$ and non-commutative polynomials $f_1$, \dots, $f_k$.  Such models were first studied for a single matrix in \cite{BGK2015} and then for multiple matrices in \cite{DGS2016}.  Here $V$ is evaluated on some $X \in M_N(\C)_{\sa}$ by replacing each term $\tr(f_j)$ by $\tr_N(f_j(X))$.  This more general class of trace polynomials is quite natural because, every polynomial function $M_N(\C)_{\sa}^d \to \R$ (that is, polynomial with respect to the real and imaginary parts of the matrix entries) that is invariant under conjugation by unitary matrices must be given by a trace polynomial, which follows from the work of Procesi \cite{Procesi1976}.  For prior work relating trace polynomials with random matrix theory, see \cite{Rains1997,Sengupta2008,Cebron2013,DHK2013,Kemp2016,Kemp2017,DGS2016}.
	
	The measure $\mu_V^{(N)}$ is an element of the classical Wasserstein manifold $\mathscr{P}(M_N(\C)_{\sa}^d)$ since it has a smooth density.  However, the density does not have a large $N$ limit since there is an $N^2$ in the exponent.  However, $-1/N^2$ times the log of density is precisely $V$, which is dimension-independent by assumption.  This leads us to the following heuristic for studying the free Wasserstein manifold:  Reparametrize $\mathcal{P}(M_N(\C)_{\sa}^d)$ in terms of $V = -(1/N^2) \log \rho$ instead of in terms of the density $\rho$.  Compute the Riemannian metric (and whatever other objects of differential equations we wish to study) in terms of $V$ rather than $\rho$.  Then study the behavior of this object as $N \to \infty$.  The reparametrization in terms of the log-density for the classical Wasserstein manifold $\mathscr{P}(\R^d)$ is explained in \S \ref{subsec:logdensity}.
	
	Following this recipe, to define the Riemannian metric for the tangent space at $V$, consider two different trace polynomials $W_1$ and $W_2$.  Then the curves $t \mapsto V + t W_j$ represent tangent vectors in $\mathscr{P}(M_N(\C)_{\sa}^d)$.  Since $V + tW_j$ is considered up to an additive constant, assume that $\int W_j\,d\mu = 0$.  It follows from the computations in \S \ref{subsec:logdensity} that the inner product of the two tangent vectors with respect to the Riemannian metric on $\mathscr{P}(M_N(\C)_{\sa}^d)$ is given by
	\begin{equation} \label{eq:metricheuristic}
		\int \ip{\nabla (L_V^{(N)})^{-1} W_1, \nabla (L_V^{(N)})^{-1} W_2 }\,d\mu_V^{(N)},
	\end{equation}
	where
	\[
	L_V^{(N)} f = \frac{1}{N^2} \Delta f - \ip{\nabla V, \nabla f}.
	\]
	If $f$ is a scalar-valued trace polynomial, then $\nabla f$ is dimension-independent and $(1/N^2) \Delta f$ on $M_N(\C)_{\sa}^d$ is given by a trace polynomial which converges coefficient-wise as $N \to \infty$ to some trace polynomial $Lf$; see \cite[\S 2]{Cebron2013}, \cite[\S 3]{DHK2013}, \cite[\S 14.1]{JekelThesis}, or Lemma \ref{lem:asymptoticLaplacian} below.  Hence, the normalization of $L_V^{(N)}$ above is dimension-independent for our random matrix setting.  The Riemannian metric for the free Wasserstein manifold should heuristically be the large $N$ limit of \eqref{eq:metricheuristic}.
	
	Several ingredients are desirable to make this heuristic precise:
	\begin{enumerate}[(1)]
		\item We want to understand the large $N$ behavior of $\mu_V^{(N)}$.
		\item We want a notion of ``trace $C^\infty$ functions'' that generalizes trace polynomials, such that $L_V$ is well-defined on any trace $C^\infty$ function.  Of course, we will replace the trace polynomials in the definition with these smooth functions.
		\item We want to study the pseudo-inverse of $L_V$ on the space of trace smooth functions (and we hope that the kernel and cokernel are $1$-dimensional).
	\end{enumerate}
	Let us discuss each of these questions in more detail.
	
	(1) In the case where $V$ is a perturbation of the quadratic, prior work has shown that $\int f\,d\mu_V^{(N)}$ converges almost surely to some deterministic limit when $f$ is a scalar-valued trace polynomial \cite{GMS2006,GS2009,Jekel2018}.  This limit is described in terms of a tuple $\mathbf{X}$ of self-adjoint operators from a von Neumann algebra $\cA$ equipped with a (faithful, normal) tracial linear functional $\tau: \cA \to \C$.  We have $\int f\,d\mu_V^{(N)} \to f(\mathbf{X})$ for all scalar-valued trace polynomials $f$, where the evaluation of $f$ on $\mathbf{X}$ is given in the same way as the evaluation on a tuple of matrices, with $\tau$ instead of $\tr_N$.  In fact, the evaluation $f(\mathbf{X})$ for a trace polynomial is completely determined by the evaluations $\tau(p(\mathbf{X}))$ for non-commutative polynomials $p$.  Thus, the (bulk) large $N$ behavior of $\mu_V^{(N)}$ is described by the non-commutative law of $\mathbf{X}$, that is, the linear functional $\C\ip{x_1,\dots,x_d} \to \C$ given by $p \mapsto \tau(p(\mathbf{X}))$.
	
	For more general $V$, a sufficient condition for such convergence to happen is if there is a unique non-commutative law $\nu_V$ that maximizes $\chi(\nu) - \nu(V)$, where $\chi$ is Voiculescu's microstates free entropy.  We discuss this approach in \S \ref{sec:freeGibbslaws}.
	
	(2) The second ingredient is to develop a notion of ``trace smooth functions'' which generalizes trace polynomials and which is closed under natural operations such as differentiation and composition.  In fact, to consider the derivatives of trace polynomials, we must consider more general objects than trace polynomials maps $M_N(\C)_{\sa}^d \to \C$.  Indeed, the gradient of such a function will be a map $M_N(\C)_{\sa}^d \to M_N(\C)^d$, which is a $d$-tuple of \emph{operator-valued trace polynomials} $M_N(\C)_{\sa}^d \to M_N(\C)$.  The operator-valued trace polynomials are linear combinations of terms such as $f_0 \tr(f_1) \dots \tr(f_k)$ where $f_0$, \dots, $f_k$ are non-commutative polynomials.  Of course, since $f_0$ can be $1$, any scalar-valued trace polynomial can be viewed as an operator-valued trace polynomial, and thus we can pass to the more general consideration of operator-valued trace polynomials.  If $f$ is an operator-valued trace polynomial, and if $\mathbf{X}$, $\mathbf{Y}_1$, \dots, $\mathbf{Y}_k$ are in $M_N(\C)_{\sa}^d$, then the iterated directional derivative
	\[
	\frac{d}{dt_1}\biggr|_{t_1=0} \dots \frac{d}{dt_k} \biggr|_{t_k=0} f(\mathbf{X} + t_1 \mathbf{Y}_1 + \dots + t_k \mathbf{Y}_k)
	\]
	defines an operator-valued trace polynomial in $\mathbf{X}$, $\mathbf{Y}_1$, \dots, $\mathbf{Y}_d$ that is multilinear in $\mathbf{Y}_1$, \dots, $\mathbf{Y}_d$.
	
	We define $C_{\tr}(\R^{*d},\mathscr{M}^k)$ as the completion of the space of operator-valued trace polynomials in $\mathbf{X}$, $\mathbf{Y}_1$, \dots, $\mathbf{Y}_k$ that are multilinear in $\mathbf{Y}_1$, \dots, $\mathbf{Y}_k$, with respect to a certain family of seminorms $\norm{f}_{C_{\tr}(\R^{*d},\mathscr{M}^k),R}$for $R > 0$.  Here for each radius $R$, the seminorm $\norm{f}_{C_{\tr}(\R^{*d},\mathscr{M}^k),R}$ is defined as follows:  Fix a tracial von Neumann algebra $(\cA,\tau)$ and $\alpha$, $\alpha_1$, \dots, $\alpha_k \in [1,\infty]$ with $1/\alpha = 1/\alpha_1 + \dots + 1/\alpha_k$.  Take the supremum of $\norm{f(\mathbf{X})[\mathbf{Y}_1,\dots,\mathbf{Y}_d]}_{L^\alpha(\cA,\tau)}$ over $\mathbf{X}$ in an operator norm ball of radius $R$ and $\mathbf{Y}_j$ in the unit ball of $L^{\alpha_j}(\cA,\tau)$.  Then take the supremum over $(\cA,\tau)$ and $\alpha$, $\alpha_1$, \dots, $\alpha_k$.
	
	Then $C_{\tr}^k(\R^{*d})$ is defined as the space of functions whose derivatives of order $k' \leq k$ are in $C_{\tr}(\R^{*d},\mathscr{M}^{k'})$.  On $C_{\tr}^\infty(\R^{*d})$, differentiation and composition are well-defined, and there is a Laplacian operator $L_V$ that describes the large $N$ behavior of $L_V^{(N)}$.
	
	\begin{remark}
		Our space $C_{\tr}^k(\R^{*d})$ is closely related to the definition in \cite{DGS2016} of trace $C^k$ functions on the operator norm ball of radius $R$.  However, the definition in \cite{DGS2016} was more complicated because it involved separating out different types of terms in the derivative and using Haagerup tensor norms.  The norms used in this paper have some of the same desirable properties, such as good behavior under conditional expectations and the ability to control the Lipschitz norms of a function with respect to $\norm{\cdot}_2$.  The definition in \cite{DGS2016} also had some unavoidable complexity due to working in setting of operator-valued free probability which replaced the scalars $\C$ with some von Neumann algebra $\mathcal{B}$.
	\end{remark}
	
	(3) We study the pseudo-inverse of $L_V$ rigorously in the case where $V$ is sufficiently close to a quadratic.  The strategy is the same as previous works such as \cite{BS2001,GS2009,GS2014,DGS2016}.  In fact, the results about the expectation with respect to $\nu_V$ discussed above in (1) and the results about the pseudo-inverse $\Psi_V$ both follow from the study of the heat semigroup $e^{tL_V}$.  Indeed, we hope to obtain the expectation map the $\mathbb{E}_V: C_{\tr}(\R^{*d}) \to \C$ associated to $\nu_V$ as
	\[
	\mathbb{E}_V f = \lim_{t \to \infty} e^{tL_V} f
	\]
	and the pseudo-inverse of $L_V$ as
	\[
	\Psi_V f = \int_0^\infty (e^{tL_V} - \mathbb{E}_V)f\,dt.
	\]
	
	The most explicit known method of constructing the heat semigroup in the free setting is using free stochastic differential equations, as in the papers cited above. Let $(\cA,\tau)$ be a tracial $\mathrm{W}^*$-algebra and $\mathbf{X} \in \cA_{\sa}^d$.  Let $\mathcal{X}(\mathbf{X},t)$ be a stochastic process solving the equation
	\[
	d\mathcal{X}(\mathbf{X},t) = d\mathcal{S}(t) - \frac{1}{2} \nabla_x V(\mathcal{X}(\mathbf{X},t))\,dt, \qquad \mathcal{X}(\mathbf{X},0) = \mathbf{X}.
	\]
	where $(\mathcal{S}(t))_{t \in [0,\infty)}$ is a free Brownian motion in $d$ variables, freely independent of $\mathbf{X}$.  Then we define $(e^{tL_V} f)(\mathbf{X}) = E_{\cA} f(\mathcal{X}(\mathbf{X},2t))$ for $\mathbf{X} \in \cA_{\sa}^d$.
	
	We prove in \S \ref{sec:pseudoinverse} that for smooth $V$, the resulting stochastic process and the heat semigroup are smooth functions of $\mathbf{X}$ and depend continuously on $V$.  This argument is closely parallel to \cite[\S 3]{DGS2016}, only with different spaces of functions and with more details given for the inductive arguments.  More importantly, the results are proved more generally in the conditional setting where the functions depend on an auxiliary $d'$-tuple of variables $\mathbf{X}'$.  This is what enables us to prove the triangular transport theorem in \S \ref{subsec:triangulartransport}.
	
	Unfortunately, we do not expect that $L_V$ will be invertible for arbitrary $V \in \mathscr{W}(\R^{*d})$.  As we discuss in \S \ref{subsec:discussion}, the work of \cite{BS2001,BG2013multi,BGK2015} and others on the $d = 1$ case shows that in general the Laplacian might have a kernel of dimension larger than $1$ when acting the $L^2$ space associated to the free Gibbs law.
	
	We conclude the discussion by pointing out an (at first) counterintuitive feature of our definition of $\mathscr{W}(\R^{*d})$:  There could in principle be many different functions $V$ satisfying Assumptions \ref{ass:freeGibbs} and \ref{ass:Laplacian} which produce the same non-commutative law $\mu_V$.  This is unavoidable because if $\mu_V$ is realized by a $d$-tuple of bounded operators with norm $ < R$, then we could perturb $V$ outside the ball of radius $R$ and end up with the same law $\mu_V$.
	
	Besides perturbing $V$ outside the ``support'' of $\mu_V$, there is another way in which such degeneracy can arise, which is easier to describe from the point of view of the tangent space.  The Riemannian metric $\ip{\cdot,\cdot}_V$ could have a very large kernel in $T_V \mathscr{W}(\R^{*d})$.  Indeed, suppose $(\cA,\tau)$ is the tracial von Neumann algebra associated to the GNS representation of $\mu_V$ and $\mathbf{X}$ is the canonical generating tuple (see Proposition \ref{prop:GNS}).  Then for tangent vectors $\dot{V}$ and $\dot{W}$, we have
	\[
	\ip{\dot{V},\dot{W}}_V = \ip{(\nabla \Psi_V \dot{V})^{\cA,\tau}(\mathbf{X}), (\nabla \Psi_V \dot{W})^{\cA,\tau}(\mathbf{X})}_\tau.
	\]
	Thus, $\dot{V}$ will be in the kernel of $\ip{\cdot,\cdot}_V$ if and only if $\nabla \Psi_V \dot{V}$ evaluates to zero on $\mathbf{X}$.  There are many functions in $C_{\tr}(\R^{*d})$ which evaluate to zero on $\mathbf{X}$; for instance, for any trace polynomial $\mathbf{f}$, there will be a non-commutative polynomial $\mathbf{g}$ with $\mathbf{f}^{\cA,\tau}(\mathbf{X}) = \mathbf{g}^{\cA,\tau}(\mathbf{X})$.
	
	The fact that $\mu_V$ does not uniquely determine $V$ might seem like a defect in the definition.  In the classical case, the space of probability measures on $\R^d$ is the completion of smooth positive densities with respect to a certain topology.  But to obtain some space of non-commutative laws from the free Wasserstein manifold defined here, one has to first quotient out by the equivalence relation that $V \sim W$ if $\mu_V = \mu_W$, that is, we must use a separation-completion rather than a completion.
	
	A heuristic explanation for why this degeneration occurs is because the random matrix models often have exponential concentration of measure as $N \to \infty$ (see e.g.\ \cite{GZ2000}).  Although the measures $\mu_V^{(N)}$ are supported on all of $M_N(\C)_{\sa}^d$, their mass concentrates on much smaller sets, namely the matricial microstate spaces of Voiculescu.  Due to the concentration of measure, one must be very careful about the normalization of various quantities associated to $V$ and $\mu_V^{(N)}$.  For instance, we earlier gave the formula $\int \ip{\nabla (L_V^{(N)})^{-1} W_1, \nabla (L_V^{(N)})^{-1} W_2 }\,d\mu_V^{(N)}$ for the Riemannian metric which turns out to be dimension-independent, but the metric could also be written as
	\[
	N^2 \int (-L_V^{(N)})^{-1} W_1 \cdot W_2\,d\mu^{(N)}.
	\]
	Thus, it turns out that $\int (-L_V^{(N)})^{-1} W_1 \cdot W_2\,d\mu^{(N)}$ goes to zero as $N \to \infty$ (we can also see this because both $(-L_V^{(N)})^{-1} W_1$ and $W_2$ are close their mean, which is zero, with high probability).  Thus, the Riemannian metric cannot be defined by this formula in the large-$N$ limit.
	
	The choice to work with globally defined functions in $C_{\tr}(\R^{*d})^d$ rather than only their projections in $L^2(\mu_V)^d$ enables us to more easily apply the ideas of classical analysis.  This is conceptually similar to how might study functions on some small and complicated compact subset $K$ of $\R^d$ by first analyzing those which extend to smooth functions in a neighborhood of $K$.  Prior work on free transport such as \cite{GS2014} and \cite{DGS2016} has also used functions that are globally defined (at least on some operator-norm ball) rather than only on the specific $d$-tuple of operators realizing the law $\mu_V$.  Since degeneration is unavoidable in any case, we might as well frame the Wasserstein manifold in terms of the globally defined functions that are more analytically tractable rather than attempting to sort out the difficult technical question of exactly how much degeneration occurs.
	
	Besides, as seen in \cite{Jekel2018,JekelExpectation,JekelThesis} as well as \S \ref{subsec:matrixapproximation} - \ref{subsec:conditional} of this paper, for $V$ sufficiently close to $\sum_j \tr(x_j^2)$, various functions $\mathbf{f}^{(N)}$ on $M_N(\C)_{\sa}^d$ associated to $\mu_V^{(N)}$ will as $N \to \infty$ be asymptotically close to corresponding non-commutative functions $\mathbf{f}$ in $C_{\tr}(\R^{*d})_{\sa}^d$ \emph{everywhere} (uniformly on each operator-norm ball) rather than only the microstate spaces associated to $\mu_V$.  These results are better than we might expect; due to concentration of measure, there is no way to deduce them simply from studying the ``bulk behavior'', or knowing the $L^2(\mu_V^{(N)})$-norms of non-commutative functions on $M_N(\C)_{\sa}^d$ as $N \to \infty$.  Another way to describe this phenomenon is that the $C_{\tr}(\R^{*d})_{\sa}^d$ functions carry more information about the large $N$ behavior of the random matrix models than could be detected from the non-commutative law $\mu_V$ alone.  However, it is unclear to what extent this generalizes when $V$ is not close to $(1/2) \sum_j \tr(x_j^2)$ or not uniformly convex.
	
	Another difficulty in framing the free Wasserstein manifold is that the non-commutative laws $\mu_V$ associated to our smooth potentials $V \in \mathscr{W}(\R^{*d})$ might not be dense in the space of all non-commutative laws.  Certainly, we can only approximate non-commutative laws that can be approximated by the non-commutative laws of matrix tuples (or laws whose associated von Neumann algebras are Connes-embeddable); and we now know that not all tracial $\mathrm{W}^*$-algebras are Connes-embeddable due to the recent work on related problems in quantum information theory \cite{JNVWY2020}.  But even after we restrict our attention to Connes-embeddable von Neumann algebras, it is unlikely than an arbitrary potential $V \in \mathscr{W}(\R^{*d})$ can be approximated by other potentials $W$ such that $L_W$ has a one-dimensional kernel, in light of the counterexamples in the single-matrix setting (see \ref{subsec:discussion}).
	
	\subsection{Outline} \label{subsec:outline}
	
	In \S \ref{sec:preliminaries}, we explain background material and terminology.  In \S \ref{subsec:operatoralgebras}, we summarize definitions and results about $\mathrm{C}^*$ and von Neumann algebras that will be used throughout the paper.  In \S \ref{subsec:logdensity}, as a heuristic reference point, we describe the classical Wasserstein manifold and give a parametrization of it in terms of the log-density rather than the density.
	
	In \S \ref{sec:NCfunc1}, we define spaces of tracial non-commutative $C^k$ functions, and describe their basic properties, such as the chain rule for composition.  In \S \ref{sec:NCfunc2}, we relate non-commutative functions with smooth functional calculus for self-adjoint operators, and we describe differential operators on non-commutative smooth functions that mimic the gradient and Laplacian of trace polynomial functions on $M_N(\C)_{\sa}^d$.
	
	In \S \ref{sec:manifold}, we define the free Wasserstein manifold $\mathscr{W}(\R^{*d})$, diffeomorphism group $\mathscr{D}(\R^{*d})$, and action $\mathscr{D}(\R^{*d}) \curvearrowright \mathscr{W}(\R^{*d})$ by transport.
	
	In \S \ref{sec:pseudoinverse}, we analyze the heat semigroup, expectation, and pseudo-inverse associated to the Laplacian $L_V$ when $V$ is sufficiently close to the quadratic $V_0$.  In particular, we construct an operator $\Psi_V$ such that $-\Psi_V L_V f = f - \mathbb{E}_V(f)$, where $\mathbb{E}_V$ is the expectation functional (which will turn out to agree with $\mu_V$).
	
	In \S \ref{sec:freeGibbslaws}, we discuss a version of Voiculescu's free entropy $\chi$ defined on (a slight generalization of) non-commutative laws.  We show that for certain $V$ (with quadratic growth at $\infty$ but not necessarily convex), there always exist non-commutative laws $\nu$ maximizing $\chi(\nu) - \nu(V)$.  Any free Gibbs law must satisfy the equation $\nu(\nabla_V^* \mathbf{h}) = 0$ (Proposition \ref{prop:integrationbyparts}).  Finally, when $\partial V$ and $\partial^2 V$ are bounded, this equation implies that $\nu$ can be realized by a $d$-tuple of bounded operators (Theorem \ref{thm:magicnormbound}).
	
	In \S \ref{subsec:constructtransport}, the results from \S \ref{sec:pseudoinverse} and \S \ref{sec:freeGibbslaws} are combined in the framework of the free Wasserstein manifold to yield a rigorous construction of transport of measure for $V$ sufficiently close to $\sum_j \tr(x_j^2)$.  More precisely, for any continuously differentiable path $t \mapsto V_t$ with $V_t$ sufficiently close to $\sum_j \tr(x_j^2)$, there is a path $t \mapsto \mathbf{f}_t$ of diffeomorphisms with $(\mathbf{f}_t)_* V_0 = V_t$, and our choice of $t \mapsto \mathbf{f}_t$ is ``infinitesimally optimal'' (Theorem \ref{thm:transport}).
	
	In the remainder of \S \ref{sec:rigoroustransport}, we adapt the technique to prove triangular transport (Theorem \ref{thm:triangulartransport}) by studying conditional expectations and transport.  An important tool for the enterprise, which is interesting in its own right, is a precise connection between non-commutative functions and functions on $N \times N$ matrices in the large $N$ limit.  In particular, similar to \cite{JekelExpectation,JekelThesis}, we show that a certain conditional expectation operator from \S \ref{sec:pseudoinverse} describes the large-$N$ limit of conditional expectations for the matrix models.
	
	Finally, \S \ref{sec:applications} suggests directions for future research.  In particular, we state and heuristically derive non-commutative versions of the heat equation, Wasserstein geodesic equation, incompressible Euler equation, and inviscid Burgers' equation.
	
	\section{Preliminaries} \label{sec:preliminaries}

\subsection{Operator algebras and free probability} \label{subsec:operatoralgebras}

We recall some standard definitions and results about $\mathrm{C}^*$ and von Neumann algebras, non-commutative laws, and free independence.  For background material on $\mathrm{C}^*$ and von Neumann algebras, see e.g.\ \cite{KadRing1,KadRing2}.

\begin{definition}[$*$-algebra]
	A \emph{unital $*$-algebra (over $\C$)} is a unital algebra $\cA$ over $\C$ equipped with a skew-linear involution $*$: $\cA \to \cA$ satisfying $(ab)^* = b^*a^*$.  We call $a^*$ the \emph{adjoint} of $a$, and we say $a$ is \emph{self-adjoint} if $a^* = a$.  We denote by $\cA_{\sa}$ the set of self-adjoint elements (which is a vector space over $\R$).
\end{definition}

\begin{definition}[$\mathrm{C}^*$-algebra]
	Let $B(H)$ denote the $*$-algebra of bounded operators on a Hilbert space $H$ (where the $*$-operation is the adjoint in the usual sense).  A \emph{(unital) $\mathrm{C}^*$-algebra} is a unital $*$-subalgebra of $B(H)$ that is closed with respect to the operator norm.
\end{definition}

\begin{definition}[$\mathrm{W}^*$-algebra]
	The \emph{$\sigma$-weak operator topology ($\sigma$-WOT)} on $B(H)$ is the topology generated by all maps $B(H) \to \C$ of the form
	\[
	T \mapsto \sum_{j=1}^\infty \ip{\xi_j, T \xi_j},
	\]
	where $(\xi_j)_{j \in \N}$ is a sequence of vectors with $\sum_j \norm{\xi_j}^2 < \infty$.  (Equivalently, the $\sigma$-WOT is weak-$\star$ topology on $B(H)$ obtained from viewing it as the dual of the space of trace-class operators.)  A \emph{von Neumann algebra} or \emph{$\mathrm{W}^*$-algebra} is a unital $*$-subalgebra of $B(H)$ that is closed in the $\sigma$-WOT.
\end{definition}

\begin{definition}[States and traces]
	If $\cA$ is a unital $*$-algebra, then a linear functional $\phi: \cA \to \C$ is said to be \emph{positive} if $\phi(a^*a) \geq 0$ for all $a \in \cA$, \emph{unital} if $\phi(1) = 1$, \emph{tracial} if $\phi(ab) = \phi(ba)$ for $a, b \in \cA$, \emph{faithful} if $\phi(a^*a) = 0$ implies $a = 0$.  If $\cA$ is a $\mathrm{W}^*$-algebra, then $\phi$ is said to be \emph{normal} if it is continuous with respect to the $\sigma$-WOT.  A \emph{state} is unital positive functional, and a \emph{trace} is a unital positive tracial functional.
\end{definition}

\begin{definition}[Tracial $\mathrm{C}^*$ and $\mathrm{W}^*$-algebras]
	A \emph{tracial $\mathrm{C}^*$-algebra} is a pair $(\cA,\tau)$ where $\cA$ is a $\mathrm{C}^*$-algebra and $\tau$ is a faithful trace.  A \emph{tracial $\mathrm{W}^*$-algebra} is a pair $(\cA,\tau)$ where $\cA$ is a $\mathrm{W}^*$-algebra and $\tau$ is a faithful normal trace.
\end{definition}

\begin{definition}[$*$-homomorphisms]
	A \emph{$*$-homomorphism} from one $*$-algebra to another is a linear map which respects multiplication and the $*$-operation.  A $*$-homomorphism of unital $*$-algebras is called \emph{unital} if it preserves $1$.  A $*$-homomorphism of $\mathrm{W}^*$-algebras is said to be \emph{normal} if it is $\sigma$-WOT continuous.  An \emph{isomorphism of tracial $\mathrm{C}^*$-algebras} is a $*$-isomorphism that preserves the trace; we make the same definition for tracial $\mathrm{W}^*$-algebras but with the added requirement that the map and its inverse are normal.
\end{definition}

\begin{lemma}[Properties of $*$-homomorphisms]
	Any $*$-homomorphism of $\mathrm{C}^*$-algebras is contractive and any injective $*$-homomorphism is isometric.
\end{lemma}

%The following result is immediate from well-known facts about $\mathrm{C}^*$ and $\mathrm{W}^*$-algebras (such as the Gelfand-Naimark-Segal construction).

%\begin{lemma} \label{lem:Wstarcompletion}
%If $(\cA,\tau)$ is a tracial $\mathrm{C}^*$-algebra, then there is a trace-preserving injective $*$-homomorphism of $(\cA,\tau)$ into a tracial $\mathrm{W}^*$-algera.
%\end{lemma}

For any tracial $\mathrm{C}^*$-algebra, there is a non-commutative analog of the $L^\alpha$ spaces for $\alpha \in [0,\infty]$ (we use $\alpha$ rather than $p$ to reserve the letter $p$ for polynomials), and they satisfy the non-commutative H\"older's inequality.

\begin{definition}[Non-commutative $L^\alpha$ norms] \label{def:NCLp}
	Let $(\cA,\tau)$ be a tracial $\mathrm{C}^*$-algebra.  For $\alpha \in [0,\infty]$ and $\mathbf{X} \in \cA^d$, we write
	\[
	\norm{\mathbf{X}}_\alpha = \begin{cases} \left( \sum_{j=1}^d \tau((X_j^*X_j)^{\alpha/2}) \right)^{1/\alpha}, & \alpha < \infty \\ \max_j \norm{X_j}_\infty, & \alpha = \infty. \end{cases}
	\]
	Here $(X_j^*X_j)^{\alpha/2}$ is defined by functional calculus.
\end{definition}

\begin{lemma}[Non-commutative H\"older's inequality] \label{lem:NCHolder}
	Let $\alpha$, $\alpha_1$, \dots, $\alpha_n \in [0,\infty]$ with $1/\alpha = \sum_{j=1}^n 1 / \alpha_j$.  Let $(\cA,\tau)$ be a tracial $\mathrm{C}^*$-algebra and let $a_1$, \dots, $a_n \in \cA$.  Then
	\[
	\norm{a_1 \dots a_n}_\alpha \leq \norm{a_1}_{\alpha_1} \dots \norm{a_n}_{\alpha_n}.
	\]
	Also, we have $\lim_{\alpha \to \infty} \norm{\mathbf{X}}_\alpha = \norm{\mathbf{X}}_\infty$ for $\mathbf{X} \in \cA^d$.
\end{lemma}

Modulo renormalization of the trace, the inequality for matrices follows from the treatment of trace-class operators in \cite{Simon2005}; see especially Thm.\ 1.15 and Thm.\ 2.8, as well as the references cited on p.\ 31.  The von Neumann algebraic setting was studied by Dixmier \cite{Dixmier1953}, and a convenient proof can be found in \cite[Thm.\ 2.4 - 2.6]{daSilva2018}; for an overview and further history see \cite[\S 2]{PX2003}.

\begin{definition}[Conditional expectation]
	Let $\cA$ be a $\mathrm{C}^*$-algebra and $\cB$ a unital $\mathrm{C}^*$-subalgebra.  A \emph{conditional expectation} $E: \cA \to \cB$ is a linear map such that
	\begin{enumerate}[(1)]
		\item $E$ is \emph{positive}, that is, it maps any operator of the form $a^*a \in \cA$ to an operator of the form $b^*b \in \cB$.
		\item $E$ is a $\cB$-$\cB$-bimodule map, that is, $E[b_1ab_2] = b_1 E[a] b_2$ for $a \in \cA$ and $b_1$, $b_2 \in \cB$.
		\item $E|_{\cB} = \id$.
	\end{enumerate}
\end{definition}

The following result about tracial $\mathrm{W}^*$-algebras is well-known.

\begin{lemma}[Conditional expectations for tracial $\mathrm{W}^*$-algebras]
	Let $(\cA,\tau)$ be a tracial $\mathrm{W}^*$-algebra and let $\cB$ be a $\mathrm{W}^*$-subalgebra.  Then there exists a unique trace-preserving conditional expectation $E: \cA \to \cB$, and this $E$ is $\sigma$-WOT continuous.  For each $a \in \cA$, the conditional expectation $E[a]$ is characterized by the condition that $\tau(E[a] b) = \tau(ab)$ for all $b \in \mathcal{B}$.  Moreover, $\norm{E[\mathbf{X}]}_\alpha \leq \norm{\mathbf{X}}_\alpha$ for any $\mathbf{X} \in \cA^d$ and $\alpha \in [1,\infty]$.
\end{lemma}

Next, we describe the space of non-commutative laws.  A non-commutative law is the analog of a linear functional $\C[x_1,\dots,x_d] \to \R$ given by $f \mapsto \int f\,d\mu$ for some compactly supported measure on $\R^d$.  Instead of $\C[x_1,\dots,x_d]$, we use the non-commutative polynomial algebra in $d$ variables.

\begin{definition}[Non-commutative polynomial algebra] \label{def:NCpolynomial}
	We denote by $\C\ip{x_1,\dots,x_d}$ the universal unital algebra generated by variables $x_1$, \dots, $x_d$.  As a vector space, $\C\ip{x_1,\dots,x_d}$ has a basis consisting of all products $x_{i_1} \dots x_{i_\ell}$ for $\ell \geq 0$ and $i_1$, \dots, $i_\ell \in \{1,\dots,d\}$.  We equip $\C\ip{x_1,\dots,x_d}$ with the unique $*$-operation such that $x_j^* = x_j$.
\end{definition}

\begin{definition}[Non-commutative law]
	A linear functional $\lambda: \C\ip{x_1,\dots,x_d} \to \C$ is said to be \emph{exponentially bounded} if there exists $R > 0$ such that $|\lambda(x_{i_1} \dots x_{i_\ell})| \leq R^\ell$ for all $\ell \in \N_0$ and $i_1$, \dots, $i_\ell \in \{1,\dots,d\}$, and in this case we say $R$ is an \emph{exponential bound} for $\lambda$.  A \emph{non-commutative law} is a unital, positive, tracial, exponentially bounded linear functional $\lambda: \C\ip{x_1,\dots,x_d} \to \C$. We denote the space of non-commutative laws by $\Sigma_d$, and we equip it with the weak-$\star$ topology (that is, the topology of pointwise convergence on $\C\ip{x_1,\dots,x_d}$).  We denote by $\Sigma_{d,R}$ the subset of $\Sigma_d$ comprised of non-commutative laws with exponential bound $R$.
\end{definition}

\begin{observation}
	The space $\Sigma_{d,R}$ is compact and metrizable.
\end{observation}

\begin{observation}
	Let $\cA$ be a $*$-algebra and $\mathbf{X} = (X_1, \dots, X_d) \in \cA_{\sa}^d$.  Then there is a unique $*$-homomorphism $\rho_{\mathbf{X}}: \C\ip{x_1,\dots,x_d} \to \cA$ such that $\rho_{\mathbf{X}}(x_j) = X_j$ for $j = 1$, \dots, $d$.
\end{observation}

\begin{definition}[Non-commutative law of a $d$-tuple]
	Let $(\cA,\tau)$ be a tracial $\mathrm{C}^*$-algebra.  Let $\mathbf{X} = (X_1,\dots,X_d) \in \cA_{\sa}^d$.  Then we define $\lambda_{\mathbf{X}}: \C\ip{x_1,\dots,x_d} \to \C$ by $\lambda_{\mathbf{X}} = \tau \circ \rho_{\mathbf{X}}$.
\end{definition}

\begin{observation}
	If $(\cA,\tau)$ and $\mathbf{X}$ are as above, then $\lambda_{\mathbf{X}}$ is a non-commutative law with exponential bound $\norm{\mathbf{X}}_\infty$.  Conversely, if $R$ is an exponential bound for $\lambda_{\mathbf{X}}$, then
	\[
	\norm{\mathbf{X}}_\infty = \max_j \lim_{n \to \infty} [\sum_j \tau(X_j^{2n})]^{1/2n} \leq R.
	\]
	Hence, $\norm{\mathbf{X}}_\infty$ is the smallest exponential bound for $\lambda_{\mathbf{X}}$ and in particular it is uniquely determined by $\lambda_{\mathbf{X}}$.
\end{observation}

In the case of a single operator $X$, we can apply the spectral theorem to show that there is a unique probability measure $\mu_X$ on $\R$ satisfying
\[
\int_{\R} f\,d\mu_X = \tau(f(X)) \text{ for } f \in C_0(\R).
\]
Since $X$ is bounded, $\mu_X$ is compactly supported and thus makes sense to evaluate on polynomials.  If $p$ is a polynomial, then $\lambda_X[p] = \int_{\R} p\,d\mu_X$.  Thus, $\lambda_X$ is simply the linear functional on polynomials corresponding to the spectral distribution.

We use the notation $\lambda_{\mathbf{X}}$ in particular when $\cA = M_N(\C)$.  We denote by $\tr_N$ the normalized trace $(1/N) \Tr$ on $M_N(\C)$; recall that this is the unique (unital) trace on $M_N(\C)$.  Thus, for any $\mathbf{X} \in M_N(\C)_{\sa}^d$, a non-commutative law $\lambda_{\mathbf{X}}$ is unambiguously specified by the previous definition.  In the $d = 1$ case, the non-commutative law is given by the empirical spectral distribution.  Note that when $\mathrm{X}$ is a random $d$-tuple of matrices, we will use the notation $\lambda_{\mathbf{X}}$ by default to refer to the empirical non-commutative law, that is, the (random) non-commutative law of $\mathbf{X}$ with respect to $\tr_N$.

The next proposition shows that any non-commutative law can be realized by a self-adjoint $d$-tuple in  some tracial $\mathrm{C}^*$ or $\mathrm{W}^*$-algebra.  This is a version of the \emph{Gelfand-Naimark-Segal construction} (or \emph{GNS construction}).  A proof can be found in \cite[Proposition 5.2.14(d)]{AGZ2009}.

\begin{proposition}[GNS construction for non-commutative laws] \label{prop:GNS}
	Let $\lambda \in \Sigma_{d,R}$.  Then we may define a semi-inner product on $\C\ip{x_1,\dots,x_d}$ by
	\[
	\ip{p,q}_\lambda = \lambda(p^*q).
	\]
	Let $H_\lambda$ be the separation-completion of $\C\ip{x_1,\dots,x_d}$ with respect to this inner product, that is, the completion of $\C\ip{x_1,\dots,x_d} / \{p: \lambda(p^*p) = 0\}$, and let $[p]$ denote the equivalence class of a polynomial $p$ in $H_\lambda$.
	
	There is a unique unital $*$-homomorphism $\pi: \C\ip{x_1,\dots,x_d} \to B(H_\lambda)$ satisfying $\rho(p)[q] = [pq]$ for $p$, $q \in \C\ip{x_1,\dots,x_d}$.  Moreover, $\norm{\pi(x_j)} \leq R$.
	
	Let $X_j = \pi(x_j)$, let $\mathbf{X} = (X_1,\dots,X_d)$ and let $\mathrm{C}^*(\mathbf{X})$ and $\mathrm{W}^*(\mathbf{X})$ denote respectively the $\mathrm{C}^*$ and $\mathrm{W}^*$-algebras generated by the image of $\pi$.  Define $\tau: \mathrm{W}^*(\mathbf{X}) \to \C$ by $\tau(T) = \ip{[1],T[1]}_\lambda$.  Then $\tau$ is a faithful normal trace on $\mathrm{W}^*(\mathbf{X})$ and in particular a faithful trace on $\mathrm{C}^*(\mathbf{X})$.
\end{proposition}

\begin{definition}
	In the situation of the previous proposition, we call $\mathrm{C}^*(\mathbf{X})$ and $\mathrm{W}^*(\mathbf{X})$, the \emph{$\mathrm{C}^*$ and $\mathrm{W}^*$-algebras associated to $\lambda$}.
\end{definition}

The operator algebras associated to $\lambda$ are canonical in the sense that any other construction would yield an isomorphic $\mathrm{W}^*$ or $\mathrm{C}^*$-algebra.  The following lemma can be deduced from the well-known properties of the GNS representation associated to a faithful trace $\tau$ on a $\mathrm{C}^*$ or $\mathrm{W}^*$-algebra $\cA$ (which gives the so-called standard form of a tracial $\mathrm{W}^*$-algebra).

\begin{lemma} \label{lem:lawisomorphism}
	Let $(\cA,\tau)$ and $(\cB,\sigma)$ be tracial $\mathrm{C}^*$-algebras.  Let $\mathbf{X} \in \cA_{\sa}^d$ and $\mathbf{Y} \in \cB_{\sa}^d$ such that $\lambda_{\mathbf{X}} = \lambda_{\mathbf{Y}}$.  Let $\mathrm{C}^*(\mathbf{X})$ and $\mathrm{C}^*(\mathbf{Y})$ be the $\mathrm{C}^*$-subalgebras of $\cA$ and $\cB$ generated by $\mathbf{X}$ and $\mathrm{Y}$ respectively.  Then there is a unique tracial  $\mathrm{C}^*$-isomorphism $\rho: \mathrm{C}^*(\mathbf{X}) \to \mathrm{C}^*(\mathbf{Y})$ such that $\rho(X_j) = Y_j$.  The same result holds with tracial $\mathrm{W}^*$-algebras rather than tracial $\mathrm{C}^*$-algebras.
\end{lemma}

Next, we review Voiculescu's definition of free independence \cite{Voiculescu1985,Voiculescu1986}, which provides a probabilistic viewpoint on classical notion of free products of tracial $\mathrm{W}^*$-algebras.  For background material, see e.g.\ \cite{VDN1992,NS2006,MS2017}.

\begin{definition}[Free independence]
	Let $\cA$ be a $*$-algebra and $\tau: \cA \to \C$ a trace.  Then unital $*$-subalgebras $(\cA_i)_{i \in I}$ are said to be \emph{freely independent} if $\tau(a_1 \dots a_\ell) = 0$ whenever $a_1 \in \cA_{i_1}$, \dots, $a_\ell \in \cA_{i_\ell}$ such that $\tau(a_j) = 0$ and $i_1 \neq i_2 \neq \dots \neq i_\ell$.  Similarly, if $I$ is an index set and $\mathbf{X}_i$ is a $d_i$-tuple of operators in $\cA$ for each $i \in I$, we say that $(\mathbf{X}_i)_{i \in I}$ \emph{freely independent} if the $*$-algebras $\cA_i$ generated by $\mathbf{X}_i$ are freely independent.
\end{definition}

\begin{lemma}[Free independence determines joint moments]
	Let $(\cA,\tau)$ be a $*$-algebra with a trace.  Suppose that $\mathbf{X}_i = (X_{i,d_1},\dots,X_{i,d_i})$ is a $d_i$-tuple of self-adjoint operators for each $i$ in some index set $I$, such that $(\mathbf{X}_i)_{i \in I}$ are freely independent.  Then for any non-commutative polynomial $p$ in $(\mathbf{X}_i)_{i\in I}$, the trace $\tau(p((\mathbf{X}_i)_{i \in I}))$ is uniquely determined from the traces $\tau(q(\mathbf{X}_i))$ for $q \in \C\ip{x_1,\dots,x_{d_i}}$ and $i \in I$.  In fact, there is a universal formula for $\tau(p((\mathbf{X}_i)_{i \in I}))$ using sums and products of the traces $\tau(q(\mathbf{X}_i))$ that does not depend on the particular $\cA$ and $\tau$.  In particular, (if $I$ is finite) the non-commutative law of $(\mathbf{X}_i)_{i \in I}$ is uniquely determined by $(\lambda_{\mathbf{X}_i})_{i \in I}$.
\end{lemma}

For proof, see \cite[Proposition 2.5.5]{VDN1992}.

\begin{lemma}[Free conditional expectations] \label{lem:freeCE}
	Let $\mathbf{X} \in \cA_{\sa}^d$ and $\mathbf{Y} \in \cA_{\sa}^{d'}$ be freely independent in $(\cA,\tau)$.  Let $E_{\mathrm{W}^*(\mathbf{X})}: \cA \to \mathrm{W}^*(\mathbf{X})$ be the unique trace-preserving conditional expectation.  If $p(\mathbf{X},\mathbf{Y})$ is a non-commutative polynomial of $\mathbf{X}$ and $\mathbf{Y}$, then $E_{\mathrm{W}^*(\mathbf{X})}[p(\mathbf{X},\mathbf{Y})]$ is a non-commutative polynomial of $\mathbf{X}$.  Furthermore, the coefficients are given by a universal formula in terms of sums and products of traces of non-commutative polynomials in $\mathbf{X}$ and traces of non-commutative polynomials in $\mathbf{Y}$.
\end{lemma}

See \cite[\S 2.5, Theorem 19]{MS2017} or \cite[proof of Lemma 2.1]{DY2020}; it can also be proved from the argument used much earlier in \cite[proof of Proposition 3.2]{Biane1997}.

\begin{lemma}[Free products]
	Let $(\cA_1,\tau_1)$, \dots, $(\cA_n,\tau_n)$ be tracial $\mathrm{W}^*$-algebras.  Then there exists a tracial $\mathrm{W}^*$-algebra
	\[
	(\cA,\tau) = (\cA_1 * \dots * \cA_n, \tau_1 * \dots * \tau_n)
	\]
	with canonical trace-preserving inclusions $\iota_j: (\cA_j, \tau_j) \to (\cA,\tau)$ such that $\cA$ is the $\mathrm{W}^*$-algebra generated by the images $\iota_1(\cA_1)$, \dots, $\iota_n(\cA_n)$ and these images are freely independent.  The free product is commutative and associative up to a canonical isomorphism.
\end{lemma}

For proof, refer to \cite[Propositions 1.5.5 and 2.5.3]{VDN1992} or \cite[Lectures 6-7]{NS2006}.

\begin{definition}[Standard semicircular family]
	A $d$-tuple $\mathbf{S} = (S_1,\dots,S_d)$ from $(\cA,\tau)$ is said to be a \emph{standard semicircular family} if $S_1$, \dots, $S_d$ are freely independent and the spectral measure of each $S_j$ with respect to $\tau$ is $(1/2\pi) \sqrt{4 - t^2} \mathbf{1}_{[-2,2]}(t)\,dt$.
\end{definition}

\begin{lemma}[Free Brownian motion] \label{lem:Brownian}
	There exists a tracial $\mathrm{W}^*$-algebra $(\cB,\sigma)$ and self-adjoint $d$-tuples $(\mathcal{S}(t))_{t \in [0,\infty)}$ from $\cB$ such that
	\begin{enumerate}[(1)]
		\item $\mathcal{S}(0) = 0$;
		\item $(\mathcal{S}(t_1) - \mathcal{S}(t_0)) / (t_1 - t_0)^{1/2}$ is a standard semicircular family for each $t_0 < t_1$;
		\item $\mathcal{S}(t_1) - \mathcal{S}(t_0)$, \dots, $\mathcal{S}(t_m) - \mathcal{S}(t_{m-1})$ are freely independent whenever $t_0 < t_1 < \dots < t_m$;
		\item $(\cB,\sigma)$ is generated as a $\mathrm{W}^*$-algebra by $(\mathcal{S}(t))_{t \in [0,\infty)}$.
	\end{enumerate}
	Moreover, $(\cB,\sigma)$ and $(\mathcal{S}(t))_{t \in [0,\infty)}$ are unique up to a $\mathrm{W}^*$-isomorphism that preserves the generators.  We call $\mathcal{S}(t)$ a \emph{$d$-variable free Brownian motion}.
\end{lemma}

For proof, refer to \cite[\S 5]{Speicher1990} or \cite[\S 2.6]{VDN1992}.

\subsection{The classical Wasserstein manifold and log-density coordinates} \label{subsec:logdensity}

To motivate our construction of the free Wasserstein manifold, we briefly review the classical Wasserstein manifold and discuss an alternate coordinate system based on minus the log-density rather than the density itself, as was done to some extent in \cite{Lafferty1988} and \cite{OV2000}.  In the following, $M$ will be a Riemannian manifold of dimension $d$.  We denote by $\ip{v,w}$ the inner product of two tangent vectors $v$ and $w$ at some point $x \in M$ with respect to the Riemannian metric, by $d_M$ the geodesic distance on $M$, and by $dx$ the canonical volume form associated to the Riemannian metric.  In this discussion, we will mostly assume that $M$ is compact because it makes the analysis simpler; and for instance, the rigorous formulation of $\mathscr{P}(\R^d)$ as a Fr\'echet manifold is easiest when $M$ is compact, see e.g.\ \cite{Lafferty1988}.  However, readers who are less familiar with Riemannian geometry may focus on the case $M = \R^d$ to understand the computations.  Our non-commutative Wasserstein manifold is the analog of the case $M = \R^d$.

\begin{definition}[Wasserstein manifold]
	We define the \emph{manifold of probability densities} or \emph{Wasserstein manifold} of $M$ by
	\[
	\mathscr{P}(M) := \left\{\rho \in C^\infty(M;\R): \rho > 0, \int_M \rho\,dx = 1 \right\}.
	\]
	For each density $\rho$, the tangent space is defined by
	\[
	T_\rho \mathscr{P}(M) := \left\{\sigma \in C^\infty(M;\R): \int_M \sigma\,dx = 0 \right\}.
	\]
\end{definition}

The Riemannian metric for $\mathscr{P}(\R^d)$ is defined in terms of the elliptic differential operator $\Delta_\rho: C^\infty(M) \to C^\infty(M)$ given by
\[
\Delta_\rho f := \nabla^\dagger(\rho \nabla f) = \rho \Delta f + \ip{\nabla \rho, \nabla f}, 
\]
where $\nabla^\dagger$ denotes the divergence operator from vector fields on $M$ to smooth functions.  When $M$ is compact, $\Delta_\rho$ defines a unbounded self-adjoint operator on $L^2(dx)$ with $\Delta_\rho \leq 0$.  The kernel is the space of constant functions and its orthogonal complement in $L^2(\rho)$ is the space of functions $\sigma$ with $\int \sigma \,dx = 0$.  Thanks to the theory of elliptic PDE, there is a pseudo-inverse operator $\Delta_\rho^{-1}: C^\infty(M) \to C^\infty(M)$ satisfying $\Delta_\rho^{-1} f = g$ if and only if $\int_M g \,dx = 0$ and $\Delta_\rho g = f - \int_M f\,dx$.

\begin{definition}[Riemannian metric on $\mathscr{P}(\R^d)$] \label{eqn:riemannian whole density space}
	Let $M$ be compact.  For each $\rho \in \mathscr{P}(M)$, we define a Riemannian metric $\ip{\cdot,\cdot}_{T_\rho \mathscr{P}(M)}$ on the tangent space by
	\[
	\ip{\sigma_1,\sigma_2}_{T_\rho \mathscr{P}(M)} := \int_M \sigma_1(x) (-\Delta_\rho^{-1} \sigma_2)(x) \,dx,
	\]
	or equivalently (using integration by parts),
	\[
	\ip{\sigma_1,\sigma_2}_{T_\rho \mathscr{P}(M)} := \int_M \ip{\nabla (\Delta_{\rho}^{-1} \sigma_1),  \nabla (\Delta_\rho^{-1} \sigma_2)}\rho(x)\,dx
	\]
\end{definition}

Next, we define alternative coordinates in terms of minus the log-density, and we compute the Riemannian metric in these new coordinates. 

\begin{definition}[Log-density manifold]
	Let
	\[
	\mathscr{W}(M) := \left\{ V \in C^\infty(M,\R): \int_M e^{-V}\,dx = 1 \right\}
	\]
	and
	\[
	T_V \mathscr{W}(M) := \left\{W \in C^\infty(M,\R): \int_M W e^{-V}\,dx = 0 \right\}.
	\]
\end{definition}

\begin{lemma}[Change of coordinates between density and log-density]
	Let $M$ be compact.  There is a bijection $\mathcal{E}: \mathscr{W}(M) \to \mathscr{P}(M)$ given by $V \mapsto e^{-V}$.  The corresponding map $d\mathcal{E}_V: T_V \mathscr{W}(M) \to T_\rho \mathscr{P}(M)$ is $W \mapsto -W e^{-V}$.  Moreover, the Riemannian metric on $\mathscr{P}(M)$ corresponds to the Riemannian metric on $\mathscr{W}(M)$ given by
	\[
	\ip{W_1,W_2}_{T_V \mathscr{W}(M)} := -\int_M W_1 (L_V^{-1} W_2) e^{-V}\,dx = \int_M \ip{\nabla(L_V^{-1} W_1), \nabla(L_V^{-1} W_1)}e^{-V}\,dx,
	\]
	where
	\[
	L_V f := \Delta f - \ip{\nabla f,\nabla V}
	\]
	and $L_V^{-1}$ is the pseudo-inverse of $L_V$ given by
	\[
	L_V(L_V^{-1} f) = L_V^{-1} (L_V f) = f - \int_M f e^{-V}\,dx, \qquad L_V^{-1}(1) = 0.
	\]
\end{lemma}

\begin{proof}
	$\mathcal{E}$ defines a bijection since the inverse is given by $\rho \mapsto -\log \rho$.  A tangent vector $W \in T_V \mathscr{W}(M)$ represents the equivalence class of the path $t \mapsto V + tW$ in $\mathscr{W}(M)$.  The corresponding path in $\mathscr{P}(M)$ is $t \mapsto e^{-(V + tW)}$.  Differentiating at $t = 0$ yields $-W e^{-V}$, hence this is the corresponding element of $T_\rho \mathscr{P}(M)$.
	
	Note that
	\[
	\Delta_{e^{-V}} f = e^{-V} \Delta f - e^{-V} \ip{\nabla V,\nabla f} = e^{-V} L_V f,
	\]
	and that $e^{-V} f$ integrates to zero with respect to $dx$ if and only if $f$ integrates to zero with respect to $e^{-V}\,dx$.  Hence,
	\[
	\Delta_{e^{-V}}^{-1}(e^{-V} f) = L_V^{-1} f,
	\]
	so
	\[
	\ip{d\mathcal{E}_V(W_1), d\mathcal{E}_V(W_2)}_{T_{e^{-V}}\mathscr{P}(M)} = -\int_M e^{-V} W_1 \Delta_{e^{-V}}^{-1}[e^{-V} W_2]\,dx = -\int_M W_1 L_V^{-1}(W_2) e^{-V}\,dx.
	\]
	Using integration by parts, this is equivalent to $\int_M \ip{\nabla L_V^{-1}(W_1), \nabla L_V^{-1}(W_2)} e^{-V}\,dx$. 
\end{proof}

We point out that $L_V$ defines a self-adjoint unbounded operator on $L^2(e^{-V}\,dx)$ satisfying $L_V \leq 0$.  In fact, $L_V = -\nabla_V^* \nabla$, where
\[
\nabla_V^* \mathbf{f} := -\nabla^\dagger \mathbf{f} + \ip{\mathbf{f},\nabla V}
\]
when $\mathbf{f}$ is a vector field on $M$.  When $M$ is compact, the kernel of $L_V$ is precisely the space of constant functions.  The operator $L_V$ seems more intrinsic than $\Delta_\rho$ since it is defined directly in terms of the measure $e^{-V} \,dx$ rather than $dx$.

Smooth transport of measure, or in other words, the transport action of the diffeomorphism group of $M$ on $\mathscr{P}(M)$, is of central importance for our work.  Let $\mathscr{D}(M)$ denote the group of diffeomorphisms of the compact Riemannian manifold $M$, where the group operation is composition.  We can consider $\mathscr{D}(M)$ as an infinite-dimensional Lie group.  The corresponding Lie algebra is the algebra of smooth vector fields on $M$, which we denote by $\Vect(M)$, and the exponential map sends a vector field $\mathbf{f}$ to the diffeomorphism obtained from the flow along $\mathbf{f}$ at time $1$.  The Lie bracket for the Lie algebra of vector fields is known as the \emph{Poisson bracket}; application of the Poisson bracket to vector fields corresponds (up to varying sign conventions) to taking the commutator of the differential operators associated to those vector fields.

\begin{observation}[Transport action]
	There is a group action $\mathscr{D}(M) \curvearrowright \mathscr{P}(M)$ given by 
	\[
	(\mathbf{f},\rho) \mapsto \mathbf{f}_* \rho := (\rho \circ \mathbf{f}^{-1}) |\det d\mathbf{f}^{-1}|,
	\]
	or in other words, the push-forward of the measure $\rho \,dx$ by the function $\mathbf{f}$ is $(\mathbf{f}_* \rho)\,dx$.  The corresponding action $\mathscr{D}(M) \curvearrowright \mathscr{W}(M)$ is given by
	\[
	(\mathbf{f},V) \mapsto \mathbf{f}_* V := V \circ \mathbf{f}^{-1} - \log |\det d\mathbf{f}^{-1}|.
	\]
\end{observation}

\begin{lemma}[Differential of transport action]
	Fix $\rho \in \mathscr{P}(M)$, and consider the map $\mathscr{S}: \mathscr{D}(M) \to \mathscr{P}(M)$ given by $\rho \mapsto \mathbf{f}_* \rho$.  Then the differential satisfies
	\[
	d\mathscr{S}_{\id}: \Vect(M) \to T_\rho \mathscr{P}(M): \mathbf{h} \mapsto -\nabla^\dagger(\rho \mathbf{h}) = -\ip{\nabla \rho, \mathbf{h}} - \rho \nabla^\dagger \mathbf{h}.
	\]
	Fix $V \in \mathscr{W}(M)$, and consider the map $\mathscr{T}: \mathscr{D}(M) \to \mathscr{W}(M)$ given by $\mathbf{f} \mapsto \mathbf{f}_*V$.  Then the differential satisfies
	\[
	d\mathscr{T}_{\id}: \Vect(M) \to T_V \mathscr{W}(M): \mathbf{h} \mapsto -\nabla_V^* \mathbf{h} = \nabla^\dagger \mathbf{h} - \ip{\nabla \mathbf{h}, \nabla V}.
	\]
\end{lemma}

\begin{proof}
	Let $\mathbf{f}_t$ be a path of diffeomorphisms with $\mathbf{f}_0 = \id$ and $\dot{\mathbf{f}}_0 = \mathbf{h}$.  Then using the product rule
	\[
	\frac{d}{dt} \Bigr|_{t=0} ((\rho \circ \mathbf{f}_t^{-1}) |\det d\mathbf{f}_t^{-1}|) = -\ip{\nabla \rho, \mathbf{h}} - \Tr(d \mathbf{h}) = -\nabla^\dagger(\rho \mathbf{h})
	\]
	and
	\[
	\frac{d}{dt} \Bigr|_{t=0} (V \circ \mathbf{f}_t^{-1} - \log |\det d \mathbf{f}_t^{-1}| ) = -\ip{\nabla V, \mathbf{h}} + \Tr(d\mathbf{h}) = -\nabla_V^* \mathbf{h}.  \qedhere
	\]
\end{proof}

If $M$ is compact, then the action of $\mathscr{D}(M)$ on $\mathscr{P}(M)$ is transitive \cite{EM1970}.  Moreover, if we fix some $\rho$, then the map $\mathbf{f} \mapsto \mathbf{f}_* \rho$ is a submersion $\mathscr{D}(M) \to \mathscr{P}(M)$, which can be used to define local coordinates on $\mathscr{P}(M)$ \cite[\S 3]{Lafferty1988}.  In hindsight, one heuristic for these results is that the map $-\nabla_V^*: \Vect(M) \to C^\infty(M)$ modulo constants has a right-inverse given by $\nabla L_V^{-1}$ since $-\nabla_V^* \nabla L_V^{-1} f = f - \int f e^{-V}\,dx$.  Thus, $\nabla L_V^{-1}$ transforms a change in $V$ into an infinitesimal transport map.  We shall use this idea to construct families of transport maps along paths in the free Wasserstein manifold.

The stabilizer in $\mathscr{D}(M)$ of some $V \in \mathscr{W}(M)$ is the group $\mathscr{D}(M,V)$ of diffeomorphisms that preserve the measure $e^{-V}\,dx$. If $\mathbf{h} \in \Vect(M)$, then $\exp(t\mathbf{h})$ preserves $V$ for all $t$ if and only if $\nabla_V^* \mathbf{h} = 0$.  Hence, Lie algebra for the stabilizer consists of \emph{divergence-free vector fields} with respect to $V$, which is the orthogonal complement in $L^2(e^{-V}\,dx)$ of the space of gradients.  For each $V$, we can define an inner product on vector fields by integrating the Riemannian metric of $M$ with respect to the measure $e^{-V}\,dx$, and this can be extended to a right-invariant Riemannian metric on the diffeomorphism group.  Geodesic equations on $\mathscr{D}(M)$ and $\mathscr{D}(M,V)$ yield respectively the inviscid Burgers' equation and incompressible Euler's equation \cite{Arnold1966}; we formulate the non-commutative versions in \S \ref{subsec:Euler}.

Next, we turn our attention to the differentials and the gradient flow of functionals on $\mathscr{P}(M)$ or $\mathscr{W}(M)$.

\begin{definition}[Wasserstein differential and gradient]
	For a $\mathscr{F}: \mathscr{P}(M) \to \R$, we denote the differential (when defined) by
	\[
	\delta_\rho \mathscr{F}(\rho): T_\rho \mathscr{P}(M) \to \R.
	\]
	Moreover, $\grad_\rho \mathscr{F}(\rho)$ is the unique element of $T_\rho \mathscr{P}(M)$ satisfying
	\[
	\ip{\grad_\rho \mathscr{F}(\rho),\sigma}_{T_\rho \mathscr{P}(M)} = \delta_\rho \mathscr{F}(\rho)[\sigma].
	\]
	For functionals $\mathscr{F}$ on $\mathscr{W}(M)$, we make the analogous definitions of $\delta_V \mathscr{F}$ and $\grad_V \mathbf{F}$.
\end{definition}

Often, the functionals are given by integration of some function of $\rho$ over $M$, and then the gradients are computed using integration by parts.  We illustrate this technique on one of the most important functionals, the \emph{entropy functional}
\[
h(\rho) := \int -\rho \log \rho\,dx.
\]

\begin{lemma}[Wasserstein gradient of entropy]
	We have
	\[
	\grad_\rho[h(\rho)] = \Delta \rho.
	\]
	and
	\[
	\grad_V[h(e^{-V})] = L_V V.
	\]
\end{lemma}

\begin{proof}
	Consider the perturbation $\rho + t \sigma$ for some $\sigma \in T_\rho \mathscr{P}(M)$.  Note that
	\begin{align*}
		\frac{d}{dt} \Bigr|_{t=0} \int -(\rho + t \sigma) \log (\rho + t \sigma) \,dx
		&= - \int \sigma (1 + \log \rho)\,dx \\
		&= \int \Delta_\rho(-\Delta_\rho^{-1} \sigma)(1 + \log \rho)\,dx \\
		&= \int \Delta_\rho (1 + \log \rho) (-\Delta_\rho)^{-1} \sigma\,dx.
	\end{align*}
	Then note that $\Delta_\rho (1 + \log \rho) = \nabla^\dagger (\rho \nabla \log \rho) = \nabla^\dagger \nabla \rho = \Delta \rho$.
	
	Similarly, consider $W \in T_V \mathscr{W}(M)$.  Let $\mathbf{h} = \nabla L_V^{-1} W$ and let $V_t = \exp(t \mathbf{h})_* V$, so that $\dot{V}_0 = -\nabla_V^* \mathbf{h} = W$.  Then
	\begin{align*}
		\frac{d}{dt} \Bigr|_{t=0} \int e^{-V_t}V_t\,dx &= \int W(1 + V)e^{-V}\,dx \\
		&= \int L_V (L_V^{-1} W)(1 + V) e^{-V}\,dx \\
		&= \int L_V(1 + V) L_V^{-1}W e^{-V}\,dx \\
		&= \ip{L_V(1 + V), W}_{T_V \mathscr{W}(M)}.
	\end{align*}
	Hence, $\grad_V[h(e^{-V})] = L_V(1 + V) = L_V V$.  Alternatively, we can deduce this from the computation for $\mathscr{P}(M)$ and the relation that $-e^{-V} L_V V = \Delta[e^{-V}]$.
\end{proof}

Hence, as observed by Otto \cite{Otto2001}, the upward gradient flow on $\mathscr{P}(M)$ for the entropy functional is described by the heat equation $\dot{\rho} = \Delta \rho$.  The corresponding equation on $\mathscr{W}(M)$ is $\dot{V} = L_V V$.

Next, we discuss Hamiltonian flows on $\mathscr{W}(M)$ and in particular the geodesic equation.  Hamiltonian flows on a the tangent manifold $TM$ are related to the natural symplectic form $TM$ coming from the Riemannian metric on $M$.  While we could write the Hamiltonian flows either in terms of the density $\rho$ or the log-density $V$, we will focus on the log-density case since it is less standard and more relevant to our work.  It will be convenient for use to reparametrize the tangent space $T_V \mathscr{W}(M)$ using $\phi = L_V^{-1} W$ as our coordinate.  More precisely, write
\[
T_V' \mathscr{W}(M) = C^\infty(M,\R) / \R 1,
\]
where $\R 1$ is the vector space of constant functions.  The map $L_V$ sends $T_V' \mathscr{W}(M)$ onto $T_V \mathscr{W}(M)$ and the Riemannian metric on $T_V' \mathscr{W}(M)$ is the Dirichlet inner product with respect to $e^{-V}\,dx$, that is,
\[
\ip{\phi_1,\phi_2}_{T_V'\mathscr{W}(M)} = \int \ip{\nabla \phi_1, \nabla \phi_2} e^{-V}\,dx.
\]
Let $T' \mathscr{W}(M)$ be the corresponding tangent bundle
\[
T' \mathscr{W}(M) = \mathscr{W}(M) \times C^\infty(M,\R) / \R 1.
\]
We denote by $\grad_V' \mathscr{F}(V) = L_V^{-1} \grad_V \mathscr{F}(V)$ the gradient of $\mathscr{F}(V)$ expressed in these new coordinates.

\begin{definition}[Hamiltonian flow]
	Let $\mathscr{H}: T' \mathscr{W}(M) \to \R: (V,\phi) \mapsto \mathscr{H}(V,\phi)$.  We call $V$ the \emph{position variable} and $\phi$ the \emph{momentum variable}.  Then the \emph{Hamiltonian flow} associated to $\mathscr{H}$ is the pair of equations
	\[
	\left\{ \begin{aligned}
		\dot{V}_t &= L_{V_t} \grad_\phi' \mathscr{H}(V,\phi) \\
		\dot{\phi}_t &= -\grad_V' \mathscr{H}(V,\phi),
	\end{aligned} \right.
	\]
	where $t \mapsto (V_t,\phi_t)$ is a path in $T' \mathscr{W}(M)$ and $\dot{~}$ denotes the time derivative.  The $L_{V_t}$ term is included to transform $T_V' \mathscr{W}(M)$ to $T_V \mathscr{W}(M)$ and thus to interpret the tangent vector as the rate of change of $V$.
\end{definition}

\begin{lemma} \label{lem:Hamiltonian}
	Let $\mathscr{F}: \mathscr{W}(M) \to \R$.  The Hamiltonian flow associated to
	\[
	\mathscr{H}(V,\phi) := \frac{1}{2} \ip{\phi,\phi}_{T_V'\mathscr{W}(M)} + \mathscr{F}(V)
	\]
	is
	\[
	\left\{ \begin{aligned}
		\dot{V}_t &= L_{V_t} \phi \\
		\dot{\phi}_t &= -\frac{1}{2} \ip{\nabla \phi, \nabla \phi} - \grad_V' \mathscr{F}(V)
	\end{aligned} \right.
	\]
\end{lemma}

\begin{proof}
	It is clear that $\grad_\phi' \mathscr{H}(V,\phi) = \phi$.  To compute $\grad_V'[\ip{\phi,\phi}_{T_V'\mathscr{W}(M)}]$, consider $\psi \in T_V' \mathscr{W}(M)$, and the corresponding vector $L_V \psi \in T_V \mathscr{W}(M)$.  Let $t \mapsto V_t$ be some path such that $\dot{V}_0 = L_V \psi$.  Note that
	\begin{align*}
		\frac{d}{dt} \Bigr|_{t=0} \ip{\phi,\phi}_{T_{V_t}'\mathscr{W}(M)} &= \frac{d}{dt} \Bigr|_{t=0} \int_M \ip{\nabla \phi, \nabla \phi} e^{-V_t}\,dx \\
		&= \int_M \ip{\nabla \phi, \nabla \phi} (-L_V \psi) e^{-V}\,dx \\
		&= \int_M \ip{\nabla \ip{\nabla \phi, \nabla \phi}, \nabla \psi} e^{-V}\,dx \\
		&= \ip{\ip{\nabla \phi, \nabla \phi}, \psi}_{T_V' \mathscr{W}(M)}.
	\end{align*}
	With this computation in hand, we obtain
	\[
	\grad_V' \mathscr{H}(V,\phi) = \frac{1}{2} \ip{\nabla \phi, \nabla \phi} + \grad_V' \mathscr{F}(V)
	\]
	which yields the asserted equations for the Hamiltonian flow.
\end{proof}

We remark that the Wasserstein Hamiltonian flow with $\mathcal{F}(V) = 0$ is the geodesic equation on $\mathscr{W}(M)$, which is closely related to optimal transport theory; we will discuss the non-commutative version in \S \ref{subsec:geodesic}. The Hamiltonian flows for nonzero $\mathcal{F}$ often arise as Nash equilibria in mean field games (see \cite{CLOY2019,LLO2021}).

\section{Non-commutative smooth functions: definition and properties} \label{sec:NCfunc1}

\subsection{Trace polynomials}

While there is a not a universally agreed upon analog of $C^\infty$ functions of several self-adjoint operators, it has at least become clear that in the random matrix setting these functions should include trace polynomials.  Trace polynomials were first studied from an algebraic viewpoint since the give all the unitarily invariant polynomials over $n \times n$ matrices for every $n$ \cite{Razmyslov1974,Procesi1976,Leron1976,Razmyslov1987}.  Their applications to Brownian motion on matrix groups and to probability theory are evident from \cite{Rains1997,Sengupta2008,Cebron2013,DHK2013,Kemp2016,Kemp2017,DGS2016}.

Trace polynomials are functions of several self-adjoint operators obtained by mixing non-commutative polynomials with applications of the trace from the ambient von Neumann algebra. Let $\C\ip{x_1,\dots,x_d}$ be the $*$-algebra of non-commutative polynomials (Definition \ref{def:NCpolynomial}).  Any non-commutative polynomial $p$ can be evaluated on self-adjoint $d$-tuples in a tracial $\mathrm{C}^*$-algebra.  If $(\cA,\tau)$ is a tracial $\mathrm{C}^*$-algebra and $\mathbf{X} = (X_1,\dots,X_d) \in \cA_{\sa}^d$, then we write $p(\mathbf{X}) =  \rho_{\mathbf{X}}(p)$, where $\rho_{\mathbf{X}}$ is the unique $*$-homomorphism $\C\ip{x_1,\dots,x_d} \to \cA$ mapping $x_j$ to $X_j$.  Then $\mathbf{X} \mapsto p(\mathbf{X})$ defines a function $p^{\cA,\tau}: \cA_{\sa}^d \to \cA$.  Moreover, there is a function $(\tr(p))^{\cA,\tau}: \cA_{\sa}^d \to \C$ given by $\mathbf{X} \mapsto \tau(p(\mathbf{X}))$.  In fact, $(\tr(p))^{\cA,\tau}(\mathbf{X})$ depends only on the non-commutative law $\lambda_{\mathbf{X}}$ and defines a continuous function on the space of laws $\Sigma_d$ (by definition of non-commutative laws).  We obtain the algebra of \emph{scalar-valued trace polynomials} $\TrP_d^0$ by taking sums and products of functions of the form $\tr(p)$, for instance,
\[
\tr(x_1 x_2) \tr(x_3) - 3 \tr(x_2) + 5 \tr(x_3)^2 \tr(x_1^2 x_2 x_3).
\]
In fact, using the Stone-Weierstrass theorem, this algebra is dense in $C(\Sigma_{d,R})$ (see \cite[Proposition 13.6.3]{JekelThesis}).

These scalar-valued trace polynomials sit inside a larger algebra $\TrP_d$ obtained by multiplying scalar-valued trace polynomials and non-commutative polynomials, which would contain for instance
\[
\tr(x_1 x_2) x_3 + x_1 - 3 \tr(x_2)1 + 5 \tr(x_3)^2 x_1^2 x_2 x_3.
\]
The space of trace polynomials is defined algebraically as follows.

\begin{definition}
	We define $\tr(\C\ip{x_1,\dots,x_d})$ to be the vector space
	\[
	\C\ip{x_1,\dots,x_d} / \Span \{pq - qp: p, q \in \C\ip{x_1,\dots,x_d} \}.
	\]
	Then $\TrP^0(\R^{*d})$ is defined to be the symmetric tensor algebra over $\tr(\C\ip{x_1,\dots,x_d})$ modulo the relation $\tr(1) = 1$.  We also define $\TrP(x_1,\dots,x_d) = \TrP^0(x_1,\dots,x_d) \otimes \C\ip{x_1,\dots,x_d}$ $*$-algebras.
\end{definition}

	For $p \in \C\ip{x_1,\dots,x_d}$, we denote the corresponding element of $\tr(\C\ip{x_1,\dots,x_d})$ by $\tr(p)$.  Elements in the algebra $\TrP(x_1,\dots,x_d)$ will be written as linear combinations of expressions such as $\tr(p_1) \dots \tr(p_n) p_0$.  Note that $\C\ip{x_1,\dots,x_d}$ has a natural $\Z_{\geq 0}^d$-grading by the degrees in each variable.  The quotient $\tr(\C\ip{x_1,\dots,x_d})$ is defined by relations $pq - qp = 0$, and it suffices to take $p$ and $q$ monomials, so that $pq - qp$ is in a single graded component.  Therefore, $\tr(\C\ip{x_1,\dots,x_d})$ inherits the $\Z_{\geq 0}^d$-grading.  From this, we obtain a grading on the tensor algebra $\TrP^0(x_1,\dots,x_d)$ and then on $\TrP(x_1, \dots, x_d)$, which is the tensor product of $\TrP^0(x_1,\dots,x_d)$ and $\C\ip{x_1,\dots,x_d}$.  We also identify $\TrP^0(x_1,\dots,x_d)$ with the subalgebra $\TrP^0(x_1,\dots,x_d) \otimes 1$ of $\TrP(x_1,\dots,x_d)$.

Just as commutative polynomials in $d$ variables can be interpreted as functions $\R^d \to \R$, a trace polynomial $f$ defines a function $\cA_{sa}^d \to \cA$ for every tracial $\mathrm{C}^*$-algebra $(\cA,\tau)$.  This is done through evaluation maps which naturally extend the evaluation maps on $\C\ip{x_1,\dots,x_d}$.

\begin{definition}
	Let $(\cA,\tau)$ be a tracial $\mathrm{C}^*$-algebra, and let $X_1$, \dots, $X_d \in \cA$ be self-adjoint.  Then we define the evaluation map $\ev_{X_1,\dots,X_d}^{\cA,\tau}: \TrP(x_1,\dots,x_d) \to \cA$ as the unique $*$-homomorphism satisfying
	\begin{align*}
		\ev_{X_1,\dots,X_d}^{\cA,\tau}(p(x_1,\dots,x_d)) &= p(X_1,\dots,X_d) \\
		\ev_{X_1,\dots,X_d}^{\cA,\tau}(\tr(p(x_1,\dots,x_d))) &= \tau(p(X_1,\dots,X_d)) 1.
	\end{align*}
\end{definition}

To see that this is well-defined, note $\ev_{X_1,\dots,X_d}^{\cA,\tau}$ passes to well-defined linear map from the $\tr(\C\ip{x_1,\dots,x_d})$ into $\cA$ since $\tau$ is invariant under cyclic symmetry.  Using the universal property of the symmetric tensor algebra, we obtain a map $\TrP^0(X_1,\dots,X_d) \to \cA$.  Finally, we tensor this map with the well-known evaluation map $\C\ip{x_1,\dots,x_d} \to \cA$ to obtain a map $\TrP(x_1,\dots,x_d) \to \cA$.

\begin{definition}
	With $(\cA,\tau)$ a tracial $\mathrm{C}^*$-algebra and $f \in \TrP(x_1,\dots,x_d)$, we define $f^{\cA,\tau}: \cA_{\sa}^d \to \cA$ by
	\[
	f^{\cA,\tau}(X_1,\dots,X_d) = \ev_{X_1,\dots,X_d}|_{\cA,\tau}(f).
	\]
\end{definition}

Thus, a trace polynomial $f$ defines a function $\cA_{\sa}^d \to \cA$.  We next explain how to differentiate the function $f^{\cA,\tau}$, and this will motivate the construction of non-commutative $C^k$ functions.  Given $f: \cA_{\sa}^d \to \cA$ for some tracial $\mathrm{C}^*$-algebra, we define
\[
\partial_j f: \cA_{\sa}^d \times \cA_{\sa} \to \cA
\]
by
\begin{equation} \label{eq:concretederivative}
	\partial_j f(X_1,\dots,X_d)[Y] = \frac{d}{dt} \biggr|_{t=0} f(X_1,\dots,X_{j-1}, X_j + tY, X_{j+1}, \dots, X_d)
\end{equation}
whenever the limit defining the derivative exists in norm.  (Of course, this definition makes sense for maps between Banach spaces in general, and one could also consider differentiation in the weak topology.)  Similarly, for $j_1 \in \{1,\dots,d\}$, we can view $\partial_{j_1} f(X)[Y_1]$ as a function of $d+1$ variables, and then take a second directional derivative with respect to the $j_2$th variable in another direction $Y_2$.  In general, we denote the iterated directional derivatives of order $k$ by
\[
\partial_{j_k} \dots \partial_{j_1} f(X_1,\dots,X_d)[Y_1,\dots,Y_k]
\]
for $j_1$, \dots, $j_k \in \{1,\dots,d\}$ and $X_1$, \dots, $X_d$ and $Y_1$, \dots, $Y_k$ in $\cA_{\sa}$.

We claim that if $f \in \TrP(x_1,\dots,x_d)$, then the directional derivative $\partial_j (f^{\cA,\tau})(X)[Y]$ is given by $g^{\cA,\tau}(X,Y)$ for some trace polynomial $g$ that is independent of $(\cA,\tau)$.  In fact, we will describe abstract differentiation operators on the algebra $\TrP(x_1,\dots,x_d)$ such that the abstract derivatives of $f$ evaluate to the directional derivatives of $f^{\cA,\tau}$ for every $(\cA,\tau)$.  Since a trace polynomial is smooth in the sense of Fr\'echet differentiation, the $k$th directional derivatives of a function $f(X_1,\dots,X_d)$ in directions $(Y_1,\dots,Y_k)$ will be multilinear in $(Y_1,\dots,Y_k)$.  Hence, the $k$th directional derivatives ought to be given by trace polynomials in $(x_1,\dots,x_d, y_1,\dots,y_k)$ that are multilinear in $(y_1,\dots,y_k)$, which motivates the following definition.

\begin{definition}
	Let $\TrP(x_1,\dots, x_d ; y_1, \dots, y_\ell)$ be the subspace of $\TrP(x_1,\dots,x_d, y_1, \dots, y_\ell)$ consisting of trace polynomials that are linear in each $y_j$, that is, it is the sum of the graded components with grading in $\Z_{\geq 0}^d \times \{1\}^\ell$.  An element $f \in \TrP(x_1,\dots,x_d; y_1, \dots, y_\ell)$ will often be denoted $f(x_1,\dots,x_d)[y_1,\dots,y_\ell]$ rather than $f(x_1,\dots,x_d,y_1,\dots,y_\ell)$.
\end{definition}

Of course, if $f \in \TrP(x_1,\dots,x_d; y_1,\dots, y_k)$, then $f|_{\cA,\tau}$ defines a map $\cA_{\sa}^{d+k} \to \cA$ that is multilinear in the last $k$ variables.  To define the abstract derivative operators, we start with the case of first-order derivatives.

\begin{lemma}
	There is a unique linear operator
	\[
	\partial_{x_j}: \TrP(x_1,\dots,x_d) \to \TrP(x_1,\dots,x_d; y)
	\]
	satisfying
	\begin{align*}
		\partial_{x_j}(x_j)[y] &= y \\
		\partial_{x_j}(x_i)[y] &= 0 \text{ for } i \neq j \\
		\partial_{x_j}[\tr(p(x))][y] &= \tr(\partial_{x_j}[p(x)][y]) \text{ for } p \in \C\ip{x_1,\dots,x_d} \\
		\partial_{x_j}[f(x) g(x)] &= \partial_{x_j}f(x)[y] g(x) + f(x) \partial_{x_j}g(x)[y].
	\end{align*}
\end{lemma}

\begin{proof}
	First, for a monomial $p(x) = x_{j(1)} \dots x_{j(k)}$, define
	\[
	\partial_{x_j} p(x) = \sum_{i: j(i) = j} x_{j(1)} \dots x_{j(i-1)} y x_{j(i+1)} \dots x_{j(k)}.
	\]
	Since monomials are a basis for $\C\ip{x_1,\dots,x_d}$, this extends to a linear operator $\C\ip{x_1,\dots,x_d} \to \C\ip{x_1,\dots,x_d,y}$.  Then observe that if $q$ is cyclically equivalent to $p$, then $\partial_{x_j} q$ is cyclically equivalent to $\partial_{x_j} p$.  Thus, $\partial_{x_j}$ also defines a map $\tr(\C\ip{x_1,\dots,x_d}) \to \tr(\C\ip{x_1,\dots,x_d,y})$.  Recall that a basis for $\TrP(x_1,\dots,x_d)$ is given by elements of the form $\tr(p_1) \dots \tr(p_n) p_0$, where $p_1$, \dots, $p_n$ are monomials up to cyclic symmetry and $p_0$ is a monomial.  Thus, there is a unique linear operator $\TrP(x_1,\dots,x_d) \to \TrP(x_1,\dots,x_d,y)$ satisfying
	\[
	\partial_{x_j}[\tr(p_1) \dots \tr(p_n) p_0] = \sum_{i=1}^n \tr(\partial_{x_j}p_i) \prod_{i' \in [n] \setminus \{i\}} \tr(p_{i'}) p_0 + \prod_{i=1}^n \tr(p_i) \partial_{x_j} p_0.
	\]
	whenever $p_0$, \dots, $p_n$ are monomials.  We leave it as an exercise to check that this operator $\partial_{x_j}$ satisfies all the desired properties and is uniquely determined by those properties, and moreover that it maps into $\TrP(x_1,\dots,x_d;y)$.
\end{proof}

\begin{remark}
	The action of $\partial_{x_j}$ can be described in words as ``find each occurrence of $x_j$ and replace it by $y$ and then add the resulting trace polynomials.''  For instance, with $d = 2$, $j = 1$,
	\[
	\partial_{x_1}[\tr(x_1 x_2) \tr(x_2) x_1^2][y]
	= \tr(y x_2) \tr(x_1) x_1^2 + \tr(x_1 x_2) \tr(x_2) y x_1 + \tr(x_1 x_2) \tr(x_2) x_1 y.
	\]
\end{remark}

To define higher order derivatives, note that $\TrP(x_1,\dots,x_d,y_1,\dots,y_k)$ is isomorphic to $\TrP(x_1,\dots,x_{d+k})$, and hence for $j=1$,\dots,$d$, we can define
\[
\partial_{x_j}: \TrP(x_1,\dots,x_d,y_1,\dots,y_k) \to \TrP(x_1,\dots,x_d,y_1,\dots,y_k; y_{k+1}),
\]
where $y_{k+1}$ stands for the extra variable $y$ that is introduced when differentiating.  In fact, this operator maps
\[
\TrP(x_1,\dots,x_d; y_1,\dots,y_k) \to \TrP(x_1,\dots,x_d; y_1, \dots, y_{k+1}).
\]

\begin{lemma} \label{lem:TrPdirectionalderivatives}
	Let $f \in \TrP(x_1,\dots,x_d;y_1,\dots,y_\ell)$, and let $(\cA,\tau)$ be a tracial $\mathrm{C}^*$-algebra.  Then
	\begin{multline*}
		\partial_{j_k} \dots \partial_{j_1} (f^{\cA,\tau})(X_1,\dots,X_d)[Y_1,\dots,Y_{k+\ell}] \\
		= (\partial_{x_{j_k}} \dots \partial_{x_{j_1}} f)|_{\cA,\tau}(X_1,\dots,X_d)[Y_1,\dots,Y_{k+\ell}]
	\end{multline*}
	for $X_1$, \dots, $X_d$, $Y_1$, \dots, $Y_{k+\ell} \in \cA_{\sa}$.  Here the left-hand side denotes the iterated directional derivative of $f^{\cA,\tau}$ as a function on $\cA_{\sa}^d$ while the right-hand side denotes abstract differentiation operators which we introduced algebraically.
\end{lemma}

\begin{proof}
	By induction, it suffices to prove the case where $k = 1$.  Then, since a function in $\Tr(x_1,\dots,x_d; y_1,\dots,y_\ell)$ can be viewed as a function of $d+\ell$ variables, we can assume without loss of generality that $\ell = 0$ by changing $d$ if necessary.  Hence, it suffices to show that for $f \in \TrP(x_1,\dots,x_d)$,
	\[
	\partial_j (f|_{\cA,\tau})(X_1,\dots,X_k)[Y_1] = (\partial_{x_j} f])_{\cA,\tau}(X_1,\dots,X_k)[Y_1].
	\]
	The two sides of the equation agree when $f(x_1,\dots,x_k) = x_i$ for some $i$, hence they agree for non-commutative monomials using the Leibniz rule and for non-commutative polynomials by linearity.  Then because both $\partial_{x_j}$ and the directional derivative operations commute with the application of the trace, the relation also holds for $f \in \tr(\C\ip{x_1,\dots,x_d})$.  Finally, by the Leibniz rule, it extends to all of $\TrP(x_1,\dots,x_d)$.
\end{proof}

\subsection{The spaces $C_{\tr}^k(\R^{*d}, \mathscr{M}(\R^{*d_1},\dots,\R^{*d_\ell}))$}

Now we are ready to define a certain non-commutative analog of $C^k$ functions.  These are, roughly speaking, functions whose derivatives up to order $k$ can be approximated by trace polynomials.  But we must first decide what norm to use for the approximation, and there are many possible choices. Thus, we will first give some motivation for our definitions. What is most important is for the resulting function spaces to have good closure properties; for instance, closure under addition, multiplication, and more generally composition.

The first derivative of a trace polynomial $f$ in $(x_1,\dots,x_d)$ is a trace polynomial in $(x_1,\dots,x_d,y_1)$ that is linear in $y_1$.  Thus, $\partial_{x_j} f(X_1,\dots,X_d)$ defines a linear map $\cA \to \cA$ for each tracial $\mathrm{C}^*$-algebra $\cA$ and $X_1$, \dots, $X_d$ in $\cA_{\sa}$.  Obviously, it is natural to consider the norm of $\partial_{x_j} f(X_1,\dots,X_d)$ as a linear map with respect to the operator norm of $\cA$.  However, $\cA$ also has a $2$-norm with respect to the trace (Definition \ref{def:NCLp}).  The $2$-norm is important in the study of von Neumann algebras since it allows us to apply Hilbert space theory.  And the $2$-norm on $M_n(\C)$ is a rescaling of the standard Euclidean norm on $M_n(\C) \cong C^{n^2}$.  Thus, we want to take into consideration
\[
\norm{\partial_{x_j} f(X_1,\dots,X_d)}_{2;2} = \sup \{ \norm{\partial_{x_j} f(X_1,\dots,X_d)[Y]}_2: \norm{Y}_2 \leq 1\}.
\]

Higher order derivatives will be multilinear forms $\cA_{\sa}^k \to \cA$.  For instance, one term might be the multilinear form $f(x_1,x_2)[y_1,y_2,y_3] = x_1 y_2 x_2^2 x_1 y_1 y_3$.  If $X_1$, $X_2 \in \cA_{sa}$, then $f(X_1,X_2)$ will not be bounded as a map from $(\cA_{\sa}, \norm{\cdot}_2)^3 \to (\cA,\norm{\cdot}_2)$.  However, by the non-commutative H\"older's inequality (Lemma \ref{lem:NCHolder}), if $\alpha$, $\alpha_1$, $\alpha_2$, $\alpha_3 \in [1,\infty]$ satisfy $1/\alpha = 1/\alpha_1 + 1/\alpha_2 + 1/\alpha_3$, then we have
\[
\norm{X_1 Y_2 X_2^2 X_1 Y_1 Y_3}_\alpha \leq \norm{X_1}_\infty^2 \norm{X_2}_\infty^2 \norm{Y_1}_{\alpha_1} \norm{Y_2}_{\alpha_2} \norm{Y_3}_{\alpha_3},
\]
where $\norm{Y}_\alpha = \tau((Y^*Y)^{\alpha/2})^{1/\alpha}$ for $\alpha < \infty$ and $\norm{Y}_\infty$ is the operator norm.

These considerations will lead to the definition of the space $C_{\tr}^k(\R^{*d})$, which we think of as an analog of the classical space $C^k(\R^d)$.  Before explaining the formal definition, let us first discuss the notation and type of object we aim to describe.  The symbol $\R^{*d}$ does not have a literal meaning but it expresses the idea of a functions of $d$ free real (that is, self-adjoint) variables.  The derivatives of these functions will live in certain spaces of functions of self-adjoint variables which output $\ell$-multilinear forms.  Thus, for instance for $f \in C_{\tr}^k(\R^{*d})$, the total derivative $\partial^k f$ will be define for each $(\cA,\tau)$ a function of $d$-tuples $\mathbf{X}$, $\mathbf{Y}_1$, \dots, $\mathbf{Y}_\ell$ which is real-multilinear in the last $\ell$ arguments (i.e.\ an $\ell$-multilinear function of $\mathbf{Y}_1$, \dots, $\mathbf{Y}_\ell$ that depends on $\mathbf{X}$).  Here, for the sake of compact notation, we want to denote a tuple $(X_1,\dots,X_d) \in \cA_{\sa}^d$ by a single letter $\mathbf{X}$, akin to the common notation for vectors in $\R^d$.  Thus the derivative $\partial^k f$ will collect all the partial derivatives of $f$ of order $k$ (discussed in the previous section) into a single gadget.

Although in many applications the variables $\mathbf{X}$ and $\mathbf{Y}_1$, \dots, $\mathbf{Y}_\ell$ will be vectors with the same number of components, we will need each of them to have a different number of components on some occasions.  The space $C_{\tr}^k(\R^{*d}, \mathscr{M}(\R^{*d_1}, \dots, \R^{*d_\ell}))^{d'}$ will describe functions which assign, to each $(\cA,\tau)$ and each $\mathbf{X}$ in $\cA_{\sa}^d$, a multilinear form $\cA_{\sa}^{d_1} \times \dots \times \cA_{\sa}^{d_\ell} \to \cA^{d'}$.

The entries of the output vector are not restricted to be self-adjoint; thus, this is the non-commutative analog of functions from $\R^d$ to the space of $\R$-multilinear maps $\R^{d_1} \times \dots \times \R^{d_\ell} \to \C^{d'}$.  Moreover, just as every $\R$-multilinear map $\R^{d_1} \times \dots \times \R^{d_\ell} \to \C^{d'}$ extends to a unique $\C$-multilinear map $\C^{d_1} \times \dots \C^{d_\ell} \to \C^{d'}$, any $\R$-multilinear map $\cA_{\sa}^{d_1} \times \dots \times \cA_{\sa}^{d_\ell} \to \cA^{d'}$ extends uniquely to a $\C$-multilinear map $\cA^{d_1} \times \dots \times \cA^{d_\ell} \to \cA$.  We will define norms of multilinear forms using the ``complexified'' versions since they are slightly better behaved (although this only makes a difference up to a constant factor).  Now let us give the precise definitions.

\begin{definition} \label{def:Lalphamultilinearnorm}
	If $\Lambda: \cA^{d_1} \times \dots \times \cA^{d_\ell} \to \cA^{d'}$ is a $\C$-multilinear form and $\alpha$, $\alpha_1$, \dots, $\alpha_\ell \in [0,\infty]$, then we define
	\[
	\norm{\Lambda}_{\alpha;\alpha_1,\dots,\alpha_\ell} = \sup \{\norm{\Lambda[Y_1,\dots,Y_\ell]}_\alpha: Y_1 \in \cA^{d_1}, \dots, Y_\ell \in \cA^{d_\ell}, \norm{Y_1}_{\alpha_1} \leq 1, \dots \norm{Y_\ell}_{\alpha_\ell} \leq 1\}.
	\]
	We also define
	\[
	\norm{\Lambda}_{\mathscr{M}^\ell,\tr} = \sup \{ \norm{\Lambda}_{\alpha;\alpha_1,\dots,\alpha_\ell}: \alpha^{-1} = \alpha_1^{-1} + \dots + \alpha_\ell^{-1} \}.
	\]
	Note that in the case $\ell = 0$, the multilinear form reduces to an element of $\cA^{d'}$ and $\norm{\Lambda}_{\mathscr{M}^0,\tr} = \norm{\Lambda}_\infty$.
\end{definition}

\begin{observation}
	Every $\mathbf{Y} \in \cA^d$ can be written uniquely as $\re(\mathbf{Y}) + i \im(\mathbf{Y})$, where $\re(\mathbf{Y})$ and $\im(\mathbf{Y}) \in \cA_{\sa}^d$, and we have $\norm{\re(\mathbf{Y})}_\alpha, \norm{\im(\mathbf{Y})}_\alpha \leq \norm{\mathbf{Y}}_\alpha$.  Therefore, We have
	\begin{align*}
		\frac{1}{2^\ell} \norm{\Lambda}_{\alpha;\alpha_1,\dots,\alpha_\ell} &\leq \sup \{\norm{\Lambda[Y_1,\dots,Y_\ell]}_\alpha: Y_1 \in \cA_{\sa}^{d_1}, \dots, Y_\ell \in \cA_{\sa}^{d_\ell}, \norm{Y_1}_{\alpha_1} \leq 1, \dots \norm{Y_\ell}_{\alpha_\ell} \leq 1\} \\
		&\leq \norm{\Lambda}_{\alpha;\alpha_1,\dots,\alpha_\ell}.
	\end{align*}
\end{observation}

\begin{definition} \label{def:basicnorm}
	Suppose that $(\cA,\tau)$ is a tracial $\mathrm{C}^*$-algebra and $f: \cA_{\sa}^d \times \cA_{\sa}^{d_1} \dots \cA_{\sa}^{d_\ell} \to \cA^{d'}$ is a function that is real-multilinear in the last $\ell$ arguments.  Then we define
	\[
	\norm{f}_{\mathscr{M}^\ell,\tr,R} = \sup \{\norm{f(\mathbf{X})}_{\mathscr{M}^\ell,\tr}: \mathbf{X} \in \cA_{\sa}^d, \norm{\mathbf{X}}_\infty \leq R \}.
	\]
	In the case $\ell = 0$, we write it simply as $\norm{\mathbf{f}}_{\tr,R}$.
\end{definition}

The seminorm of a function $f$ in $C_{\tr}^k(\R^{*d},\mathscr{M}(\R^{*d_1},\dots,\R^{*d_\ell}))^{d'}$ with radius $R$ will be defined below essentially as the supremum of $\norm{f^{\cA,\tau}}_{\mathscr{M}^\ell,\tr,R}$ over tracial $\mathrm{C}^*$-algebras $(\cA,\tau)$, but there is a small technical issue that the classes of tracial $\mathrm{C}^*$-algebras and of tracial $\mathrm{W}^*$-algebras are not sets.  However, this issue is easily resolved as follows (for a moment, we assume a greater background knowledge about operator algebras):  There does exist a set $\mathbb{W}$ of isomorphism class representatives for tracial $\mathrm{W}^*$-algebras that are separable in $\sigma$-WOT.  This is because a separable tracial $\mathrm{W}^*$-algebra with a choice of a countable set of self-adjoint generators is equivalent to a non-commutative law in countably many variables, that is, unital, positive, tracial, exponentially bounded linear maps $\C\ip{x_j: j \in \N} \to \C$.  These linear functionals evidently form a set.  Isomorphism between the $\mathrm{W}^*$-algebras defines an equivalence relation on the space of laws, hence we can define $\mathbb{W}$ as the set of equivalence classes.  Of course, if we take the supremum over separable tracial $\mathrm{W}^*$-algebras, the supremum is the same as if we used all tracial $\mathrm{W}^*$-algebras since
\[
\norm{f(\mathbf{X})[\mathbf{Y}_1,\dots,\mathbf{Y}_\ell]}_\alpha
\]
can be evaluated only using the $\sigma$-WOT-separable subalgebra $\mathrm{W}^*(\mathbf{X};\mathbf{Y}_1,\dots,\mathbf{Y}_\ell)$ and its trace.  Moreover, it is the same as the supremum over all tracial $\mathrm{C}^*$-algebras, since any tracial $\mathrm{C}^*$-algebra can be completed to a tracial $\mathrm{W}^*$-algebra through the Gelfand-Naimark-Segal construction.

\begin{definition}
	We denote by $\TrP(\R^{*d},\mathscr{M}(\R^{*d_1},\dots,\R^{*d_\ell}))^{d'}$ vector space of $d'$-tuples $\mathbf{g}$ of trace polynomials in the indeterminates or formal variables
	\[
	\mathbf{x} = (x_1,\dots,x_d), \quad \mathbf{y}_1 = (y_{1,1},\dots,y_{1,d_1}), \quad \dots, \quad \mathbf{y}_\ell = (y_{\ell,1},\dots,y_{\ell,d_\ell})
	\]
	that are multilinear in $\mathbf{y}_1$, \dots, $\mathbf{y}_\ell$ (as above).
\end{definition}

We observe that for every $\mathbf{g} \in \TrP(\R^{*d},\mathscr{M}(\R^{*d_1},\dots,\R^{*d_\ell})^{d'}$, we have
\[
\sup_{(\cA,\tau) \in \mathbb{W}} \norm{\mathbf{g}}_{\mathscr{M}^\ell,\tr,R} < \infty.
\]
To verify this, it suffices to check the case $d' = 1$.  By linearity, we reduce to the case where $g = p_0 \tr(p_1) \dots \tr(p_n)$ where $p_0$, \dots, $p_n$ are non-commutative monomials in $\mathbf{x} = (x_1, \dots, x_d)$ and $\mathbf{y}_1$, \dots, $\mathbf{y}_\ell$, such that each $y_j$ occurs exactly once in the entire expression.  When evaluating this function on $\mathbf{X}$ and $\mathbf{Y}_1 \in \cA_{\sa}^{d_1}$, \dots, $\mathbf{Y}_\ell \in \cA_{\sa}^{d_\ell}$ for some $(\cA,\tau) \in \mathbb{W}$, one estimates the result by applying the non-commutative H\"older's inequality to $\tau(p_i)$ for each $i$, using $\norm{\mathbf{Y}_j}_{\alpha_j}$ and $\norm{\mathbf{X}}_\infty$ for each occurrence of $X_i$ (and $\norm{\mathbf{X}}_\infty$ in turn is bounded by $R$).

\begin{definition}
	We define $C_{\tr}(\R^{*d},\mathscr{M}(\R^{d_1},\dots,\R^{d_\ell}))^{d'}$ as the set of tuples $(\mathbf{f}^{\cA,\tau})_{(\cA,\tau) \in \mathbb{W}}$ such that $f^{\cA,\tau}: \cA_{\sa}^d \times \cA_{\sa}^{d_1} \times \dots \times \cA_{\sa}^{d_\ell} \to \cA^{d'}$ that are real-multilinear in the last $\ell$ variables and such that for every $R > 0$ and $\epsilon > 0$, there exists a $d'$-tuple $\mathbf{g} \in \TrP(\R^{*d},\mathscr{M}(\R^{*d_1},\dots,\R^{*d_\ell})^{d'}$ such that 
	\[
	\sup_{(\cA,\tau) \in \mathbb{W}} \norm{\mathbf{f}^{\cA,\tau} - \mathbf{g}^{\cA,\tau}}_{\mathscr{M}^\ell,\tr,R} < \epsilon.
	\]
	We also define
	\[
	\norm{\mathbf{f}}_{C_{\tr}(\R^{*d},\mathscr{M}(\R^{*d_1},\dots,\R^{*d_\ell}))^{d'},R} = \sup_{(\cA,\tau) \in \mathbb{W}} \norm{\mathbf{f}^{\cA,\tau}}_{\mathscr{M}^\ell,\tr,R}.
	\]
	Because writing down $\R^{*d_1}$, \dots, $\R^{*d_\ell}$ is rather cumbersome, we will also use the shorthand
	\[
	\norm{\mathbf{f}}_{C_{\tr}(\R^{*d},\mathscr{M}^\ell)^{d'},R}
	\]
	when the dimensions $d_1$, \dots, $d_\ell$ are understood from context.  Finally, we write
	\[
	C_{\tr}(\R^{*d},\mathscr{M}^\ell(\R^{*d})) = C_{\tr}(\R^{*d},\mathscr{M}(\underbrace{\R^{*d},\dots,\R^{*d}}_{\ell})).
	\]
\end{definition}

Evidently, there is a canonical linear map
\[
\TrP(\R^{*d},\mathscr{M}(\R^{*d_1},\dots,\R^{*d_\ell}))^{d'} \to C_{\tr}(\R^{*d},\mathscr{M}(\R^{*d_1},\dots,\R^{*d_\ell}))^{d'}.
\]
In fact, this map is injective.  For any trace polynomial $f$, it makes sense to evaluate $f^{M_N(\C),\tr_N}$ on arbitrary matrix $d$-tuples (not necessarily self-adjoint), although this extended evaluation map does not respect the $*$-operation.  Let $\mathcal{E}_0$ be an orthonormal basis for $M_N(\C)_{\sa}^d$ as a real inner product space, hence also an orthonormal basis for $M_N(\C)^d$ as a complex inner product space.  For any trace polynomial $f$ and $b \in \mathcal{E}_0$, the function $g(\mathbf{X}) = \ip{b,f^{M_N(\C),\tr_N}(\mathbf{X})}_{\tr_N}$ is a complex analytic function in the coefficients $z_b = \ip{b,\mathbf{X}}$.  Hence, by analytic continuation, it is uniquely determined by the values of $g$ when $z_b \in \R$, that is, by $g$ restricted to self-adjoint $d$-tuples.  Since this is true for each basis element $b$, we see that if $f^{M_N(\C),\tr_N} = 0$ for self-adjoint $\mathbf{X}$, then it is zero for arbitrary $d$-tuple of $N \times N$ matrices.  If a trace polynomial $f$ satisfies $f^{M_N(\C),\tr_N} = 0$ for all $N$, then $f$ must equal zero by \cite[Corollary 4.4]{Procesi1976}.  Hence if $\mathbf{f}^{\cA,\tau} = \mathbf{g}^{\cA,\tau}$ for all $(\cA,\tau) \in \mathscr{W}$, then $\mathbf{f} = \mathbf{g}$ as trace polynomials, which is what we wanted to prove.  While this is not essential to any of our main results, it is notationally and conceptually convenient to treat $\TrP(\R^{*d},\mathscr{M}(\R^{*d_1},\dots,\R^{*d_\ell}))^{d'}$ as a dense subspace of $C_{\tr}(\R^{*d},\mathscr{M}(\R^{*d_1},\dots,\R^{*d_\ell}))^{d'}$.

The following observations are straightforward exercises:
\begin{itemize}
	\item $C_{\tr}(\R^{*d},\mathscr{M}(\R^{*d_1},\dots,\R^{*d_\ell}))$ is a Fr\'echet space with respect to the family of seminorms $\norm{f}_{C_{\tr}(\R^{*d},\mathscr{M}^\ell),R}$ for $R > 0$ (or for any countable set of values of $R$ which tends to $\infty$).
	\item If $\mathbf{f} \in C_{\tr}(\R^{*d},\mathscr{M}(\R^{*d_1},\dots,\R^{*d_\ell}))^{d'}$, then it makes sense to evaluate $\mathbf{f}$ on any tuple $(\mathbf{X},\mathbf{Y}_1,\dots,\mathbf{Y}_\ell) \in \cA_{\sa}^d \times \cA^{d_1} \times \dots \times \cA^{d_\ell}$ for any tracial $\mathrm{C}^*$-algebra $(\cA,\tau)$.  Indeed, we restrict to the $\mathrm{C}^*$-algebra generated by $\mathbf{X}$ and $\mathbf{Y}_1$, \dots, $\mathbf{Y}_\ell$, then complete it to a tracial $\mathrm{W}^*$-algebra.
	\item Given such an $(\cA,\tau)$ and $\mathbf{X}$, $\mathbf{Y}_1$, \dots, $\mathbf{Y}_\ell$, the evaluation $f^{\cA,\tau}(\mathbf{X})[\mathbf{Y}_1,\dots,\mathbf{Y}_\ell]$ is always a $d'$-tuple from the $\mathrm{C}^*$-algebra generated by $\mathbf{X}$, $\mathbf{Y}_1$, \dots, $\mathbf{Y}_\ell$ because $f$ can be approximated in $\norm{\cdot}_{\mathscr{M}^\ell,\tr,R}$ by trace polynomials.  Moreover, the value of $f(\mathbf{X})[\mathbf{Y}_1,\dots,\mathbf{Y}_d]$ only depends on $\tau|_{\mathrm{C}^*(\mathbf{X},\mathbf{Y}_1,\dots,\mathbf{Y}_\ell)}$.
	\item There is a unique $*$-operation on $C_{\tr}(\R^{*d},\mathscr{M}(\R^{*d_1},\dots,\R^{*d_\ell}))^{d'}$ that is continuous and extends the $*$-operation on trace polynomials.  This is given by
	\[
	(f^*)^{\cA,\tau}(\mathbf{X})[\mathbf{Y}_1,\dots,\mathbf{Y}_\ell] = (f^{\cA,\tau}(\mathbf{X})[\mathbf{Y}_1^*,\dots,\mathbf{Y}_\ell^*])^*.
	\]
	This $*$-operation is isometric with respect to each of the seminorms $\norm{\cdot}_{C_{\tr}(\R^{*d},\mathscr{M}^\ell)^{d'},R}$ for $R > 0$.
\end{itemize}

\begin{definition} \label{def:traceCk}
	For $k \in \N_0 \cup \{\infty\}$, we define $C_{\tr}^k(\R^{*d},\mathscr{M}^\ell)^{d'}$ as the set of tuples $\mathbf{f} = (\mathbf{f}^{\cA,\tau})_{(\cA,\tau) \in \mathbb{W}}$ such that for $k' \leq k$, there exists a function
	\[
	\mathbf{f}_{k'} \in C(\R^{*d},\mathscr{M}(\R^{*d_1},\dots,\R^{*d_\ell}, \underbrace{\R^{*d}, \dots, \R^{*d}}_{k'} ))^{d'}
	\]
	such that for every $(\cA,\tau) \in \mathbb{W}$, for $\mathbf{X}$, $\mathbf{Y}_1 \in \cA_{\sa}^{d_1}$, \dots, $\mathbf{Y}_\ell \in \cA_{\sa}^{d_\ell}$, and $\mathbf{Y}_{\ell+1}, \dots, \mathbf{Y}_{\ell+k'} \in \cA_{\sa}^d$, we have
	\[
	\frac{d}{dt_{k'}} \biggr|_{t_{k'} = 0} \dots \frac{d}{dt_1} \biggr|_{t_1 = 0}\mathbf{f}^{\cA,\tau}(\mathbf{X} + t_1 \mathbf{Y}_{\ell+1} + \dots + t_{k'} \mathbf{Y}_{\ell+k'})[\mathbf{Y}_1,\dots,\mathbf{Y}_\ell] = \mathbf{f}_{k'}^{\cA,\tau}(\mathbf{X})[\mathbf{Y}_1,\dots,\mathbf{Y}_{\ell+k'}].
	\]
	In other other words, for each $(\cA,\tau) \in \mathbb{W}$, each iterated directional derivative of $\mathbf{f}^{\cA,\tau}$ exists, and it agrees some function in $C_{\tr}(\R^{*d},\mathscr{M}^{\ell+k'})^{d'}$ that is independent of the choice of $(\cA,\tau)$.  For each $k' \leq k$, the function $\mathbf{f}_{k'}$ is uniquely determined, and we will denote this function by $\partial^{k'} \mathbf{f}$.
\end{definition}

The following observations are immediate:
\begin{itemize}
	\item If $\mathbf{f} = (\mathbf{f}^{\cA,\tau})_{(\cA,\tau) \in \mathbb{W}} \in C_{\tr}^k(\R^{*d},\mathscr{M}(\R^{*d_1},\dots,\R^{*d_\ell}))^{d'}$, and if $k' \leq k$, then $\partial^{k'} \mathbf{f}$ is an element of $C_{\tr}^{k-k'}(\R^{*d},\mathscr{M}^{\ell+k'})^{d'}$.
	\item Every element of $\TrP(\R^{*d},\mathscr{M}(\R^{*d_1},\dots,\R^{*d_\ell}))^{d'}$ defines an element of $C_{\tr}^\infty(\R^{*d},\mathscr{M}(\R^{*d_1},\dots,\R^{*d_\ell}))^{d'}$.
	\item $C_{\tr}^k(\R^{*d},\mathscr{M}(\R^{*d_1},\dots,\R^{*d_\ell}))^{d'}$ is a Fr\'echet space with the topology given by the seminorms
	\[
	\norm{\partial^{k'} \mathbf{f}}_{C_{\tr}(\R^{*d},\mathscr{M}^{\ell+k'})^{d'},R}
	\]
	for $R > 0$ and $k' \leq k$.
	\item If $k \leq k'$, then
	\[
	C_{\tr}^k(\R^{*d},\mathscr{M}(\R^{*d_1},\dots,\R^{*d_\ell}))^{d'} \subseteq C_{\tr}^{k'}(\R^{*d},\mathscr{M}(\R^{*d_1},\dots,\R^{*d_\ell}))^{d'},
	\]
	and the inclusion map is continuous.
	\item If $d_1 \leq d_2$, then there is a continuous inclusion
	\[
	C_{\tr}^k(\R^{*d_1},\mathscr{M}(\R^{*d_1},\dots,\R^{*d_\ell}))^{d'} \to C_{\tr}^k(\R^{*d_2},\mathscr{M}(\R^{*d_1},\dots,\R^{*d_\ell}))^{d'}
	\]
	given by sending $\mathbf{f}$ to the function $(X_1,\dots,X_{d_2}) \mapsto \mathbf{f}(X_1,\dots,X_{d_1})$.
\end{itemize}

It is often convenient to work with bounded functions so as not to worry about growth conditions at $\infty$.  Thus, we define the following $BC_{\tr}^k$ spaces.

\begin{definition}
	For $\mathbf{f} \in C_{\tr}^k(\R^{*d},\mathscr{M}(\R^{*d_1},\dots,\R^{*d_\ell}))^{d'}$, we define
	\[
	\norm{\mathbf{f}}_{BC_{\tr}(\R^{*d},\mathscr{M}^\ell)^{d'}} := \sup_R \norm{\mathbf{f}}_{C_{\tr}(\R^{*d},\mathscr{M}^\ell)^{d'},R}.
	\]
	For $k \in \N_0 \cup \{\infty\}$, we define $BC_{\tr}^k(\R^{*d},\mathscr{M}^\ell)^{d'}$ as the set of $\mathbf{f} \in C_{\tr}^k(\R^{*d},\mathscr{M}^\ell)^{d'}$ such that
	\[
	\norm{\partial^{k'} \mathbf{f}}_{BC_{\tr}(\R^{*d},\mathscr{M}^\ell)^{d'}} < \infty
	\]
	for $k' \in \N_0$ with $k' \leq k$.
\end{definition}

We equip $BC_{\tr}^k(\R^{*d},\mathscr{M}(\R^{*d_1},\dots,\R^{*d_\ell}))^{d'}$ with the topology given by these seminorms.  If $k < \infty$, there are only finitely many of these seminorms, so we have a Banach space.  Note that this topology on $BC_{\tr}^k(\R^{*d},\mathscr{M}^\ell)^{d'}$ is stronger than the subspace topology from $C_{\tr}^k(\R^{*d},\mathscr{M}^\ell)^{d'}$.  Moreover, $BC_{\tr}^k(\R^{*d},\mathscr{M}^\ell)^{d'}$ is a Banach space for $k \in \N_0$ and a Fr\'echet space for $k = \infty$.

\begin{remark}
	At this point, it may not be clear whether there are any nontrivial functions $BC_{\tr}^k(\R^{*d},\mathscr{M}^\ell)^{d'}$.  However, it turns out that these functions are quite abundant.  It follows from Proposition \ref{prop:smoothfunctionalcalculus} below that if $\phi: \R \to \R$ is a function whose Fourier transform satisfies $\int_{\R} |s^n \phi(s)|\,ds < \infty$ for all $n$, then an element of $BC_{\tr}^\infty(\R)$ is defined applying $\phi$ to self-adjoint operators through functional calculus.  Furthermore, it follows Theorem \ref{thm:chainrule} below that $BC_{\tr}^\infty$ functions are closed under composition (hence also under multiplication).  Moreover, if $f \in BC_{\tr}^\infty(\R^{*d},\mathscr{M}(\R^{*d_1}, \dots, \R^{*d_\ell}))$, then so is $\tr(f)$.
\end{remark}

\subsection{Continuity and differentiability properties}

Functions in $C_{\tr}(\R^{*d},\mathscr{M}(\R^{*d_1},\dots,\R^{*d_\ell}))^{d'}$ have the following continuity property, which is a type of uniform continuity for $\mathbf{X}$ in the $\norm{\cdot}_\infty$-ball of radius $R$.

\begin{lemma} \label{lem:Ctrcontinuity}
	Let $\mathbf{f} = (\mathbf{f}^{\cA,\tau})_{(\cA,\tau) \in \mathbb{W}} \in C_{\tr}(\R^{*d},\mathscr{M}(\R^{*d_1},\dots,\R^{*d_\ell}))^{d'}$.  Then for every $R > 0$ and $\epsilon > 0$, there exists a $\delta > 0$ such that for every $(\cA,\tau) \in \mathbb{W}$, if $\mathbf{X}$ and $\mathbf{X}' \in \cA_{\sa}^d$ with $\norm{\mathbf{X}}_\infty \leq R$ and $\norm{\mathbf{X}'}_\infty\leq R$ and $\norm{\mathbf{X} - \mathbf{X}'}_\infty < \delta$ for each $i$, then $\norm{\mathbf{f}^{\cA,\tau}(\mathbf{X}) - \mathbf{f}^{\cA,\tau}(\mathbf{X}')}_{\mathscr{M}^\ell,\tr} < \epsilon$.
\end{lemma}

\begin{proof}
	First, consider the case where $\mathbf{f} \in \TrP(\R^{*d},\mathscr{M}(\R^{*d_1},\dots,\R^{*d_\ell}))^{d'}$.  Let $\mathbf{X}$ and $\mathbf{X}'$ be self-adjoint $d$-tuples from $(\cA,\tau)$ with $\norm{\mathbf{X}}_\infty \leq R$ and $\norm{\mathbf{X}'} \leq R$ and $\norm{\mathbf{X} - \mathbf{X}'}_\infty < \delta$.  Let $\alpha$, $\alpha_1$, \dots $\alpha_\ell \in [1,\infty]$ with $1/\alpha = 1/\alpha_1 + \dots + 1/\alpha_\ell$, and let $\mathbf{Y}_1 \in \cA^{d_1}$, \dots, $\mathbf{Y}_\ell \in \cA^{d_\ell}$ with $\norm{\mathbf{Y}_j}_{\alpha_j} \leq 1$.  It follows from Lemma \ref{lem:TrPdirectionalderivatives} that
	\[
	\frac{d}{dt} \mathbf{f}^{\cA,\tau}((1 - t)\mathbf{X} + t \mathbf{X}')[\mathbf{Y}_1,\dots,\mathbf{Y}_\ell]
	= (\partial \mathbf{f})^{\cA,\tau}((1 - t)\mathbf{X} + t \mathbf{X}')[\mathbf{Y}_1,\dots,\mathbf{Y}_\ell,\mathbf{X}'-\mathbf{X}].
	\]
	Since $\norm{(1 - t)\mathbf{X} + t\mathbf{X}'}_\infty \leq R$ for $t \in [0,1]$, we get
	\begin{multline*}
		\norm{(\partial f)^{\cA,\tau}((1 - t)\mathbf{X} + t \mathbf{X}')[\mathbf{Y}_1,\dots,\mathbf{Y}_\ell,\mathbf{X}'-\mathbf{X}]}_\alpha \\
		\leq \norm{\partial \mathbf{f}}_{C_{\tr}(\R^{*d},\mathscr{M}^{\ell+1})^{d'},R} \norm{\mathbf{Y}_1}_{\alpha_1} \dots \norm{\mathbf{Y}_\ell}_{\alpha_\ell} \norm{\mathbf{X}' - \mathbf{X}}_\infty \\
		\leq \norm{\partial_{x_1} \mathbf{f}}_{C_{\tr}(\R^{*d},\mathscr{M}^{\ell+1})} \delta.
	\end{multline*}
	Hence,
	\[
	\norm{\mathbf{f}^{\cA,\tau}(\mathbf{X}')[\mathbf{Y}_1,\dots,\mathbf{Y}_\ell] - \mathbf{f}^{\cA,\tau}(\mathbf{X})[\mathbf{Y}_1,\dots,\mathbf{Y}_\ell]}_\alpha
	\leq \norm{\partial \mathbf{f}}_{C_{\tr}(\R^{*d},\mathscr{M}^{\ell+1})^{d'},R} \delta.
	\]
	This implies the desired uniform continuity property for $\mathbf{f} \in \TrP(\R^{*d},\mathscr{M}(\R^{*d_1},\dots,\R^{*d_\ell}))^{d'}$.
	
	In general, if $\mathbf{f} \in C_{\tr}(\R^{*d}, \mathscr{M}^{\ell+1})^{d'}$, then there is a sequence of trace polynomials $\mathbf{f}^{(n)}$ that converge to $\mathbf{f}$ in $C_{\tr}(\R^{*d}, \mathscr{M}^{\ell+1})^{d'}$.  For a given $R > 0$, this implies that $\mathbf{f}^{(n)} \to \mathbf{f}$ with respect to $\norm{\cdot}_{C_{\tr}(\R^{*d}, \mathscr{M}^{\ell+1})^{d'},R}$.  The uniform continuity property asserted in the lemma holds for $\mathbf{f}$ by the principle that uniform continuity is preserved under uniform limits.
\end{proof}

Next, we discuss how the non-commutative derivatives defined in this paper related to the more standard notions of Fr\'echet differentiation for functions between Banach spaces.  While this discussion is of interest in its own right, it is also helpful for our proof of the chain rule in the next section, since it allows us to deduce properties of $C_{\tr}(\R^{*d},\mathscr{M}(\R^{*d_1},\dots,\R^{*d_\ell}))$ from the better known properties of Fr\'echet derivatives.

Let $\mathcal{X}$ and $\mathcal{Y}$ be Banach spaces over $\R$, and let $f: \mathcal{X} \to \mathcal{Y}$.  We say that $f$ is \emph{Fr\'echet-differentiable at $x_0 \in \mathcal{X}$} if there is a bounded linear map $T: \mathcal{X} \to \mathcal{Y}$ such that
\[
\lim_{x \to x_0} \frac{\norm{f(x) - f(x_0) - T(x - x_0)}}{\norm{x - x_0}} = 0.
\]
This $T$ is unique and is denoted $Df(x_0)$.  We say that $f$ is \emph{Fr\'echet-$C^1$} if $f$ is Fr\'echet-differentiable at every point and $x \mapsto Df(x)$ is a continuous function $\mathcal{X} \to \mathscr{L}(\mathcal{X},\mathcal{Y})$, where $\mathscr{L}(\mathcal{X},\mathcal{Y})$ is the Banach space of bounded linear transformations $\mathcal{X} \to \mathcal{Y}$.  By induction, we say that $f$ is \emph{Fr\'echet-$C^k$} if it is Fr\'echet-differentiable at every point and $Df$ is Fr\'echet-$C^{k-1}$.  We say that $f$ is \emph{Fr\'echet-$C^\infty$} if it is Fr\'echet-$C^k$ for every $k \in \N_0$.

If $f$ is Fr\'echet-$C^k$, then the $k$th-order Fr\'echet derivatives $D^k f$ are multilinear maps $\mathcal{X}^k \to \mathcal{Y}$ defined as follows.  For $k = 2$, note that $D(Df)(x)$ is an element of $\mathscr{L}(\mathcal{X},\mathscr{L}(\mathcal{X},\mathcal{Y}))$.  But a linear map from $\mathcal{X}$ to $\mathscr{L}(\mathcal{X},\mathcal{Y})$ is equivalent to a bilinear map $\mathcal{X} \times \mathcal{X} \to \mathcal{Y}$.  The operator norm on $\mathscr{L}(\mathcal{X},\mathscr{L}(\mathcal{X},\mathcal{Y}))$ agrees with the norm on bilinear forms given by
\[
\norm{\Lambda} = \sup \{\norm{\Lambda[x_1,x_2]}: \norm{x_1}, \norm{x_2} \leq 1\}.
\]
In a similar way, let $\mathscr{M}^k(\mathcal{X},\mathcal{Y})$ be the space of $k$-linear forms $\mathcal{X}^k \to \mathcal{Y}$.  Then the $k$-fold application of $D$ to a Fr\'echet-$C^k$ function $f$ produces a function $D^k f$ from $\mathcal{X}$ to $\mathscr{M}^k(\mathcal{X},\mathcal{Y})$.

The spaces $C_{\tr}^k(\R^{*d},\mathscr{M}(\R^{*d_1},\dots,\R^{*d_\ell}))^{d'}$ can be described alternatively as follows.

\begin{lemma}
	Let $\mathbf{f} = (\mathbf{f}^{\cA,\tau})_{(\cA,\tau) \in \mathbb{W}}$ be a tuple of functions $\cA_{\sa}^d \times \cA_{\sa}^{d_1} \times \dots \times \cA_{\sa}^{d_\ell} \to \cA^{d'}$ that is multilinear in the last $\ell$ variables.  Then $f \in C_{\tr}^k(\R^d,\mathscr{M}(\R^{*d_1},\dots,\R^{*d_\ell}))^{d'}$ if and only if the following hold:
	\begin{enumerate}[(1)]
		\item For each $(\cA,\tau)$, $f^{\cA,\tau}$ is a Fr\'echet-$C^k$ function $\cA_{\sa}^d \to \mathscr{M}^\ell(\cA_{\sa},\cA^{d'})$, where $\cA_{\sa}^d$ and $\cA^{d
		}$ are viewed as Banach spaces with respect to $\norm{\cdot}_\infty$.
		\item For $k' \leq k$, there exists \[
		\mathbf{f}_{k'} \in C_{\tr}(\R^{*d},\mathscr{M}(\R^{*d_1},\dots,\R^{*d_\ell} \times \underbrace{\R^{*d} \times \dots \times \R^{*d}}_{k'}))^{d'}
		\]
		such that for all $(\cA,\tau) \in \mathbb{W}$,
		\[
		D^{k'} (\mathbf{f}^{\cA,\tau}) =  \mathbf{f}_{k'}^{\cA,\tau}.
		\]
	\end{enumerate}
\end{lemma}

\begin{proof}
	Suppose that $\mathbf{f} \in C_{\tr}^k(\R^{*d},\mathscr{M}^\ell)^{d'}$.  By Definition \ref{def:traceCk}, this means that all the iterated directional derivatives up of order $k' \leq k$ exist and are given by functions $\mathbf{f}_{k'}$ in $C_{\tr}(\R^{*d},\mathscr{M}^{\ell+k'})^{d'}$.  Now observe that for each $(\cA,\tau)$, the function $f_{k'}$ defines a continuous map from $\cA_{\sa}^d$ to the space of multilinear forms
	\[
	\cA_{\sa}^{d_1} \times \dots \times \cA_{\sa}^{d_\ell} \times \cA_{\sa}^d \times \dots \times \cA_{\sa}^d \to \cA^{d'}
	\]
	endowed with $\norm{\cdot}_{\infty;\infty,\dots,\infty}$.  This follows from Lemma \ref{lem:Ctrcontinuity} because for a multilinear form $\Lambda: \cA_{\sa}^{\ell+k'} \to \cA^{d'
	}$, we have $\norm{\Lambda}_{\infty;\infty,\dots,\infty} \leq \norm{\Lambda}_{\mathscr{M}^\ell,\tr}$.  Once we have this continuity, it is a standard argument to show that $f^{\cA,\tau}$ is Fr\'echet-$C^k$; this is a generalization of the well-known fact from multivariable calculus that if a function has continuous iterated directional derivatives up to order $k$, then it is $C^k$.
	
	The converse direction of the lemma is immediate.  Indeed, the combination of statements (1) and (2) is stronger than Definition \ref{def:traceCk} since Fr\'echet-differentiability implies the existence of directional derivatives.
\end{proof}

\begin{remark}[Equality of mixed partials]
	The equality of mixed partials generalizes to the setting of Fr\'echet differentiation:  If $f$ is a Fr\'echet-$C^k$ function, then $D^k f$ is a symmetric multilinear form, that is, it is invariant under permutation of the arguments.  For $\mathbf{f} \in C_{\tr}^k(\R^{*d}, \mathscr{M}(\R^{*d_1},\dots,\R^{*d_\ell}))^{d'}$ and $\sigma$ in the symmetric group $\Perm(\ell)$, we denote by $\mathbf{f}_\sigma \in C_{\tr}^k(\R^{*d},\mathscr{M}(\R^{*d_{\sigma^{-1}(1)}}, \dots, \R^{*d_{\sigma^{-1}(\ell)}}))$ the function given by
	\[
	(\mathbf{f}_\sigma)^{\cA,\tau}(\mathbf{X})[\mathbf{Y}_1,\dots,\mathbf{Y}_\ell] = \mathbf{f}^{\cA,\tau}(\mathbf{X})[\mathbf{Y}_{\sigma^{-1}(1)},\dots,\mathbf{Y}_{\sigma^{-1}(\ell)}].
	\]
	This defines a right action of $\Perm(\ell)$ on $C_{\tr}^k(\R^{*d},\mathscr{M}^\ell)^{d'}$, and this action is isometric for each seminorm $\norm{\cdot}_{C_{\tr}(\R^{*d},\mathscr{M}^\ell),R}$.  %We denote the subspace of fixed points by $C_{\tr}^k(\R^{*d},\mathscr{S}^\ell)^{d'}$, where $\mathscr{S}^\ell$ stands for symmetric $\ell$-linear forms.
	
	Equality of mixed partials means that if $\mathbf{f} \in C_{\tr}^k(\R^{*d},\mathscr{M}(\R^{*d_1},\dots,\R^{*d_\ell}))^{d'}$, then $(\partial^k \mathbf{f})_\sigma = \partial^k \mathbf{f}$ for every permutation $\sigma$ that only affects the last $k$ elements (that is, the indices corresponding to the multilinear arguments introduced by differentiation).
\end{remark}

\begin{remark}[Lipschitz bounds] \label{rem:Lipschitz}
	Similar reasoning as in the proof of Lemma \ref{lem:Ctrcontinuity} shows the following Lipschitz-type bound:  Let $\mathbf{f} \in C_{\tr}^1(\R^{*d};\mathscr{M}(\R^{*d_1},\dots,\R^{*d_\ell}))^{d'}$.  Then for $(\cA,\tau) \in \mathbb{W}$ and $R > 0$ and $\alpha_1$, \dots $\alpha_\ell, \alpha, \beta \in [1,\infty]$ with $1/\alpha = 1/\alpha_1 + \dots + 1/\alpha_\ell + 1/\beta$, we have
	\begin{multline*}
		\norm{\mathbf{f}^{\cA,\tau}(\mathbf{X}')[\mathbf{Y}_1,\dots,\mathbf{Y}_\ell] - \mathbf{f}^{\cA,\tau}(\mathbf{X})[\mathbf{Y}_1,\dots,\mathbf{Y}_\ell]}_\alpha \\
		\leq \norm{\partial \mathbf{f}}_{C_{\tr}(\R^{*d},\mathscr{M}^{\ell+1})^{d'},R} \norm{\mathbf{X} - \mathbf{X}'}_\beta \norm{\mathbf{Y}_1}_{\alpha_1} \dots \norm{\mathbf{Y}_\ell}_{\alpha_\ell}
	\end{multline*}
	In particular, taking $\ell = 0$, we see that for every $\mathbf{f} \in C_{\tr}^1(\R^{*d})^{d'}$, for every $\alpha \in [1,\infty]$, for every $(\cA,\tau) \in \mathbb{W}$, the function $\mathbf{f}^{\cA,\tau}$ is Lipschitz with respect to $\norm{\cdot}_\alpha$ on the $\norm{\cdot}_\infty$ ball of $\cA_{\sa}^d$ radius $R$, with Lipschitz constant bounded by $\norm{\partial \mathbf{f}}_{C_{\tr}(\R^{*d_1},\mathscr{M}(\R^{*d_1}))^{d_2},R}$.
\end{remark}

\subsection{Composition}

In this section, we will discuss composition of functions in $C_{\tr}^{k,\ell}(\R^{*d})^{d'}$ and the chain rule.  The first lemma describes composition in our spaces of non-commutative continuous functions.

\begin{lemma} \label{lem:composition}
	Let $\mathbf{f} \in C_{\tr}(\R^{*d'},\mathscr{M}(\R^{*d_1},\dots,\R^{*d_n}))^{d''}$ for some $n, d' \in \N_0$ and $d''$, $d_1$, \dots, $d_n \in \N$.  Let $\mathbf{g} \in C_{\tr}(\R^{*d})_{\sa}^{d'}$ for some $d \in \N_0$.  For each $m = 1$, \dots, $n$, let $\mathbf{h}_m \in C_{\tr}(\R^{*d}, \mathscr{M}(\R^{*d_{m,1}},\dots,\R^{*d_{m,\ell_m}}))^{d_m}$ for some $\ell_m \in \N_0$ and $d_{m,1}$, \dots, $d_{m,\ell_m}$.  Let $L_m = \ell_1 + \dots + \ell_m$.  Then there exists a (unique) function
	\[
	\mathbf{f}(\mathbf{g}) \# [\mathbf{h}_1,\dots,\mathbf{h}_n] \in C_{\tr}(\R^{*d},\mathscr{M}(\R^{*d_{1,1}}, \dots, \R^{*d_{1,\ell_1}}, \dots \dots, \R^{*d_{m,1}}, \dots, \R^{*d_{n,\ell_n}}))^{d''}
	\]
	given by
	\begin{multline*}
		(\mathbf{f}(\mathbf{g}) \# [\mathbf{h}_1,\dots,\mathbf{h}_n])^{\cA,\tau}(\mathbf{X})[\mathbf{Y}_1,\dots,\mathbf{Y}_{L_n}]
		\\ :=
		\mathbf{f}^{\cA,\tau}(\mathbf{g}^{\cA,\tau}(\mathbf{X}))[\mathbf{h}_1^{\cA,\tau}(\mathbf{X})[\mathbf{Y}_1,\dots,\mathbf{Y}_{L_1}], \dots, \mathbf{h}_n^{\cA,\tau}(\mathbf{X})[\mathbf{Y}_{L_{n-1}+1},\dots,\mathbf{Y}_{L_n}]].
	\end{multline*}
	Moreover, if we fix $R > 0$ and if
	\[
	R' = \norm*{\mathbf{g}}_{C_{\tr}(\R^{*d})^{d'},R},
	\]
	then
	\begin{multline*}
		\norm{\mathbf{f}(\mathbf{g}) \# [\mathbf{h}_1,\dots,\mathbf{h}_n]}_{C_{\tr}(\R^{*d_1},\mathscr{M}^{L_n})^{d''},R'} \\
		\leq \norm{\mathbf{f}}_{C_{\tr}(\R^{*d'},\mathscr{M}^n)^{d''},R'} \norm{\mathbf{h}_1}_{C_{\tr}(\R^{*d_1},\mathscr{M}^{\ell_1}),R} \dots \norm{\mathbf{h}_n}_{C_{\tr}(\R^{d_1},\mathscr{M}^{\ell_n}),R}.
	\end{multline*}
	Moreover, the composition map
	\begin{multline*}
		C_{\tr}(\R^{*d})_{\sa}^{d'} \times \prod_{m=1}^n C_{\tr}(\R^{*d}, \mathscr{M}(\R^{*d_{m,1}},\dots,\R^{*d_{m,\ell_m}}))^{d_m} \\
		\to C_{\tr}(\R^{*d},\mathscr{M}(\R^{*d_{1,1}}, \dots, \R^{*d_{1,\ell_1}}, \dots \dots, \R^{*d_{m,1}}, \dots, \R^{*d_{n,\ell_n}}))^{d''}
	\end{multline*}
	is jointly continuous.
\end{lemma}

\begin{proof}
	Let $\mathbf{F} = \mathbf{f}(\mathbf{g})[\mathbf{h}_1,\dots,\mathbf{h}_n]$.  Fix $R$ and let $R'$ be as above.  We begin by proving the inequality that for each $(\cA,\tau)$,
	\begin{equation} \label{eq:niceHolderinequality}
		\norm{\mathbf{F}^{\cA,\tau}}_{\mathscr{M}^{L_n},\tr,R} \leq \norm{\mathbf{f}^{\cA,\tau}}_{\mathscr{M}^n,\tr,R'} \norm{\mathbf{h}_1^{\cA,\tau}}_{\mathscr{M}^{\ell_1},\tr,R} \dots \norm{\mathbf{h}_n^{\cA,\tau}}_{\mathscr{M}^{\ell_n},\tr,R}.
	\end{equation}
	Let $\alpha$, $\alpha_1$, \dots, $\alpha_{L_n} \in [1,\infty]$ such that
	\[
	\frac{1}{\alpha} = \frac{1}{\alpha_1} + \dots + \frac{1}{\alpha_{L_n}}.
	\]
	Let $\beta_1$, \dots, $\beta_n$ be given by
	\[
	\frac{1}{\beta_m} = \sum_{j=1}^{\ell_m} \frac{1}{\alpha_{L_{m-1}+j}}.
	\]
	Let $\mathbf{X} \in \cA_{\sa}^d$ with $\norm{\mathbf{X}} \leq R$.  For each $m \leq n$ and $j \leq \ell_m$, let $\mathbf{Y}_{L_{m-j}+j} \in \cA^{d_{m,j}}$ such that $\norm{\mathbf{Y}_i}_{\alpha_i} \leq 1$ for each $i = 1$, \dots $L_n$.  Note that
	\[
	\norm{\mathbf{g}^{\cA,\tau}(\mathbf{X})}_\infty \leq \norm{\mathbf{g}}_{\tr,R} \leq R'.
	\]
	Hence,
	\begin{multline*}
		\norm{\mathbf{F}^{\cA,\tau}(\mathbf{X})[\mathbf{Y}_1,\dots,\mathbf{Y}_{L_n}]}_\alpha
		\leq \norm{\mathbf{f}^{\cA,\tau}}_{\mathscr{M}^n,\tr,R'} \norm{\mathbf{h}_1^{\cA,\tau}(\mathbf{X})[\mathbf{Y}_1, \dots, Y_{L_1}]}_{\beta_1} \dots \\ \dots \norm{\mathbf{h}_n^{\cA,\tau}(\mathbf{X})[\mathbf{Y}_{L_{n-1}+1}, \dots, \mathbf{Y}_{L_n}]}_{\beta_n}.
	\end{multline*}
	Moreover, for each $m$, by the definition of $\beta_m$ and of $\norm{\mathbf{h}_m^{\cA,\tau}}_{\mathscr{M}^{\ell_m}\tr,R}$, we have
	\[
	\norm{\mathbf{h}_m^{\cA,\tau}(\mathbf{X})[\mathbf{Y}_{L_{m-1}+1}, \dots, \mathbf{Y}_{L_m}]}_{\beta_m}
	\leq \norm{\mathbf{h}_m^{\cA,\tau}}_{\mathscr{M}^{\ell_1},\tr,R} \norm{\mathbf{Y}_{L_{m-1}+1}}_{\alpha_{L_{m-1}+1}} \dots \norm{\mathbf{Y}_{L_m}}_{\alpha_{L_m}} \leq \norm{\mathbf{h}_m^{\cA,\tau}}_{\mathscr{M}^{\ell_n},\tr,R}.
	\]
	Therefore, \eqref{eq:niceHolderinequality} holds.
	
	Now let us prove that $\mathbf{F} \in C_{\tr}(\R^{*d},\mathscr{M}(\R^{*d_{1,1}}, \dots, \R^{*d_{1,\ell_1}}, \dots \dots, \R^{*d_{m,1}}, \dots, \R^{*d_{n,\ell_n}}))^{d''}$.  We proceed in several steps.
	\begin{enumerate}[(1)]
		\item Suppose that $\mathbf{f}$, $\mathbf{g}$, and the $\mathbf{h}_m$'s are all trace polynomials.  Then clearly $\mathbf{F}$ is a trace polynomial.
		\item Next, suppose that $\mathbf{f}$ and the $\mathbf{h}_m$'s are trace polynomials, while $\mathbf{g}$ is in $C_{\tr}(\R^{*d})_{\sa}^{d'}$.  Let $\mathbf{g}^{(N)} \in \TrP(\R^{*d})_{\sa}^{d'}$ such that $\mathbf{g}^{(N)} \to \mathbf{g}$ in $C_{\tr}(\R^{*d})_{\sa}^{d'}$ as $N \to \infty$.  If we fix $R > 0$, then
		\[
		R^* := \sup_N \norm{\mathbf{g}^{(N)}}_{C_{\tr}(\R^{*d})^{d'},R} < \infty.
		\]
		Applying Lemma \ref{lem:Ctrcontinuity} with the radius $R^*$, we see that
		\[
		\lim_{N \to \infty} \sup_{(\cA,\tau) \in \mathbb{W}} \norm{\mathbf{f}^{\cA,\tau}((\mathbf{g}^{(N)})^{\cA,\tau}) - \mathbf{f}^{\cA,\tau}(\mathbf{g}) }_{\mathscr{M}^n,\tr,R} = 0.
		\]
		Let $\mathbf{F}^{(N)}$ be defined analogously to $\mathbf{F}$ except using $\mathbf{g}^{(N)}$ instead of $\mathbf{g}$.  By the same argument as \eqref{eq:niceHolderinequality},
		\begin{multline*}
			\norm{(\mathbf{F}^{(N)})^{\cA,\tau} - \mathbf{F}^{\cA,\tau}}_{\mathscr{M}^{L_n},\tr,R} \\ \leq \norm{\mathbf{f}^{\cA,\tau}((\mathbf{g}^{(N)})^{\cA,\tau}) - \mathbf{f}^{\cA,\tau}(\mathbf{g}^{\cA,\tau}) }_{\mathscr{M}^n,\tr,R}  \norm{\mathbf{h}_1}_{\mathscr{M}^{\ell_1},\tr,R} \dots \norm{\mathbf{h}_n}_{\mathscr{M}^{\ell_n},\tr,R}.
		\end{multline*}
		Hence,
		\[
		\lim_{N \to \infty} \sup_{(\cA,\tau) \in \mathbb{W}} \norm{ (\mathbf{F}^{(N)})^{\cA,\tau} - \mathbf{F}^{\cA,\tau} }_{\mathscr{M}^{L_n},\tr,R} = 0,
		\]
		so that $\mathbf{F} \in C_{\tr}(\R^{*d},\mathscr{M}(\R^{*d_{1,1}}, \dots, \R^{*d_{1,\ell_1}}, \dots \dots, \R^{*d_{m,1}}, \dots, \R^{*d_{n,\ell_n}}))^{d''}$ because this space is complete with respect to the family of seminorms.
		\item Next, suppose $\mathbf{f}$ is a trace polynomial, while $\mathbf{h}_m \in C_{\tr}(\R^{*d},\mathscr{M}(\R^{*d_{m,1}},\dots,\R^{d_{m,\ell_m}}))_{\sa}^{d_m}$ and $\mathbf{g} \in C_{\tr}(\R^{*d})_{\sa}^{d'}$.  We approximate $\mathbf{h}_m$ by trace polynomials $\mathbf{h}_m^{(N)}$ as $N \to \infty$.  Then using \eqref{eq:niceHolderinequality}, we conclude that the function $\mathbf{F}^{(N)}$ obtained from composing $\mathbf{f}$ with $\mathbf{g}$ and $\mathbf{h}_m^{(N)}$ converges to $\mathbf{F}$ with respect to the seminorms used to define $C_{\tr}(\R^{*d},\mathscr{M}(\R^{*d_{1,1}}, \dots, \R^{*d_{1,\ell_1}}, \dots \dots, \R^{*d_{m,1}}, \dots, \R^{*d_{n,\ell_n}}))^{d''}$, hence $\mathbf{F}$ is in this space.
		\item Finally, we consider the general case.  In the last step we approximate $\mathbf{f}$ by trace polynomials $\mathbf{f}^{(N)}$ as $N \to \infty$.  The argument is similar to the previous step, so we leave the details as an exercise.
	\end{enumerate}
	Finally, to prove continuity, it suffices to show that given $\mathbf{f}$, $\mathbf{g}$, $\mathbf{h}_1$, \dots, $\mathbf{h}_n$ and given $R_1$ and $\epsilon > 0$, there exist $R_2$, $\delta_1$, $\delta_2$, and $\eta_1$, \dots, $\eta_n$ such that if
	\begin{align*}
		\norm{\mathbf{f}' - \mathbf{f}}_{C_{\tr}(\R^{*d_2},\mathscr{M}^n),R_2} &< \delta_2, \\
		\norm{\mathbf{g}' - \mathbf{g}}_{C_{\tr}(\R^{*d_1}),R_1} &< \delta_1, \\
		\norm{\mathbf{h}_m' - \mathbf{h}_m}_{C_{\tr}(\R^{*d_1},\mathscr{M}^{\ell_m}),R_1} &< \eta_m,
	\end{align*}
	then
	\[
	\norm{\mathbf{F}' - \mathbf{F}}_{C_{\tr}(\R^{*d_1},\mathscr{M}^{L_n})^{d_3},R_1} < \epsilon.
	\]
	Let $R_2 = \norm{\mathbf{g}}_{C_{\tr}(\R^{*d_1})^{d_2},R_1} + 1$.  Then by choosing $\delta_2$ small enough, we can guarantee that $\norm{\mathbf{g}'}_{C_{\tr}(\R^{*d_1})^{d_2},R_1} < R_2$.  Then we use the uniform continuity of $\mathbf{f}$ as in (2) to control the error when we swap out $\mathbf{g}$ for $\mathbf{g}'$.  Proceeding as in (3) and (4), we can control the errors when swapping out $\mathbf{f}$ for $\mathbf{f}'$ and $\mathbf{h}_m$ for $\mathbf{h}_m'$ by choosing $\delta_1$ and $\eta_1$, \dots, $\eta_n$ small enough.  We leave the details as an exercise.
\end{proof}

\begin{theorem} \label{thm:chainrule}
	Let $k \in \N_0 \cup \{\infty\}$ and $n \in \N_0$.  Let $\mathbf{f} \in C_{\tr}^k(\R^{*d'},\mathscr{M}(\R^{*d_1},\dots,\R^{*d_n}))^{d''}$ for some $d' \in \N_0$ and $d''$, $d_1$, \dots, $d_n \in \N$.  Let $\mathbf{g} \in C_{\tr}^k(\R^{*d})_{\sa}^{d'}$ for some $d \in \N_0$.  For each $m = 1$, \dots, $n$, let $\mathbf{h}_m \in C_{\tr}^k(\R^{*d}, \mathscr{M}(\R^{*d_{m,1}},\dots,\R^{*d_{m,\ell_m}}))^{d_m}$ for some $\ell_m \in \N_0$ and $d_{m,1}$, \dots, $d_{m,\ell_m}$.  Let $L_m = \ell_1 + \dots + \ell_m$.  Then 
	\[
	\mathbf{f}(\mathbf{g}) \# [\mathbf{h}_1,\dots,\mathbf{h}_n] \in C_{\tr}^k(\R^{*d},\mathscr{M}(\R^{*d_{1,1}}, \dots, \R^{*d_{1,\ell_1}}, \dots \dots, \R^{*d_{m,1}}, \dots, \R^{*d_{n,\ell_n}}))^{d''},
	\]
	and for $k' \leq k$, we have
	\begin{multline*}
		\partial^{k'} [\mathbf{f}(\mathbf{g}) \# [\mathbf{h}_1,\dots,\mathbf{h}_n]] \\
		= \sum_{j=0}^{k'} \sum_{\substack{(B_1,\dots,B_n,B_1',\dots,B_j') \\ \text{partition of } [L_n+k'], \\ \min B_1' < \dots < \min B_j'}}
		\left( \partial^j \mathbf{f}(\mathbf{g}) \# [\partial^{|B_1|} \mathbf{h}_1, \dots, \partial^{|B_n|} \mathbf{h}_n, \partial^{|B_1'|} \mathbf{g}, \dots, \partial^{|B_j'|}\mathbf{g}] \right)_\sigma,
	\end{multline*}
	where $\sigma$ is the permutation given by
	\begin{align*}
		(\sigma(1), \dots, \sigma(L_n + k')) = (I_1,\dots,I_n,B_1,\dots,B_n,B_1',\dots,B_j'),
	\end{align*}
	where
	\[
	I_m = \{|B_1| + \dots + |B_m| + L_{m+1} + 1,\dots,L_{m+1} + |B_1| + \dots + |B_m| + L_m\},
	\]
	and where each of the sets $I_i$, $B_i$, and $B_i'$ is interpreted in the definition of $\sigma$ as a list of elements in order from least to greatest.  Here the blocks $B_1$, \dots, $B_n$, $B_1'$, \dots, $B_j'$ are regarded as an ordered tuple rather than a set, so that the same partition (set of blocks) can occur several times.  Moreover, the composition map
	\begin{multline*}
		C_{\tr}^k(\R^{*d})_{\sa}^{d'} \times \prod_{m=1}^n C_{\tr}^k(\R^{*d}, \mathscr{M}(\R^{*d_{m,1}},\dots,\R^{*d_{m,\ell_m}}))^{d_m} \\
		\to C_{\tr}^k(\R^{*d},\mathscr{M}(\R^{*d_{1,1}}, \dots, \R^{*d_{1,\ell_1}}, \dots \dots, \R^{*d_{m,1}}, \dots, \R^{*d_{n,\ell_n}}))^{d''}
	\end{multline*}
	is jointly continuous.
\end{theorem}

\begin{remark}
	It is immediate from the theorem that the $BC_{\tr}^k$ spaces are also closed under composition.
\end{remark}

\begin{proof}
	Fix $(\cA,\tau) \in \mathbb{W}$.  Then by iteratively applying the chain rule for Fr\'echet-$C^k$ functions (which is standard), we obtain the formula asserted above with $\mathbf{f}^{\cA,\tau}$, $\mathbf{g}^{\cA,\tau}$, and $\mathbf{h}^{\cA,\tau}$ rather than $\mathbf{f}$, $\mathbf{g}$, and $\mathbf{h}_m$.  Because of Lemma \ref{lem:composition}, the resulting expression is an element of $C_{\tr}(\R^{*d_1},\mathscr{M}^{L_n+k'})^{d_3}$.
	
	To explain the formula, note that when we apply $\partial$ iteratively $k'$ times, the operator $\partial$ at each stage could ``hit'' three different things:
	\begin{enumerate}[(1)]
		\item It could differentiate $\partial^j \mathbf{f}(\mathbf{g})$ by the chain rule which will change it to $\partial^{j+1} \mathbf{f}(\mathbf{g})$ and produce another term $\partial \mathbf{g}$, which we append as the $(j+1)$th argument for $\partial^{j+1} \mathbf{f}(\mathbf{g})$ (thus, setting $t_{j+1} = 0$).
		\item It could differentiate an already existing term $\partial^{t_i} \mathbf{g}$ that is one of the multilinear arguments (which was originally produced by step (1)).
		\item It could differentiate one of the multilinear arguments $\partial^{s_m} \mathbf{h}_m$.
	\end{enumerate}
	We arrive at the formula by keeping track of all these possibilities.  Here $B_m$ represents the set of time indices when $\mathbf{h}_m$ is differentiated and $B_i'$ represents the set of indices in which the $i$th derivative of $\mathbf{g}$ is appended and differentiated.  Since the copies are appended in order, we have $\min B_1' < \dots < \min B_j'$.  The first $L_n$ input vectors into $\partial^{k'} [\mathbf{f}(\mathbf{g}) \# [\mathbf{h}_1,\dots,\mathbf{h}_n]]$ are supposed to represent the multilinear arguments in the positions that already existed at stage $0$; or in other words, $\mathbf{Y}_{L_{m-1}+1}$, \dots, $\mathbf{Y}_{L_m}$ should be plugged into the first $\ell_m$ places of $\mathbf{h}_m$ for each $m$, which is the index set $I_m$.  The permutation $\sigma$ is defined to put these vectors into the correct locations, and the same for the tangent vectors corresponding to differentiation of the terms of the form $\mathrm{h}_i$ or $\partial^i \mathbf{g}$.
	
	Continuity of the composition operation follows from the formula for derivatives and the continuity claim in Lemma \ref{lem:composition}.
\end{proof}

\begin{corollary} \label{cor:Ckalgebra}
	$C_{\tr}^k(\R^{*d})$ is a $*$-algebra.
\end{corollary}

\begin{proof}
	We already explained the $*$-operation on $C_{\tr}^k(\R^{*d})$.  If $f$ and $g$ are self-adjoint, then the product $fg$ is the same as $h(f,g)$ where $h(x_1,x_2) = x_1x_2 \in \TrP(\R^{*2})$.  Since $h$ is $C_{\tr}^\infty$, it follows from Theorem \ref{thm:chainrule} that if $f$ and $g$ are $C_{\tr}^k$ and self-adjoint, then $fg$ is $C_{\tr}^k$.  The restriction of self-adjointness for $f$ and $g$ can be removed by decomposing a general element into its real and imaginary (that is, self-adjoint and anti-self-adjoint) parts.
\end{proof}

\begin{corollary} \label{cor:tracemap}
	There is a continuous map
	\[
	\tr: C_{\tr}^k(\R^{*d},\mathscr{M}(\R^{*d_1},\dots,\R^{*d_\ell})) \to C_{\tr}^k(\R^{*d},\mathscr{M}(\R^{*d_1},\dots,\R^{*d_\ell}))
	\]
	defined by
	\[
	(\tr(f))^{\cA,\tau}(\mathbf{X})[\mathbf{Y}_1,\dots,\mathbf{Y}_\ell] = \tau(f^{\cA,\tau}(\mathbf{X})[\mathbf{Y}_1,\dots,\mathbf{Y}_\ell]).
	\]
	Moreover, $\partial^{k'}[\tr(f)] = \tr[\partial^{k'} f]$ for $k' \leq k$.
\end{corollary}

\begin{proof}
	The trace $\tr$ can be viewed as an element $g$ of $C_{\tr}^\infty(\R^{*0},\mathscr{M}(\R^{*1}))$ that is given by $g^{\cA,\tau}[Y] = \tau(Y)$. Recall that $|\tau(X)| \leq \norm{X}_\alpha$ for every $\alpha \in [1,\infty]$ and hence $\norm{g}_{C_{\tr}(\R^{*0},\mathscr{M}^1),R} = 1$ for all $R$.  Also, $\partial^k g = 0$ for $k \geq 1$.  For $f \in C_{\tr}^k(\R^{*d},\mathscr{M}^\ell)_{\sa}$, we define $\tr(f) := g[f]$.  Then the relation $\partial^{k'}[\tr(f)] = \tr[\partial^{k'} f]$ follows from the chain rule.  A general $f \in C_{\tr}^k(\R^{*d},\mathscr{M}(\R^{*d_1},\dots,\R^{*d_\ell}))$ can be broken into its self-adjoint and anti-self-adjoint parts, and thus the map $\tr$ can be extended to all of $C_{\tr}^k(\R^{*d},\mathscr{M}(\R^{*d_1},\dots,\R^{*d_\ell}))$.
\end{proof}

As a consequence, if $\mathbf{f}$, $\mathbf{g} \in C_{\tr}^k(\R^{*d})^{d'}$, we can define a new function $\ip{\mathbf{f},\mathbf{g}}_{\tr} \in \tr(C_{\tr}^k(\R^{*d}))$ by
\[
\ip{\mathbf{f},\mathbf{g}}_{\tr}^{\cA,\tau}(\mathbf{X}) = \ip{\mathbf{f}^{\cA,\tau}(\mathbf{X}),\mathbf{g}^{\cA,\tau}(\mathbf{X})}_\tau.
\]
In particular, we will denote by $\ip{\mathbf{x},\mathbf{x}}_{\tr}$ the function whose evaluation on $(\cA,\tau)$ and $\mathbf{X}$ is $\norm{\mathbf{X}}_2^2$.

\subsection{An inverse function theorem}

The following result is a version of the inverse function theorem.  Although it would be possible to prove inverse function theorems on an operator norm ball, it is sufficient for our purposes to use the ``cheap'' global version that comes from a contraction mapping principle.

\begin{proposition}[Global inverse function theorem] \label{prop:IFT}
	Let $k \geq 1$.  Let $\mathbf{f} \in C_{\tr}^k(\R^{*d})_{\sa}^d$ for some $k \geq 1$.  Suppose that for some $0 < K < K'$, we have $\norm{\partial \mathbf{f} - K' \Id}_{BC_{\tr}(\R^{*d},\mathscr{M}^1)^d} \leq K$.  Then there exists (a unique) $\mathbf{g} \in C_{\tr}^k(\R^{*d})_{\sa}^d$ such that $\mathbf{f} \circ \mathbf{g} = \mathbf{g} \circ \mathbf{f} = \id$.
	
	Let us denote this function by $\mathbf{f}^{-1}$.  For a given $K' < K$, we have continuity of the map
	\[
	\mathbf{f} \mapsto \mathbf{f}^{-1}: \Bigl\{\mathbf{f} \in C_{\tr}^k(\R^{*d})_{\sa}^d: \norm{\partial \mathbf{f} - K' \Id}_{BC_{\tr}(\R^{*d},\mathscr{M}^1(\R^{*d}))^d} \leq K \Bigr\} \to C_{\tr}^k(\R^{*d})_{\sa}^d,
	\]
	where we use the subspace topology from $C_{\tr}^k(\R^{*d})_{\sa}^d$ on the domain.
\end{proposition}

\begin{proof}
	By substituting $(1/K') \mathbf{f}$ for $\mathbf{f}$ and $\mathbf{g}(K'(\cdot))$ for $\mathbf{g}$, we may assume without loss of generality that $K' = 1$.  Define $\mathbf{g}_0 = \id$ and inductively
	\[
	\mathbf{g}_{n+1} = \id + (\id - \mathbf{f}) \circ \mathbf{g}_n.
	\]
	Note that $\norm{(\id - \mathbf{f})^{\cA,\tau}(\mathbf{X}) - (\id - \mathbf{f})^{\cA,\tau}(\mathbf{Y})}_\infty \leq K \norm{\mathbf{X} - \mathbf{Y}}_\infty$ for $\mathbf{X}$, $\mathbf{Y} \in \cA_{\sa}^d$ for any $(\cA,\tau) \in \mathbb{W}$.  It follows that
	\[
	\norm{(\id - \mathbf{f}) \circ \mathbf{h} - (\id - \mathbf{f}) \circ \mathbf{h}' }_{C_{\tr}(\R^{*d})^d,R} \leq K \norm{\mathbf{h} - \mathbf{h}'}_{C_{\tr}(\R^{*d})^d,R}
	\]
	for $\mathbf{h}$, $\mathbf{h}' \in C_{\tr}(\R^{*d})_{\sa}^d$ and $R > 0$.  In particular, for $R > 0$,
	\[
	\norm{\mathbf{g}_{n+1} - \mathbf{g}_n}_{C_{\tr}(\R^{*d})_{\sa}^d,R} \leq K^n \norm{\mathbf{g}_1 - \mathbf{g}_0}_{C_{\tr}(\R^{*d})_{\sa}^d,R} = K^n \norm{\id - \mathbf{f}}_{C_{\tr}(\R^{*d})_{\sa}^d,R}.
	\]
	Hence, $\mathbf{g}_n$ converges as $n \to \infty$ to some $\mathbf{g} \in C_{\tr}(\R^{*d})_{\sa}^d$, which must also $\mathbf{g} = \id + (\id - \mathbf{f}) \circ \mathbf{g}$, or in other words $\mathbf{f} \circ \mathbf{g} = \id$.  Since $\id - \mathbf{f}$ is $K$-Lipschitz on $\cA_{\sa}^d$ for any $(\cA,\tau)$ and $K < 1$, it follows that $\mathbf{f}^{\cA,\tau}$ is injective.  Thus, in the relation $\mathbf{f}^{\cA,\tau} \circ \mathbf{g}^{\cA,\tau} \circ \mathbf{f}^{\cA,\tau} = \mathbf{f}^{\cA,\tau}$, we may cancel $\mathbf{f}^{\cA,\tau}$ on the left-hand side and thus obtain $\mathbf{g} \circ \mathbf{f} = \id$.  Since the rate of convergence in $\norm{\cdot}_{C_{\tr}(\R^{*d})^d,R}$ only depends on $K$ and $\norm{\id - \mathbf{f}}_{C_{\tr}(\R^{*d})_{\sa}^d,R}$, it follows that $\mathbf{g}$ depends continuously on $\mathbf{f}$ in $C_{\tr}^k(\R^{*d})_{\sa}^d$.
	
	Note that by the chain rule and induction, $\mathbf{g}_n \in C_{\tr}^k(\R^{*d})_{\sa}^d$ and we have for $1 \leq k' \leq k$ that
	\[
	\partial^{k'} \mathbf{g}_{n+1} = \sum_{j=1}^{k'} \sum_{\substack{(B_1,\dots,B_j) \\ \text{partition of } [k'] \\ \min B_1 < \dots < \min B_j}} (\partial^j(\id -  \mathbf{f}) \circ \mathbf{g}_n)[\partial^{|B_1|} \mathbf{g}_n,\dots,\partial^{|B_j|}\mathbf{g}_n].
	\]
	We claim that $\partial^{k'} \mathbf{g}_n$ converges as $n \to \infty$.  We first describe the candidate limit functions $\mathbf{g}^{(k')}$ as fixed points of the equation where we substitute $\mathbf{g}^{(k')}$ for $\partial^{k'} \mathbf{g}_n$ and $\partial^{k'} \mathbf{g}_{n+1}$.  Of course $\mathbf{g}^{(0)}$ will simply be $\mathbf{g}$.  Separating out the $j = 1$ term on the right-hand side, this equation becomes
	\[
	\mathbf{g}^{(k')} = (\Id - \partial \mathbf{f} \circ \mathbf{g}) \# \mathbf{g}^{(k')} - \sum_{j=2}^{k'} \sum_{\substack{(B_1,\dots,B_j) \\ \text{partition of } [k'] \\ \min B_1 < \dots < \min B_j}} (\partial^j \mathbf{f} \circ \mathbf{g})[\mathbf{g}^{(|B_1|)},\dots,\mathbf{g}^{(|B_j|)}].
	\]
	Since $\norm{\Id - \partial \mathbf{f}}_{BC_{\tr}(\R^{*d},\mathscr{M}^1)^d} \leq K < 1$, it follows that the right-hand side is $K$-contractive as a function of $\mathbf{g}^{(k')}$.  Thus, we may construct the functions $\mathbf{g}^{(k')}$ by induction on $k'$; assuming the previous terms have been defined, $\mathbf{g}^{(k')}$ is obtained by iteration of the right-hand side, starting with the function $\Id$ for $k' = 1$ and $0$ for $k' > 1$.  The rate of convergence of the iterates with respect to $\norm{\cdot}_{C_{\tr}(\R^{*d}),R}$ is controlled completely by the constant $K$, the norms of the derivatives of $\mathbf{f}$ on the ball of radius $R' := \norm{\mathbf{g}}_{C_{\tr}(\R^{*d})^d,R}$, and the norms of the previous terms $\mathbf{g}^{(j)}$ on the ball of radius $R$.  In particular, it follows that $\mathbf{g}^{(k')} \in C_{\tr}(\R^{*d},\mathscr{M}^{k'}(\R^{*d}))^d$ depends continuously on $\mathbf{f} \in C_{\tr}^k(\R^{*d})_{\sa}^d$ using induction on $k'$.  Indeed, once we know the claim for $j < k'$, then the iterates for $\mathbf{g}^{(k')}$ depend continuously on $\mathbf{f}$, and the preceding remarks show that for each $R$, the rate of convergence will be uniform on some open set in $C_{\tr}^k(\R^{*d})_{\sa}^d$ containing $\mathbf{f}$.
	
	To finish the proof, it only remains to show that $\mathbf{g}$ is in $C_{\tr}^k(\R^{*d})_{\sa}^d$ and $\partial^{k'} \mathbf{g} = \mathbf{g}^{(k')}$ for $k' \leq k$.  To this end, it suffices to show that $\partial^{k'} \mathbf{g}_n \to \mathbf{g}^{(k')}$ as $n \to \infty$.  We proceed by induction on $k' \geq 1$ (with $k' = 0$ already proved).  Subtracting the relations for $\partial^{k'} \mathbf{g}_{n+1}$ and $\mathbf{g}^{(k')}$, we get
	\begin{multline*}
		\partial^{k'} \mathbf{g}_{n+1} - \mathbf{g}^{(k')} = (\Id - \partial \mathbf{f} \circ \mathbf{g}_n) \# (\partial^{k'} \mathbf{g}_n - \mathbf{g}^{(k')}) + (\partial \mathbf{f} \circ \mathbf{g}_n - \partial \mathbf{f} \circ \mathbf{g}) \# \mathbf{g}^{(k')} \\
		+ \sum_{j=2}^{k'} \sum_{\substack{(B_1,\dots,B_j) \\ \text{partition of } [k'] \\ \min B_1 < \dots < \min B_j}} \bigl[ (\partial^j(\id -  \mathbf{f}) \circ \mathbf{g}_n)[\partial^{|B_1|} \mathbf{g}_n,\dots,\partial^{|B_j|}\mathbf{g}_n] - (\partial^j \mathbf{f} \circ \mathbf{g})[\mathbf{g}^{(|B_1|)},\dots,\mathbf{g}^{(|B_j|)}] \bigr].
	\end{multline*}
	Let $\epsilon_{n,R}$ be the norm of $(\partial \mathbf{f} \circ \mathbf{g}_n - \partial \mathbf{g}) \# \mathbf{g}^{(k')}$ plus the norms of the terms in the summation.  By the induction hypothesis and by continuity of composition $\epsilon_{n,R} \to 0$ as $n \to \infty$, and we also have
	\[
	\norm{\partial^{k'} \mathbf{g}_{n+1} - \mathbf{g}^{(k')}}_{C_{\tr}(\R^{*d},\mathscr{M}^{k'}(\R^{*d}))^d,R} \leq K \norm{\partial^{k'} \mathbf{g}_n - \mathbf{g}^{(k')}}_{C_{\tr}(\R^{*d},\mathscr{M}^{k'}(\R^{*d}))^d,R} + \epsilon_{n,R}.
	\]
	A straightforward induction on $n$ shows that
	\[
	\norm{\partial^{k'} \mathbf{g}_n - \mathbf{g}^{(k')}}_{C_{\tr}(\R^{*d},\mathscr{M}^{k'}(\R^{*d}))^d,R} \leq K^n \norm{\partial^{k'} \mathbf{g}_0 - \mathbf{g}^{(k')}}_{C_{\tr}(\R^{*d},\mathscr{M}^{k'}(\R^{*d}))^d,R} + \sum_{m=0}^n K^m \epsilon_{n-m,R}.
	\]
	Clearly, the first term on the right-hand side goes to zero as $n \to \infty$.  For the second term, note that the bi-infinite sequence $(\mathbf{1}_{m\leq n} \epsilon_{n-m,R})_{m,n}$ is bounded and $\lim_{n \to \infty} \mathbf{1}_{m\leq n} \epsilon_{n-m,R} = 0$.  Because $\sum_{m=0}^\infty K^m < \infty$, the dominated convergence theorem implies that
	\[
	\lim_{n \to \infty} \sum_{m=0}^n K^m \epsilon_{n-m,R} = \lim_{n \to \infty} \sum_{m=0}^\infty K^m \mathbf{1}_{m \leq n} \epsilon_{n-m} = 0.
	\]
	Thus, $\partial^{k'} \mathbf{g}_n \to \mathbf{g}^{(k')}$ as desired.
\end{proof}

\section{Non-commutative smooth functions: connections} \label{sec:NCfunc2}

\subsection{Scalar-valued functions, non-commutative laws, and operator algebras}

The trace map in Corollary \ref{cor:tracemap} leads to the following definition.

\begin{definition}
	We denote the image of $\tr$ in $C_{\tr}^k(\R^{*d},\mathscr{M}(\R^{*d_1},\dots,\R^{*d_\ell}))$ by \newline $\tr(C_{\tr}^k(\R^{*d},\mathscr{M}(\R^{*d_1},\dots,\R^{*d_\ell})))$.
\end{definition}

\begin{observation}
	Let $f \in C_{\tr}^k(\R^{*d},\mathscr{M}(\R^{*d_1},\dots,\R^{*d_\ell}))$.  Then the following are equivalent:
	\begin{enumerate}[(1)]
		\item $f \in \tr(C_{\tr}^k(\R^{*d},\mathscr{M}(\R^{*d_1},\dots,\R^{*d_\ell})))$,
		\item $f^{\cA,\tau}(\mathbf{X})[\mathbf{Y}_1,\dots,\mathbf{Y}_\ell] \in \C$ for every $(\cA,\tau)$ and $\mathbf{X}$, $\mathbf{Y}_1$, \dots, $\mathbf{Y}_\ell \in \cA_{\sa}$.
		\item $f = \tr(f)$.
	\end{enumerate}
\end{observation}

Thus, $\tr(C_{\tr}^k(\R^{*d},\mathscr{M}^\ell)$ may be viewed as the subspace of $C_{\tr}^k(\R^{*d},\mathscr{M}(\R^{*d_1},\dots,\R^{*d_\ell}))$ consisting of scalar-valued functions.  Similarly, $f \in \tr(C_{\tr}^k(\R^{*d},\mathscr{M}(\R^{*d_1},\dots,\R^{*d_\ell})))$ is self-adjoint if and only if $f^{\cA,\tau}$ is real-valued for every $(\cA,\tau) \in \mathbb{W}$.

Non-commutative laws can be characterized as certain linear functionals on $C_{\tr}(\R^{*d})$.  To state this result, we use the following definitions.

\begin{definition}
	We say that $f \in C_{\tr}^k(\R^{*d})$ is \emph{positive} if $f^{\cA,\tau}(\mathbf{X}) \geq 0$ in $\cA$ for every $(\cA,\tau) \in \mathbb{W}$ and $\mathbf{X} \in \cA_{\sa}^d$.  We say that a map $\Phi: C_{\tr}^k(\R^{*d_1}) \to C_{\tr}(\R^{*d_2})$ is \emph{positive} if it maps positive elements to positive elements.
\end{definition}

\begin{definition}
	Let $\cA$ be an algebra.  We say that map $\Phi: C_{\tr}^k(\R^{*d}) \to \cA$ is \emph{multiplicative over $\tr(C_{\tr}(\R^{*d}))$} if $\Phi(fg) = \Phi(f) \Phi(g)$ whenever $f \in \tr(C_{\tr}^k(\R^{*d}))$. 
\end{definition}

\begin{lemma} \label{lem:Claw}
	The following three sets are in bijection with each other:
	\begin{enumerate}[(1)]
		\item the space $\Sigma_d$ of non-commutative laws $\lambda$,
		\item the set of continuous positive algebra homomorphisms $\rho: \tr(C_{\tr}(\R^{*d})) \to \C$,
		\item the set of continuous unital positive maps $\Phi: C_{\tr}(\R^{*d}) \to \C$ that are multiplicative over $\tr(C_{\tr}(\R^{*d}))$ and satisfy $\Phi = \Phi \circ \tr$.
	\end{enumerate}
	The bijections are given by
	\begin{align*}
		\lambda &= \rho \circ \tr|_{\C\ip{x_1,\dots,x_d}} \\
		\lambda &= \Phi |_{\C\ip{x_1,\dots,x_d}} \\
		\Phi &= \rho \circ \tr \\
		\rho &= \Phi|_{\tr(C_{\tr}(\R^{*d}))}
	\end{align*}
\end{lemma}

\begin{proof}
	First, we show the bijection between (2) and (3).  Note that $\tr$ is a continuous unital positive map $C_{\tr}(\R^{*d}) \to \tr(C_{\tr}(\R^{*d}))$ that is multiplicative over $\tr(C_{\tr}(\R^{*d}))$.  Hence, if $\rho$ satisfies (2), then $\Phi = \rho \circ \tr$ satisfies (3).  Conversely, if $\Phi$ satisfies (3), then $\Phi|_{\tr(C_{\tr}(\R^{*d}))}$ satisfies (2), and the maps $\rho \mapsto \rho \circ \tr$ and $\Phi \mapsto \Phi|_{\tr(C_{\tr}(\R^{*d}))}$ are mutually inverse.
	
	Next, we show the bijection between (1) and (2).  If $\rho$ satisfies (2), then let $\lambda(p) = \rho(\tr(p))$ for $p \in \C\ip{x_1,\dots,x_d}$.  Since $\rho$ is an algebra homomorphism it is unital and hence $\lambda(1) = 1$.  Also, $\lambda(pq) = \lambda(qp)$ since $\tr(pq) = \tr(qp)$ in $C_{\tr}(\R^{*d})$.  Thirdly, $\tr(p^*p)$ is positive in $\tr(C_{\tr}(\R^{*d}))$, hence $\lambda(p^*p) \geq 0$.  Finally, since $\rho$ is continuous, there exists $R > 0$ and $\delta > 0$ such that
	\[
	\norm{f}_{C_{\tr}(\R^{*d}),R} \leq \delta \implies |\rho(\tr(f))| < 1.
	\]
	Taking $p(x) = x_{i_1} \dots x_{i_\ell}$, we have $\norm{p}_{C_{\tr}(\R^{*d}),R} = R^\ell$ and hence
	\[
	|\lambda(p)| = |\rho(\tr(p))| \leq \frac{R^\ell}{\delta}.
	\]
	Since this holds for all $\ell$, we know $\lambda$ is exponentially bounded and hence is a non-commutative law.
	
	Conversely, suppose that $\lambda$ is a non-commutative law in $\Sigma_{d,R}$.  Let $\mathbf{X}$ be a $d$-tuple of self-adjoint operators in $(\cA,\tau)$ which realize the law $\lambda$.  Then define $\rho: \tr(C_{\tr}(\R^{*d})) \to \C$ by $\rho(f) = f(\mathbf{X})$.  Clearly, $f$ is a positive homomorphism, and also $\rho$ is continuous since $|\rho(f)| \leq \norm{f}_{C_{\tr}(\R^{*d}),R}$.
	
	Now, let us show that the maps $\lambda \mapsto \rho$ and $\rho \mapsto \lambda$ described above are mutually inverse.  If we start with $\lambda$ and define $\rho(f) = f(\mathbf{X})$ using $\cA$, $\tau$, and $\mathbf{X}$ as above, then $\rho(\tr(p)) = \tau(p(\mathbf{X})) = \lambda(p)$.  On the other hand, suppose we start with $\rho$ and let $\lambda = \rho \circ \tr|_{\C\ip{x_1,\dots,x_d}}$.  Let $\mathbf{X}$ be a tuple realizing the law $\lambda$.  Then clearly $\rho(\tr(p)) = \tau(p(\mathbf{X}))$.  Since $\rho$ is a homomorphism, it follows that $\rho(f) = f(\mathbf{X})$ holds for all scalar-valued trace polynomials.  But the trace polynomials are dense in $C_{\tr}(\R^{*d})$ and hence this equality holds for all $f$.
\end{proof}

This lemma allows us to describe the push-forward of non-commutative laws by functions $\mathbf{f} \in C_{\tr}(\R^{*d})_{\sa}^{d'}$.  Indeed, if $\mathbf{f} \in C_{\tr}(\R^{*d})_{\sa}^{d'}$, then there is a continuous positive homomorphism $\tr(\C_{\tr}(\R^{*d'})) \to \tr(C_{\tr}(\R^{*d}))$ given by $g \mapsto g \circ \mathbf{f}$.  Continuity follows because $\mathbf{f}$ is bounded in $\norm{\cdot}_\infty$ on each $\norm{\cdot}_\infty$-ball.  If $\rho$ is a positive homomorphism $\tr(C_{\tr}(\R^{*d})_{\sa}) \to \C$, then $\mathbf{f}_* \rho := \rho \circ \mathbf{f}$ is a continuous positive homomorphism $\tr(C_{\tr}(\R^{*d'})) \to \C$.  Since the continuous homomorphisms are in bijection with non-commutative laws, there is a corresponding push-forward operation $\mathbf{f}_*: \Sigma_d \to \Sigma_{d'}$.  Furthermore, the push-forward map $\mathbf{f}_*$ is characterized by the property that for every  $(\cA,\tau) \in \mathbb{W}$ and $\mathbb{X} \in \cA_{\sa}^d$, we have $\lambda_{\mathbf{f}(\mathbf{X})} = \mathbf{f}_* \lambda_{\mathbf{X}}$.

	Push-forwards of non-commutative laws lead naturally to inclusions and isomorphisms of tracial $\mathrm{C}^*$- and $\mathrm{W}^*$-algebras.  The next observation is immediate from Lemma \ref{lem:lawisomorphism}.

\begin{observation}
	Let $\mathbf{f} \in C_{\tr}(\R^{*d})_{\sa}^d$.  Let $\mu \in \Sigma_d$, and let $(\cA_1,\tau_1)$ be the $\mathrm{W}^*$ GNS representation of $\mu$, and let $\mathbf{X} \in (\cA_1)_{\sa}^d$ be the canonical generators having the non-commutative law $\mu$.  Similarly, let $(\cA_2,\tau_2)$ be the GNS representation for $\mathbf{f}_* \mu$ with its canonical generators $\mathbf{Y} \in (\cA_2)_{\sa}^{d'}$.  Then there is a unique inclusion map $\iota: (\cA_2,\tau_2) \to (\cA_1,\tau_1)$ of tracial $\mathrm{W}^*$-algebras such that $\iota(\mathbf{Y}) = \mathbf{f}^{\cA_1,\tau_1}(\mathbf{X})$.  We also have $\iota(\mathrm{C}^*(\mathbf{Y})) \subseteq \mathrm{C}^*(\mathbf{X})$.
\end{observation}

\begin{observation} \label{obs:pushforwardisomorphism}
	Consider the same situation as above, and suppose there exists a function $\mathbf{g} \in C_{\tr}(\R^{*d'})_{\sa}^d$ such that $\mathbf{g}^{\cA_2,\tau_2}(\mathbf{Y}) = \mathbf{X}$.  Then $\iota$ is an isomorphism of tracial $\mathrm{W}^*$-algebras, which also restricts to an isomorphism $\mathrm{C}^*(\mathbf{Y}) \to \mathrm{C}^*(\mathbf{X})$.
\end{observation}

\begin{observation} \label{obs:pushforwardisomorphism2}
	Suppose that $\mathbf{f} \in C_{\tr}(\R^{*d})_{\sa}^{d'}$ and $C_{\tr}(\R^{*d'})_{\sa}^d$ satisfy $\mathbf{f} \circ \mathbf{g} = \id$ and $\mathbf{g} \circ \mathbf{f} = \id$.  Let $\mu \in \Sigma_d$.  Then by the previous observations there is an isomorphism of the tracial $\mathrm{W}^*$-algebras associated to $\mu$ and $\mathbf{f}_*\mu$ respectively, which also restricts to an isomorphism of the $\mathrm{C}^*$-algebras associated to the two laws.
\end{observation}

\begin{remark}
	If $\mathbf{f}$ and $\mathbf{g}$ as above satisfy $\mathbf{f} \circ \mathbf{g} = \id$ and $\mathbf{g} \circ \mathbf{f} = \id$, then we must have $d = d'$.  This is because $\mathbf{f}$ defines a homeomorphism $M_N(\C)_{\sa}^d \to M_N(\C)_{\sa}^{d'}$ for every $N$, so it follows from the invariance of domain theorem in topology (and in fact, we would only need the homeomorphism for a single value of $N$ to make this conclusion).  However, if we only assume that $\mathbf{g}^{\cA,\tau}(\mathbf{f}^{\cA,\tau}(\mathbf{X})) = \mathbf{X}$ for a particular $d$-tuple of operators $\mathbf{X}$ in a particular $(\cA,\tau)$, then it is a difficult question whether $d$ must equal $d'$, and the answer will likely depend on the properties of the tuple $\mathbf{X}$.
\end{remark}

\subsection{One-variable functional calculus} \label{subsec:smoothfunctionalcalculus}

\begin{lemma}
	If $\phi \in C(\R)$, then the function $f = (f^{\cA,\tau})_{(\cA,\tau) \in \mathbb{W}}$ given by $f^{\cA,\tau}(X) = \phi(X)$ for every $(\cA,\tau) \in \mathbb{W}$ and $X \in \cA_{\sa}$ is an element of $C_{\tr}(\R^{*1})$.
\end{lemma}

\begin{proof}
	Let $(\phi^{(N)})_{N \in \N}$ be a sequence of polynomials which converge to $\phi$ uniformly on compact subsets of $\R$.  By the spectral mapping theorem, for any $(\cA,\tau)$ and any self-adjoint operator $X$ in $\cA$ with $\norm{X} \leq R$, we have
	\[
	\norm{\phi^{(N)}(X) - \phi(X)}_\infty \leq \sup_{t \in [-R,R]} |\phi^{(N)}(t) - \phi(t)|.
	\]
	Hence, the sequence of polynomials $\phi^{(N)}(x) \in \C[x] \subseteq C_{\tr}(\R)$ converges in $C_{\tr}(\R)$ to some function $f$, which clearly must satisfy $f^{\cA,\tau}(X) = \phi(X)$ for self-adjoint $X$ in $(\cA,\tau)$.
\end{proof}

\begin{definition}
	Given $\phi \in C(\R)$, we denote the corresponding element of $C_{\tr}(\R)$ by $\phi(x)$, where $x$ is the same formal variable used for defining the trace polynomials in $C_{\tr}(\R)$.  Similarly, for $j \leq d$, we may define an element $\phi(x_j)$ in $C_{\tr}(\R^{*d})$ as the element sending a self-adjoint tuple $(X_1,\dots,X_d)$ in $(\cA,\tau)$ to $\phi(X_j)$.
\end{definition}

Under what conditions is $\phi(x) \in C_{\tr}^k(\R^{*d})$?  Peller, Aleksandrov, and Nazarov have studied the free difference quotients of functions on the real line for the sake of understanding the perturbations of self-adjoint operators \cite{Peller2006,AlPe2010a,AlPe2010b,ANP2016,AlPe2017}, and concluded that Besov spaces are natural spaces of functions on $\R$ that lead to operator $C^k$ functions; for a self-contained development of operator $C^k$ functions, see \cite{Nikitopoulos2020}.  However, we do not need the full strength of these results, and we will be content to directly apply one of the key basic ideas, Fourier decomposition, to our current context.  We also point out that the recent papers \cite{CGP2019} and \cite{Parraud2020} have applied the same functional calculus/Fourier decomposition techniques to study the finer properties of random matrix models.  We begin by describing the non-commutative derivatives of the complex exponential $e^{ix} \in C_{\tr}(\R)$ for each $t \in \R$.  In the formula for derivatives, we recall that the theory of Riemann integration is valid for continuous functions on polytopes taking values in a Fr\'echet space, with all the same proofs that are learned in undergraduate calculus.

\begin{lemma}
	For each $t \in \R$, the function $e^{itx}$ is in $BC_{\tr}^\infty(\R)$ and satisfies
	\begin{equation} \label{eq:exponentialderivativebound}
		\norm{\partial^k[e^{itx}]}_{BC_{\tr}(\R,\mathscr{M}^k)} \leq t^k.
	\end{equation}
	The derivatives are given explicitly as follows.  Let $\Delta_k$ denote the simplex
	\[
	\Delta_k := \{(s_0,\dots,s_k): s_j \geq 0, s_0 + \dots + s_k = 1\},
	\]
	and let $\rho_k$ be the standard uniform probability measure on $\Delta_k$.  Then
	\begin{equation} \label{eq:exponentialderivatives}
		\partial^k[e^{itx}][y_1,\dots,y_k] = \frac{(it)^k}{k!} \sum_{\sigma \in \Perm(k)} \int_{\Delta_k} e^{its_0x} y_{\sigma(1)} e^{its_1x} \dots y_{\sigma(k)} e^{its_kx} \,d\rho_k(s_0,\dots,s_k).
	\end{equation}
	Here $y_1$, \dots, $y_k$ denote the formal variables occurring as multilinear arguments of the derivative, and the integral is interpreted as a Riemann integral with values in the Fr\'echet space $C_{\tr}(\R,\mathscr{M}^k)$.
\end{lemma}

\begin{proof}
	First, we prove the formula for the derivative.  Consider the projection map $\pi_k: \R^{k+1} \to \R^k$ onto the first $k$ coordinates.  Note that $\pi_k$ gives an affine bijection from $\Delta_k$ onto the simplex $\{s_j \geq 0, s_0 + \dots + s_{k-1} \leq 1\}$, and therefore this map is measure-preserving up to a constant factor.  The Lebesgue measure on $\R^k$ assigns total mass $1/k!$ to the simplex $\pi_k(\Delta_k)$ and hence \eqref{eq:exponentialderivatives} is equivalent to
	\begin{multline} \label{eq:exponentialderivatives2}
		\partial^k[e^{itx}][y_1,\dots,y_k] \\ = (it)^k \sum_{\sigma \in \Perm(k)} \int_{\pi_k(\Delta_k)} e^{its_0x} y_{\sigma(1)} e^{its_1x} \dots y_{\sigma(k)} e^{it(1 - s_0 + \dots + s_{k-1})x} \,ds_0 \dots ds_{k-1}.
	\end{multline}
	
	We prove this formula by induction.  First, consider $k = 1$.  For $n \in \N$, the function $x^n$ is in $C_{\tr}(\R)$ with $\norm{x^n}_{C_{\tr}(\R),R} = R^n$.  Moreover, using the product rule,
	\[
	\partial[x^n][y] = \sum_{m=0}^{n-1} x^{n-1-m} y x^m,
	\]
	so clearly $\norm{\partial[x^n]}_{C_{\tr}(\R,\mathscr{M}^1),R} \leq nR^{n-1}$.  It follows that the series
	\[
	\sum_{n=0}^\infty \frac{1}{n!} (itx)^n
	\]
	converges in $C_{\tr}^1(\R)$.  This series must agree with $e^{itx}$ since they agree when evaluating on any self-adjoint operator $X$.  We thus have
	\begin{multline*}
		\partial[e^{itx}][y] = \sum_{n=0}^\infty \frac{(it)^n}{n!} \sum_{m=0}^{n-1} x^{n-1-m} y x^m
		= \sum_{\ell,m \geq 0} \frac{(it)^{\ell+m+1}}{(\ell+m+1)!} x^\ell y x^m \\
		= it \sum_{\ell,m \geq 0} \frac{1}{(\ell+m+1)!} (itx)^\ell y (itx)^m.
	\end{multline*}
	Observe that by repeated integration by parts
	\begin{multline*}
		\int_0^1 \frac{1}{\ell!} s^\ell \frac{1}{m!} (1 - s)^m\,ds = \int_0^1 \frac{1}{(\ell + 1)!} s^{\ell+1} \frac{1}{(m-1)!} (1 - s)^{m-1}\,ds = \dots \\
		= \int_0^1 \frac{1}{(\ell+m)!} s^{\ell+m}\,ds = \frac{1}{(\ell + m + 1)!},
	\end{multline*}
	so that
	\begin{align*}
		\partial[e^{itx}][y]
		&= it \sum_{\ell,m \geq 0} \left( \int_0^1 \frac{1}{\ell!} s^\ell \frac{1}{m!} (1 - s)^m\,ds \right) (itx)^\ell y (itx)^m \\
		&= it \int_0^1 \sum_{\ell,m \geq 0} \frac{1}{\ell!} (itsx)^\ell \frac{1}{m!} (it(1 - s)x)^m\,ds \\
		&= it \int_0^1 e^{itsx} y e^{it(1 - s)x}\,ds.
	\end{align*}
	Note that $(itsx)^\ell y (itx(1-s))^m$ is an element of $C_{\tr}(\R,\mathscr{M}^1)$ that depends continuously on $s$ and its norm on the $R$-ball is bounded by $(|t|R)^{\ell+m}$.  This implies uniform convergence of the series and hence the $C_{\tr}(\R,\mathscr{M}^1)$-valued summation and integration are defined and exchangeable.  This proves \eqref{eq:exponentialderivatives} and hence \eqref{eq:exponentialderivatives2} in the case $k = 1$.
	
	For the induction step, assume \eqref{eq:exponentialderivatives} holds for $k$.  Then by applying the product rule inside the integral, we evaluate $\partial^{k+1}[e^{itx}][y_1,\dots,y_k,y_{k+1}]$ as
	\begin{multline*}
		\frac{(it)^k}{k!} \sum_{\sigma \in \Perm(k)} \int_{\Delta_k} \sum_{\ell=0}^k e^{its_0x} y_{\sigma(1)} \dots e^{its_{\ell-1}x} y_{\sigma(\ell)} \, \partial[e^{its_\ell}][y_{k+1}] \\
		\, y_{\sigma(\ell+1)} e^{its_{\ell+1}x} \dots y_{\sigma(k)} e^{its_kx} \,d\rho_k(s_0,\dots,s_k).
	\end{multline*}
	Using the $k = 1$ case,
	\[
	\partial[e^{its_\ell}][y_{k+1}] = its_\ell \int_0^1 e^{its_\ell u x} y_{k+1} e^{ist_\ell(1-u)x}\,du = it \int_0^{s_\ell} e^{itvx} y_{k+1} e^{it(s_\ell - v)x}\,dv.
	\]
	We substitute this into the above equation.  Then we observe for any function $\phi$ on  $\Delta_{k+1}$, we have
	\begin{multline*}
		k! \int_{\Delta_k} \int_0^{s_\ell} \phi(s_0,\dots,s_k,s_\ell-v)\,dv\,d\rho_k(s_0,\dots,s_k) \\
		= (k+1)! \int_{\Delta_{k+1}} \phi(s_0,\dots,s_{k+1})\,d\rho_{k+1}(s_0,\dots,s_{k+1}),
	\end{multline*}
	which follows using the parametrization of $\Delta_k$ by $\pi_k(\Delta_k)$.  Also, recall that $\rho_k$ is permutation invariant.  Thus, $\partial^{k+1}[e^{itx}][y_1,\dots,y_k,y_{k+1}]$ becomes
	\begin{multline*}
		\frac{(it)^{k+1}}{(k+1)!} \sum_{\sigma \in \Perm(k)} \sum_{\ell=0}^k \int_{\Delta_{k+1}} e^{its_0x} y_{\sigma(1)} \dots e^{its_{\ell-1}x} y_{\sigma(\ell)} e^{its_\ell x} y_{k+1} \\
		e^{its_{\ell+1}x} y_{\sigma(\ell+1)} \dots e^{its_kx} y_{\sigma(k)} e^{its_{k+1}x} \,d\rho_{k+1}(s_0,\dots,s_{k+1}).
	\end{multline*}
	It is a straightforward combinatorial manipulation to reduce this to \eqref{eq:exponentialderivatives} for $k + 1$; the idea is that by choosing a permutation $\sigma \in \Perm(k)$ and then inserting $k+1$ at every possible position before, between, or after the existing elements, we achieve every permutation of $k+1$ elements.
	
	Now note that for any operator $X$, $e^{itX}$ is unitary.  This implies that $\norm{e^{itx}}_{BC_{\tr}(\R)} = 1$.  By substituting this into \eqref{eq:exponentialderivatives}, we get \eqref{eq:exponentialderivativebound}.
\end{proof}

The role of the Fourier transform is to decompose a function on $\R$ into a linear combination of complex exponentials.  For $\phi \in L^1(\R)$, the \emph{Fourier transform} is given by
\[
\widehat{\phi}(s) = \int_{\R} e^{-2\pi i st} \phi(t)\,dt.
\]
If $\widehat{\phi} \in L^1(\R)$, then we have the \emph{Fourier inversion formula}
\[
\phi(t) = \int_{\R} e^{2\pi i ts}\widehat{\phi}(s)\,ds.
\]
The Fourier transform extends to a well-defined operator on the space of \emph{tempered distributions} and in particular is well-defined for any continuous function of polynomial growth at $\infty$.  We also have
\[
\widehat{\phi'}(s) = 2\pi i s \widehat{\phi}(s)
\]
for all tempered distributions.  In particular, this implies that if $s^k \widehat{\phi}(s)$ is in $L^1(\R)$, then $(d/dt)^k \phi$ is in $BC(\R)$.  In fact, we will show a similar property for the non-commutative derivatives of $\phi(x)$ in $C_{\tr}(\R)$.

\begin{proposition} \label{prop:smoothfunctionalcalculus}
	Let $k \in \N$.
	\begin{enumerate}[(1)]
		\item Suppose that $\phi \in BC(\R)$ and that $\int_{\R} (1 + |s|^k) |\widehat{\phi}(s)|\,ds$ is finite.  Then $\phi(x) \in BC_{\tr}^k(\R)$ with
		\[
		\norm{\partial^\ell \phi(x)}_{BC_{\tr}(\R,\mathscr{M}^\ell)} \leq \int_{\R} |(2 \pi is)^\ell \widehat{\phi}(s)|\,ds
		\]
		for each $\ell \leq k$.
		\item If $\phi \in C^{k+2}(\R)$, then $\phi(x) \in C_{\tr}^k(\R)$.
	\end{enumerate}
\end{proposition}

\begin{proof}
	(1) In light of \eqref{eq:exponentialderivativebound}, we have for every $R > 0$ and $\ell \leq k$ that
	\[
	\norm{\partial^\ell(e^{2\pi i sx})}_{C_{\tr}(\R),R} \leq |2\pi s|^\ell.
	\]
	Moreover, the map $s \mapsto \partial^k[e^{2\pi i sx}]$ from $\R$ to $C_{\tr}(\R,\mathscr{M}^\ell)$ is continuous by continuity of composition in Lemma \ref{lem:composition}.  Moreover, $\widehat{\phi}$ is continuous.  Thus, the improper Riemann integral
	\[
	\int_{\R} \partial^\ell[e^{2\pi i s x}] \widehat{\phi}(s)\,ds = \lim_{S \to \infty} \int_{-S}^S \partial^\ell[e^{2\pi i s x} \widehat{\phi}(s)\,ds
	\]
	is well-defined in $C_{\tr}(\R,\mathscr{M}^\ell)$ for each $\ell \leq k$.  Or equivalently, the improper Riemann integral $\int_{\R} e^{2\pi i s x} \widehat{\phi}(s)\,ds$ is well-defined in $C_{\tr}^k(\R)$.  By evaluating this on any self-adjoint operator $X$ and using the spectral decomposition of $X$, we see that $\phi(x) = \int_{\R} e^{2 \pi i s x} \widehat{\phi}(s)\,ds$ in $C_{\tr}(\R)$.  Therefore, $\phi \in C_{\tr}^k(\R)$.  Also,
	\[
	\partial^\ell[\phi(x)] = \int_{\R} \partial^\ell[e^{2\pi i s x}] \widehat{\phi}(s)\,ds,
	\]
	so that $\norm{\partial^\ell[\phi(x)]}_{C_{\tr}(\R,\mathscr{M}^\ell),R} \leq \int_{\R} |(2\pi i s)^k \widehat{\phi}(s)|\,ds$ for all $R$, which implies that $\phi \in BC_{\tr}^k(\R)$.
	
	(2) Since the definition of $C_{\tr}^k(\R)$ requires approximation of $\phi(x)$ and its derivatives on each operator norm ball, it suffices to show that $\phi(x)$ agrees with a $C_{\tr}^k(\R)$ function on each operator norm ball.  Fix $R$, and let $\psi \in C_c^{k+2}(\R)$ such that $\psi|_{[-R,R]} = \phi|_{[-R,R]}$.  Clearly, $\psi(x)$ agrees with $\phi(x)$ on the operator norm ball of radius $R$.  Note that $s^\ell \widehat{\psi}(s)$ is bounded for $\ell \leq k+2$.  In particular, $(1 + |s|^k) |\widehat{\psi}(s)|$ is bounded by a constant times $1 / (1 + s^2)$, and hence it is integrable.  Thus, (1) shows that $\psi \in C_{\tr}^k(\R)$ as required.
\end{proof}

The following is a technical variant of the previous proposition which we will use later in the proof of Theorem \ref{thm:magicnormbound}.  The point is that we can control $\partial \phi(x)$ with only information about $\widehat{\phi'}$ and not $\widehat{\phi}$.

\begin{lemma} \label{lem:Fouriertransformmisc}
	Suppose that $\phi \in C^1(\R)$ with polynomial growth at $\infty$.  If $s \widehat{\phi}(s)$ is in $C(\R) \cap L^1(\R)$, then $\phi(x) \in C_{\tr}^1(\R)$ with $\partial \phi(x) \in BC_{\tr}(\R,\mathscr{M}(\R^{*1}))$.
\end{lemma}

\begin{proof}
	Note that for any $R > 0$, $(1 - e^{-Rs^2}) \widehat{\phi}(s)$ is in $C(\R) \cap L^1(\R)$.  Thus, we may define
	\[
	\phi_R(t) = \int_{\R} e^{2\pi i ts} (1 - e^{-Rs^2}) \widehat{\phi}(s)\,ds.
	\]
	Thus, $\widehat{\phi_R}(s) = (1 - e^{-Rs^2}) \widehat{\phi}(s)$ and $\widehat{\phi_R'}(s) = 2\pi i s (1 - e^{-Rs^2}) \widehat{\phi}(s)$.  Because $2\pi i s \widehat{\phi}(s)$ is in $L^1(\R) \cap C(\R)$, we have $2 \pi i s (1 - e^{-Rs^2}) \widehat{\phi}(s) \to 2\pi i s \widehat{\phi}(s)$ in $L^1(\R)$ as $R \to \infty$.  In particular, it follows that $\phi_R' \to \phi'$ uniformly, hence $\phi_R - \phi_R(0) \to \phi - \phi(0)$ uniformly on compact sets, and so $\phi_R(x) - \phi_R(0) + \phi(0) \to \phi(x)$ in $C_{\tr}(\R)$.  Now because $2\pi i s \widehat{\phi}_R(s) \to 2\pi i s \widehat{\phi}(s)$ in $L^1(\R)$, we see in particular that $2\pi i s \widehat{\phi}_R(s)$ is Cauchy in $L^1(\R)$ as $R \to \infty$, and hence $\partial \phi_R(x)$ is Cauchy in $BC_{\tr}(\R,\mathscr{M}(\R^{*1}))$ as $R \to \infty$, and thus converges to some limit.  The limit must give the Fr\'echet derivative of $\phi(x)$ and hence $\phi \in C_{\tr}^1(\R)$ and $\partial \phi \in BC_{\tr}(\R)$.
\end{proof}

\subsection{The gradient, divergence, and Laplacian}

A function $f \in \tr(C_{\tr}^1(\R^{*d}))$ defines for each $(\cA,\tau) \in \mathbb{W}$ a map $\cA_{\sa}^d \to \C$.  Since $\cA_{\sa}^d$ is contained in the Hilbert space $L^2(\cA,\tau)_{\sa}^d$, it makes sense at least formally to speak of the gradient of $f$.  In fact, taking $\cA = M_N(\C)$ with its canonical trace $\tr_N$, we obtain a $C^1$ function $f^{M_N(\C),\tr_N}: M_N(\C)_{\sa}^d \to \C$, which certainly has a gradient with respect to the inner product coming from $\tr_N$.  The rigorous construction of the gradient in fact makes sense for $f \in \tr(\C_{\tr}^1(\R^{*d},\mathscr{M}(\R^{*d_1},\dots,\R^{*d_\ell})))$.  We start with an auxiliary technical lemma.

\begin{lemma} \label{lem:duality}
	There is a Fr\'echet-space isomorphism
	\[
	\Phi: \tr(C_{\tr}(\R^{*d},\mathscr{M}(\R^{*d_1},\dots,\R^{*d_\ell},\R^{*d}))) \to C_{\tr}(\R^{*d},\mathscr{M}(\R^{*d_1},\dots,\R^{*d_\ell}))^d
	\]
	such that $\Phi(g)$ is the unique element satisfying
	\begin{equation} \label{eq:duality1}
		g^{\cA,\tau}(\mathbf{X})[\mathbf{Y}_1,\dots,\mathbf{Y}_\ell,\mathbf{Y}] = \ip{\mathbf{Y},  \Phi(g)^{\cA,\tau}(\mathbf{X})[\mathbf{Y}_1,\dots,\mathbf{Y}_\ell] }_{\tau}.
	\end{equation}
	Furthermore, we have
	\begin{equation} \label{eq:duality2}
		\norm{\Phi(g)}_{C_{\tr}(\R^{*d},\mathscr{M}^\ell),R} \leq \norm{g}_{C_{\tr}(\R^{*d},\mathscr{M}^{\ell+1}),\R} \leq d \norm{\Phi(g)}_{C_{\tr}(\R^{*d},\mathscr{M}^\ell),R}
	\end{equation}
	Finally, for $k \in \N_0 \cup \{\infty\}$, $\Phi$ maps $\tr(C_{\tr}^k(\R^{*d},\mathscr{M}(\R^{*d_1},\dots,\R^{*d_\ell},\R^{*d}))$ isomorphically (as Fr\'echet spaces) onto $C_{\tr}^k(\R^{*d},\mathscr{M}(\R^{*d_1},\dots,\R^{*d_\ell})^d$, and it satisfies
	\begin{equation} \label{eq:duality3}
		\partial^{k'}(\Phi(g)) = \Phi((\partial^{k'} g)_\sigma) \text{ for } k' \leq k, 
	\end{equation}
	where $\sigma$ is the permutation of $\{1,\dots,\ell+1+k'\}$ that moves $\ell + 1$ to the last position and leaves the other indices in the same order.
\end{lemma}

\begin{proof}
	Consider a trace polynomial $g$ in $C_{\tr}(\R^{*d},\mathscr{M}(\R^{*d_1},\dots,\R^{*d_\ell}))$ that is expressed as a product of monomials
	\[
	\tau(g_1(\mathbf{x},\mathbf{y}_1,\dots,\mathbf{y}_\ell)) \dots \tau(g_k(\mathbf{x},\mathbf{y}_1,\dots,\mathbf{y}_\ell)) \tau(h_1(\mathbf{x},\mathbf{y}_1,\dots,\mathbf{y}_\ell) y_i h_2(\mathbf{x},\mathbf{y}_1,\dots,\mathbf{y}_\ell)),
	\]
	such that the overall expression is multilinear in $\mathbf{y}_1$, \dots, $\mathbf{y}_\ell$, $\mathbf{y}$, where $\mathbf{y} = (y_1,\dots,y_d)$.  Then set
	\[
	\Phi(g) = (\underbrace{0,\dots,0}_{i-1}, h_2 h_1, \underbrace{0,\dots,0}_{d-i}).
	\]
	Straightforward computation checks that $\Phi(g)$ satisfies \eqref{eq:duality1}.  The map $\Phi$ extends to all trace polynomials by linearity.
	
	Next, we must be pass to the completion $C_{\tr}(\R^{*d},\mathscr{M}(\R^{*d_1},\dots,\R^{*d_\ell},\R^{*d}))$.  To this end, we first show \eqref{eq:duality2} in the special case where $g$ is a trace polynomial.  Let $(\cA,\tau) \in \mathbb{W}$, let $\mathbf{X} \in \cA_{\sa}^d$ with $\norm{\mathbf{X}}_\infty \leq R$, let $\alpha$, $\alpha_1$, \dots, $\alpha_\ell \in [1,\infty]$ with $1/\alpha = 1 / \alpha_1 + \dots + 1 / \alpha_\ell$, and let $\mathbf{Y}_j \in \cA^{d_j}$ with $\norm{\mathbf{Y}_j}_{\alpha_j} \leq 1$.  Let $1/\alpha + 1/\beta = 1$, and let $\mathbf{Y} \in \cA^d$ with $\norm{\mathbf{Y}}_\beta \leq 1$.  Then
	\[
	| \ip{\mathbf{Y}, \Phi(g)^{\cA,\tau}(\mathbf{X})[\mathbf{Y}_1,\dots,\mathbf{Y}_\ell]}_\tau | = |g^{\cA,\tau}(\mathbf{X})[\mathbf{Y}_1,\dots,\mathbf{Y}_\ell,\mathbf{Y}]| \leq \norm{g}_{C_{\tr}(\R^{*d},\mathscr{M}^{\ell+1}),R}.
	\]
	Since $\mathbf{Y}$ was arbitrary with $\norm{\mathbf{Y}}_\beta \leq 1$, we have
	\[
	\norm{\Phi(g)^{\cA,\tau}(\mathbf{X})[\mathbf{Y}_1,\dots,\mathbf{Y}_\ell]}_\alpha \leq \norm{g}_{C_{\tr}(\R^{*d},\mathscr{M}^{\ell+1}),R}.
	\]
	Then taking the supremum over $\mathbf{X}$, $\mathbf{Y}_1$, \dots, $\mathbf{Y}_\ell$ and $\alpha$, $\alpha_1$, \dots, $\alpha_\ell$ satisfying the conditions given above, and over $(\cA,\tau) \in \mathbb{W}$, we obtain
	\[
	\norm{\Phi(g)}_{C_{\tr}(\R^{*d},\mathscr{M}^\ell)^d,R} \leq \norm{g}_{C_{\tr}(\R^{*d},\mathscr{M}^{\ell+1}),R}.
	\]
	Conversely, to estimate $g$ in terms of $\Phi(g)$, let $(\cA,\tau)$ and $\mathbf{X}$ be as above and consider $\alpha$, $\alpha_1$, \dots, $\alpha_\ell$, $\beta$ with $1/\alpha = 1/ \alpha_1 + \dots + 1 /\alpha_\ell + 1 / \beta$.  For $j = 1$, \dots, $\ell$, let $\mathbf{Y}_j \in \cA^{d_j}$ with $\norm{\mathbf{Y}_j}_{\alpha_j} \leq 1$ and let $\mathbf{Y} \in \cA^d$ with $\norm{\mathbf{Y}}_\beta \leq 1$.  Let $\beta'$ be such that $1 / \alpha_1 + \dots + 1 / \alpha_\ell + 1 / \beta' = 1$.  Then $\beta' \leq \beta$ and hence $\norm{\mathbf{Y}}_{\beta'} \leq d \norm{\mathbf{Y}}_\beta \leq d$.  Since $g^{\cA,\tau}(\mathbf{X})[\mathbf{Y}_1,\dots,\mathbf{Y}_\ell,\mathbf{Y}]$ is a scalar, its norm in $L^\alpha(\cA,\tau)$ is equal to its absolute value, hence
	\[
	|g^{\cA,\tau}(\mathbf{X})[\mathbf{Y}_1,\dots,\mathbf{Y}_\ell] | = | \ip{\mathbf{Y}, \Phi(g)^{\cA,\tau}(\mathbf{X})[\mathbf{Y}_1,\dots,\mathbf{Y}_\ell]}_\tau | \leq d \norm{\Phi(g)}_{C_{\tr}(\R^{*d},\mathscr{M}^\ell),R}.
	\]
	Hence, \eqref{eq:duality2} holds when $f$ is a trace polynomial. It follows that the map $\Phi$ extends to the unique map
	\[
	\tr(C_{\tr}(\R^{*d},\mathscr{M}(\R^{*d_1},\dots,\R^{*d_\ell},\R^{*d})) \to C_{\tr}(\R^{*d},\mathscr{M}(\R^{*d_1},\dots,\R^{*d_\ell})^d
	\]
	and that this map (still denoted by $\Phi$) is injective.  To see that $\Phi$ is surjective, let $\mathbf{h} \in C_{\tr}(\R^{*d},\mathscr{M}(\R^{*d_1},\dots,\R^{*d_\ell}))^d$.  Let $g \in \tr(C_{\tr}(\R^{*d},\mathscr{M}(\R^{*d_1},\dots,\R^{*d_\ell},\R^{*d})))$ be given by
	\[
	g^{\cA,\tau}(\mathbf{X})[\mathbf{Y}_1,\dots,\mathbf{Y}_\ell,\mathbf{Y}] = \ip{\mathbf{Y}, \mathbf{h}^{\cA,\tau}(\mathbf{X})[\mathbf{Y}_1,\dots,\mathbf{Y}_\ell]}_\tau.
	\]
	Then $\Phi(g) = \mathbf{h}$.  So $\Phi$ is a linear isomorphism.  Continuity of $\Phi$ and $\Phi^{-1}$ is clear from \eqref{eq:duality2}.
	
	Finally, one checks \eqref{eq:duality3} directly from the characterization \eqref{eq:duality1} of $\Phi$, and it follows that $\Phi$ maps $\tr(C_{\tr}^k(\R^{*d},\mathscr{M}(\R^{*d_1},\dots,\R^{*d_\ell},\R^{*d})))$ isomorphically onto $C_{\tr}^k(\R^{*d},\mathscr{M}(\R^{*d_1},\dots,\R^{*d_\ell}))^d$.
\end{proof}

\begin{definition} \label{def:gradient}
	For $f \in \tr(C^1(\R^{*d},\mathscr{M}(\R^{*d_1},\dots,\R^{*d_\ell}))$, we define $\nabla f := \Phi(\partial f)$, where $\Phi$ is the map in the previous lemma.  Equivalently, $\nabla f$ is characterized by the relation that for every $(\cA,\tau)$, for $\mathbf{X} \in \cA_{\sa}^d$, and $\mathbf{Y}_1 \in \cA_{\sa}^{d_1}$, \dots, $\mathbf{Y}_\ell \in \cA_{\sa}^{d_\ell}$, and $\mathbf{Y} \in \cA_{\sa}^d$, we have
	\[
	(\partial f)^{\cA,\tau}(\mathbf{X})[\mathbf{Y}_1,\dots,\mathbf{Y}_\ell,\mathbf{Y}] = \ip{\mathbf{Y},  \nabla f^{\cA,\tau}(\mathbf{X})[\mathbf{Y}_1,\dots,\mathbf{Y}_\ell] }_{\tau}.
	\]
\end{definition}

The previous lemma implies in particular that for each $R > 0$,
\begin{equation} \label{eq:gradientestimate}
	\norm{\nabla f}_{C_{\tr}(\R^{*d},\mathscr{M}^\ell),R}
	\leq \norm{\partial f}_{C_{\tr}(\R^{*d},\mathscr{M}^{\ell+1}),R} 
	\leq d \norm{\nabla f}_{C_{\tr}(\R^{*d},\mathscr{M}^\ell),R}.
\end{equation}
Also, for $k \in \N_0 \cup \{\infty\}$, we have $f \in \tr(C_{\tr}^{k+1}(\R^{*d},\mathscr{M}(\R^{*d_1},\dots,\R^{*d_\ell})))$ if and only if $\nabla f$ is in $C_{\tr}^k(\R^{*d},\mathscr{M}(\R^{*d_1},\dots,\R^{*d_\ell}))^d$.  Intuition for the gradient comes from the following special cases.

\begin{remark}
	Suppose that $f(x) = \tau(\phi(x))$ for some $C^1$ function $\phi: \R \to \C$.  Then we claim that $f \in \tr(C_{\tr}^1(\R^{*d}))$ and $\nabla f(x) = \phi'(x)$.  To prove this, first consider the case where $\phi(t) = t^n$.  Then
	\[
	\partial f^{\cA,\tau}(X)[Y] = \sum_{j=0}^{n-1} \tau(X^j Y X^{n-1-j}) = \tau(n X^{n-1} Y) = \tau(\phi'(X)Y),
	\]
	so that $\nabla f^{\cA,\tau}(X) = \phi'(X)$.  By linearity, the same holds whenever $\phi$ is a polynomial.  Finally, if $\phi$ is $C^1$, then there exist polynomials $\phi_N$ such that $\phi_N \to \phi$ and $\phi_N' \to \phi'$ uniformly on compact subsets of $\R$.  Hence, $\nabla[\tr(\phi_N(x))] = \phi_N'(x) \to \phi'(x)$ in $C_{\tr}(\R)$, which implies that $\partial[\tr(\phi_N(x))]$ converges in $C_{\tr}(\R,\mathscr{M}(\R))$.  The limit clearly gives $\partial[\tr(\phi(x))]$, hence $\nabla[\tr(\phi(x))] = \phi'(x)$ as desired.
\end{remark}

\begin{remark}
	Suppose that $f(x) = \tau(p(x))$ for some non-commutative polynomial $p$.  Then $\nabla f$ as defined in Definition \ref{def:gradient} is the same as the cyclic gradient of the non-commutative polynomial $p$ introduced by Voiculescu in \cite{VoiculescuFE5,Voiculescu2002cyclo,Voiculescu2004}.  For further explanation, see \cite{Cebron2013}, \cite[\S 3]{DHK2013}, \cite[\S 14.1]{JekelThesis}.
\end{remark}

Consider the matrix algebra $(M_N(\C),\tr_N)$. Recall that $M_N(\C)_{\sa}^d$ with the inner product coming from $\tr_N$ is a real inner-product space of dimension $dN^2$, and hence can be mapped by a linear isometry onto $\R^{dN^2}$.  Hence, the classical gradient, divergence, Jacobian, and Hessian all make sense for $M_N(\C)_{\sa}^d$.  If $f \in \tr(C_{\tr}^1(\R^{*d}))$, then $f^{M_N(\C),\tr_N}: M_N(\C)_{\sa}^d \to \C$ has its gradient given by $(\nabla f)^{M_N(\C),\tr_N}$.  Moreover, if $\mathbf{f} \in C_{\tr}^1(\R^{*d})^d$, then the Jacobian matrix of $\mathbf{f}^{M_N(\C),\tr_N}(\mathbf{X})$ corresponds to the linear transformation $(\partial \mathbf{f})^{M_N(\C),\tr_N}(\mathbf{X}): M_N(\C)_{\sa}^d \to M_N(\C)^d$.

It is natural to ask whether the \emph{divergence} also has an analog defined on $C_{\tr}(\R^{*d})^d$.  Recall that if $\mathbf{f}: \R^d \to \C^d$, then $\Div(\mathbf{f}) = \sum_{j=1}^d \partial_j f_j$.  The divergence is the trace of the Jacobian matrix $Df$ (that is, the Fr\'echet derivative).  Moreover, it can be expressed in probabilistic terms as follows.  Let $\mathbf{Z}$ be a standard Gaussian (random) vector in $\R^d$.  Then
\[
\Div(\mathbf{f})(\mathbf{x}) = \Tr(D\mathbf{f}(\mathbf{x})) = \E[\ip{\mathbf{Z},D\mathbf{f}(\mathbf{x}) \mathbf{Z}}].
\]
Now the analog of the standard Gaussian vector in free probability is a standard semicircular family $\mathbf{S} = (S_1,\dots,S_d)$, where the $S_j$'s are freely independent of each other and each $S_j$ has the spectral measure $(1/2\pi) \sqrt{4 - t^2} \mathbf{1}_{[-2,2]}(t)\,dt$.  Let $(\cB,\sigma)$ be the tracial $\mathrm{W}^*$-algebra generated by the standard semicircular family $\mathbf{S}$.  Then we want to define, for $\mathbf{f} \in C_{\tr}^1(\R^{*d})^d$,
\[
\Div(\mathbf{f})^{\cA,\tau}(\mathbf{X}) = \ip{\mathbf{S}, \partial \mathbf{f}^{\cA*\cB,\tau*\sigma}(\mathbf{X})[\mathbf{S}]}_{\tau * \sigma},
\]
where $(\cA * \cB, \tau * \sigma)$ denotes the $\mathrm{W}^*$-algebraic free product of $(\cA,\tau)$ and $(\cB,\sigma)$.  As in the case of the gradient, we will phrase the definition in greater generality to work with multilinear forms.  As in the study of the gradient, we begin with an auxiliary technical lemma.

\begin{lemma} \label{lem:tracelike}
	Let $\ell \in \N_0$ and $d$, $d'$, $d_1$, \dots, $d_\ell \in \N$.  Let $(\cB,\sigma)$ be the tracial $\mathrm{W}^*$-algebra generated by a standard semicircular family $\mathbf{S}$.
	\begin{enumerate}[(1)]
		\item There exists a unique continuous map
		\[
		\Upsilon: C_{\tr}(\R^{*d}, \mathscr{M}(\R^{*d_1}, \dots, \R^{*d_\ell}, \R^{*d}, \R^{*d}))^{d'} \to C_{\tr}(\R^{*d}, \mathscr{M}(\R^{*d_1},\dots,\R^{*d_\ell}))^{d'}
		\]
		satisfying
		\begin{equation} \label{eq:upsilondefinition}
			\Upsilon(f)^{\cA,\tau}(\mathbf{X})[\mathbf{Y}_1,\dots,\mathbf{Y}_\ell] = E_{\cA}[f^{\cA*\cB,\tau*\sigma}(\mathbf{X})[\mathbf{Y}_1,\dots,\mathbf{Y}_\ell,\mathbf{S},\mathbf{S}]],
		\end{equation}
		where $E_{\cA}: \cA * \cB \to \cA$ is the unique trace-preserving conditional expectation.
		\item We have
		\[
		\norm{\Upsilon(f)}_{C_{\tr}(\R^{*d},\mathscr{M}^\ell)^{d'},R} \leq \norm{f}_{C_{\tr}(\R^{*d},\mathscr{M}^{\ell+2})^{d'},R}.
		\]
		\item For $k \in \N_0 \cup \{\infty\}$, $\Upsilon$ maps $C_{\tr}^k(\R^{*d}, \mathscr{M}(\R^{*d_1}, \dots, \R^{*d_\ell}, \R^{*d}, \R^{*d}))$ into \newline $C_{\tr}^k(\R^{*d}, \mathscr{M}(\R^{*d_1},\dots,\R^{*d_\ell})$, and we have
		\[
		\partial^{k'}(\Upsilon(f)) = \Upsilon((\partial^{k'}f)_\pi) \text{ for } k' \leq k,
		\]
		where $\pi$ is the permutation of $\{1,\dots,\ell+k'+2\}$ that moves the elements $\ell + 1$ and $\ell + 2$ to the end and keeps the others in the same order.
	\end{enumerate}
\end{lemma}

\begin{proof}
	First, we show that if $f \in C_{\tr}(\R^{*d}, \mathscr{M}(\R^{*d_1}, \dots, \R^{*d_\ell}, \R^{*d}, \R^{*d}))^{d'} \to C_{\tr}(\R^{*d}, \mathscr{M}(\R^{*d_1},\dots,\R^{*d_\ell}))^{d'}$ is a trace polynomial, then there is a trace polynomial $\Upsilon(f)$ satisfying \eqref{eq:upsilondefinition} (which is clearly uniquely determined by this relation).  We may consider each coordinate $1$, \dots, $d'$ individually and thus assume without loss of generality that $d' = 1$.  By linearity, it suffices to consider the case where $f = \tr(p_1) \dots \tr(p_n) q$ where $p_1$, \dots, $p_n$, $q$ are non-commutative monomials (and $f$ satisfies the appropriate multilinearity conditions).  We then consider the following cases.  To make the discussion clearer, we shall assume the polynomial is evaluated on some $(\cA,\tau)$, $\mathbf{X}$, $\mathbf{Y}_1$, \dots, $\mathbf{Y}_\ell$, and $\mathbf{S}$ as in \eqref{eq:upsilondefinition} when referring to the different arguments of the function, but of course the statements are equally valid for all instances of $(\cA,\tau)$, $\mathbf{X}$, and so forth.
	\begin{enumerate}[(a)]
		\item Suppose that one of the monomials $p_j$ is linear in $\mathbf{S}$, or more precisely, it contains one occurrence of $S_i$ for one value of $i$.  Then it will evaluate to zero by free independence.  Thus, we may take $\Upsilon(f) = 0$.
		\item Similarly, if one of the monomials $p_j$ contains an occurrence of $S_i$ and $S_j$ for $i \neq j$, then it has the form
		\[
		g_1(\mathbf{X}, \mathbf{Y}_1, \dots, \mathbf{Y}_\ell) S_i g_2(\mathbf{X}, \mathbf{Y}_1, \dots, \mathbf{Y}_\ell) S_j g_3(\mathbf{X}, \mathbf{Y}_1, \dots, \mathbf{Y}_\ell)
		\]
		where the $g_j$'s are non-commutative monomials.  By free independence, the trace will be zero, and hence we may again take $\Upsilon(f) = 0$.
		\item Suppose that one of the monomials $p_j$ contains two occurrences of $S_i$ for some $i$.  Then it has the form
		\[
		g_1(\mathbf{X}, \mathbf{Y}_1, \dots, \mathbf{Y}_\ell) S_i g_2(\mathbf{X}, \mathbf{Y}_1, \dots, \mathbf{Y}_\ell) S_i g_3(\mathbf{X}, \mathbf{Y}_1, \dots, \mathbf{Y}_\ell)
		\]
		where the $g_j$'s are non-commutative monomials.  By free independence the trace is $\tr(g_3 g_1) \tr(g_2)$ evaluated on $\mathbf{X}$, $\mathbf{Y}_1$, \dots, $\mathbf{Y}_\ell$.  Thus, $\Upsilon(f)$ is obtained from $f$ by replacing $\tr(p_j)$ with $\tr(g_3g_1) \tr(g_2)$.
		\item Suppose that $q$ contains an occurrence of $S_i$ and an occurrence of $S_j$ for $i \neq j$.  Then using free independence (similar to case (2)), we see that $E_{\cA}[q(\mathbf{X},\mathbf{Y}_1,\dots,\mathbf{Y}_\ell,\mathbf{S},\mathbf{S})] = 0$, so we can take $\Upsilon(f) = 0$.
		\item Suppose that $q$ contains two occurrences of $S_i$ for some $i$.  Then $q(\mathbf{X},\mathbf{Y}_1,\dots,\mathbf{Y}_\ell,\mathbf{S},\mathbf{S})$ can be written as
		\[
		g_1(\mathbf{X}, \mathbf{Y}_1, \dots, \mathbf{Y}_\ell) S_i g_2(\mathbf{X}, \mathbf{Y}_1, \dots, \mathbf{Y}_\ell) S_i g_3(\mathbf{X}, \mathbf{Y}_1, \dots, \mathbf{Y}_\ell).
		\]
		Since the remaining terms in $f$ are scalar-valued, they can be factored out of the conditional expectation $E_{\cA}$.  The conditional expectation onto $\cA$ of $q(\mathbf{X},\mathbf{Y}_1,\dots,\mathbf{Y}_\ell,\mathbf{S},\mathbf{S})$ will be
		\[
		g_1(\mathbf{X}, \mathbf{Y}_1, \dots, \mathbf{Y}_\ell) \tau[g_2(\mathbf{X}, \mathbf{Y}_1, \dots, \mathbf{Y}_\ell)] g_3(\mathbf{X}, \mathbf{Y}_1, \dots, \mathbf{Y}_\ell).
		\]
		Hence, $\Upsilon(f)$ will be obtained from $f$ by replacing $q$ by $g_1 g_3 \tr(g_2)$.
	\end{enumerate}
	
	Next, let us prove (2) for the trace polynomial case.  In all the above computations with free independence, we only had to use the first and second moments of $\mathbf{S}$ with respect to the trace $\sigma$.  Thus, we would have gotten the same result if we took $S_1$, \dots, $S_d$ to be freely independent operators, each of which has as its spectral distribution the Bernoulli measure $(1/2)(\delta_{-1} + \delta_1)$.  In particular, for these operators $\norm{\mathbf{S}}_\infty = 1$.  Thus, (2) follows directly from our definitions of the norms.
	
	Then using (2), we can extend the claim about existence of $\Upsilon(f)$ satisfying \eqref{eq:upsilondefinition} from the case of trace polynomial $f$ to general $f \in C_{\tr}(\R^{*d}, \mathscr{M}(\R^{*d_1}, \dots, \R^{*d_\ell}, \R^{*d}, \R^{*d}))^{d'}$.  The extended map $\Upsilon$ clearly still satisfies (2), which in turn implies it is continuous.
	
	Finally, to prove (3), the equality $\partial^{k'}(\Upsilon(f)) = \Upsilon((\partial^{k'}f)_\sigma)$ can be checked directly from \eqref{eq:upsilondefinition} since the substitution of $\mathbf{S}$ into two places commutes with the operation of Fr\'echet differentiation.  But the relation $\partial^{k'}(\Upsilon(f)) = \Upsilon((\partial^{k'}f)_\sigma)$ implies that $\Upsilon$ maps $C_{\tr}^k(\R^{*d}, \mathscr{M}(\R^{*d_1}, \dots, \R^{*d_\ell}, \R^{*d}, \R^{*d}))$ into $C_{\tr}^k(\R^{*d}, \mathscr{M}(\R^{*d_1},\dots,\R^{*d_\ell})$.
\end{proof}

\begin{remark} \label{rem:alternativeupsilon}
	In the proof, we saw that the ``cross terms'' that mix $S_i$ and $S_j$ for $i \neq j$ will cancel.  Thus, we can in fact rewrite $\Upsilon$ as
	\[
	\Upsilon(f)^{\cA,\tau}(\mathbf{X})[\mathbf{Y}_1,\dots,\mathbf{Y}_\ell] = \sum_{j=1}^d E_{\cA}[f^{\cA*\cB,\tau*\sigma}(\mathbf{X})[\mathbf{Y}_1,\dots,\mathbf{Y}_\ell,\tilde{S}_j,\tilde{S}_j]],
	\]
	where $\tilde{S}_j = (0,\dots,0,S_j,0,\dots,0)$ where $S_j$ occurs in the $j$th position.
\end{remark}

\begin{definition} \label{def:freedivergence}
	We define the \emph{divergence}
	\[
	\nabla^\dagger: C_{\tr}^1(\R^{*d})^d \to \tr(C_{\tr}(\R^{*d}))
	\]
	by $\nabla^\dagger = \Upsilon \circ \partial \circ \Phi^{-1}$ where $\Phi$ is as in Lemma \ref{lem:duality} and $\Upsilon$ is as in Lemma \ref{lem:tracelike}.  In other words,
	\[
	\nabla^\dagger (\mathbf{f})^{\cA,\tau}(\mathbf{X}) = \ip{\mathbf{S}, \partial \mathbf{f}^{\cA*\cB,\sigma*\tau}(\mathbf{X})[\mathbf{S}] }_{\tau * \sigma},
	\]
	where $(\cB,\sigma)$ is the tracial $\mathrm{W}^*$-algebra generated by a standard semicircular family $\mathbf{S} = (S_1,\dots,S_d)$.
\end{definition}

We can define a similar operation more generally on multilinear forms.

\begin{definition}
	Let $\ell \in \N_0$ and $d$, $d_1$,\dots, $d_\ell \in \N$, we define
	\[
	\partial^\dagger: C_{\tr}^1(\R^{*d},\mathscr{M}(\R^{*d_1},\dots,\R^{*d_\ell},\R^{*d})) \to C_{\tr}(\R^{*d},\mathscr{M}(\R^{*d_1},\dots,\R^{*d_\ell}))
	\]
	by $\partial^\dagger = \Upsilon \circ \partial$.
\end{definition}

This leads to the definition of the free Laplacian.

\begin{definition} \label{def:freeLaplacian}
	Define
	\[
	L: C_{\tr}^2(\R^{*d},\mathscr{M}(\R^{*d_1},\dots,\R^{*d_\ell}))^{d'} \to C_{\tr}(\R^{*d},\mathscr{M}(\R^{*d_1},\dots,\R^{*d_\ell}))^{d'}
	\]
	by $L:= \partial^\dagger \partial$.
\end{definition}

\begin{observation}
	If $f \in \tr(C_{\tr}^2(\R^{*d}))$, we have $Lf = \nabla^\dagger \nabla f$.
\end{observation}

\begin{remark}
	In the next section, we shall state an analog of the classical fact that the divergence is the trace of the Jacobian and the Laplacian is the trace of the Hessian after we discuss the trace on $C_{\tr}(\R^{*d},\mathscr{M}(\R^{*d}))^d$.
\end{remark}

\begin{remark} \label{rem:partialdifferentiation}
	There is a generalization of all the above differential operators to functions that depend not only on $\mathbf{X}$ but also on an auxiliary variable $\mathbf{X}'$.  More precisely, let $\ell \in \N_0$, let $d, d', d'' \in \N$, and let $d_1$, \dots, $d_\ell \in \N$.  Then we may consider $d''$-tuples of functions of $(\cA,\tau)$ and $\mathbf{X} \in \cA_{\sa}^d$, $\mathbf{X}' \in \cA_{\sa}^{d'}$, and $\mathbf{Y}_j \in \cA^{d_j}$.  Let
	\[
	\partial_{\mathbf{x}}: C_{\tr}^1(\R^{*(d+d')},\mathscr{M}(\R^{*d_1},\dots,\R^{*d_\ell})) \to C_{\tr}(\R^{*(d+d')},\mathscr{M}(\R^{*d_1},\dots,\R^{*d_\ell},\R^{*d}))
	\]
	be the operation of differentiation with respect to the first $d$-variables, which are represented by the formal variable $\mathbf{x} = (x_1,\dots,x_d)$.  Lemma \ref{lem:duality} generalizes to define an isomorphism
	\[
	\Phi: \tr(C_{\tr}(\R^{*(d+d')},\mathscr{M}(\R^{*d_1},\dots,\R^{*d_\ell},\R^{*d}))) \to C_{\tr}(\R^{*(d+d')},\mathscr{M}(\R^{*d_1},\dots,\R^{*d_\ell}))^d,
	\]
	and hence Definition \ref{def:gradient} generalizes to define $\nabla_{\mathbf{x}}$.  Moreover, Lemma \ref{lem:tracelike} generalizes to define a map
	\[
	C_{\tr}(\R^{*(d+d')}, \mathscr{M}(\R^{*d_1}, \dots, \R^{*d_\ell}, \R^{*d}, \R^{*d}))^{d''} \to C_{\tr}(\R^{*(d+d')}, \mathscr{M}(\R^{*d_1},\dots,\R^{*d_\ell}))^{d''}
	\]
	by
	\[
	\Upsilon(f)^{\cA,\tau}(\mathbf{X},\mathbf{X}')[\mathbf{Y}_1,\dots,\mathbf{Y}_\ell] = E_{\cA}[f^{\cA*\cB,\tau*\sigma}(\mathbf{X},\mathbf{X}')[\mathbf{Y}_1,\dots,\mathbf{Y}_\ell,\mathbf{S},\mathbf{S}]].
	\]
	Hence, we can define $\partial_{\mathbf{x}}^\dagger$ and $L_{\mathbf{x}}$ analogously to $\partial^\dagger$ and $L$.  Finally, if $L_{\mathbf{x}'}$ denotes the Laplacian with respect to the last $d'$ variables rather than the first $d$ variables, and if $L$ denotes the Laplacian with respect to the entire collection of variables $(\mathbf{x},\mathbf{x}')$, we have
	\[
	L_{\mathbf{x}} + L_{\mathbf{x}'} = L.
	\]
	This follows from Remark \ref{rem:alternativeupsilon}.
\end{remark}

\subsection{The $*$-algebra $C_{\tr}(\R^{*d},\mathscr{M}(\R^{*d}))^d$, its trace, and its log-determinant}

In this section, we endow $C_{\tr}(\R^{*d},\mathscr{M}(\R^{*d}))^d$ with the structure of a tracial $*$-algebra, which we view as a tracial non-commutative analog of $C(\R^d,M_d(\C))$ with the pointwise adjoint and trace operations.

Recall that if $\mathbf{F} \in C_{\tr}(\R^{*d},\mathscr{M}(\R^{*d}))^d$, then for each $(\cA,\tau) \in \mathbb{W}$ and $\mathbf{X} \in \cA_{\sa}^d$, $\mathbf{F}^{\cA,\tau}(\mathbf{X})$ defines a (complex) linear transformation $\cA^d \to \cA^d$. Moreover, for $\mathbf{F}, \mathbf{G} \in C_{\tr}(\R^{*d},\mathscr{M}(\R^{*d}))^d$, we have
\[
(\mathbf{F} \# \mathbf{G})^{\cA,\tau}(\mathbf{X})[\mathbf{Y}] = \mathbf{F}^{\cA,\tau}(\mathbf{X})[\mathbf{G}^{\cA,\tau}(\mathbf{X})[\mathbf{Y}]].
\]
By Lemma \ref{lem:composition}, $\mathbf{F} \# \mathbf{G} \in C_{\tr}(\R^{*d},\mathscr{M}(\R^{*d}))^d$, and more generally, by Theorem \ref{thm:chainrule}, if $\mathbf{F}$ and $\mathbf{G}$ are in $C_{\tr}^k(\R^{*d},\mathscr{M}(\R^{*d}))^d$, then so is $\mathbf{F} \# \mathbf{G}$.  In other words, $C_{\tr}^k(\R^{*d},\mathscr{M}(\R^{*d}))^d$ is an algebra under $\#$-multiplication.

Moreover, the identity element of $C_{\tr}(\R^{*d},\mathscr{M}(\R^{*d}))$ is the function $\Id$ given by
\[
\Id(\mathbf{x})[\mathbf{y}] = \mathbf{y}.
\]
(We use the lowercase $\id$ to denote the identity function in $C_{\tr}(\R^{*d})^d$.)

In fact, for $k \in \N_0$, $C_{\tr}(\R^{*d},\mathscr{M}(\R^{*d}))^d$ behaves like a Banach algebra in the following way.  This will be useful for proving smoothness of functions defined by $\#$-power series, such as the logarithm used in the proof of Proposition \ref{prop:logdeterminant}.

\begin{lemma} \label{lem:submultiplicativenorm}
	Let $k \in \N_0$.  For $\mathbf{F} \in C_{\tr}^k(\R^{*d},\mathscr{M}(\R^{*d}))^d$, define
	\[
	\norm{\mathbf{F}}_{C_{\tr}^k(\R^{*d},\mathscr{M}^1)^d,R} = \sum_{j=0}^k \frac{1}{j!} \norm{\partial^j \mathbf{F}}_{C_{\tr}(\R^{*d},\mathscr{M}^{1+j})^d,R}.
	\]
	Then
	\[
	\norm{\mathbf{F} \# \mathbf{G}}_{C_{\tr}^k(\R^{*d},\mathscr{M}^1,R} \leq \norm{\mathbf{F}}_{C_{\tr}^k(\R^{*d},\mathscr{M}^1)^d,R} \norm{\mathbf{G}}_{C_{\tr}^k(\R^{*d},\mathscr{M}^1)^d,R}.
	\]
\end{lemma}

\begin{proof}
	Let $k' \leq k$.  We apply the formula from Theorem \ref{thm:chainrule} to compute $\partial^{k'}[\mathbf{F} \# \mathbf{G}]$ by taking $n = 1$ and $\mathbf{f} = \mathbf{F}$ and $\mathbf{g} = \id$ and $\mathbf{h}_1 = \mathbf{G}$.  Note that $|B_i'| = 1$ and hence $|B_1| = k' - j$.  Since the blocks $B_i'$ must have their minimal elements ordered, they are uniquely determined by the choice of the block $B_1$.  Thus,
	\[
	\partial^{k'} [\mathbf{F} \# \mathbf{G}] = \sum_{B_1 \subseteq \{2,\dots,k'+1\}} \partial^{k' - |B_1|} \mathbf{F} \# [\partial^{|B_1|} \mathbf{G}, \Id, \dots, \Id]_\sigma,
	\]
	where $\sigma$ is the permutation sending $1$ to $1$ and mapping $2$, \dots, $1 + |B_1|$ onto $B_1$ and sending the rest of $2 + |B_1|$, \dots, $1 + k'$ in order onto the remaining points in $[k'+1]$.  For each $j \leq k'$, there are $k'$ choose $j$ choices of $B_1$ with $|B_1| = j$, which results in the estimate
	\[
	\norm{\partial^{k'} [\mathbf{F} \# \mathbf{G}]}_{C_{\tr}(\R^{*d},\mathscr{M}^{k'+1})^d,R} \leq \sum_{j=1}^{k'} \binom{k'}{j} \norm{\partial^{k'-j} \mathbf{F}}_{C_{\tr}(\R^{*d},\mathscr{M}^{1+k'-j})^d,R} \norm{\partial^j \mathbf{G}}_{C_{\tr}(\R^{*d},\mathscr{M}^{1+j})^d,R}.
	\]
	Hence,
	\begin{align*}
		\norm{\mathbf{F} \# & \mathbf{G}}_{C_{\tr}^k(\R^{*d},\mathscr{M}^1)^d,R} \\
		&= \sum_{k'=0}^k \frac{1}{k'!} \norm{\partial^{k'} [\mathbf{F} \# \mathbf{G}]}_{C_{\tr}(\R^{*d},\mathscr{M}^{k'+1})^d,R} \\
		&\leq \sum_{k'=0}^k \sum_{j=1}^{k'} \frac{1}{(k'-j)! j!} \norm{\partial^{k'-j} \mathbf{F} }_{C_{\tr}(\R^{*d},\mathscr{M}^{1+k'-j})^d,R} \norm{\partial^j \mathbf{G} }_{C_{\tr}(\R^{*d},\mathscr{M}^{1+j})^d,R} \\
		&\leq \left( \sum_{i=0}^k \frac{1}{i!} \norm{\partial^i \mathbf{F} }_{C_{\tr}(\R^{*d},\mathscr{M}^{1+i})^d,R} \right)
		\left( \sum_{j=1}^k \frac{1}{j!} \norm{\partial^j \mathbf{G} }_{C_{\tr}(\R^{*d},\mathscr{M}^{1+j})^d,R} \right) \\
		&= \norm{\mathbf{F}}_{C_{\tr}^k(\R^{*d},\mathscr{M}^1)^d,R} \norm{\mathbf{G}}_{C_{\tr}^k(\R^{*d},\mathscr{M}^1)^d,R}.  \qedhere
	\end{align*}
\end{proof}

Next, we claim that  $C_{\tr}^k(\R^{*d},\mathscr{M}(\R^{*d}))^d$ is a $*$-algebra with respect to some involution $\varstar$ that is compatible with the $\#$-multiplication structure.  Recall that we have already defined an involution $*$ by pointwise application of $*$, that is, $(\mathbf{F}^*)^{\cA,\tau}(\mathbf{X})[\mathbf{Y}] = \mathbf{F}^{\cA,\tau}(\mathbf{X})[\mathbf{Y}]^*$ for $\mathbf{X}$, $\mathbf{Y} \in \cA_{\sa}^d$.  However, this involution is analogous to applying entrywise complex conjugation to a matrix rather than taking the adjoint.  To prevent ambiguity, we will use the symbol $\varstar$ for the new adjoint operation.

\begin{lemma} \label{lem:staroperation}
	There exists a unique involution $\varstar$ on $C_{\tr}(\R^{*d},\mathscr{M}(\R^{*d}))^d$ such that for every $(\cA,\tau) \in \mathbb{W}$ and $\mathbf{X} \in \cA_{\sa}^d$ and $\mathbf{Y}_1$, $\mathbf{Y}_2 \in \cA^d$, we have
	\begin{equation} \label{eq:fancyadjoint}
		\ip{(\mathbf{F}^{\varstar})^{\cA,\tau}(\mathbf{X})[\mathbf{Y}_1], \mathbf{Y}_2}_\tau = \ip{\mathbf{Y_1}, \mathbf{F}^{\cA,\tau}(\mathbf{X})[\mathbf{Y}_2]}_\tau.
	\end{equation}
	Moreover, $\varstar$ defines a continuous map $C_{\tr}^k(\R^{*d}, \mathscr{M}(\R^{*d})) \to C_{\tr}^k(\R^{*d},\mathscr{M}(\R^{*d}))^d$ for every $k$ with
	\begin{equation} \label{eq:superadjointequality}
		\norm{\partial^k \mathbf{F}^{\varstar}}_{C_{\tr}(\R^{*d},\mathscr{M}^{k+1})^d,R} = \norm{\partial^k \mathbf{F}}_{C_{\tr}(\R^{*d},\mathscr{M}^{k+1}),R} \text{ for } R > 0,
	\end{equation}
	and hence for $k \in \N$ and $R > 0$,
	\begin{equation} \label{eq:superadjointequality2}
		\norm{\mathbf{F}^{\varstar}}_{C_{\tr}^k(\R^{*d},\mathscr{M}^1)^d,R} = \norm{\mathbf{F}}_{C_{\tr}^k(\R^{*d},\mathscr{M}^1)^d,R}.
	\end{equation}
	We also have
	\begin{equation} \label{eq:staroperation}
		(\mathbf{F} \# \mathbf{G})^{\varstar} = \mathbf{G}^{\varstar} \# \mathbf{F}^{\varstar}.
	\end{equation}
\end{lemma}

\begin{example}
	Let $p_{i,j}$ and $q_{i,j}$ for $i, j = 1$, \dots, $d$ be non-commutative polynomials (or more generally operator-valued trace polynomials).  Define $\mathbf{F} \in C_{\tr}(\R^{*d},\mathscr{M}(\R^{*d}))^d$ by
	\[
	(\mathbf{F}^{\cA,\tau}(\mathbf{X})[\mathbf{Y}])_i = \sum_{j=1}^d p_{i,j}(\mathbf{X}) Y_j q_{i,j}(\mathbf{X}),
	\]
	where $( \cdot)_i$ denotes the $i$th component of the $d$-tuple.  Then
	\[
	((\mathbf{F}^{\varstar})^{\cA,\tau}(\mathbf{X})[\mathbf{Y}])_i = \sum_{j=1}^d p_{j,i}(\mathbf{X})^* Y_j q_{j,i}(\mathbf{X})^*;
	\]
	this follows from the lemma and a direct computation with traciality that the expression here satisfies \eqref{eq:fancyadjoint} for $\mathbf{F}$.
	For another example, let $\mathbf{G}\in C_{\tr}(\R^{*d},\mathscr{M}(\R^{*d}))^d$ be given by
	\[
	(\mathbf{G}^{\cA,\tau}(\mathbf{X})[\mathbf{Y}])_i = \sum_{j=1}^d p_{i,j}(\mathbf{X}) \tau(Y_j q_{i,j}(\mathbf{X})).
	\]
	Then
	\[
	((\mathbf{G}^{\varstar})^{\cA,\tau}(\mathbf{X})[\mathbf{Y}])_i = q_{j,i}(\mathbf{X})^* \tau(Y_j p_{j,i}(\mathbf{X})^*).
	\]
\end{example}

\begin{proof}[Proof of Lemma \ref{lem:staroperation}]
	Let $\Phi: \tr(C_{\tr}(\R^{*d},\mathscr{M}^{k+2}(\R^{*d}))) \to C_{\tr}(\R^{*d},\mathscr{M}^{k+1}(\R^{*d}))^d$ be as in Lemma \ref{lem:duality} for each $k \in \N$.  Let $\sigma$ be the element of $\Perm(k+2)$ that switches the last $2$ indices.  Then we define $\Omega: C_{\tr}(\R^{*d},\mathscr{M}^{k+1}(\R^{*d}))^d \to C_{\tr}(\R^{*d},\mathscr{M}^{k+1}(\R^{*d}))^d$ by
	\[
	\Omega(\mathbf{F}) := \Phi(\Phi^{-1}(\mathbf{F})_\sigma^*),
	\]
	In the case $k = 1$, $\Omega$ defines a map from $C_{\tr}(\R^{*d},\mathscr{M}(\R^{*d}))$ to itself, and we define $F^{\varstar} := \Omega(\mathbf{F})$.  By Lemma \ref{lem:duality}, $\Omega$ is a continuous involution.  By direct computation from \eqref{eq:duality1}, for any $k$, for any $(\cA,\tau) \in \mathbb{W}$ and $\mathbf{X}$, $\mathbf{Y}$, $\mathbf{Y}_1$, \dots, $\mathbf{Y}_{k+1} \in \cA^d$, we have
	\[
	\ip{\Omega(\mathbf{F})^{\cA,\tau}(\mathbf{X})[\mathbf{Y}_1,\dots,\mathbf{Y}_{k+1}], \mathbf{Y}}_\tau = \ip{\mathbf{Y}_{k+1}, \mathbf{F}^{\cA,\tau}(\mathbf{X})[\mathbf{Y}_1,\dots,\mathbf{Y}_k,\mathbf{Y}]}_\tau,
	\]
	and hence in particular \eqref{eq:fancyadjoint} holds.  Moreover, for any $k$, if $1/\alpha = 1/\alpha_1 + \dots + 1/\alpha_{k+1}$ and $1/\alpha + 1/\beta = 1$, then
	\begin{align*}
		& \quad \norm{\Omega(\mathbf{F})^{\cA,\tau}(\mathbf{X})}_{\alpha;\alpha_1,\dots,\alpha_k} \\
		&= \sup \{ \norm{\Omega(\mathbf{F})^{\cA,\tau}(\mathbf{X})[\mathbf{Y}_1,\dots,\mathbf{Y}_{k+1}]}_{\tau,\alpha}:  \norm{\mathbf{Y}_j}_{\tau,\alpha_j} \leq  1 \} \\
		&= \sup \{ \ip{\mathbf{Y},\Omega(\mathbf{F})^{\cA,\tau}(\mathbf{X})[\mathbf{Y}_1,\dots,\mathbf{Y}_{k+1}]}_{\tau}:  \norm{\mathbf{Y}}_\beta \leq 1, \norm{\mathbf{Y}_j}_{\tau,\alpha_j} \leq  1 \} \\
		&= \sup \{ \ip{\mathbf{F}^{\cA,\tau}(\mathbf{X})[\mathbf{Y}_1,\dots,\mathbf{Y}_k,\mathbf{Y}], \mathbf{Y}_{k+1}}_{\tau}:  \norm{\mathbf{Y}}_\beta \leq 1, \norm{\mathbf{Y}_j}_{\tau,\alpha_j} \leq  1 \} \\
		&= \norm{\mathbf{F}^{\cA,\tau}(\mathbf{X})}_{(1 - 1/\alpha_{k+1})^{-1};\alpha_1,\dots,\alpha_k,\beta}.
	\end{align*}
	It follows that
	\[
	\norm{\Omega(\mathbf{F})}_{C_{\tr}(\R^{*d},\mathscr{M}^{k+1})^d,R} = \norm{\mathbf{F}}_{C_{\tr}(\R^{*d},\mathscr{M}^{k+1})^d,R}
	\]
	for all $R$.  Then we observe that $\partial[\Omega(\mathbf{F})] = \Omega[(\partial \mathbf{F})_\sigma]$, and hence by induction $\partial^j[\Omega(\mathbf{F})]$ is $\Omega$ of a permutation of $\partial^j \mathbf{F}$ whenever $\mathbf{F}$ is a $C_{\tr}^j$ function.  It follows that $\varstar$, which is the $k = 1$ case of $\Omega$, satisfies \eqref{eq:superadjointequality} and \eqref{eq:superadjointequality2}.  Finally, to show \eqref{eq:staroperation}, note that by \eqref{eq:fancyadjoint}, we have for any $(\cA,\tau)$, $\mathbf{X}$, $\mathbf{Y}_1$, $\mathbf{Y}_2 \in \cA_{\sa}^d$ that
	\[
	\ip{\mathbf{Y}_1, [(\mathbf{F} \# \mathbf{G})^{\varstar}]^{\cA,\tau}(\mathbf{X})[\mathbf{Y}_2]}_\tau = \ip{\mathbf{Y}_1, [\mathbf{G}^{\varstar} \# \mathbf{F}^{\varstar}]^{\cA,\tau}(\mathbf{X})[\mathbf{Y}_2]}_\tau.
	\]
	By linearity, the same relation holds if $\mathbf{Y}_1$ is taken from $\cA^d$ rather than $\cA_{\sa}^d$.  This implies that $[(\mathbf{F} \# \mathbf{G})^{\varstar}]^{\cA,\tau}(\mathbf{X})[\mathbf{Y}_2] = \mathbf{Y}_1, [\mathbf{G}^{\varstar} \# \mathbf{F}^{\varstar}]^{\cA,\tau}(\mathbf{X})[\mathbf{Y}_2]$, and since $(\cA,\tau)$, $\mathbf{X}$, and $\mathbf{Y}_2$ were arbitrary \eqref{eq:staroperation} holds.
\end{proof}

Next, we construct a trace functional on $C_{\tr}(\R^{*d},\mathscr{M}(\R^{*d}))^d$.

\begin{lemma} \label{lem:tracehash}
	There exists a unique linear functional $\Tr_{\#}: C_{\tr}(\R^{*d},\mathscr{M}(\R^{*d}))^d \to \tr(C_{\tr}(\R^{*d}))$ satisfying
	\begin{equation} \label{eq:semicirculartrace}
		[\Tr_\#(\mathbf{F})]^{\cA,\tau}(\mathbf{X}) = \ip{\mathbf{S},\mathbf{F}^{\cA*\cB,\tau*\sigma}(\mathbf{X})[\mathbf{S}]}_{\tau*\sigma}
	\end{equation}
	for $(\cA,\tau) \in \mathbb{W}$, where $(\cB,\sigma)$ is the tracial $\mathrm{W}^*$-algebra generated by a standard free semicircular family $\mathbf{S} = (S_1,\dots,S_d)$. We have
	\begin{equation} \label{eq:adjointtrace}
		\Tr_\#(\mathbf{F}^{\varstar}) = \Tr_\#(\mathbf{F})^*
	\end{equation}
	and
	\begin{equation} \label{eq:traciality}
		\Tr_\#(\mathbf{F} \# \mathbf{G}) = \Tr_\#(\mathbf{G} \# \mathbf{F}).
	\end{equation}
	Furthermore, $\Tr_{\#}$ maps $C_{\tr}^k(\R^{*d},\mathscr{M}(\R^{*d}))^d$ into $\tr(C_{\tr}^k(\R^{*d}))$ for each $k \in \N_0 \cup \{\infty\}$, and we have for $k' \leq k$ that
	\begin{equation} \label{eq:traceboundedness}
		\norm{\partial^{k'} \Tr_{\#}(\mathbf{F})}_{C_{\tr}(\R^{*d},\mathscr{M}^{1+k'}),R} \leq d \norm{\partial^{k'} \mathbf{F}}_{C_{\tr}(\R^{*d},\mathscr{M}^{1+k'})^d,R}.
	\end{equation}
\end{lemma}

\begin{proof}
	We define $\Tr_{\#}(\mathbf{F}) = \Upsilon \circ \Phi^{-1}(\mathbf{F})$ where $\Phi$ is as in Lemma \ref{lem:duality} and $\Upsilon$ is as in Lemma \ref{lem:tracelike}.  Then \eqref{eq:semicirculartrace} is verified from the definitions of $\Phi$ and $\Upsilon$.  The relation \eqref{eq:adjointtrace} follows because
	\[
	\ip{\mathbf{S},\mathbf{F}^{\cA*\cB,\tau*\sigma}(\mathbf{X})[\mathbf{S}]}_{\tau*\sigma} = \ip{(\mathbf{F}^{\varstar})^{\cA*\cB,\tau*\sigma}(\mathbf{X})[\mathbf{S}],\mathbf{S}}_{\tau*\sigma} = \overline{\ip{\mathbf{S},(\mathbf{F}^{\varstar})^{\cA*\cB,\tau*\sigma}(\mathbf{X})[\mathbf{S}]}_{\tau*\sigma}}.
	\]
	The claim about $C_{\tr}^k$ functions and \eqref{eq:traceboundedness} follow from \eqref{eq:duality2} and \eqref{eq:duality3} together with Lemma \ref{lem:tracelike} (2) and (3).
	
	It remains to prove \eqref{eq:traciality}.  By density and by continuity of the composition operations, it suffices to consider elements $\mathbf{F}$, $\mathbf{G}$ of $C_{\tr}(\R^{*d},\mathscr{M}(\R^{*d}))^d$ given by trace polynomials.  Then there are trace polynomials $F_{i,j,k,\ell}$ fo $i, j \in [d]$ and $k = 1,\dots,K$ and $\ell = 1, \dots, 4$ such that for all $(\cA,\tau)$,
	\[
	F_i^{\cA,\tau}(\mathbf{X})[\mathbf{Y}] = \sum_{k=1}^K \sum_{j=1}^d \left( F_{i,j,k,1}^{\cA,\tau}(\mathbf{X}) Y_j F_{i,j,k,2}^{\cA,\tau}(\mathbf{X}) + F_{i,j,k,3}^{\cA,\tau}(\mathbf{X}) \tau(F_{i,j,k,4}^{\cA,\tau}(\mathbf{X})Y_j) \right)
	\]
	and similarly we may write
	\[
	G_i^{\cA,\tau}(\mathbf{X})[\mathbf{Y}] = \sum_{k'=1}^{K'} \sum_{j=1}^d \left( G_{i,j,k',1}^{\cA,\tau}(\mathbf{X}) Y_j G_{i,j,k',2}^{\cA,\tau}(\mathbf{X}) + G_{i,j,k',3}^{\cA,\tau}(\mathbf{X}) \tau(G_{i,j,k',4}^{\cA,\tau}(\mathbf{X})Y_j) \right).
	\]
	By free independence,
	\[
	(\tau*\sigma)(F_{i,j,k,4}^{\cA,\tau}(\mathbf{X})S_j) = 0
	\]
	so that
	\[
	F_i^{\cA*\cB,\tau*\sigma}(\mathbf{X})[\mathbf{S}] = F_i^{\cA,\tau}(\mathbf{X})[\mathbf{Y}] = \sum_{k=1}^K \sum_{j=1}^d F_{i,j,k,1}^{\cA,\tau}(\mathbf{X}) S_j F_{i,j,k,2}^{\cA,\tau}(\mathbf{X}).
	\]
	Again using free independence, we have
	\[
	(\tau*\sigma)(G_{i,m,k',4}^{\cA,\tau}(\mathbf{X}) F_{m,j,k,1}^{\cA,\tau}(\mathbf{X}) S_j F_{m,j,k,2}^{\cA,\tau}(\mathbf{X}) ) = 0.
	\]
	Hence,
	\[
	\left( \mathbf{G}^{\cA*\cB,\tau*\sigma}(\mathbf{X})[ \mathbf{F}^{\cA*\cB,\tau*\sigma}(\mathbf{X})[\mathbf{S}]] \right)_i = \sum_{k,k'} \sum_{j=1}^d \sum_{m=1}^d G_{i,m,k',1}^{\cA,\tau}(\mathbf{X}) F_{m,j,k,1}^{\cA,\tau}(\mathbf{X}) S_j F_{m,j,k,2}^{\cA,\tau}(\mathbf{X})
	G_{i,m,k',2}^{\cA,\tau}(\mathbf{X}),
	\]
	and thus
	\[
	\ip{\mathbf{S}, \mathbf{G}^{\cA*\cB,\tau*\sigma}(\mathbf{X})[ \mathbf{F}^{\cA*\cB,\tau*\sigma}(\mathbf{X})[\mathbf{S}]]}_{\tau * \sigma} = \sum_{k,k'} \sum_{i,j,m=1}^d (\tau * \sigma)\left[ S_i G_{i,m,k',1}^{\cA,\tau}(\mathbf{X}) F_{m,j,k,1}^{\cA,\tau}(\mathbf{X}) S_j F_{m,j,k,2}^{\cA,\tau}(\mathbf{X})
	G_{i,m,k',2}^{\cA,\tau}(\mathbf{X})\right].
	\]
	If $i \neq j$, then the trace of the expression in the sum is zero by free independence.  Moreover, the $i = j$ can be evaluated using free independence as follows:
	\begin{multline*}
		\sum_{k,k'} \sum_{j,m=1}^d (\tau * \sigma)\left[ S_j G_{i,m,k',1}^{\cA,\tau}(\mathbf{X}) F_{m,j,k,1}^{\cA,\tau}(\mathbf{X}) S_j F_{m,j,k,2}^{\cA,\tau}(\mathbf{X})
		G_{i,m,k',2}^{\cA,\tau}(\mathbf{X})\right] \\
		= \sum_{k,k'} \sum_{j,m=1}^d \tau\left[G_{i,m,k',1}^{\cA,\tau}(\mathbf{X}) F_{m,j,k,1}^{\cA,\tau}(\mathbf{X})\right] \tau\left[F_{m,j,k,2}^{\cA,\tau}(\mathbf{X})
		G_{i,m,k',2}^{\cA,\tau}(\mathbf{X})\right].
	\end{multline*}
	This expression is invariant if we switch $\mathbf{F}$ and $\mathbf{G}$, by applying traciality of $\tau$ and interchanging the indices $j$ and $m$.  Thus, \eqref{eq:traciality} holds.
\end{proof}

We will next discuss the log-determinant described by the trace $\Tr_\#$ on $C_{\tr}(\R^{*d},\mathscr{M}(\R^{*d}))^d$.  It is easiest to define this trace in terms of the Fuglede-Kadison determinant on tracial $\mathrm{W}^*$-algebras.  To this end, let us interpret the trace $\Tr_\#$ in terms of traces on a $\mathrm{C}^*$-algebra.

Observe that for each $(\cA,\tau) \in \mathbb{W}$ and each $\mathbf{X} \in \cA_{\sa}^d$ with $\norm{\mathbf{X}}_\infty \leq R$, the function $\mathbf{F}(\mathbf{X})$ defines a bounded linear transformation $\pi_{\mathbf{X}}^{\cA,\tau}(\mathbf{F}): L^2(\mathcal{A},\tau)^d \to L^2(\mathcal{A},\tau)^d$ with
\[
\norm{\pi_{\mathbf{X}}^{\cA,\tau}(\mathbf{F})} \leq \norm{\mathbf{F}}_{C_{\tr}(\R^{*d},\mathscr{M}(\R^{*d}))^d,R}.
\]
We define a $\mathrm{C}^*$-semi-norm on $C_{\tr}(\R^{*d},\mathscr{M}(\R^{*d}))^d$ by
\[
\norm{\mathbf{F}}_{\mathbf{C}^*,R} = \sup \{ \norm{\pi_{\mathbf{X}}^{\cA,\tau}(\mathbf{F})}: (\cA,\tau) \in \mathbb{W}, \mathbf{X} \in \cA_{\sa}^d, \norm{\mathbf{X}}_\infty \leq R \}.
\]
The separation-completion of $C_{\tr}(\R^{*d},\mathscr{M}(\R^{*d}))^d$ with respect to this seminorm is thus a $\mathrm{C}^*$-algebra.  We will (temporarily) denote this $\mathbf{C}^*$-algebra by $\mathcal{C}_R$ and the quotient map $C_{\tr}(\R^{*d},\mathscr{M}(\R^{*d}))^d \to \mathcal{C}_R$ by $\pi_R$.  Letting $(\cB,\sigma)$ be the tracial $\mathrm{W}^*$-algebra generated by a free semicircular family $\mathbf{S}$, we have
\[
|\Tr_{\#}(\mathbf{F})^{\cA,\tau}(\mathbf{X})| = \ip{\mathbf{S}, \pi_{\mathbf{X}}^{\cA*\cB,\tau*\sigma}(\mathbf{F}) \mathbf{S}}_{\tau * \sigma} \leq d \norm{\pi_{\mathbf{X}}^{\cA*\cB,\tau*\sigma}(\mathbf{F})}.
\]
Thus, $\mathbf{F} \mapsto (1/d) \Tr_{\#}(\mathbf{F})^{\cA,\tau}$ passes to a well-defined tracial state $\tr_{\mathbf{X}}^{\cA,\tau}$ on the $\mathrm{C}^*$-algebra $\mathcal{C}_R$.  In particular, after constructing the GNS representation of $\mathcal{C}_R$ associated to $\tr_{\mathbf{X}}^{\cA,\tau}$, we can obtain a tracial $\mathrm{W}^*$-algebra as the WOT-closure of the image of this representation.

For an algebra $\mathcal{A}$, let $GL(\mathcal{A})$ denote the group of invertible elements.  For $\mathbf{F} \in GL(C_{\tr}(\R^{*d},\mathscr{M}(\R^{*d}))^d)$ and $(\cA,\tau) \in \mathbb{W}$ and $\mathbf{X} \in \cA_{\sa}^d$ with $\norm{\mathbf{X}}_\infty \leq R$, consider the Fuglede-Kadison log-determinant
\[
\log \Delta_{\mathbf{X}}^{\cA,\tau}(\mathbf{F}) := d \tr_{\mathbf{X}}^{\cA,\tau} \log \pi_R(\mathbf{F}^{\varstar} \mathbf{F})^{1/2}.
\]
It follows from the work of Fuglede and Kadison \cite[Theorem 1, property $1^\circ$]{FK1952} that
\[
\log \Delta_{\mathbf{X}}^{\cA,\tau}(\mathbf{F} \# \mathbf{G}) = \log \Delta_{\mathbf{X}}^{\cA,\tau}(\mathbf{F}) + \log \Delta_{\mathbf{X}}^{\cA,\tau}(\mathbf{G}).
\]
Our goal is to show that if $\mathbf{F}$ is in $GL(C_{\tr}^k(\R^{*d},\mathscr{M}(\R^{*d}))^d)$, then the log-determinant defines a function in $\tr(C_{\tr}(\R^{*d}))$.  We will use the path-connectedness of the general linear group.

\begin{lemma}
	Let $k \in \N_0 \cup \{\infty\}$.  Then $GL(C_{\tr}^k(\R^{*d},\mathscr{M})^d)$ is path-connected.
\end{lemma}

\begin{proof}
	Let $\mathbf{tr} \in C_{\tr}^\infty(\R^{*d},\mathscr{M}(\R^{*d}))^d$ denote the function $\mathbf{tr}(\mathbf{x})[\mathbf{y}] = (\tr(y_1),\dots,\tr(y_d))$.  Note that $\mathbf{tr} \# \mathbf{tr} = \mathbf{tr}$ and $\mathbf{tr}^{\varstar} = \mathbf{tr}$.
	
	There is a $*$-homomorphism $\phi: M_d(\C) \to C_{\tr}^\infty(\R^{*d},\mathscr{M}(\R^{*d}))^d$ given by
	\[
	\phi(M)(\mathbf{x}) = \left(\sum_{j=1}^d m_{1,j} x_j, \dots, \sum_{j=1}^d m_{d,j} x_j \right).
	\]
	Since $\phi(M)$ commutes with the self-adjoint idempotent $\mathbf{tr}$, the $*$-algebra $\mathcal{N}$ generated by $\phi(M_d(\C))$ and $\mathbf{tr}$ is isomorphic to $M_d(\C) \oplus M_d(\C)$, where for matrices $M_1$, $M_2 \in M_n(\C)$, the element $M_1 \oplus M_2$ in $M_d(\C) \oplus M_d(\C)$ corresponds to $M_1 (\Id - \mathbf{tr}) + M_2 \mathbf{tr}$.  Thus, $GL(\mathcal{N})$ is path-connected.
	
	It remains to show that every $\mathbf{F}$ in $C_{\tr}^k(\R^{*d},\mathscr{M}(\R^{*d}))^d$ is path-connected to some element of $GL(\mathcal{N})$.  For $t \in [0,1]$, let $\mathbf{F}(t \id)$ be the composition of $\mathbf{F}$ with $t \id$.  By Theorem \ref{thm:chainrule}, $t \mapsto \mathbf{F}(t \id)$ is a continuous function $[0,1] \to C_{\tr}^k(\R^{*d},\mathscr{M}(\R^{*d}))^d$.  Since $\mathbf{F} \mapsto \mathbf{F}(t \id)$ is a $*$-homomorphism, $\mathbf{F}(t \id) \in GL(C_{\tr}^k(\R^{*d},\mathscr{M}(\R^{*d}))^d)$ for all $t$. Hence, $\mathbf{F}$ is path-connected to $\mathbf{F}(0) = \mathbf{F} \circ (0 \id)$ in $GL(C_{\tr}^k(\R^{*d},\mathscr{M}(\R^{*d}))^d)$.  In the case where $\mathbf{F}$ is a trace polynomial, it is easy to check that $\mathbf{F}(0) \in \mathcal{N}$ since all the monomials involving $\mathbf{x}$ will disappear when we compose with the zero function.  Since $\mathcal{N}$ is closed, it follows that $\mathbf{F}(0) \in \mathcal{N}$ for all $\mathbf{F} \in GL(C_{\tr}^k(\R^{*d},\mathscr{M}(\R^{*d}))^d)$.
\end{proof}

\begin{proposition} \label{prop:logdeterminant}
	Let $k \in \N_0 \cup \{\infty\}$.  Then there exists a unique map
	\[\log \Delta_\#: GL(C_{\tr}^k(\R^{*d},\mathscr{M}(\R^{*d}))^d) \to \tr(C_{\tr}^k(\R^{*d}))
	\]
	such that for each $(\cA,\tau) \in \mathbb{W}$ and $\mathbf{X} \in \cA_{\sa}^d$, we have
	\[
	(\log \Delta_\#(\mathbf{F}))^{\cA,\tau}(\mathbf{X}) = \log \Delta_{\mathbf{X}}^{\cA,\tau}(\mathbf{F}).
	\]
	Moreover, $\log \Delta_{\#}$ is a continuous group homomorphism with respect to multiplication in the domain and addition in the codomain.
\end{proposition}

\begin{proof}
	The claim for $k = \infty$ will follow if we can prove it for $k < \infty$, so assume $k < \infty$.  Let $\mathbf{F} \in GL(C_{\tr}^k(\R^{*d},\mathscr{M}(\R^{*d}))^d$, and fix $R > 0$.  Since there is a continuous path from $\mathbf{F}$ to $\Id$, we can write
	\[
	\mathbf{F} = \mathbf{F}_1 \dots \mathbf{F}_n
	\]
	with $\norm{\mathbf{F}_j^{\varstar} \mathbf{F}_j - \Id}_{C_{\tr}^k(\R^{*d},\mathscr{M}(\R^{*d}))^d} < 1$.  Then by additivity of the Fuglede-Kadison determinant, for each $(\cA,\tau) \in \mathbb{W}$ and $\mathbf{X} \in \cA_{\sa}^d$ with $\norm{\mathbf{X}}_\infty \leq R$, we have
	\[
	\log \Delta_{\mathbf{X}}^{\cA,\tau}(\mathbf{F}) = \sum_{j=1}^n \log \Delta_{\mathbf{X}}^{\cA,\tau}(\mathbf{F}_j).
	\]
	Since $\norm{\mathbf{F}_j^{\varstar} \mathbf{F}_j - \Id}_{C_{\tr}^k(\R^{*d},\mathscr{M}(\R^{*d}))^d,R} < 1$ and because of Lemma \ref{lem:submultiplicativenorm} we have convergence of the power series 
	\[
	\log_\#(\mathbf{F}_j^{\varstar} \mathbf{F}_j) = - \sum_{m=1}^\infty \frac{1}{m} (\id - \mathbf{F}_j^{\varstar} \mathbf{F}_j)^{\# m}
	\]
	with respect to $\norm{\cdot}_{C_{\tr}^k(\R^{*d},\mathscr{M}(\R^{*d}))^d,R}$.  Since the representation $\pi_{\mathbf{X}}^{\cA,\tau}$ is bounded by in norm by $\norm{\cdot}_{C_{\tr}(\R^{*d},\mathscr{M}(\R^{*d}))^d,R}$ and respects analytic functional calculus, we have
	\[
	\log \Delta_{\mathbf{X}}^{\cA,\tau}(\mathbf{F}_j) = - \frac{1}{2} \sum_{m=1}^\infty \frac{1}{m} (\Tr_\#[(\Id - \mathbf{F}_j^{\varstar} \mathbf{F}_j)^{\# m}])^{\cA,\tau}(\mathbf{X}).
	\]
	Because of convergence of the series
	\begin{equation} \label{eq:logseries}
		- \frac{1}{2} \sum_{j=1}^n \sum_{m=1}^\infty \frac{1}{m} \Tr_\#[(\Id - \mathbf{F}_j^{\varstar} \mathbf{F}_j)^{\# m}]
	\end{equation}
	in $\norm{\cdot}_{C_{\tr}^k(\R^{*d},\mathscr{M}(\R^{*d}))^d,R}$, it follows that $\log \Delta_{\mathbf{X}}^{\cA,\tau}(\mathbf{F})$ is a Fr\'echet-$C^k$ function of $\mathbf{X}$ on the ball over radius $R$, and that this function, as well as its derivatives up to order $k$, be approximated on the ball of radius $R$ of every $(\cA,\tau) \in \mathbb{W}$ by functions in $C_{\tr}^k(\R^{*d},\mathscr{M}(\R^{*d}))^d$, where the approximation of the $k'$ derivative occurs with respect to $\norm{\cdot}_{C_{\tr}(\R^{*d},\mathscr{M}^{k'}),R}$.  Since this holds for every $R$, we conclude that $\log \Delta_{\mathbf{X}}^{\cA,\tau}(\mathbf{F})$ defines a function $\log \Delta_{\#}(\mathbf{F})$ in $\tr(C_{\tr}^k(\R^{*d}))$.
	
	The fact that $\log \Delta_\#(\mathbf{F} \# \mathbf{G}) = \log \Delta_\#(\mathbf{F}) + \log \Delta_\#(\mathbf{G})$ follows immediately from additivity of the Fuglede-Kadison determinant.  Next, to prove continuity of $\log \Delta_\#$, it suffices to check continuity at the point $\Id$.  Fix $R > 0$.  Then in a neighborhood of $\Id$, the power series expansion $\log_\#$ converges uniformly with respect to $\norm{\cdot}_{C_{\tr}^k(\R^{*d},\mathscr{M}(\R^{*d}))^d,R}$, and hence in this neighborhood $\log \Delta_\#(\mathbf{F}^{\varstar} \mathbf{F})$ and its derivatives up to order $k$ depend continuously on $\mathbf{F}$ respect to $\norm{\cdot}_{C_{\tr}^k(\R^{*d},\mathscr{M}(\R^{*d}))^d,R}$ in the domain and $\sum_{k'=0}^k \norm{\partial^{k'}(\cdot)}_{C_{\tr}(\R^{*d},\mathscr{M}^{k'}),R}$ in the target space.
\end{proof}

The following gives an explicit formula for $\partial \log \Delta_\#(\mathbf{F})$ which is helpful for assessing the boundedness properties of the derivative.

\begin{lemma} \label{lem:differentiatedeterminant}
	Let $\mathbf{F} \in GL(C_{\tr}^1(\R^{*d},\mathscr{M}(\R^{*d}))^d)$ and let $\mathbf{G}$ be the $\#$-inverse of $\mathbf{F}$.  For $(\cA,\tau) \in \mathbb{W}$ and $\mathbf{X}$, $\mathbf{Y} \in \cA_{\sa}^d$, we have
	\[
	\partial[\log \Delta_\#(\mathbf{F})]^{\cA,\tau}(\mathbf{X})[\mathbf{Y}] = \ip*{\mathbf{S},[\mathbf{G} \# \partial \mathbf{F} + \mathbf{G}^{\varstar} \# \partial \mathbf{F}^{\varstar}]^{\cA*\cB,\tau*\sigma}(\mathbf{X})[\mathbf{S},\mathbf{Y}]}_{\tau * \sigma},
	\]
	where $(\cB,\sigma)$ is the tracial $\mathrm{W}^*$-algebra generated by a family of freely independent operators $\mathbf{S}$ each of which has mean zero and variance $1$.  In particular, if $\mathbf{G} \in BC_{\tr}(\R^{*d},\mathscr{M}(\R^{*d}))$ and $\partial \mathbf{F} \in BC_{\tr}(\R^{*d},\mathscr{M}^2)$, then $\partial [\log \Delta_\#(\mathbf{F})] \in BC_{\tr}(\R^{*d},\mathscr{M}(\R^{*d}))$.
\end{lemma}

\begin{proof}
	Let us compute the directional derivatives.  Fix $(\cA,\tau) \in \mathbb{W}$.  Let $\mathbf{X}$ and $\mathbf{Y} \in \cA_{\sa}^d$, and let
	\[
	\Phi(t) = \pi_{\mathbf{X}+t\mathbf{Y}}^{\cA * \cB, \tau * \sigma}(\mathbf{F}).
	\]
	Note that for $\mathbf{Z} \in \cA_{\sa}^d$,
	\[
	\frac{d}{dt} \biggr|_{t=0} [\Phi(t)\mathbf{Z}] = \partial \mathbf{F}^{\cA*\cB,\tau*\sigma}(\mathbf{X}+t\mathbf{Y})[\mathbf{Z},\mathbf{Y}].
	\]
	Note that $\partial \mathbf{F}^{\cA*\cB,\tau*\sigma}(\mathbf{X}+t\mathbf{Y})[-,\mathbf{Y}]$ defines a bounded operator on $L^2(\cA,\tau)^d$ which depends continuously on $t$, and hence $\Phi(t)$ is differentiable in the operator norm. In particular.  For $t$ in a neighborhood of zero, $\Phi(t)$ is contained in some interval of the form $[\epsilon, 2R - \epsilon]$.  We can compute $(d/dt)|_{t=0} \log \Phi(t)^* \Phi(t)$ using the power series for $\log$ centered at $R$.  If we also apply the fact that $\ip{\mathbf{S}, (-) \mathbf{S}}_{\tau * \sigma}$ is tracial on the algebra generated by $\Phi(0)$ and $\Phi'(0)$ (for the same reason that $\Tr_\#$ is a trace), we obtain 
	\begin{align*}
		\frac{d}{dt}\biggr|_{t=0} \ip*{\mathbf{S}, \frac{1}{2} \log \Phi(t)^* \Phi(t) \mathbf{S} }_{\tau * \sigma} &= \ip*{\mathbf{S}, (\Phi(0)^* \Phi(0))^{-1} \frac{d}{dt}|_{t=0}[\Phi(t)^* \Phi(t)] \mathbf{S}}_{\tau * \sigma} \\
		&= \ip*{\mathbf{S}, \Phi(0)^{-1}(\Phi(0)^*)^{-1} [\Phi'(0)^* \Phi(0) + \Phi(0)^* \Phi'(0)] \mathbf{S}}_{\tau * \sigma} \\
		&= \ip*{\mathbf{S},[(\Phi(0)^*)^{-1} \Phi'(0)^* + \Phi(0)^{-1} \Phi'(0)] \mathbf{S}}_{\tau * \sigma},
	\end{align*}
	where the last equality follows using traciality.  This reduces to the asserted formula.  The boundedness statement then follows by inspection from the formula and the definitions of the norms.
\end{proof}

\subsection{Large $N$ limits of differential operators on $M_N(\C)_{\sa}^d$}

We have defined non-commutative analogs of the gradient, divergence, and Laplacian as well as the trace on matrix-valued functions.  Note that if $\mathbf{f} \in C_{\tr}^1(\R^{*d})^d$, then $\partial \mathbf{f}$ is the analog of the Jacobian, and we have
\[
\nabla^\dagger \mathbf{f} = \Tr_{\#}(\partial \mathbf{f}).
\]
For $f \in \tr(C_{\tr}^2(\R^{*d}))$, the analog of the Hessian matrix would be $\partial \nabla f$, and it is straightforward to check that
\[
Lf = \Tr_{\#}(\partial \nabla f).
\]

Let us now explain how the differential operators on non-commutative smooth functions describe in some sense the large $N$ limit of differential operators on $M_N(\C)_{\sa}^d$.  We have already seen that if $f \in \tr(C^1(\R^{*d}))$, then $(\nabla f)^{M_N(\C),\tr_N}$ is the classical gradient of $f^{M_N(\C),\tr_N}$ as a function on the $dN^2$-dimensional inner product space $M_N(\C)_{\sa}^d$, where the inner product is the one defined by $\tr_N$.  If $\mathbf{f} \in C_{\tr}^1(\R^{*d})^d$, then the classical divergence of $\mathbf{f}^{M_N(\C),\tr_N}$ does not equal $(\nabla^\dagger \mathbf{f})^{M_N(\C),\tr_N}$ precisely, but they agree asymptotically as $N \to \infty$ in the following sense.

\begin{lemma} \label{lem:asymptoticdivergence}
	Let $\mathbf{f} \in C_{\tr}^1(\R^{*d})^d$.  Let $\Div(\mathbf{f}^{M_N(\C),\tr_N})$ denote the classical divergence of $\mathbf{f}^{M_N(\C),\tr_N}$ as a function on the inner product space $M_N(\C)_{\sa}^d$.  Then for every $R > 0$,
	\[
	\lim_{N \to \infty} \norm*{\frac{1}{N^2} \Div(\mathbf{f}^{M_N(\C),\tr_N}) - (\nabla^\dagger \mathbf{f})^{M_N(\C),\tr_N}}_{\tr,R} = 0,
	\]
	where $\norm{\cdot}_{\tr,R}$ is as in Definition \ref{def:basicnorm} for $\cA = M_N(\C)$.  Or more explicitly,
	\begin{multline*}
		\lim_{N \to \infty} \sup \biggl\{ \norm*{\frac{1}{N^2} \Div(\mathbf{f}^{M_N(\C),\tr_N})(\mathbf{X}) - (\nabla^\dagger \mathbf{f})^{M_N(\C),\tr_N}(\mathbf{X})}_\infty: \\ \mathbf{X} \in M_N(\C)_{\sa}^d, \norm{\mathbf{X}}_\infty \leq R \biggr\} = 0.
	\end{multline*}
\end{lemma}

Of course, the previous lemma also applies to the Laplacian of functions $f \in \tr(C_{\tr}(\R^{*d}))$ since the Laplacian is the divergence of the gradient.  Similar statements hold more generally for the Laplacian of functions $\mathbf{f} \in C_{\tr}(\R^{*d},\mathscr{M}(\R^{*d_1},\dots,\R^{*d_\ell}))$.  Note that $\mathbf{f}^{M_N(\C),\tr_N}$ is a map from $M_N(\C)_{\sa}^d$ to the vector space of multilinear forms $M_N(\C)_{\sa}^{d_1} \times \dots \times M_N(\C)_{\sa}^{d_\ell} \to M_N(\C)$.  The classical Laplacian of vector-valued functions on a real inner product space is defined as the sum of the second directional derivatives over an orthonormal basis (which is the same as choosing a vector basis for the target space and computing the Laplacian coordinatewise).  As per Remark \ref{rem:partialdifferentiation}, we will state the next lemma more generally in the case of the Laplacian with respect to a subset of the variables.

\begin{lemma} \label{lem:asymptoticLaplacian}
	Let $\mathbf{f} \in C_{\tr}^2(\R^{*(d+d')},\mathscr{M}(\R^{*d_1},\dots,\R^{*d_\ell}))$.  Let $\Delta_{\mathbf{x}}$ denote the Laplacian with respect to $\mathbf{x}$ of a function of variables $(\mathbf{x},\mathbf{x}') \in M_N(\C)_{\sa}^d \times M_N(\C)_{\sa}^{d'}$.  Then for every $R > 0$, we have
	\[
	\lim_{N \to \infty} \norm*{\frac{1}{N^2} \Delta_{\mathbf{x}}[\mathbf{f}^{M_N(\C),\tr_N}] - [L_{\mathbf{x}} \mathbf{f}]^{M_N(\C),\tr_N}}_{\mathscr{M}^\ell,\tr,R} = 0,
	\]
	where $\norm{\cdot}_{\mathscr{M}^\ell,\tr,R}$ is as in Definition \ref{def:basicnorm}.
\end{lemma}

Because the Laplacian and the divergence are both defined in terms of the map $\Upsilon$ in Lemma \ref{lem:tracelike} (and its generalization in Remark \ref{rem:partialdifferentiation}), Lemmas \ref{lem:asymptoticdivergence} and \ref{lem:asymptoticLaplacian} will follow from relating $\Upsilon$ to the trace map in the finite-dimensional setting, as we will do in Lemma \ref{lem:classicaltrace}.

We begin with some notation.  Let $d, d', \ell \in \N_0$ and $d''$, $d_1$, \dots, $d_\ell \in \N$.  Let $\mathscr{M}(M_N(\C)_{\sa}^{d_1},\dots,M_N(\C)_{\sa}^{d_\ell}; M_N(\C)^{d''})$ denote the space of real-multilinear forms $M_N(\C)_{\sa}^{d_1} \times \dots \times M_N(\C)_{\sa}^{d_\ell} \to M_N(\C)^{d''}$.

Let $\mathcal{E}$ be an orthonormal basis of $M_N(\C)_{\sa}^d$.  Then we define
\begin{multline*}
	\Upsilon^{(N)}: \mathscr{M}(M_N(\C)_{\sa}^{d_1},\dots,M_N(\C)_{\sa}^{d_\ell},M_N(\C)_{\sa}^d,M_N(\C)_{\sa}^d; M_N(\C)^{d''}) \\ \to \mathscr{M}(M_N(\C)_{\sa}^{d_1},\dots,M_N(\C)_{\sa}^{d_\ell}; M_N(\C)^{d''})
\end{multline*}
by
\begin{equation} \label{eq:upsilonN}
	(\Upsilon^{(N)}\Lambda)[\mathbf{Y}_1,\dots,\mathbf{Y}_\ell] = \sum_{\mathbf{E} \in \mathcal{E}} \Lambda[\mathbf{Y}_1,\dots,\mathbf{Y}_\ell,\mathbf{E},\mathbf{E}].
\end{equation}

\begin{lemma} \label{lem:asymptotictrace}
	Let $\Upsilon^{(N)}$ be as above and let
	\[
	\Upsilon: C_{\tr}(\R^{*(d+d')},\mathscr{M}(\R^{*d_1},\dots,\R^{*d_\ell},\R^{*d},\R^{*d}))^{d''} \to C_{\tr}(\R^{*(d+d')},\mathscr{M}(\R^{*d_1},\dots,\R^{*d_\ell}))^{d''}
	\]
	be given by
	\[
	(\Upsilon \mathbf{f})^{\cA,\tau}(\mathbf{X},\mathbf{X}')[\mathbf{Y}_1,\dots,\mathbf{Y}_\ell] = E_{\cA}[ \mathbf{f}^{\cA*\cB,\tau*\sigma}(\mathbf{X},\mathbf{X}')[\mathbf{Y}_1,\dots,\mathbf{Y}_\ell,\mathbf{S},\mathbf{S}]],
	\]
	where $(\cB,\sigma)$ is the tracial $\mathrm{W}^*$-algebra generated by a standard semicircular $d$-tuple $\mathbf{S}$.  Then for $\mathbf{f} \in C_{\tr}(\R^{*(d+d')},\mathscr{M}(\R^{*d_1},\dots,\R^{*d_\ell},\R^{*d},\R^{*d}))^{d''}$, for every $R > 0$,
	\begin{equation} \label{eq:asymptotictrace}
		\lim_{N \to \infty} \norm*{\Upsilon^{(N)} \mathbf{f}^{M_N(\C),\tr_N} - (\Upsilon \mathbf{f})^{M_N(\C),\tr_N}}_{\mathscr{M}^\ell,\tr,R} = 0.
	\end{equation}
\end{lemma}

\begin{proof}
	Note that we can also write
	\begin{equation} \label{eq:upsilonN2}
		(\Upsilon^{(N)}\Lambda)[\mathbf{Y}_1,\dots,\mathbf{Y}_\ell] = \mathbb{E} \Lambda[\mathbf{Y}_1,\dots,\mathbf{Y}_\ell,\mathbf{Z},\mathbf{Z}],
	\end{equation}
	where $\mathbf{Z}$ is a standard Gaussian random vector in $M_N(\C)_{\sa}^d$, that is, a Gaussian random vector with mean zero and covariance matrix $I$.  In this case $\mathbf{S}^{(N)} = (1/N^2) \mathbf{Z}$ is Gaussian unitary ensemble.  It is well-known that
	\[
	\mathbb{E} \norm{\mathbf{S}^{(N)}}_\infty^2 \leq C
	\]
	for some constant independent of $N$ (and in fact much more is true); see Lemma \ref{lem:operatornormtailbound} and the references cited in the discussion preceding that lemma.  It follows that for $\Lambda \in \mathscr{M}(M_N(\C)_{\sa}^{d_1},\dots,M_N(\C)_{\sa}^{d_\ell}, M_N(\C)_{\sa}^d, M_N(\C)_{\sa}^d; M_N(\C)^{d''})$, we have
	\[
	\norm{\Upsilon^{(N)} \Lambda}_{\mathscr{M}^\ell,\tr} \leq C \norm{\Lambda}_{\mathscr{M}^{\ell+2},\tr}.
	\]
	In particular, for $\mathbf{f} \in C_{\tr}(\R^{*(d+d')},\mathscr{M}(\R^{*d_1},\dots,\R^{*d_\ell},\R^{*d},\R^{*d}))^{d''}$, we have
	\[
	\norm{\Upsilon^{(N)} \mathbf{f}^{M_N(\C),\tr_N}}_{\mathscr{M}^\ell,\tr,R} \leq C \norm{\mathbf{f}}_{C_{\tr}(\R^{*d},\mathscr{M}^{\ell+2})^{d''},R}.
	\]
	Therefore, it suffices to prove \eqref{eq:asymptotictrace} for a dense set of $\mathbf{f} \in C_{\tr}(\R^{*(d+d')},\mathscr{M}(\R^{*d_1},\dots,\R^{*d_\ell},\R^{*d},\R^{*d}))^{d''}$, for instance for those given by trace polynomials.  Furthermore, it suffices to consider the case $d'' = 1$ since we can handle each coordinate of $\mathbf{f}$ individually.
	
	To evaluate $\Upsilon^{(N)}$ for trace polynomials, we use the following magic formula:
	\begin{equation} \label{eq:magicformula}
		\frac{1}{N^2} \sum_{\mathbf{E} \in \mathcal{E}} A E_i B E_j C = \mathbb{E} \left[ A S_i^{(N)} B S_j^{(N)} C\right] = \delta_{i=j} A \tr_N(B) C \text{ for } A, B, C \in M_N(\C).
	\end{equation}
	This can be proved, for instance, by direct computation using the orthonormal basis $\mathcal{E}_0$ given by $M_N(\C)_{\sa}$
	\[
	\mathcal{E}_0 = \{N^{1/2} E_{j,j}\}_{j=1}^N \cup \{(N/2)^{1/2}(E_{j,k} + E_{k,j})\}_{j < k} \cup \{(N/2)^{1/2}(iE_{j,k} - i E_{k,j})\}_{j < k}.
	\]
	For further detail, see \cite[Lemma 4.1]{Sengupta2008} or \cite[Proposition 3.1]{DHK2013}.  Furthermore, using traciality and the properties of orthonormal bases, we get
	\begin{multline} \label{eq:magicformula2}
		\frac{1}{N^2} \sum_{\mathbf{E} \in \mathcal{E}_0} \tr_N(AE_i) \tr_N(BE_jC) = \mathbb{E} \left[ \tr_N(AS_i) \tr_N(BS_jC) \right] \\
		= \frac{1}{N^2} \delta_{i=j} \tr_N(ACB) = \frac{1}{N^2} \delta_{i=j} \tr_N(BAC).
	\end{multline}
	This implies also that
	\begin{equation} \label{eq:magicformula3}
		\frac{1}{N^2} \sum_{\mathbf{E} \in \mathcal{E}_0} \tr_N(AE_i) BE_jC = \mathbb{E} \tr_N(AS_i) \tr_N(BS_jC) = \frac{1}{N^2} \delta_{i=j} BAC;
	\end{equation}
	this follows by computing the inner product of this matrix with any $D \in M_N(\C)$ using \eqref{eq:magicformula2} with $CD$ instead of $C$.
	
	By linearity, it suffices to evaluate $\Upsilon^{(N)}$ on the following types of polynomials in $C_{\tr}(\R^{*(d+d')},\mathscr{M}(\R^{*d_1},\dots,\R^{*d_\ell},\R^{*d},\R^{*d}))$.
	\begin{enumerate}[(a)]
		\item Suppose that
		\[
		f(\mathbf{x},\mathbf{x}')[\mathbf{y}_1,\dots,\mathbf{y}_\ell,\mathbf{s},\mathbf{s}] = f_1(\mathbf{x},\mathbf{x}',\mathbf{y}_1,\dots,\mathbf{y}_\ell) s_i f_2(\mathbf{x},\mathbf{x}',\mathbf{y}_1,\dots,\mathbf{y}_\ell) s_j f_3(\mathbf{x},\mathbf{x}',\mathbf{y}_1,\dots,\mathbf{y}_\ell),
		\]
		for some trace polynomials $f_1$, $f_2$, $f_3$.  Then we use \eqref{eq:magicformula} to compute that
		\begin{align*}
			\Upsilon^{(N)} & f^{M_N(\C),\tr_N}(\mathbf{X},\mathbf{X}')[\mathbf{Y}_1,\dots,\mathbf{Y}_\ell] \\
			&= \delta_{i=j} f_1^{M_N(\C),\tr_N}(\mathbf{X},\mathbf{X}',\mathbf{Y}_1,\dots,\mathbf{Y}_\ell) \tr_N[ f_2^{M_N(\C),\tr_N}(\mathbf{X},\mathbf{X}',\mathbf{Y}_1,\dots,\mathbf{Y}_\ell)] \\
			& \qquad f_3^{M_N(\C),\tr_N}(\mathbf{X},\mathbf{X}',\mathbf{Y}_1,\dots,\mathbf{Y}_\ell) \\
			&= (\Upsilon f)^{M_N(\C),\tr_N}(\mathbf{X},\mathbf{X}')[\mathbf{Y}_1,\dots,\mathbf{Y}_\ell].
		\end{align*}
		Hence, \eqref{eq:asymptotictrace} holds.
		\item Suppose that
		\[
		f(\mathbf{x},\mathbf{x}')[\mathbf{y}_1,\dots,\mathbf{y}_\ell,\mathbf{s},\mathbf{s}] = \tr[f_1(\mathbf{x},\mathbf{x}',\mathbf{y}_1,\dots,\mathbf{y}_\ell) s_i] f_2(\mathbf{x},\mathbf{x}',\mathbf{y}_1,\dots,\mathbf{y}_\ell) s_j f_3(\mathbf{x},\mathbf{x}',\mathbf{y}_1,\dots,\mathbf{y}_\ell).
		\]
		Then using \eqref{eq:magicformula3}, we get
		\begin{multline*}
			\Upsilon^{(N)} f^{M_N(\C),\tr_N}(\mathbf{X},\mathbf{X}')[\mathbf{Y}_1,\dots,\mathbf{Y}_\ell] \\
			= \frac{1}{N^2} \delta_{i=j} f_2^{M_N(\C),\tr_N}(\mathbf{X},\mathbf{X}',\mathbf{Y}_1,\dots,\mathbf{Y}_\ell) f_1(\mathbf{X},\mathbf{X}',\mathbf{Y}_1,\dots,\mathbf{Y}_\ell) \\ f_3^{M_N(\C),\tr_N}(\mathbf{X},\mathbf{X}',\mathbf{Y}_1,\dots,\mathbf{Y}_\ell).
		\end{multline*}
		As $N \to \infty$, the $\norm{\cdot}_{\mathscr{M}^\ell,\tr,R}$ of this expression tends to zero.  Moreover, $\Upsilon f = 0$, so \eqref{eq:asymptotictrace} holds.
		\item Finally, suppose that
		\begin{multline*}
			f(\mathbf{x},\mathbf{x}')[\mathbf{y}_1,\dots,\mathbf{y}_\ell,\mathbf{s},\mathbf{s}] \\
			= \tr[f_1(\mathbf{x},\mathbf{x}',\mathbf{y}_1,\dots,\mathbf{y}_\ell) s_i] \tr[f_2(\mathbf{x},\mathbf{x}',\mathbf{y}_1,\dots,\mathbf{y}_\ell) s_j] f_3(\mathbf{x},\mathbf{x}',\mathbf{y}_1,\dots,\mathbf{y}_\ell).
		\end{multline*}
		Then using \eqref{eq:magicformula2}, we see that $\norm{\Upsilon^{(N)} f^{M_N(\C),\tr_N}}_{\mathscr{M}^\ell,\tr,R} \to 0$ as $N \to \infty$, and also $\Upsilon f = 0$.
	\end{enumerate}
	This completes the argument.
\end{proof}

As consequences, we obtain Lemmas \ref{lem:asymptoticdivergence} and \ref{lem:asymptoticLaplacian} as well as the following lemma about the trace and log-determinant of linear transformations.

\begin{lemma} \label{lem:classicaltrace}
	Let $\mathbf{F} \in C_{\tr}(\R^{*d},\mathscr{M}^1)^d$. Then $\mathbf{F}^{M_N(\C),\tr_N}(\mathbf{X})$ defines a linear transformation $M_N(\C)^d \to M_N(\C)^d$, which has a well-defined trace $\Tr(\mathbf{F}^{M_N(\C),\tr_N}(\mathbf{X}))$.  Then for each $R > 0$,
	\[
	\lim_{N \to \infty} \sup_{\substack{\mathbf{X} \in M_N(\C)_{\sa}^d \\ \norm{\mathbf{X}}_\infty \leq R}} \left| \frac{1}{N^2} \Tr[\mathbf{F}^{M_N(\C),\tr_N}(\mathbf{X})] - [\Tr_{\#}(\mathbf{F})]^{M_N(\C),\tr_N}(\mathbf{X}) \right| = 0.
	\]
	Similarly, for each $\mathbf{F} \in GL(C_{\tr}(\R^{*d},\mathscr{M}^1)^d)$ and for every $R > 0$, we have
	\[
	\lim_{N \to \infty} \sup_{\substack{\mathbf{X} \in M_N(\C)_{\sa}^d \\ \norm{\mathbf{X}}_\infty \leq R}} \left| \frac{1}{N^2} \log \bigl|\det [\mathbf{F}^{M_N(\C),\tr_N}(\mathbf{X})] \bigr| - [\log \Delta_{\#}(\mathbf{F})]^{M_N(\C),\tr_N}(\mathbf{X}) \right| = 0.
	\]
\end{lemma}

\begin{proof}
	The first claim is immediate since the trace was defined in terms of $\Upsilon$ in Lemma \ref{lem:tracehash}.  The claim about the log-determinant follows by expressing the log-determinant as the trace of some function as in the proof of Proposition \ref{prop:logdeterminant}; see \eqref{eq:logseries}.
\end{proof}

We also have the following refinement which allows for uniform convergence on $\norm{\cdot}_2$-balls if $\partial \mathbf{F}$ is bounded.

\begin{lemma} \label{lem:classicaltrace2}
	Let $\mathbf{F} \in C_{\tr}^1(\R^{*d},\mathscr{M}^1)^d$ with $\partial \mathbf{F} \in BC_{\tr}(\R^{*d},\mathscr{M}^2)^d$.  Then for each $R > 0$,
	\[
	\lim_{N \to \infty} \sup_{\substack{\mathbf{X} \in M_N(\C)_{\sa}^d \\ \norm{\mathbf{X}}_2 \leq R}} \left| \frac{1}{N^2} \Tr[\mathbf{F}^{M_N(\C),\tr_N}(\mathbf{X})] - [\Tr_{\#}(\mathbf{F})]^{M_N(\C),\tr_N}(\mathbf{X}) \right| = 0.
	\]
	Similarly, if $\mathbf{F} \in GL(C_{\tr}^1(\R^{*d},\mathscr{M}^1))^d$ with $\#$-inverse given by $\mathbf{G}$, and if $\mathbf{G} \in BC_{\tr}(\R^{*d},\mathscr{M}(\R^{*d}))^d$ and $\partial \mathbf{F} \in BC_{\tr}(\R^{*d},\mathscr{M}(\R^{*d},\R^{*d}))^d$, then
	\[
	\lim_{N \to \infty} \sup_{\substack{\mathbf{X} \in M_N(\C)_{\sa}^d \\ \norm{\mathbf{X}}_2 \leq R}} \left| \frac{1}{N^2} \log \bigl|\det [\mathbf{F}^{M_N(\C),\tr_N}(\mathbf{X})] \bigr| - [\log \Delta_{\#}(\mathbf{F})]^{M_N(\C),\tr_N}(\mathbf{X}) \right| = 0.
	\]
\end{lemma}

\begin{proof}
	Fix $R > 0$ and $R' > 0$.  Let $\phi_{R'}(t) = \max(-R',\min(t,R'))$.  For $(\cA,\tau) \in \mathbb{W}$ and $X \in \cA_{\sa}$, we have
	\[
	\norm{\phi_{R'}(X) - X}_1 \leq \tau(1_{\R \setminus [-R',R']}(X)) \leq \frac{1}{R'} \norm{X}_2^2
	\]
	using properties of functional calculus and Chebyshev's inequality.  Hence, letting $\mathbf{g}_{R'}^{\cA,\tau}(\mathbf{X}) = (\phi_{R'}(X_1),\dots,\phi_{R'}(X_d))$, we have
	\[
	\norm{\mathbf{g}_{R'}^{\cA,\tau}(\mathbf{X}) - \mathbf{X}}_1 \leq \frac{1}{R'} \norm{\mathbf{X}}_2^2.
	\]
	Now $\mathbf{F}(\mathbf{g}_{R'}) \in C_{\tr}(\R^{*d},\mathscr{M}^1)^d$.  Moreover, if $\mathbf{S} \in \cA_{\sa}^d$ and if $\norm{\mathbf{X}}_2 \leq R$, then by Remark \ref{rem:Lipschitz},
	\[
	|\ip{\mathbf{S}, \mathbf{F}^{\cA,\tau}(\mathbf{g}_{R'}^{\cA,\tau}(\mathbf{X})) [\mathbf{S}] }_{\tau} - \ip{\mathbf{S}, \mathbf{F}^{\cA,\tau}(\mathbf{X})[\mathbf{S}]}_{\tau}| 
	\leq \norm{\partial \mathbf{F}}_{BC_{\tr}(\R^{*d},\mathscr{M}^2)} \norm{\mathbf{S}}_\infty^2 \norm{\mathbf{g}_{R'}^{\cA,\tau}(\mathbf{X}) - \mathbf{X}}_1
	\]
	In particular, since $\Tr(\mathbf{F}^{M_N(\C),\tr_N}(\mathbf{X})$ is computed using Gaussian random vectors by \eqref{eq:upsilonN2}, and since the Gaussian unitary ensemble $\mathbf{S}^{(N)}$ satisfies $\mathbb{E} \norm{\mathbf{S}^{(N)}}_\infty^2 \leq C$ for some constant $C$, this implies that for each $N$
	\[
	\sup_{\substack{\mathbf{X} \in M_N(\C)_{\sa}^d \\ \norm{\mathbf{X}}_2 \leq R}} \left| \frac{1}{N^2} \Tr[\mathbf{F}^{M_N(\C),\tr_N}(\mathbf{X})] - \frac{1}{N^2} \Tr[(\mathbf{F}  \circ \mathbf{g}_{R'})^{M_N(\C),\tr_N}(\mathbf{X})] \right| \leq C \frac{R^2}{R'} \norm{\partial \mathbf{F}}_{BC_{\tr}(\R^{*d},\mathscr{M}^2)}.
	\]
	A similar bound holds for the error from replacing $\mathbf{F}$ with $\mathbf{F} \circ \mathbf{g}_{R'}$ in $\Tr_{\#}$.  Since $\norm{g_{R'}^{\cA,\tau}(\mathbf{X})}_\infty \leq R'$, we have
	\[
	\lim_{N \to \infty} \sup_{\substack{\mathbf{X} \in M_N(\C)_{\sa}^d \\ \norm{\mathbf{X}}_2 \leq R}} \left| \frac{1}{N^2} \Tr[(\mathbf{F} \circ \mathbf{g}_{R'})^{M_N(\C),\tr_N}(\mathbf{X})] - [\Tr_{\#}(\mathbf{F} \circ \mathbf{g})]^{M_N(\C),\tr_N}(\mathbf{X}) \right| = 0.
	\]
	Thus,
	\[
	\limsup_{N \to \infty} \sup_{\substack{\mathbf{X} \in M_N(\C)_{\sa}^d \\ \norm{\mathbf{X}}_2 \leq R}} \left| \frac{1}{N^2} \Tr[\mathbf{F}^{M_N(\C),\tr_N}(\mathbf{X})] - [\Tr_{\#}(\mathbf{F})]^{M_N(\C),\tr_N}(\mathbf{X}) \right| \leq \frac{2R^2}{R'} \norm{\partial \mathbf{F}}_{BC_{\tr}(\R^{*d},\mathscr{M}^2)}.
	\]
	Since $R'$ was arbitrary, we have finished proving the first claim.  The proof of the second claim is similar using Lemma \ref{lem:differentiatedeterminant}.
\end{proof}

\section{The free Wasserstein manifold and diffeomorphism group} \label{sec:manifold}

This section will give the definition of the free Wasserstein manifold $\mathscr{W}(\R^{*d})$ consisting of non-commutative log-densities $V$, the non-commutative diffeomorphism group $\mathscr{D}(\R^{*d})$, and the transport action $\mathscr{D}(\R^{*d}) \curvearrowright \mathscr{W}(\R^{*d})$.  It will explain as many results as can be proved by computation, and then sketch other ideas that will be carried out rigorously in the rest of the paper when $V$ is sufficiently close to the quadratic function $(1/2) \ip{\mathbf{x},\mathbf{x}}_{\tr}$.

\subsection{Definition of the manifolds} \label{subsec:manifolddefinition}

\begin{definition}
	We define the free Wasserstein manifold $\mathscr{W}(\R^{*d})$ be the set of $V \in \tr(C_{\tr}^\infty(\R^{*d}))$ such that $a \ip{\mathbf{x},\mathbf{x}}_{\tr} + b \leq V \leq a' \ip{\mathbf{x},\mathbf{x}}_{\tr} + b'$ for some $a, a' > 0$ and $b, b' \in \R$, considered modulo additive constants.  Here the inequality means that for every $(\cA,\tau) \in \mathbb{W}$ and $\mathbf{X} \in \cA_{\sa}^d$, we have $a \norm{\mathbf{X}}_2^2 + b \leq V^{\cA,\tau}(\mathbf{X}) \leq a' \norm{\mathbf{X}}_2^2 + b'$.
\end{definition}

\begin{definition}
	We define the tangent space $T_V \mathscr{W}(\R^{*d})$ as the set of equivalence classes of continuously differentiable paths $t \mapsto V_t$ from some interval $(-\epsilon,\epsilon)$ to $\tr(C_{\tr}^\infty(\R^{*d}))_{\sa}$ such that $V_0 = V$ modulo constants and such that $a \ip{\mathbf{x},\mathbf{x}}_{\tr} + b \leq V \leq a' \ip{\mathbf{x},\mathbf{x}}_{\tr} + b'$ for some $a, a' > 0$ and $b, b' \in \R$.  Here $t \mapsto V_t$ and $t \mapsto W_t$ are considered to be equivalent if $\dot{V}_0 = \dot{W}_0$ modulo constant functions.  Here ``continuously differentiable'' is interpreted in terms of the Fr\'echet topology on $\tr(C_{\tr}(\R^{*d}))_{\sa}$.
\end{definition}

\begin{definition}
	For $k \in \N_0 \cup \{\infty\}$, we define $\Diff_{\tr}^k(\R^{*d})$ as the space of functions $\mathbf{f} \in C_{\tr}^k(\R^{*d})$ such that $\mathbf{f}$ has an inverse function $\mathbf{f}^{-1} \in C_{\tr}^k(\R^{*d})$.  Similarly, we define $\BDiff_{\tr}^k(\R^{*d})$ as the space of functions $\mathbf{f} \in \Diff_{\tr}^k(\R^{*d})$ such that $\partial \mathbf{f}$, \dots, $\partial^k \mathbf{f}$ and $\partial \mathbf{f}^{-1}$, \dots, $\partial^k \mathbf{f}^{-1}$ are bounded.  We also use the notation $\Diff_{\tr}(\R^{*d}) = \Diff_{\tr}^\infty(\R^{*d})$ and $\BDiff_{\tr}(\R^{*d}) = \BDiff_{\tr}^\infty(\R^{*d})$.
\end{definition}

	\begin{observation}
	It follows from the chain rule that $\Diff_{\tr}^k(\R^{*d})$ and $\BDiff_{\tr}^k(\R^{*d})$ are groups under composition.
\end{observation}

\begin{definition}
	Let $\mathscr{D}(\R^{*d}) := \Diff_{\tr}(\R^{*d}) \cap \BDiff_{\tr}^1(\R^{*d})$.  We define $T_{\mathbf{f}} \mathscr{D}(\R^{*d})$ as the set of continuously differentiable paths $t \mapsto \mathbf{f}_t$ from some interval $(-\epsilon,\epsilon)$ to $\mathscr{D}(\R^{*d})$ such that $\mathbf{f}_0 = \mathbf{f}$, the derivatives $\partial \mathbf{f}_t$ and $\partial \mathbf{f}_t^{-1}$ are uniformly bounded, and the maps $t \mapsto \mathbf{f}_t$ and $t \mapsto \mathbf{f}_t^{-1}$ are continuously differentiable $(-\epsilon,\epsilon) \to C_{\tr}^\infty(\R^{*d})$.  Here $t \mapsto \mathbf{f}_t$ and $t \mapsto \mathbf{g}_t$ are considered equivalent if $\dot{\mathbf{f}}_0 = \dot{\mathbf{g}}_0$.
\end{definition}

\begin{lemma} \label{lem:groupaction}
	There is a group action $\mathscr{D}(\R^{*d}) \curvearrowright \mathscr{W}(\R^{*d})$ given by
	\[
	(\mathbf{f},V) \mapsto \mathbf{f}_* V := V \circ \mathbf{f}^{-1} - \log \Delta_{\#}(\partial \mathbf{f}^{-1}).
	\]
	More generally, this formula defines an action $\Diff_{\tr}^{k+1}(\R^{*d}) \curvearrowright \tr(C_{\tr}^k(\R^{*d}))_{\sa}$.
\end{lemma}

\begin{proof}
	First, note that if $V \in \tr(C_{\tr}^k(\R^{*d}))_{\sa}$ and $\mathbf{f} \in \Diff_{\tr}^{k+1}(\R^{*d})$, then $\mathbf{f}_* V \in \tr(C_{\tr}^k(\R^{*d}))_{\sa}$.  Indeed, Theorem \ref{thm:chainrule} shows that $V \circ \mathbf{f}^{-1} \in \tr(C_{\tr}^k(\R^{*d}))_{\sa}$, and Proposition \ref{prop:logdeterminant} shows that $\log \Delta_\#(\partial \mathbf{f}^{-1}) \in \tr(C_{\tr}^k(\R^{*d}))_{\sa}$.
	
	To show that $\mathbf{f}_* (\mathbf{g}_* V) = (\mathbf{f} \circ \mathbf{g})_* V$, observe that
	\begin{align*}
		V \circ (\mathbf{f} \circ \mathbf{g})^{-1} - \log \Delta_{\#}(\partial(\mathbf{f} \circ \mathbf{g})^{-1}) &= (V \circ \mathbf{g}^{-1}) \circ \mathbf{f}^{-1} - \log \Delta_{\#}((\partial\mathbf{g}^{-1} \circ \mathbf{f}^{-1}) \# \partial \mathbf{f}^{-1}) \\
		&= (V \circ \mathbf{g}^{-1} - \log \Delta_{\#}(\partial \mathbf{g}^{-1})) \circ \mathbf{f}^{-1} - \log \Delta_{\#}(\partial \mathbf{f}^{-1}).
	\end{align*}
	
	To complete the proof that $\mathscr{D}(\R^{*d})$ acts on $\mathscr{W}(\R^{*d})$, it suffices to show that if $\mathbf{f} \in \BDiff_{\tr}^1(\R^{*d})$ and $V \in \tr(C_{\tr}(\R^{*d}))_{\sa}$ satisfies $a \ip{\mathbf{x},\mathbf{x}}_{\tr} + b \leq V \leq a' \ip{\mathbf{x},\mathbf{x}}_{\tr} + b'$, then $\mathbf{f}_* V$ satisfies similar bounds.  Now $\partial \mathbf{f}^{-1}$ and its inverse $\partial \mathbf{f} \circ \mathbf{f}^{-1}$ are both bounded.  This implies a uniform bound, independent of $R$, on the $\mathrm{C}^*$-norms $\norm{\partial \mathbf{f}^{-1}}_{\mathrm{C}^*,R}$ and $\norm{(\partial \mathbf{f}^{-1})^{\# -1}}_{\mathrm{C}^*,R}$ used in the definition of $\log \Delta_\#$.  Hence, $\log \Delta_{\#}(\partial \mathbf{f}^{-1})$ is bounded.  Thus, it remains to show that $V \circ \mathbf{f}^{-1}$ has quadratic upper and lower bounds.  But note that $\mathbf{f}^{-1}$ and $\mathbf{f}$ both have bounded first derivative, and thus they are both uniformly Lipschitz with respect to $\norm{\cdot}_2$ by Remark \ref{rem:Lipschitz}, and hence for all $(\cA,\tau) \in \mathbb{W}$ and $\mathbf{X} \in \cA_{\sa}^d$,
	\[
	\norm{\mathbf{f}^{-1}(0)}_2 + \frac{1}{\norm{\partial \mathbf{f}}_{BC_{\tr}(\R^{*d},\mathscr{M}^1)}} \norm{\mathbf{X}}_2 \leq \norm{(\mathbf{f}^{-1})^{\cA,\tau}(\mathbf{X})}_2 \leq \norm{\mathbf{f}^{-1}(0)}_2 + \norm{\partial \mathbf{f}^{-1}}_{BC_{\tr}(\R^{*d},\mathscr{M}^1)} \norm{\mathbf{X}}_2.
	\]
	Substituting this into the given bounds for $V$ completes the argument.
\end{proof}

The group action $\mathscr{D}(\R^{*d}) \curvearrowright \mathscr{W}(\R^{*d})$ produces a map from $T_{\id}(\mathscr{D}(\R^{*d}))$ to $T_V \mathscr{W}(\R^{*d})$.  This transformation from ``infinitesimal transport maps'' to perturbations of $V$ is described as follows.  For the classical analog, see \cite[Theorem 3.5]{Lafferty1988}.

\begin{lemma} \label{lem:groupactiontangent}
	Let $(-\epsilon,\epsilon) \to \mathscr{D}(\R^{*d}): t \mapsto \mathbf{f}_t$ be a tangent vector at $\id$ in $\mathscr{D}(\R^{*d})$, and let $V \in \mathscr{W}(\R^{*d})$.  Then $t \mapsto V_t := (\mathbf{f}_t)_* V$ is a tangent vector at $V$ in $\mathscr{W}(\R^{*d})$.  Moreover, we have
	\[
	\dot{V}_0 = -\nabla_V^* \dot{\mathbf{f}}_0,
	\]
	where
	\[
	\nabla_V^* \mathbf{h} := -\Tr_{\#}(\partial \mathbf{h}) + \partial V \# \mathbf{h} \text{ for } \mathbf{h} \in C_{\tr}^1(\R^{*d})^d.
	\]
\end{lemma}

\begin{proof}
	Let $\mathbf{g}_t = \mathbf{f}_t^{-1}$.  Note that $\dot{V}_t = \partial V(\mathbf{g}_t)[\dot{\mathbf{g}}_t]$, which depends continuously on $t$ in $\tr(C_{\tr}^\infty(\R^{*d}),\mathscr{M}(\R^{*d}))$ by Theorem \ref{thm:chainrule}.  Next, we claim that
	\[
	\frac{d}{dt} \log \Delta_\#(\partial \mathbf{g}_t) = \Tr_{\#}(\partial \dot{\mathbf{g}}_t \# \partial \mathbf{f}_t \circ \mathbf{g}_t ).
	\]
	Let $\mathbf{g}_{s,t} = \mathbf{g}_s \circ \mathbf{g}_t^{-1}$.  Then for small $\delta \in \R$, we have
	\[
	\partial \mathbf{g}_{t + \delta} = (\partial \mathbf{g}_{t+\delta,t} \circ \mathbf{g}_t) \# \partial \mathbf{g}_t,
	\]
	hence
	\[
	\log \Delta_{\#}(\partial \mathbf{g}_{t + \delta}) - \log \Delta_{\#}(\partial \mathbf{g}_t) = (\log \Delta_{\#} \partial \mathbf{g}_{t+\delta,t}) \circ \mathbf{g}_t.
	\]
	Note $\mathbf{g}_{t+\delta,t} \to \id$ in $C_{\tr}(\R^{*d})^d$ as $\delta \to 0$ and satisfies
	\[
	\frac{d}{d\delta} \biggr|_{\delta = 0} \mathbf{g}_{t+\delta,t} = \dot{\mathbf{g}}_t \circ \mathbf{g}_t^{-1}.
	\]
	For each $R > 0$ and $k > 0$, the series expansion
	\[
	\log \Delta_{\#}(\partial \mathbf{g}_{t+\delta,t}) = - \frac{1}{2} \sum_{m=1}^\infty \frac{1}{m} \Tr_{\#}[(\Id - (\partial \mathbf{g}_{t+\delta,t})^{\varstar} \# \partial \mathbf{g}_{t+\delta,t})^{\#m}]
	\]
	converges in $\norm{\cdot}_{C^k(\R^{*d},\mathscr{M}(\R^{*d}))^d,R}$ for sufficiently small $\delta$.  Therefore,
	\begin{align*}
		\frac{d}{d\delta} \biggr|_{\delta = 0} \log \Delta_{\#}(\partial \mathbf{g}_{t+\delta,t})
		&= \frac{1}{2} \Tr_{\#} \left( \frac{d}{d\delta} \biggr|_{\delta = 0} (\mathbf{g}_{t+\delta,t})^{\varstar} \# \partial \mathbf{g}_{t+\delta,t} \right) \\
		&= \frac{1}{2} \Tr_{\#} \left( \frac{d}{d\delta} \biggr|_{\delta = 0} (\partial \mathbf{g}_{t+\delta,t} +  (\partial \mathbf{g}_{t+\delta,t})^{\varstar}) \right).
	\end{align*}
	Now $\partial \mathbf{g}_{t+\delta,t}^{\cA,\tau}(\mathbf{X})$ maps $\cA_{\sa}^d \to \cA_{\sa}^d$ for any $(\cA,\tau)$.  Therefore, if $(\cB,\sigma)$ is the tracial $\mathrm{W}^*$-algebra generated by a semicircular $d$-tuple $\mathbf{S}$, then $\partial \mathbf{g}_{t+\delta,t}^{\cA*\cB,\sigma*\tau}(\mathbf{X})[\mathbf{S}]$ is self-adjoint and hence
	\[
	\ip{\mathbf{S}, \partial \mathbf{g}_{t+\delta,t}^{\cA*\cB,\sigma*\tau}(\mathbf{X})[\mathbf{S}]}_{\tau*\sigma} = \ip{\partial \mathbf{g}_{t+\delta,t}^{\cA*\cB,\sigma*\tau}(\mathbf{X})[\mathbf{S}],\mathbf{S}}_{\tau*\sigma} = \ip{\mathbf{S},((\partial \mathbf{g}_{t+\delta,t})^{\varstar})^{\cA*\cB,\sigma*\tau}(\mathbf{X})[\mathbf{S}]}_{\tau*\sigma}.
	\]
	Hence, $\Tr_{\#}((\partial \mathbf{g}_{t+\delta,t})^{\varstar}) = \Tr_{\#}(\partial \mathbf{g}_{t+\delta,t})$, which implies that
	\begin{align*}
		\frac{d}{d\delta} \biggr|_{\delta = 0} \log \Delta_{\#}(\partial \mathbf{g}_{t+\delta,t}) &= \Tr_{\#} \left( \frac{d}{d\delta} \biggr|_{\delta = 0} \partial \mathbf{g}_{t+\delta,t}  \right) \\
		&= \Tr_{\#}(\partial(\dot{g}_t \circ \mathbf{g}_t^{-1})) \\
		&= \Tr_{\#}(\partial \mathbf{g}_t \circ \mathbf{g}_t^{-1} \# \partial (\mathbf{g}_t^{-1})).
	\end{align*}
	Thus,
	\[
	\frac{d}{dt} \log \Delta_{\#}(\partial \mathbf{g}_t)
	= \Tr_{\#}(\partial \dot{\mathbf{g}}_t \circ \mathbf{g}_t^{-1} \# \partial (\mathbf{g}_t^{-1})) \circ \mathbf{g}_t
	= \Tr_{\#}(\partial \dot{\mathbf{g}}_t \# \partial \mathbf{f}_t \circ \mathbf{g}_t ).
	\]
	This is continuous in $t$ by Theorem \ref{thm:chainrule} and Proposition \ref{prop:logdeterminant}.  Hence, $t \mapsto \log \Delta_{\#}(\partial \mathbf{g}_t)$ is continuously differentiable as desired.  The above computations also show that
	\[
	\dot{V}_0 = \frac{d}{dt} \biggr|_{t=0} [V \circ \mathbf{g}_t - \log \Delta_{\#}(\partial \mathbf{g}_t)]
	= \partial V \# \dot{\mathbf{g}}_0 - \Tr_{\#}(\partial \dot{\mathbf{g}}_0)
	= -\partial V \# \dot{\mathbf{f}}_0 + \Tr_{\#}(\partial \dot{\mathbf{f}}_0)
	= -\nabla_V^* \dot{\mathbf{f}}_0.
	\qedhere
	\]
\end{proof}

\subsection{Paths from infinitesimal transport}

Given a tangent vector $t \mapsto \mathbf{f}_t$ of the identity in $\mathscr{D}(\R^{*d})$, the function $\dot{\mathbf{f}}_0 \in C_{\tr}(\R^{*d})_{\sa}^d$ can be viewed as a $d$-dimensional vector field.  The next lemma describes how to construct a path in $\mathscr{D}(\R^{*d})$ as the flow of a family of vector fields.

\begin{lemma} \label{lem:flow}
	Let $t \mapsto \mathbf{h}_t$ be a continuous map $[0,T] \to C_{\tr}^1(\R^{*d})_{\sa}^d$ such that $\norm{\partial \mathbf{h}_t}_{BC_{\tr}(\R^{*d},\mathscr{M}^1)^d}$ is bounded by a constant $M$.  Then there exist continuous maps $t \mapsto \mathbf{f}_t$ and $t \mapsto \mathbf{g}_t$ from $[0,T]$ to $C_{\tr}^1(\R^{*d})_{\sa}^d$ satisfying
	\begin{align*}
		\mathbf{f}_t &= \id + \int_0^t \mathbf{h}_u \circ \mathbf{f}_u\,du \\
		\mathbf{g}_t &= \id - \int_0^t \mathbf{h}_{t-u} \circ \mathbf{g}_u\,du
	\end{align*}
	and
	\[
	\mathbf{f}_t \circ \mathbf{g}_t = \mathbf{g}_t \circ \mathbf{f}_t = \id
	\]
	and
	\begin{align*}
		\norm{\partial \mathbf{f}_t }_{BC_{\tr}(\R^{*d},\mathscr{M}^1)^d} &\leq e^{Mt}, &
		\norm{\partial \mathbf{g}_t }_{BC_{\tr}(\R^{*d},\mathscr{M}^1)^d} &\leq e^{Mt}.
	\end{align*}
	Furthermore, for $k \geq 1$, if $t \mapsto \mathbf{h}_t$ is a continuous map into $C_{\tr}^k(\R^{*d})_{\sa}^d$, then so are $t \mapsto \mathbf{f}_t$ and $t \mapsto \mathbf{g}_t$.  If in addition $\norm{\partial^{k'} \mathbf{h}_t}_{BC_{\tr}(\R^{*d},\mathscr{M}^{k'})^d}$ is bounded for each $1 \leq k' \leq k$, then the same holds for $\mathbf{f}_t$ and $\mathbf{g}_t$.
\end{lemma}

\begin{proof}
	We focus first on the function $\mathbf{f}_t$ and its derivatives.  We construct the solution $\mathbf{f}_t$ through Picard iteration.  Let
	\begin{align*}
		\mathbf{f}_{t,0} &= \id \\
		\mathbf{f}_{t,n+1} &= \id + \int_0^t \mathbf{h}_u \circ \mathbf{f}_{u,n}
		\,du.
	\end{align*}
	As in \S \ref{sec:NCfunc2}, we understand the right-hand side in terms of Riemann integration for functions with values in a Fr\'echet space.  The same arguments used in single various calculus shows that for any continuous function $\gamma$ from $[0,T]$ into a Fr\'echet space $\mathcal{Y}$, the Riemann integral $\int_0^T \gamma$ is well-defined.  Moreover, $\int_0^t \gamma$ is continuously differentiable with derivative equal to $\gamma$.  Now $C_{\tr}(\R^{*d})_{\sa}^d$ is a Fr\'echet space and the composition operation is continuous, so by induction $\mathbf{f}_{t,n}$ is a well-defined and continuous function $[0,T] \to C_{\tr}(\R^{*d})_{\sa}^d$.  
	
	Next, since $\partial \mathbf{h}_u$ is bounded by $M$ for all $u$, we know that for every $(\cA,\tau) \in \mathbb{W}$, the function $\mathbf{h}_u^{\cA,\tau}: \cA_{\sa}^d \to \cA_{\sa}^d$ is $M$-Lipschitz with respect to $\norm{\cdot}_\infty$.  It follows that
	\[
	\norm{\mathbf{h}_u \circ \mathbf{f}_{u,n} - \mathbf{h}_u \circ \mathbf{f}_{u,n-1}}_{C_{\tr}(\R^{*d}),R} \leq M \norm{\mathbf{f}_{u,n} - \mathbf{f}_{u,n-1}}_{C_{\tr}(\R^{*d}),R}.
	\]
	Therefore,
	\[
	\norm{\mathbf{f}_{t,n+1} - \mathbf{f}_{t,n}}_{C_{\tr}(\R^{*d}),R} \leq M \int_0^t \norm{\mathbf{f}_{u,n} - \mathbf{f}_{u,n-1}}_{C_{\tr}(\R^{*d}),R}\,du.
	\]
	By induction,
	\[
	\norm{\mathbf{f}_{t,n+1} - \mathbf{f}_{t,n}}_{C_{\tr}(\R^{*d}),R} \leq \frac{M^nt^n}{n!} \sup_{t \in [0,T]} \norm{\mathbf{f}_{t,1} - \id}_{C_{\tr}(\R^{*d},R)}.
	\]
	For each $R$, the right-hand side goes to zero.  Hence, $\mathbf{f}_{t,n}$ converges to some function $\mathbf{f}_t$ in $C_{\tr}(\R^{*d})$ as $n \to \infty$ uniformly for all $t$, which satisfies the integral equation as desired.
	
	For $k \geq 1$, suppose that $t \mapsto \mathbf{h}_t$ is a continuous map into $C_{\tr}^k(\R^{*d})^d$, and we will show that $t \mapsto \mathbf{f}_t$ is as well.  Because the composition operation on $C_{\tr}^k$ functions is continuous, we obtain by the chain rule that for $n \in \N_0$,
	\[
	\partial \mathbf{f}_{t,n+1} = \Id + \int_0^t (\partial \mathbf{h}_u \circ \mathbf{f}_{u,n}) \# \partial \mathbf{f}_{u,n}\,du
	\]
	and for $2 \leq k' \leq k$,
	\[
	\partial^{k'} \mathbf{f}_{t,n+1} = \sum_{j=1}^{k'} \sum_{\substack{B_1, \dots, B_j \\ \text{partition of } [k'] \\ \min B_1 < \dots < \min B_j}} \int_0^t (\partial^j \mathbf{h}_u \circ \mathbf{f}_{u,n}) \# [\partial^{|B_1|} \mathbf{f}_{u,n}, \dots, \partial^{|B_j|} \mathbf{f}_{u,n}]\,du.
	\]
	We want to show that $\partial^{k'} \mathbf{f}_{t,n}$ converges as $n \to\infty$ in order to conclude that $\mathbf{f}_t$ is in $C_{\tr}^k(\R^{*d})_{\sa}^d$.
	
	First, we construct the limiting functions.  For $1 \leq k' \leq k$, we claim that there is a continuous function $t \mapsto \mathbf{f}_t^{(k')}$ from $[0,T]$ to $C_{\tr}(\R^{*d},\mathscr{M}(\R^{*d},\dots,\R^{*d}))$ (here the multilinear form has $k'$ arguments) that satisfies
	\begin{equation} \label{eq:flowderivative}
		\mathbf{f}_t^{(1)} = \Id + \int_0^t (\partial \mathbf{h}_u \circ \mathbf{f}_u) \# \partial \mathbf{f}_u^{(1)}\,du
	\end{equation}
	and for $2 \leq k' \leq k$,
	\begin{equation} \label{eq:flowderivative2}
		\mathbf{f}_{t,n+1}^{(k')} = \sum_{j=1}^{k'} \sum_{\substack{B_1, \dots, B_j \\ \text{partition of } [k'] \\ \min B_1 < \dots < \min B_j}} \int_0^t (\partial^j \mathbf{h}_u \circ \mathbf{f}_u) \# [\mathbf{f}_u^{(|B_1|)}, \dots, \mathbf{f}_u^{(|B_j|)}]\,du.
	\end{equation}
	We proceed by strong induction.  Let $k' \geq 1$ and suppose the claim holds for all $1 \leq \ell < k'$.  Note that the right-hand side only has one term which depends on $\mathbf{f}_u^{(k')}$, namely the term $(\partial \mathbf{h}_u \circ \mathbf{f}_u) \# \mathbf{f}_u^{(k')}$ for $j = 1$.  All the other terms $\mathbf{f}_u^{(|B_i|)}$ are already defined by inductive hypothesis and bounded in $\norm{\cdot}_{C_{\tr}(\R^{*d},\mathscr{M}^{|B_i|})^d,R}$.  Since $\partial \mathbf{h}_u$ is bounded by $M$, the right-hand side is thus $M$-Lipschitz in $\mathbf{f}_u^{(k')}$ with respect to $\norm{\cdot}_{C_{\tr}(\R^{*d},\mathscr{M}^{k'})^d,R}$.  Thus, a solution $\mathbf{f}_u^{(k')}$ exists by Picard iteration by the same argument as we used for $\mathbf{f}_t$.
	
	Let $\mathbf{f}_t^{(0)} = \mathbf{f}_t$.  Next, we show by strong induction on $k'$ that for each $R > 0$, we have $\partial^{k'} \mathbf{f}_{t,n} \to \mathbf{f}_t^{(k')}$ in $\norm{\cdot}_{C_{\tr}(\R^{*d},\mathscr{M}^{k'})}$ as $n \to \infty$ uniformly for $t \in [0,T]$.  Suppose $k' \geq 1$ and the claim holds for $\ell < k'$.  Fix $R > 0$.  Observe that
	\begin{align*}
		\partial^{k'} \mathbf{f}_{t,n+1} - \mathbf{f}_t^{(k')} &= \int_0^t (\partial \mathbf{h}_u \circ \mathbf{f}_{u,n}) \# (\partial^{k'} \mathbf{f}_{t,n} - \mathbf{f}_t^{(k')})\,du \\
		&\quad + \int_0^t (\partial \mathbf{h}_u \circ \mathbf{f}_u - \partial \mathbf{h}_u - \circ \mathbf{f}_{u,n}) \# \mathbf{f}_t^{(k')}\,du \\
		&\quad + \sum_{j=2}^{k'} \sum_{\substack{B_1, \dots, B_j \\ \text{partition of } [k'] \\ \min B_1 < \dots < \min B_j}} \int_0^t (\partial^j \mathbf{h}_u \circ \mathbf{f}_{u,n}) \# [\partial^{|B_1|} \mathbf{f}_{u,n}, \dots, \partial^{|B_j|} \mathbf{f}_{u,n}] \\
		-
		&\quad \sum_{j=2}^{k'} \sum_{\substack{B_1, \dots, B_j \\ \text{partition of } [k'] \\ \min B_1 < \dots < \min B_j}} (\partial^j \mathbf{h}_u \circ \mathbf{f}_u) \# [\mathbf{f}_u^{(|B_1|)}, \dots, \mathbf{f}_u^{(|B_j|)}] \,du.
	\end{align*}
	For $n \geq 1$, let
	\begin{multline*}
		\epsilon_{n,R} = \sup_{t \in [0,T]} \norm*{(\partial \mathbf{h}_u \circ \mathbf{f}_u - \partial \mathbf{h}_u - \circ \mathbf{f}_{u,n}) \# \mathbf{f}_t^{(k')}}_{C_{\tr}(\R^{*d},\mathscr{M}^{k'}),R}  \\
		+ \sum_{j=2}^{k'} \sum_{\substack{B_1, \dots, B_j \\ \text{partition of } [k'] \\ \min B_1 < \dots < \min B_j}} \norm*{ (\partial^j \mathbf{h}_u \circ \mathbf{f}_{u,n}) \# [\partial^{|B_1|} \mathbf{f}_{u,n}, \dots, \partial^{|B_j|} \mathbf{f}_{u,n}] - (\partial^j \mathbf{h}_u \circ \mathbf{f}_u) \# [\mathbf{f}_u^{(|B_1|)}, \dots, \mathbf{f}_u^{(|B_j|)}] }_{C_{\tr}(\R^{*d},\mathscr{M}^{k'})^d,R}.
	\end{multline*}
	By the inductive hypothesis and continuity of composition, we have $\epsilon_{n,R} \to 0$ as $n \to \infty$.  We have
	\[
	\norm{\partial^{k'} \mathbf{f}_{t,0} - \mathbf{f}_t^{(k')}}_{C_{\tr}(\R^{*d},\mathscr{M}^{k'}),R} \leq \sup_{u \in [0,T]} \norm{\mathbf{f}_u^{(k')}}_{C_{\tr}(\R^{*d},\mathscr{M}^{k'})^d,R} =: K
	\]
	and
	\[
	\norm{\partial^{k'} \mathbf{f}_{t,n+1} - \mathbf{f}_t^{(k')}}_{C_{\tr}(\R^{*d},\mathscr{M}^{k'})^d,R} \leq \int_0^t \left( M \norm{\partial^{k'} \mathbf{f}_{t,n} - \mathbf{f}_t^{(k')}}_{C_{\tr}(\R^{*d},\mathscr{M}^{k'})^d,R} + \epsilon_{n,R} \right)\,du.
	\]
	A straightforward induction on $n$ shows that
	\[
	\norm{\partial^{k'} \mathbf{f}_{t,n} - \mathbf{f}_t^{(k')}}_{C_{\tr}(\R^{*d},\mathscr{M}^{k'}),R} \leq \frac{KM^nt^n}{n!} + \sum_{\ell=1}^n \frac{\epsilon_{n-\ell,R}  M^\ell t^\ell}{\ell!}.
	\]
	Let $\epsilon_{n,R} = 0$ for $n \leq 0$.  Then
	\[
	\sum_{\ell=1}^n \frac{\epsilon_{n-\ell,R}  M^\ell t^\ell}{\ell!} = \sum_{\ell=1}^\infty \frac{\epsilon_{n-\ell,R}  M^\ell t^\ell}{\ell!} \to 0
	\]
	as $n \to \infty$ using the dominated convergence theorem because $(\epsilon_{n-\ell,R})_{n,\ell \in \N}$ is bounded and $\epsilon_{n-\ell,R} \to 0$ as $n \to \infty$ and $\sum_{m=1}^\infty (Mt)^m/m!$ converges.  Therefore, $\norm{\partial^{k'} \mathbf{f}_{t,n} - \mathbf{f}_t^{(k')}}_{C_{\tr}(\R^{*d},\mathscr{M}^{k'}),R} \to 0$ as $n \to \infty$ as desired.
	
	Because $\partial^{k'} \mathbf{f}_{t,n} \to \mathbf{f}_t^{(k')}$ as $n \to \infty$ for each $k' \leq k$, we conclude that $\mathbf{f}_t \in C_{\tr}^k(\R^{*d})^d$ and $\partial^{k'} \mathbf{f}_t = \mathbf{f}_t^{(k'})$ for $k' \leq k$.  We already showed that $\mathbf{f}_t^{(k')}$ depends continuously on $t$ in $C_{\tr}(\R^{*d},\mathscr{M}(\R^{*d},\dots,\R^{*d}))^d$ and therefore $t \mapsto \mathbf{f}_t$ is a continuous map from $[0,T]$ into $C_{\tr}^k(\R^{*d})^d$.
	
	The bound $\norm{\mathbf{f}_t}_{C_{\tr}(\R^{*d},\mathscr{M}^1)^d} \leq e^{Mt}$ follows from \eqref{eq:flowderivative} by the same argument as Gr\"onwall's inequality in classical ordinary differential equations.  Similarly, if $\partial^{k'} \mathbf{h}_t$ is uniformly bounded for each $k' \leq k$, then one can obtain a Gr\"onwall-type bound and \eqref{eq:flowderivative2} to show that $\partial^{k'} \mathbf{f}_t$ is uniformly bounded for $k' \leq k$.  We leave the details to the reader.
	
	It remains to show that the same claims hold for $\mathbf{g}_t$ as for $\mathbf{f}_t$.  By applying the foregoing argument to a subinterval of $[0,T]$, we obtain functions $\mathbf{f}_{t,s}$ for $s, t \in [0,T]$ such that $t \mapsto \mathbf{f}_{t,s}$ is continuous and
	\[
	\mathbf{f}_{t,s} = \id + \int_s^t \mathbf{h}_u \circ \mathbf{f}_{u,s} \,du.
	\]
	Also, $\mathbf{f}_{t,s} \in C_{\tr}^1(\R^{*d})_{\sa}^d$ and $\norm{\partial \mathbf{f}_{t,s}}_{BC_{\tr}(\R^{*d},\mathscr{M}^1)^d} \leq e^{M|t-s|}$.  One can verify from the integral equations that $\mathbf{f}_{t_1,t_2} \circ \mathbf{f}_{t_2,t_3} = \mathbf{f}_{t_1,t_3}$, which is a standard idea in ordinary differential equations.  In particular, since $\mathbf{f}_t = \mathbf{f}_{t,0}$, the inverse function is given by $\mathbf{g}_t = \mathbf{f}_{0,t}$, which satisfies the integral equation asserted in the proposition after switching the order of the endpoints in the Riemann integral.
\end{proof}

\begin{remark}
	Of course, the lemma applies equally well to negative time intervals.  It also works for unbounded time intervals with the hypotheses and conclusions modified to state uniform bounds on each compact time interval rather than for all time.
\end{remark}

An important special case is when $\mathbf{h}$ is independent of $t$.  Let $\mathbf{h} \in C_{\tr}^\infty(\R^{*d})_{\sa}^d$ with $\partial \mathbf{h}$ bounded.  Then there is a one-parameter group $(\mathbf{f}_t)_{t \in \R}$ in $\mathscr{D}(\R^{*d})$ solving the equation
\[
\mathbf{f}_t = \id + \int_0^t \mathbf{h} \circ \mathbf{f}_u\,du.
\]
In the spirit of Lie theory, we will denote $\mathbf{f}_t$ by $\exp(t \mathbf{h})$.  This description of one-parameter subgroups naturally gives rise to a Lie bracket on $C_{\tr}^\infty(\R^{*d})_{\sa}^d$ analogous to the classical Lie bracket on vector fields associated to the classical diffeomorphism group of $\R^d$ (also known as the Poisson bracket).  Suppose $\mathbf{h}_1$, $\mathbf{h}_2 \in C_{\tr}^\infty(\R^{*d})_{\sa}^d$ have bounded first derivatives.  Then using continuity of $t \mapsto \exp(t \mathbf{h})$ and the differential equation above, one can compute that
\[
\exp(t \mathbf{h}_1) \circ \exp(t \mathbf{h}_2) \circ \exp(-t \mathbf{h}_1) \circ \exp(-t \mathbf{h}_2) = \id + t^2 [\mathbf{h}_1,\mathbf{h}_2] + o(t^2),
\]
where
\[
[\mathbf{h}_1, \mathbf{h}_2] := \partial \mathbf{h}_1 \# \mathbf{h}_2 - \partial \mathbf{h}_2 \# \mathbf{h}_1,
\]
and where ``$o(t^2)$'' means $o(t^2)$ with respect to each of the seminorms in $C_{\tr}^\infty(\R^{*d})_{\sa}^d$.  It is an exercise to check that the Lie bracket is a continuous map $C_{\tr}^\infty(\R^{*d})_{\sa}^d \times C_{\tr}^\infty(\R^{*d})_{\sa}^d \to C_{\tr}^\infty(\R^{*d})_{\sa}^d$ and satisfies the Jacobi identity.  In the special case of non-commutative polynomials and power series, this Lie bracket was studied by \cite[\S 6.1 and \S 6.5]{Voiculescu2002}.

The classical idea that vector fields represent differential operators adapts to this setting as well.  For any $\mathbf{h} \in C_{\tr}^\infty(\R^{*d})_{\sa}^d$, let $\delta_{\mathbf{h}}: C_{\tr}^\infty(\R^{*d}) \to C_{\tr}^\infty(\R^{*d})$ be the map $\partial_{\mathbf{h}} f := \partial f \# \mathbf{h}$.  It follows from the product rule (which is a special case of Theorem \ref{thm:chainrule}) that $\partial_{\mathbf{h}}(fg) = (\partial_{\mathbf{h}} f) \cdot g + f \cdot (\partial_{\mathbf{h}} g)$, that is, $\partial_{\mathbf{h}}$ is a derivation on the algebra $C_{\tr}^\infty(\R^{*d})$.  We also have
\[
\partial_{\mathbf{h}_1} \partial_{\mathbf{h}_2} f = \partial(\partial f \# \mathbf{h}_2) \# \mathbf{h}_1 = \partial^2 f \# [\mathbf{h}_2,\mathbf{h}_1] - \partial f \# \partial \mathbf{h}_2 \# \mathbf{h}_1,
\]
hence
\[
(\partial_{\mathbf{h}_1} \partial_{\mathbf{h}_2} - \partial_{\mathbf{h}_2} \partial_{\mathbf{h}_1})f = -\partial_{[\mathbf{h}_1,\mathbf{h}_2]} f.
\]
In other words, $\mathbf{h} \mapsto - \partial_{\mathbf{h}}$ is a Lie algebra homomorphism from $C_{\tr}^\infty(\R^{*d})_{\sa}^d$ to the Lie algebra of derivations on $C_{\tr}^\infty(\R^{*d})$.

The next lemma describes how the flows $(\mathbf{f}_t)$ of Lemma \ref{lem:flow} will act upon some $V \in \tr(C_{\tr}^1(\R^{*d}))_{\sa}$.  This is the basic computation that underlies our results about free transport.

\begin{lemma} \label{lem:transport0}
	Let $t \mapsto V_t$ be continuously differentiable map $[0,T] \to \tr(C_{\tr}^1(\R^{*d}))_{\sa}$ and let $\dot{V}_t$ be its time derivative.  Let $t \mapsto \mathbf{h}_t$ be a continuous map $[0,T] \to C_{\tr}^1(\R^{*d})_{\sa}^d$ with $\norm{\partial \mathbf{h}_t}_{BC_{\tr}(\R^{*d},\mathscr{M}^1)^d}$ bounded, and let $\mathbf{f}_t$ be the solution from Lemma \ref{lem:flow} to the equation
	\begin{equation} \label{eq:flow}
		\mathbf{f}_t = \id + \int_0^t \mathbf{h}_u \circ \mathbf{f}_u \,du.
	\end{equation}
	Then we have in $\tr(C_{\tr}(\R^{*d}))$ that
	\begin{equation} \label{eq:transdiff}
		\frac{d}{dt} [(\mathbf{f}_t^{-1})_* V_t ] = (\dot{V}_t + \nabla_{V_t}^* \mathbf{h}_t) \circ \mathbf{f}_t.
	\end{equation}
	In particular, $V_t = (\mathbf{f}_t)_* V_0$ modulo constants for all $t$ if and only if $-\nabla_{V_t}^* \mathbf{h}_t = \dot{V}_t$ modulo constants for all $t$.
\end{lemma}

\begin{proof}
	For $s,t \in [0,T]$, let $\mathbf{f}_{t,s}$ be the solution to the equation
	\[
	\mathbf{f}_{t,s} = \id + \int_s^t \mathbf{h}_u \circ \mathbf{f}_{u,s} \,du,
	\]
	which is guaranteed to exist by Lemma \ref{lem:flow}.  Then for $t \in [0,T]$ and $\epsilon \in \R$ such that $t + \epsilon \in [0,T]$, we have $\mathbf{f}_{t+\epsilon} = \mathbf{f}_{t+\epsilon,t} \circ \mathbf{f}_t$.  Moreover,
	\[
	(\mathbf{f}_t^{-1})_*V_t = V_t \circ \mathbf{f}_t - \log \Delta_\# \partial \mathbf{f}_t,
	\]
	and
	\[
	(\mathbf{f}_{t+\epsilon}^{-1})_* V_{t+\epsilon} = V_{t+\epsilon} \circ \mathbf{f}_{t+\epsilon,t} \circ \mathbf{f}_t - \log \Delta_\# \partial \mathbf{f}_{t+\epsilon,t} \circ \mathbf{f}_t -  \log \Delta_{\#} \partial \mathbf{f}_t.
	\]
	Therefore,
	\begin{equation} \label{eq:transdifflem}
		(\mathbf{f}_{t+\epsilon}^{-1})_* V_{t+\epsilon} - (\mathbf{f}_t^{-1})_*V_t
		=
		\Bigl( (V_{t+\epsilon} - V_t) \circ \mathbf{f}_{t+\epsilon,t} + [V_t \circ \mathbf{f}_{t+\epsilon,t} - V_t] - \log \Delta_\# \partial \mathbf{f}_{t+\epsilon,t} \Bigr) \circ \mathbf{f}_t.
	\end{equation}
	By continuity of composition (see Lemma \ref{lem:composition}), we have
	\[
	\lim_{\epsilon \to 0} \frac{V_{t+\epsilon} - V_t}{\epsilon} \circ \mathbf{f}_{t+\epsilon,t} = \dot{V}_t \circ \mathbf{f}_{t,t} = \dot{V}_t \text{ in } \tr(C_{\tr}(\R^{*d})).
	\]
	Meanwhile, regarding the last two terms on the right-hand side of \eqref{eq:transdifflem}, we have
	\[
	[V_t \circ \mathbf{f}_{t+\epsilon,t} - V_t] - \log \Delta_\# \partial \mathbf{f}_{t+\epsilon,t} = (\mathbf{f}_{t+\epsilon,t}^{-1})_* V_t - V_t.
	\]
	The same reasoning as in Lemma \ref{lem:groupactiontangent} shows that
	\[
	\frac{d}{d\epsilon} \biggr|_{\epsilon = 0} (\mathbf{f}_{t+\epsilon,t}^{-1})_* V_t = \nabla_{V_t}^* \mathbf{h}_t
	\]
	holds in $\tr(C_{\tr}(\R^{*d}))_{\sa}$.  However, $\mathbf{g}_{s,t}$ is replaced by $\mathbf{f}_{s,t}$, which results in the sign of $\mathbf{h}_t$ changing in the final formula.  Moreover, since we have only assumed that $\mathbf{h}_t$ is $C_{\tr}^1(\R^{*d})_{\sa}^d$ rather than $C_{\tr}^\infty(\R^{*d},\mathscr{M}(\R^{*d}))$, we only have $\partial \mathbf{f}_{s,t} \in C_{\tr}(\R^{*d},\mathscr{M}(\R^{*d}))$. Altogether,
	\[
	\lim_{\epsilon \to 0} \frac{1}{\epsilon} \Bigl( (\mathbf{f}_{t+\epsilon}^{-1})_* V_{t+\epsilon} - (\mathbf{f}_t^{-1})_*V_t \Bigr)
	=
	\Bigl( \dot{V}_t + \ip{\nabla V_t,\mathbf{h}_t}_{\tr} - \Tr_{\#}(\partial \mathbf{h}_t) \Bigr) \circ \mathbf{f}_t,
	\]
	which proves \eqref{eq:transdiff}.  The final claim of the Proposition follows immediately.
\end{proof}

The case where $\mathbf{h}$ is independent of $t$ is worthy of special note, since it gives a description of one-parameter subgroups of $\mathscr{D}(\R^{*d})$ that stabilize some $V \in \mathscr{W}(\R^{*d})$ (the analog of measure-preserving transformations).

\begin{corollary} \label{cor:measurepreservinggroup}
	Let $V \in \tr(C_{\tr}^1(\R^{*d}))_{\sa}$, and let $\mathbf{h} \in C_{\tr}(\R^{*d})_{\sa}^d$ with $\partial \mathbf{h} \in BC_{\tr}(\R^{*d},\mathscr{M}(\R^{*d}))^d$.  Let $\mathbf{f}_t = \id + \int_0^t \mathbf{h} \circ \mathbf{f}_u\,du$.  Then $(\mathbf{f}_t)_* V = V$ for all $t$ if and only if $\nabla_V^* \mathbf{h} = 0$.
\end{corollary}

\begin{remark}
	Voiculescu \cite[\S 6.12]{Voiculescu2002cyclo} studied the related notion of diffeomorphisms that preserve a given non-commutative law $\mu$.  If there is a law $\mu_V$ canonically associated to $V$ (as described below), then $V$ may not be uniquely determined by $\mu_V$, and thus preserving $\mu_V$ is a weaker condition than preserving $V$.
\end{remark}

Note that the stabilizer $\mathscr{D}(\R^{*d},V) := \{\mathbf{f} \in \mathscr{D}(\R^{*d}): \mathbf{f}_* V = V\}$ is a subgroup that is closed under limits with respect to convergence of $\mathbf{f}$ and $\mathbf{f}^{-1}$ in $C_{\tr}^1(\R^{*d})^d$.  Based on Corollary \ref{cor:measurepreservinggroup}, the tangent space of the subgroup $\mathscr{D}(\R^{*d},V)$ at the identity should naturally be identified with (a subspace of) $\ker(\nabla_V^*) \subseteq C_{\tr}^\infty(\R^{*d})_{\sa}^d$.  Thus, we expect that $\ker(\nabla_V^*)$ is closed under Lie brackets.  To give a rigorous justification for this, we observe the following identity.

	\begin{lemma} \label{lem:Liebracket}
	For $V \in \tr(C_{\tr}^\infty(\R^{*d}))_{\sa}$ and $\mathbf{h}_1$, $\mathbf{h}_2 \in C_{\tr}^\infty(\R^{*d})_{\sa}^d$,
	\[
	\nabla_V^* [\mathbf{h}_1, \mathbf{h}_2] = \partial(\nabla_V^* \mathbf{h}_1) \# \mathbf{h}_2 - \partial(\nabla_V^* \mathbf{h}_2) \# \mathbf{h}_1.
	\]
\end{lemma}

\begin{proof}
	Fix $(\cA,\tau) \in \mathbb{W}$.  Let $(\cB,\sigma)$ be the tracial $\mathrm{W}^*$-algebra generated by a freely independent standard semicircular $d$-tuple $\mathbf{S}$.  Then
	\begin{align*}
		&\nabla_V^*(\partial \mathbf{h}_1 \# \mathbf{h}_2)^{\cA,\tau}(\mathbf{X}) \\
		=& -\ip{\mathbf{S}, \partial(\partial \mathbf{h}_1 \# \mathbf{h}_2)^{\cA*\cB,\tau*\sigma}(\mathbf{X})[\mathbf{S}]}_{\tau*\sigma} + (\partial V \# \partial \mathbf{h}_1 \# \mathbf{h}_2)^{\cA,\tau}(\mathbf{X}) \\
		=& -\ip{\mathbf{S}, \partial \mathbf{h}_1 \# \partial \mathbf{h}_2)^{\cA*\cB,\tau*\sigma}(\mathbf{X})[\mathbf{S}]}_{\tau*\sigma}
		-\ip{\mathbf{S}, \partial^2 \mathbf{h}_1^{\cA*\cB,\tau*\sigma}(\mathbf{X})[\mathbf{h}_2^{\cA,\tau}(\mathbf{X}), \mathbf{S}]}_{\tau*\sigma} \\
		& \quad + (\partial V \# \partial \mathbf{h}_1 \# \mathbf{h}_2)^{\cA,\tau}(\mathbf{X}) \\
		=& -\Tr_{\#}(\partial \mathbf{h}_1 \# \partial \mathbf{h}_2)^{\cA,\tau}(\mathbf{X})
		-\ip{\mathbf{S}, \partial^2 \mathbf{h}_1^{\cA*\cB,\tau*\sigma}(\mathbf{X})[\mathbf{S},\mathbf{h}_2^{\cA*\cB,\tau*\sigma}(\mathbf{X})]}_{\tau*\sigma} \\
		& \quad +(\partial V \# \partial \mathbf{h}_1)^{\cA,\tau}(\mathbf{X})[ \mathbf{h}_2)^{\cA,\tau}(\mathbf{X})] \\
		=& -\Tr_{\#}(\partial \mathbf{h}_1 \# \partial \mathbf{h}_2)^{\cA,\tau}(\mathbf{X}) + \partial(\nabla_V^* \mathbf{h}_1)^{\cA,\tau}(\mathbf{X})[\mathbf{h}_2^{\cA,\tau}(\mathbf{X})].
	\end{align*}
	Therefore,
	\[
	\nabla_V^*(\partial \mathbf{h}_1 \# \mathbf{h}_2) = - \Tr(\partial \mathbf{h}_1 \# \partial \mathbf{h}_2) + \partial(\nabla_V^* \mathbf{h}_1) \# \mathbf{h}_2.
	\]
	When we subtract $\nabla_V^*(\partial \mathbf{h}_2 \# \mathbf{h}_1)$ from $\nabla_V^*(\partial \mathbf{h}_1 \# \mathbf{h}_2)$, the terms $\Tr_\#(\partial \mathbf{h}_1 \# \partial \mathbf{h}_2)$ and $\Tr_\#(\partial \mathbf{h}_2 \# \partial \mathbf{h}_1)$ cancel.
\end{proof}

\subsection{The Laplacian and the Riemannian metric} \label{subsec:metric}

Recall that the Riemannian metric on the classical Wasserstein manifold is given by
\[
\int \ip{\nabla L_V^{-1} \dot{V}_1, \nabla L_V^{-1} \dot{V}_2}\,d\mu_V
\]
for two tangent vectors $\dot{V}_1$ and $\dot{V}_2$ at the point $V$ such that $\int \dot{V}_j \,d\mu_V = 0$.  To define the Riemannian metric in free case, we must describe how to associate a non-commutative law $\mu_V$ to some $V \in \mathscr{W}(\R^{*d})$ as well as how to invert $L_V^{-1}$ on the space of functions with expectation zero.  As this section is primarily concerned with formal computation, we will state the necessary ingredients as hypotheses.

There are several ways to approach the problem of associating a non-commutative law $\mu_V$ to a potential $V$.  We will assume here that $\mu_V$ is characterized by $\nabla_V^* \mathbf{h}$ having expectation zero for all $\mathbf{h} \in \tr(C_{\tr}^\infty(\R^{*d}))$, a relation known as the \emph{Dyson-Schwinger equation}.  The analogous property in the classical setting is that
\[
\int (\ip{\nabla V, \mathbf{h}} - \operatorname{div}(\mathbf{h}))\,d\mu = 0,
\]
which holds for the Gibbs measure $d\mu(x) = e^{-V}\,dx / \int e^{-V}$ for the potential $V$ using integration by parts.  In \S \ref{sec:freeGibbslaws}, we will argue that for many choices of $V$, there exist non-commutative laws satisfying the Dyson-Schwinger equation.

\begin{assumption} \label{ass:freeGibbs}
	Suppose that $V \in \mathscr{W}(\R^{*d})$ and there is a unique non-commutative law $\mu_V \in \Sigma_d$ that satisfies the Dyson-Schwinger equation
	\begin{equation} \label{eq:DSE0}
		\tilde{\mu}_V[\nabla_V^* \mathbf{h}] = 0
	\end{equation}
	for $\mathbf{h} \in C_{\tr}^\infty(\R^{*d})^d$, where $\tilde{\mu}_V$ is the positive homomorphism $\tr(C_{\tr}^\infty(\R^{*d})) \to \C$ corresponding to $\mu_V$.
\end{assumption}

The second hypothesis is invertibility of the Laplacian associated to $V$, which we will discuss in \S \ref{sec:pseudoinverse} for potentials $V$ close to $(1/2) \sum_j \tr(x_j^2)$.

\begin{definition}
	For $V \in \mathscr{W}(\R^{*d})$, we define $L_V: \tr(C_{\tr}^\infty(\R^{*d})) \to \tr(C_{\tr}^\infty(\R^{*d}))$ by
	\[
	L_V f := -\nabla_V^* \nabla f = \Tr_{\#}(\partial \nabla f) - \partial V \# \nabla f.
	\]
\end{definition}

\begin{assumption} \label{ass:Laplacian}
	Suppose Assumption \ref{ass:freeGibbs} holds and there is a continuous linear transformation $\Psi_V: \tr(C_{\tr}^\infty(\R^{*d})) \to \ker(\tilde{\mu}_V) \subseteq \tr(C_{\tr}^\infty(\R^{*d}))$ such that $-L_V \Psi_V f = -\Psi_V L_V f = f - \tilde{\mu}_V(f)$.
\end{assumption}

\begin{definition}
	Suppose that $V \in \mathscr{W}(\R^{*d})$ satisfies Assumptions \ref{ass:freeGibbs} and \ref{ass:Laplacian}.  Then we define a formal Riemannian metric $\ip{\cdot,\cdot}_V$ on $T_V \mathscr{W}(\R^{*d})$ by
	\[
	\ip{\dot{V},\dot{W}}_{T_V \mathscr{W}(\R^{*d})} = \tilde{\mu}(\ip{\nabla \Psi_V \dot{V}, \nabla \Psi_V \dot{W}}_{\tr}),
	\]
	where by abuse of notation $\dot{V}$ represents an equivalence class of paths $t \mapsto V_t$ in the tangent space with $\dot{V}_0 = \dot{V}$.
\end{definition}

The operator $\Psi_V$ has another use besides defining the Riemannian metric.  We saw in Lemma \ref{lem:groupactiontangent} that a vector field $\mathbf{h}$, viewed as a tangent vector to $\id$ in $\mathscr{D}(\R^{*d})$, produces a tangent vector $\dot{V} = -\nabla_V^* \mathbf{h}$ to $V$ in $\mathscr{W}(\R^{*d})$.  The operator $\Psi_V$ allows us to reverse this transformation, since for any $\dot{V}$, the vector field $-\nabla \Psi_V \dot{V}$ satisfies
\[
\dot{V} = -\nabla_V^* (-\nabla \Psi_V \dot{V}).
\]
Furthermore, if we go from a vector field $\mathbf{h}$ by $\nabla_V^*$ to a perturbation $\dot{V} = -\nabla_V^* \mathbf{h}$ and then back by $-\nabla \Psi_V$ to a vector field $\nabla \Psi_V \nabla_V^* \mathbf{h}$, then see that any vector field is equivalent modulo $\ker(\nabla_V^*)$ to a gradient.  The operator
\[
\mathbb{P}_V = \nabla \Psi_V \nabla_V^*: C_{\tr}^\infty(\R^{*d})^d \to C_{\tr}^\infty(\R^{*d})^d
\]
thus represents the ``projection of vector fields onto gradients'', and $1 - \mathbb{P}_V$ is the free version of the Leray projection in fluid dynamics.  The operators $L_V$, $\nabla$, $\nabla_V^*$, $\Psi_V$, and $\mathbb{P}_V$ satisfy the following relations.

\begin{proposition} \label{prop:Laplacianrelations2}
	Suppose that $V \in \mathscr{W}(\R^{*d})$ satisfies Assumptions \ref{ass:freeGibbs} and \ref{ass:Laplacian}.  Consider the operators
	\[
	\C \xrightarrow{\iota} \tr(C_{\tr}^\infty(\R^{*d})) \xrightarrow{\nabla} C_{\tr}^\infty(\R^{*d})^d,
	\]
	where $\iota$ maps a scalar to the corresponding constant function, and
	\[
	C_{\tr}^\infty(\R^{*d})^d \xrightarrow{\nabla_V^*} \tr(C_{\tr}^\infty(\R^{*d})) \xrightarrow{\tilde{\mu}_V} \C.
	\]
	Then
	\begin{enumerate}[(1)]
		\item $\ker(\nabla) = \ker(L_V) = \iota(\C)$.
		\item $\im(\nabla_V^*) = \im(L_V) = \ker(\tilde{\mu}_V)$.
		\item $-L_V \Psi_V L_V = L_V$ and $-\Psi_V L_V \Psi_V = \Psi_V$.
		\item $\mathbb{P}_V^2 = \mathbb{P}_V$.
		\item Every $\mathbf{f} \in C_{\tr}^\infty(\R^{*d})^d$ can be uniquely written as $\mathbf{f} = \nabla g + \mathbf{h}$ where $g \in \tr(C_{\tr}^\infty(\R^{*d}))$ and $\nabla_V^* \mathbf{h} = 0$.  Here $\nabla g = \mathbb{P}_V \mathbf{f}$.
	\end{enumerate}
\end{proposition}

\begin{proof}
	(1) Clearly, $\iota(\C) \subseteq \ker(\nabla) \subseteq \ker(L_V)$.  Conversely, if $f \in \ker(L_V)$, then $f = -\Psi_V L_V f + \tilde{\mu}_V f = \tilde{\mu}_V f \in \iota(\C)$.
	
	(2) Clearly, $\im(L_V) \subseteq \im(\nabla_V^*)$.  Moreover, \eqref{eq:DSE0} says precisely that $\im(\nabla_V^*) \subseteq \ker(\tilde{\mu}_V)$.  Finally, if $f \in \ker(\tilde{\mu}_V)$, then $f = -L_V \Psi_V f + \tilde{\mu}_V f = \nabla_V^* \nabla \Psi_V f + 0$.
	
	(3) Note that $-L_V \Psi_V L_V f = L_V(f - \tilde{\mu}_V(f)) = L_V f$ and $-\Psi_V L_V \Psi_V f = \Psi_V f - \tilde{\mu}_V(\Psi_V f) = \Psi_V f$ since $\im(\Psi_V) \subseteq \ker(\tilde{\mu}_V)$.
	
	(4) Note that $\nabla \Psi_V \nabla_V^* \nabla \Psi_V \nabla_V^* = -\nabla \Psi_V L_V \Psi_V \nabla_V^* = \nabla \Psi_V \nabla_V^*$.
	
	(5) To show existence, fix $\mathbf{f}$ and let $g = \Psi_V \nabla_V^* \mathbf{f}$ and $\mathbf{h} = \mathbf{f} - \nabla g = (1 - \mathbb{P}_V) \mathbf{f}$.  Then $\nabla_V^* \mathbf{h} = \nabla_V^* \mathbf{f} - \nabla_V^* \nabla \Psi_V \nabla_V^* \mathbf{f} = (1 + L_V)\nabla_V^* \mathbf{f} = \tilde{\mu}_V \nabla_V^* \mathbf{f} = 0$.  For uniqueness, note that $\mathbb{P}_V \mathbf{f}$ must equal $\nabla g$, and hence $\mathbf{h}$ must equal $(1 - \mathbb{P}_V) \mathbf{f}$.
\end{proof}

In the classical setting, $\mathbb{P}_V$ is the $L^2$-orthogonal projection of the space of vector fields onto the subspace of gradients.  Thus, $\mathbb{P}_V \mathbf{h}$ is a vector field which will produce the same perturbation of $V$ through the transport action as $\mathbf{h}$ does, and which has $L^2$ norm less than or equal to that of $\mathbf{h}$.  That is, $\mathbb{P}_V$ is an infinitesimal version of optimal transport.  For the same idea to apply in the free setting, we would like to show that $\ker(\nabla_V^*)$ and $\im(\nabla)$ are orthogonal with respect to $\tilde{\mu}_V$.

Although this is merely an integration-by-parts computation in the classical case, the same approach does not directly work in the free setting because (despite our choice of notation) $\nabla_V^*$ is not actually the adjoint of $\nabla$.  Rather, it is the large $N$ limit of $1/N^2$ times the adjoint of $\nabla$ on $L^2(\mu_V^{(N)})$, where $\mu_V^{(N)}$ is the measure on $M_N(\C)_{\sa}^d$ with density proportional to $e^{-N^2 V}$.  The adjointness relation as written does not make sense in the large $N$ limit because of the factor of $1/N^2$.

There is another natural heuristic for why $\ker(\nabla_V^*)$ and $\im(\nabla)$ are orthogonal.  If $\mathbf{h} \in \ker(\nabla_V^*)$ with appropriate boundedness assumptions, then $\mathbf{h}$ should generate a one-parameter group of measure-preserving transformations $\mathbf{f}_t$ for $V$ by Corollary \ref{cor:measurepreservinggroup}.  If we differentiate the equation $\tilde{\mu}_V[g \circ \mathbf{f}_t] = \tilde{\mu}_V[g]$ at $t = 0$, we get $\tilde{\mu}_V[\ip{\nabla g, \mathbf{h}}_{\tr}] = 0$.  However, to make a rigorous argument, it is easier to directly use the Lie bracket identity Lemma \ref{lem:Liebracket} (related to the group of measure-preserving transformations) together with the Dyson-Schwinger equation.

\begin{proposition} \label{prop:Laplacianrelations3}
	Suppose that $V$ satisfies Assumption \ref{ass:freeGibbs}, and in (3) - (5) suppose also that $V$ satisfies Assumption \ref{ass:Laplacian}.
	\begin{enumerate}[(1)]
		\item $\tilde{\mu}_V[\ip{\nabla \nabla_V^* \mathbf{h}_1,\mathbf{h}_2}_{\tr}] = \tilde{\mu}_V[\ip{\mathbf{h}_1,\nabla \nabla_V^* \mathbf{h}_2}_{\tr}]$ for $\mathbf{h}_1$, $\mathbf{h}_2 \in C_{\tr}(\R^{*d})^d$.
		\item $\tilde{\mu}_V[\ip{\nabla L_V g_1, \nabla g_2}_{\tr}] = \tilde{\mu}_V[\ip{\nabla g_1, \nabla L_V g_2}_{\tr}]$ for $g_1$, $g_2 \in \tr(C_{\tr}^\infty(\R^{*d}))$.
		\item $\tilde{\mu}_V[\ip{\nabla \Psi_V g_1, \nabla g_2}_{\tr}] = \tilde{\mu}_V[\ip{\nabla g_1, \nabla \Psi_V g_2}_{\tr}]$ for $g_1$, $g_2 \in \tr(C_{\tr}^\infty(\R^{*d}))$.
		\item If $g \in \tr(C_{\tr}^\infty(\R^{*d}))$ and $\mathbf{h} \in \ker(\nabla_V^*)$, then $\tilde{\mu}_V[\ip{\nabla g, \mathbf{h}}_{\tr}] = 0$.
		\item $\tilde{\mu}_V[\ip{\mathbb{P}_V \mathbf{h}_1,\mathbf{h}_2}_{\tr}] = \tilde{\mu}_V[\ip{\mathbf{h}_1, \mathbb{P}_V \mathbf{h}_2}_{\tr}]$ for $\mathbf{h}_1$, $\mathbf{h}_2 \in C_{\tr}(\R^{*d})^d$.
	\end{enumerate}
\end{proposition}

\begin{proof}
	(1) By complex-linearity, it suffices to consider the case when $\mathbf{h}_1$ and $\mathbf{h}_2$ are self-adjoint.  By Lemma \ref{lem:Liebracket}, we have
	\[
	\nabla_V^*[\mathbf{h}_1,\mathbf{h}_2] = \ip{\nabla \nabla_V^* \mathbf{h}_1, \mathbf{h}_2}_{\tr} - \ip{\nabla \nabla_V^* \mathbf{h}_2,\mathbf{h}_1}_{\tr}.
	\]
	When we apply $\tilde{\mu}_V$, the left-hand side evaluates to zero, hence
	\[
	\tilde{\mu}_V[\ip{\nabla \nabla_V^* \mathbf{h}_1, \mathbf{h}_2}_{\tr}] = \tilde{\mu}_V[\ip{\nabla \nabla_V^* \mathbf{h}_2,\mathbf{h}_1}_{\tr}] = \tilde{\mu}_V[\ip{\mathbf{h}_1,\nabla \nabla_V^* \mathbf{h}_2}_{\tr}],
	\]
	since $\mathbf{h}_1$ and $\nabla \nabla_V^* \mathbf{h}_2$ are self-adjoint (which follows since $\nabla_V^* \mathbf{h}_2$ is real-valued).
	
	(2) Substitute $\mathbf{h}_j = \nabla g_j$ into (1) and apply $\nabla_V^* \nabla = -L_V$.
	
	(3) Substitute $\Psi_V g_j$ for $g_j$ in (2) and note that $\nabla L_V \Psi_V g_j = \nabla[\tilde{\mu}_V[g_j] - g_j] = -\nabla g_j$.
	
	(4) Note
	\begin{align*}
		\tilde{\mu}_V[\ip{\nabla g, \mathbf{h}}_{\tr}]
		&= -\tilde{\mu}_V[\ip{\nabla \nabla_V^* \nabla \Psi_V g, \mathbf{h}}_{\tr}] \\
		&= -\tilde{\mu}_V[\ip{\nabla \Psi_V g, \nabla \nabla_V^* \mathbf{h}}_{\tr}] \\
		&= 0.
	\end{align*}
	
	(5)  Since $\mathbb{P}_V \mathbf{h}_1 \in \im(\nabla)$ and $(1 - \mathbb{P}_V) \mathbf{h}_2 \in \ker(\nabla_V^*)$, they are orthogonal with respect to $\tilde{\mu}_V \circ \ip{\cdot,\cdot}_{\tr}$.  Therefore,
	\[
	\tilde{\mu}_V[\ip{\mathbb{P}_V \mathbf{h}_1,\mathbf{h}_2}_{\tr}] = \tilde{\mu}_V[\ip{\mathbb{P}_V \mathbf{h}_1, \mathbb{P}_V \mathbf{h}_2}_{\tr}].
	\]
	By symmetrical reasoning, this equals $\tilde{\mu}_V[\ip{\mathbf{h}_1,\mathbb{P}_V\mathbf{h}_2}_{\tr}]$.
\end{proof}

In contrast to the situation with $\nabla$, the adjoint of the operator $\partial$ can be understood directly from the Dyson-Schwinger equation.  The following lemma is related to computations in \cite[Proposition 21]{Shlyakhtenko2009}.

\begin{lemma} \label{lem:adjointness}
	Let $V$ satisfy Asssumptions \ref{ass:freeGibbs} and \ref{ass:Laplacian}.  Define
	\[
	\partial_V^*: C_{\tr}^1(\R^{*d},\mathscr{M}(\R^{*d}))^d \to C_{\tr}(\R^{*d})^d
	\]
	by
	\[
	\partial_V^* \mathbf{F} = \mathbf{F} \# \nabla V - \partial^\dagger \mathbf{F}.
	\]
	Then for $\mathbf{f} \in C_{\tr}^2(\R^{*d})^d$ and $\mathbf{F} \in C_{\tr}^2(\R^{*d},\mathscr{M}(\R^{*d}))^d$, we have
	\[
	\tilde{\mu}_V \ip{\mathbf{f}, \partial_V^* \mathbf{F}}_{\tr} = \tilde{\mu}_V \Tr_{\#}[(\partial \mathbf{f})^{\varstar} \mathbf{F}].
	\]
\end{lemma}

\begin{remark}
	We can define an semi-inner product on $C_{\tr}^\infty(\R^{*d})^d$ by $(\mathbf{f},\mathbf{g}) \mapsto \tilde{\mu}_V \ip{\mathbf{f},\mathbf{g}}_{\tr}$.  We can also define a semi-inner product on $C_{\tr}^\infty(\R^{*d},\mathscr{M}(\R^{*d}))^d$ by $(\mathbf{F},\mathbf{G}) \mapsto \tilde{\mu}_V \Tr_\#(\mathbf{F} \# \mathbf{G})$.  The lemma then says that $\partial_V^*$ is formally the adjoint of $\partial$ with respect to these inner products.
\end{remark}

\begin{proof}
	We apply \eqref{eq:DSE0} with $\mathbf{h} = (\mathbf{F}^{\varstar} \# \mathbf{f})^*$.  Observe that
	\[
	\partial V \# \mathbf{h}
	= \ip{\nabla V, \mathbf{h}}_{\tr}
	= \ip{\mathbf{h}^*,\nabla V}_{\tr}
	= \ip{\mathbf{F}^{\varstar} \# \mathbf{f}, \nabla V}_{\tr}
	= \ip{\mathbf{f}, \mathbf{F} \# \nabla V}_{\tr}.
	\]
	Next, we compute $\Tr_{\#}(\partial \mathbf{h})$.  Let $\Phi$ and $\Upsilon$ be the maps in Lemmas \ref{lem:duality} and \ref{lem:tracelike} respectively.  Then $(\cA,\tau) \in \mathbb{W}$ and $\mathbf{X}$, $\mathbf{Y} \in \cA_{\sa}^d$, we have
	\begin{align*}
		\Phi^{-1}(\mathbf{h})^{\cA,\tau}(\mathbf{X})[\mathbf{Y}]
		&= \ip{\mathbf{Y}, \mathbf{h}^{\cA,\tau}(\mathbf{X})}_{\tau} \\
		&= \ip{\mathbf{h}^{\cA,\tau}(\mathbf{X})^*,\mathbf{Y}}_{\tau} \\
		&= \ip{(\mathbf{F}^{\varstar})^{\cA,\tau}(\mathbf{X})[\mathbf{f}^{\cA,\tau}(\mathbf{X})], \mathbf{Y}}_\tau \\
		&= \ip{\mathbf{f}^{\cA,\tau}(\mathbf{X}), \mathbf{F}^{\cA,\tau}(\mathbf{X})[\mathbf{Y}]}_\tau.
	\end{align*}
	Now
	\[
	\Tr_{\#}(\partial \mathbf{h}) = \Upsilon(\Phi^{-1}(\partial \mathbf{h})) = \Upsilon(\partial \Phi^{-1} (\mathbf{h}))),
	\]
	where the last equality follows from \eqref{eq:duality3} and the fact that $\Upsilon(\mathbf{g}_\pi) = \Upsilon(\mathbf{g})$ when $\pi$ is the permutation that switches the last two indices.  Let $(\cB,\sigma)$ be generated by a standard semicircular $d$-tuple $\mathbf{S}$.  Using our previous expression for $\Phi^{-1}(\mathbf{h})$, we have
	\begin{align*}
		\Upsilon(\partial \Phi^{-1}(\mathbf{h}))^{\cA,\tau}(\mathbf{X}) &= \frac{d}{dt} \Bigr|_{t=0} \ip{\mathbf{f}^{\cA*\cB,\tau*\sigma}(\mathbf{X}+t\mathbf{S}), \mathbf{F}^{\cA*\cB,\tau*\sigma}(\mathbf{X} + t\mathbf{S})[\mathbf{S}]}_{\tau*\sigma} \\
		&= \ip{\partial \mathbf{f}^{\cA*\cB,\tau*\sigma}(\mathbf{X})[\mathbf{S}], \mathbf{F}^{\cA*\cB,\tau*\sigma}(\mathbf{X})[\mathbf{S}]}_{\tau*\sigma} + \ip{\mathbf{f}^{\cA*\cB,\tau*\sigma}(\mathbf{X}), \partial \mathbf{F}^{\cA*\cB,\tau*\sigma}(\mathbf{X})[\mathbf{S},\mathbf{S}]}_{\tau*\sigma} \\
		&= \ip{\mathbf{S}, (\partial \mathbf{f}^{\varstar} \# \mathbf{F})^{\cA*\cB,\tau*\sigma}(\mathbf{X})[\mathbf{S}]}_{\tau*\sigma} + \ip{\mathbf{f}^{\cA,\tau}(\mathbf{X}), E_{\cA} \partial \mathbf{F}^{\cA*\cB,\tau*\sigma}(\mathbf{X})[\mathbf{S},\mathbf{S}]}_\tau \\
		&= \Tr_{\#}[(\partial \mathbf{f})^{\varstar} \mathbf{F}]^{\cA,\tau}(\mathbf{X}) + \ip{\mathbf{f}^{\cA,\tau}(\mathbf{X}), (\partial^\dagger \mathbf{F})^{\cA,\tau}(\mathbf{X})}_\tau.
	\end{align*}
	Thus, we get
	\[
	\Tr_{\#}(\partial \mathbf{h}) = \Tr_{\#}[(\partial \mathbf{f})^{\varstar} \mathbf{F}] + \ip{\mathbf{f}, \partial^\dagger \mathbf{F}}_{\tr}.
	\]
	So the Dyson-Schwinger equation yields
	\[
	\tilde{\mu}_V \ip{\mathbf{f}, \mathbf{F} \# \nabla V}_{\tr} = \tilde{\mu}_V \Tr_{\#}[(\partial \mathbf{f})^{\varstar} \mathbf{F}] + \tilde{\mu}_V \ip{\mathbf{f}, \partial^\dagger \mathbf{F}}_{\tr},
	\]
	which is the desired equality.
\end{proof}

\subsection{Strategy and discussion} \label{subsec:discussion}

A natural strategy to produce transport maps from one point $V_0$ to another $V_1$ in $\mathscr{W}(\R^{*d})$ is as follows.  Suppose we are given a path $t \mapsto V_t$ from $[0,1]$ into the free Wasserstein manifold.  Suppose all the $V_t$'s satisfy Assumptions \ref{ass:freeGibbs} and \ref{ass:Laplacian}.  Assume without loss generality that $\dot{V}_t$ has expectation zero under $\mu_{V_t}$.  Let $\mathbf{h}_t = -\nabla \Psi_{V_t} \dot{V}_t$, so that $-\nabla_{V_t}^* \mathbf{h}_t = \dot{V}_t$.  Let $\mathbf{f}_t$ solve the equation $\mathbf{f}_t = \id + \int_0^t \mathbf{h}_u \circ \mathbf{f}_u\,du$.  Then $(\mathbf{f}_t)_* V_0$ should equal $V_t$ for all $t$.  Of course, carrying this out rigorously requires additional analytic assumptions.

The remainder of the paper will show that Assumptions \ref{ass:freeGibbs} and \ref{ass:Laplacian} hold and the transport strategy can be carried out rigorously for potentials $V \in C_{\tr}^\infty(\R^{*d})$ of the form $V(\mathbf{x}) = (1/2) \sum_j \tr(x_j^2) + W(\mathbf{x})$ such that $\partial W$ is uniformly bounded and $\partial \nabla W$ is uniformly bounded by a constant strictly less than $1$.  More precisely, \S \ref{sec:pseudoinverse} will study the heat semigroup associated to $L_V$, and from there the associated expectation $\mathbb{E}_V: \tr(C_{\tr}(\R^{*d})) \to \C$ and the pseudo-inverse $\Psi_V$ of the Laplacian $L_V$.  These results will imply that $V$ satisfies Assumption \ref{ass:Laplacian}, and that there is a unique law $\mu_V$ satisfying $\tilde{\mu}_V (L_V f) = 0$ for all $f \in \tr(C_{\tr}^2(\R^{*d}))$.  However, this alone does not imply that $\mu_V$ satisfies \eqref{eq:DSE0}.

	Next, \S \ref{sec:freeGibbslaws} will study the free Gibbs laws associated to a potential $V$, that is, non-commutative law maximizing a certain free entropy functional.  These results will imply that if $\partial W$ and $\partial^2 W$ are bounded (here there are no restrictions on the constant), then there exists a non-commutative law $\nu$ satisfying the Dyson-Schwinger equation $\tilde{\nu}[ \nabla_V^* \mathbf{h}] = 0$ for all sufficiently smooth $\mathbf{h}$.  Hence, in the situation where $\partial \nabla W$ is uniformly smaller than $1$, we have existence and uniqueness of a law $\mu_V$ satisfying \eqref{eq:DSE0}, or in other words, $V$ satisfies Assumption \ref{ass:freeGibbs}.

In order to execute the strategy for constructing transport, we need $\mathbf{h}_t = -\nabla_{V_t}^* \Psi_{V_t} \dot{V}_t$ to have uniformly bounded first derivative and to depend continuously on $t$ in order to apply Lemmas \ref{lem:flow} and \ref{lem:transport0}.  Thus, in our construction of $\Psi_V$ in \S \ref{sec:pseudoinverse}, we have to estimate the derivatives of $\Psi_V f$ and show that $\Psi_V f$ depends continuously on $V$ and $f$ jointly.  The continuity property of course increases the amount of technical work, but it follows quite naturally from the stochastic construction of heat semigroup provided that we have uniform bounds on $\partial V$ and $\partial \nabla V$.  On the other hand, to get $\mathbf{h}_t$ to have bounded first derivative with our methods requires us to assume that $\partial^3 V_t$ is bounded and that $\partial \dot{V}_t$ and $\partial^2 \dot{V}_t$ are bounded.

In \S \ref{sec:rigoroustransport}, we complete the argument for transport by showing that $(\mathbf{f}_1)_* \mu_{V_0} = \mu_{V_1}$, and this yields an isomorphism of the $\mathrm{C}^*$ and $\mathrm{W}^*$-algebras associated to $\mu_{V_0}$ and $\mu_{V_1}$.  In \S \ref{subsec:triangulartransport}, assuming a smaller bound for $\partial^2 V - \Id$, we construct transport functions $\mathbf{h}_t$ and $\mathbf{f}_t$ which are triangular, in the sense that
\[
\mathbf{f}_t(x_1,\dots,x_d) = (f_{t,1}(x_1),f_{t,2}(x_1,x_2),\dots,f_{t,d}(x_1,\dots,x_d)).
\]
This produces a triangular isomorphism of $\mathrm{C}^*$ and $\mathrm{W}^*$-algebras./

It is natural to ask what the minimal assumptions are on $V_0$ and $V_1$ to obtain isomorphisms of the associated $\mathrm{C}^*$ and $\mathrm{W}^*$-algebras.  First, although we assume that $V \in \tr(C_{\tr}^\infty(\R^{*d}))$ throughout, the proof would work just as well if $V$ is merely in $\tr(C_{\tr}^3(\R^{*d}))$ (with of course the required bounds on the derivatives).  We did not wish to get mired down with writing the precise smoothness assumptions needed for each result.  In any case, the smoothness assumptions needed in this proof may not be optimal.  For instance, von Neumann algebraic triangular transport was constructed in \cite{JekelExpectation,JekelThesis} using only assumptions on the first two derivatives of $V$.  We do not yet verified that this would be sufficient for $\mathrm{C}^*$-algebraic triangular transport.

More generally, do we expect such results to hold for functions $V$ which are not perturbations of a quadratic, and especially those which are not even convex?  Unfortunately, the $\mathrm{C}^*$-isomorphism can fail even for $d = 1$ with $V \in \tr(C_{\tr}^\infty(\R^{*d}))$.

Random matrix theorists have carried out a detailed analysis of the case (among others) where $d = 1$ and $V(X) = \tr(f(X))$ for some smooth $f: \R \to \R$; see \cite{BdMPS,BS2001,BG2013,BG2013multi,BGK2015}.  Of course, by \S \ref{subsec:smoothfunctionalcalculus}, such a $V$ will be in $\tr(C_{\tr}^\infty(\R^{*d}))$.  As in \cite[\S 7.1]{BS2001}, consider $f(t) = t^4 / 4 - ct^2$, or $V(x) = \tr(x^4) / 4 - c \tr(x^2)$.  Let $\mu^{(N)}$ be the associated measure on $M_N(\C)_{\sa}$, and let $X^{(N)}$ be a random matrix chosen according to this measure.  It was shown that for large enough $c$, the empirical spectral distribution of $X^{(N)}$ converges in probability to a measure $\rho$ on $\R$ whose support is the disjoint union of two closed intervals.  If $X$ is a self-adjoint operator in $(\cA,\tau)$ with spectral distribution $\rho$, then $\mathrm{C}^*(X) \cong C[0,1] \oplus C[0,1]$.  In particular, it is not isomorphic to the $\mathrm{C}^*$-algebra generated by a self-adjoint operator $S$ with the semicircular distribution.

As a side note, the function $\tr(x^4) / 4 - c \tr(x^2)$ is not a bounded perturbation of $(1/2) \tr(x^2)$, hence not among the class of functions studied in this paper.  However, one can easily modify the function $t^4/4 - ct^2$ near $\infty$ so that it is a bounded perturbation of some constant times $t^2$.  If this modification is close enough to $\infty$, and the values of the modified function remain sufficiently large in that region, then the support of the limiting distribution can be forced to stay inside a bounded set where the function was not changed (using similar techniques as \cite[\S 7.1]{BS2001}, \cite[\S 18.2]{JekelThesis}), and hence the limiting distribution will still be $\rho$ because of \cite[Theorem 1]{BdMPS}.  Similarly, one could consider a function such as $f(t) = t^2/2 + a e^{-bt^2}$ for large constants $a$ and $b$.  By choosing the coefficients correctly, one could presumably produce similar behavior to $t^4/4 - ct^2$ in that the limiting empirical spectral distribution would have a support with two components.

Such examples are an obstruction to $\mathrm{C}^*$ transport results for free Gibbs laws for general $V$.  These examples will in fact fail Assumptions \ref{ass:freeGibbs} and \ref{ass:Laplacian}.  Indeed, by reweighting the pieces of $\mu_V$ on each component of the support, one can obtain a continuum of measures that satisfy the Dyson-Schwinger equation, although it turns out that often there is still a unique maximizer of entropy.  Moreover, if we consider a smooth function $f$ on $\R$ that is constant on each component of the support, then $\nabla (f(x)) = f'(x)$ will evaluate to zero in $L^2$ of the free Gibbs law for $V$.  Although this is not technically the same as $\nabla(f(x))$ being zero in $C_{\tr}(\R^{*d})^d$, this behavior still suggests an obstacle to inverting $L_V$ modulo constant functions.  On the other hand, \cite{BG2013multi} and \cite{BGK2015} were able to invert the Laplacian on $L^2$ modulo a finite-dimensional kernel (still for a single matrix).  It is an intriguing possibility that something like this could work for the multi-matrix setting and lead to a transport result that applies as long as $\mathbf{h}_t$ is in a certain subspace of $C_{\tr}^\infty(\R^{*d})_{\sa}^d$ complementary to the kernel of $L_{V_t}$.

We also remark that since $\mathrm{W}^*$-isomorphism is weaker than $\mathrm{C}^*$-isomorphism, there could be situations in which the former is possible even when the latter is not.  In the case of a single self-adjoint operator, topological obstructions, such as disconnected support, disappear when we pass from the algebra of continuous functions to the $L^\infty$ space.  On the other hand, Brown showed that finite free entropy for a non-commutative law is not sufficient to guarantee $\mathrm{W}^*$-isomorphism with the law of a semicircular family \cite{Brown2005}.  However, we do not know of any counterexamples to having a $\mathrm{W}^*$-isomorphism between $\mu_V$ and the law of a free semicircular family for any smooth $V$ with quadratic growth at $\infty$. Voiculescu conjectured such a $\mathrm{W}^*$-isomorphism for a certain class of potentials in \cite{Voiculescu2006}.

\section{Pseudo-inverse of the Laplacian $L_V$} \label{sec:pseudoinverse}

As we saw in \S \ref{subsec:logdensity} and \S \ref{sec:manifold}, the Laplacian associated to $V$ plays an important role in converting between perturbations of $V$ and infinitesimal transport maps, both in the classical case and in the non-commutative case.  Recall that for $V \in \tr(C_{\tr}^\infty(\R^{*d}))$, the associated Laplacian is defined by
\[
L_Vf = Lf - \sum_{j=1}^d \partial_{x_j} f \# \nabla_{x_j} V.
\]
For each $k \in \N_0 \cup \{\infty\}$, this operator is a continuous linear transformation $C_{\tr}^{k+2}(\R^{*d}) \to C_{\tr}^k(\R^{*d})$.

We seek sufficient conditions for $L_V$ to have a one-dimensional kernel and a well-behaved pseudo-inverse $\Psi_V$.  We will use this in \S \ref{subsec:constructtransport} to verify that $V$ satisfies Assumption \ref{ass:Laplacian}.  As discussed in \S \ref{subsec:discussion}, we do not expect this to hold in all cases, so we will assume that $V$ is close in a certain sense to the quadratic $(1/2) \ip{\mathbf{x},\mathbf{x}}_{\tr}$.  Following similar ideas to \cite{BS2001,BCG2003,Dabrowski2010,GS2009,GS2014,DGS2016} and especially \cite{DGS2016}, since we cannot work directly with the density in the free setting, we will instead recover $\mathbb{E}_V$ and $\Psi_V$ from the heat semigroup $(e^{tL_V})_{t \in [0,\infty)}$, which in turn will be constructed from a free stochastic process $\mathcal{X}(\mathbf{X},t)$ solving the equation
\[
d\mathcal{X}(\mathbf{X},t) = d\mathcal{S}(t) - \frac{1}{2} \nabla_x V(\mathcal{X}(\mathbf{X},t))\,dt, \qquad \mathcal{X}(\mathbf{X},0) = \mathbf{X},
\]
where $(\mathcal{S}(t))_{t \in [0,\infty)}$ is a free Brownian motion in $d$ variables, freely independent of $\mathbf{X}$.  We remark that the technical development of free SDE theory owes a great deal to the work of Biane \cite{Biane1997}, Biane and Speicher \cite{BS1998,BS2001}, and Dabrowski \cite{Dabrowski2010,Dabrowski2017}, although due to the simple nature of the SDE considered here, we opt for a self-contained treatment which does not require any background in free stochastic analysis.

In fact, the SDE construction only depends on $V$ through its gradient $\nabla V$ and nothing about the construction of the SDE and heat semigroup requires us to use a gradient.  Hence, we will prove the results with $\nabla V$ replaced by a function $\mathbf{J} \in C_{\tr}^\infty(\R^{*d})_{\sa}^d$ which is sufficiently close to the identity function.  As motivation, note that in the case where $\mathbf{J} = \nabla V$, the condition $\norm{\partial \mathbf{J} - \Id}_{BC_{\tr}(\R^{*d},\mathscr{M}(\R^{*d}))} < 1$ would mean that the Hessian of $V$ is within $1$ of $\Id$.  In the classical world, this implies that $V$ is uniformly convex.

\begin{definition} \label{def:J}
	For constants $c \in (0,1)$ and $a \in \R$, we define
	\[
	\mathscr{J}_{a,c}^d := \{\mathbf{J} \in C_{\tr}^\infty(\R^{*d}): \norm{\mathbf{J} - \id}_{BC_{\tr}(\R^{*d})^d} \leq a, \norm{\partial \mathbf{J} - \Id}_{BC_{\tr}(\R^{*d},\mathscr{M}^1)^d} \leq 1 - c\}.
	\]
	We also define
	\[
	L_{\mathbf{J}} f := Lf - \partial f \# \mathbf{J}.
	\]
\end{definition}

Thus, in particular, the earlier operator $L_V$ would equal $L_{\nabla V}$ in this notation.  This will not cause any confusion because $V$ and $\nabla V$ are different types of objects:  $V$ is a scalar-valued function while $\nabla V$ is a $d$-tuple of operator-valued functions.  A precise statement of our results is as follows.

\begin{definition} \label{def:Vheatsemigroup1}
	Let $\mathbf{J} \in \mathscr{J}_{a,c}^d$.  Let $(\cA,\tau)$ be a tracial $\mathrm{W}^*$-algebra, let $(\cB,\sigma)$ be the tracial $\mathrm{W}^*$-algebra generated by a $d$-tuple of self-adjoint free Brownian motions $(\mathcal{S}_1(t),\dots,\mathcal{S}_d(t))$ for $t \in [0,\infty)$, and let $(\cA * \cB, \tau * \sigma)$ be the tracial free product of $(\cA,\tau)$ and $(\cB,\sigma)$.  For $\mathbf{X} = (X_1,\dots,X_d) \in \cA_{\sa}^d$, let $\mathcal{X}(\mathbf{X},t) = \mathcal{X}^{\cA,\tau}(\mathbf{X},t)$ be the solution to the integral equation
	\[
	\mathcal{X}(\mathbf{X},t) = X + \mathcal{S}(t) + \int_0^t \mathbf{J}(\mathcal{X}(\mathbf{X},u))\,du
	\]
	(which we will show is well-defined in Lemma \ref{lem:SDEsolution}).  Note that $\mathcal{X}$ is a function $\cA_{\sa}^d \times [0,\infty) \to (\cA * \cB)_{\sa}^d$.  For $f \in C_{\tr}(\R^{*d})$, we define
	\[
	(e^{tL_{\mathbf{J}}}f)^{\cA,\tau}(\mathbf{X}) := E_{\cA}[f^{\cA * \cB, \tau * \sigma}(\mathcal{X}(\mathbf{X},2t))],
	\]
	where $E_{\cA}: \cA * \cB \to \cA$ is the unique trace-preserving conditional expectation.
\end{definition}

\begin{theorem} \label{thm:Vheatsemigroup1}
	Let $\mathbf{J} \in \mathscr{J}_{a,c}^d$ for some $a \in \R$ and $c \in (0,1)$.  Let $f \in C_{\tr}^k(\R^{*d})$.
	\begin{enumerate}[(1)]
		\item We have $e^{tL_{\mathbf{J}}} f \in C_{\tr}^k(\R^{*d})$.
		\item As $t \to \infty$, the function $e^{tL_{\mathbf{J}}} f$ converges in $C_{\tr}^k(\R^{*d})$ to a constant $\mathbb{E}_{\mathbf{J}}f$.
		\item The integral $\Psi_{\mathbf{J}} f = \int_0^\infty [e^{tL_{\mathbf{J}}} - \mathbb{E}_{\mathbf{J}}]f\,dt$ makes sense as an improper Riemann integral in $C_{\tr}^k(\R^{*d})$.
		\item We have
		\[
		-L_{\mathbf{J}} \Psi_{\mathbf{J}} + \mathbb{E}_{\mathbf{J}} = -\Psi_{\mathbf{J}} L_{\mathbf{J}} + \mathbb{E}_{\mathbf{J}} = \id
		\]
		as operators $C_{\tr}^k(\R^{*d}) \to C_{\tr}^k(\R^{*d})$.
	\end{enumerate}
\end{theorem}

This theorem is a summary of the results we will prove in this section.  In particular,
\begin{enumerate}[(1)]
	\item See Lemma \ref{lem:heatsemigroupbounds}.
	\item See Proposition \ref{prop:kernelprojection} and \eqref{eq:heatsemigroupestimate2}.
	\item See Proposition \ref{prop:pseudoinverse}.
	\item See Proposition \ref{prop:Laplacianrelations}.
\end{enumerate}
Actually, as we are interested in studying conditional distributions and conditional transport, we will prove a more general result, which allows $\mathbf{J}$ and $f$ to depend on an auxiliary variable $\mathbf{x}'$.  We will furthermore allow the function $\mathbf{f}$ to be in $C_{\tr}^k(\R^{*(d+d')},\mathscr{M}(\R^{*d_1},\dots,\R^{*d_\ell}))^{d''}$ for some $\ell \in \N_0$ and $d_1$, \dots, $d_\ell$, and $d'' \in \N$.  The more general definition of the heat semigroup is as follows.

\begin{definition}
	Consider formal variables $\mathbf{x} = (x_1,\dots,x_d)$ and $\mathbf{x}' = (x_1',\dots,x_{d'}')$.  Let $\pi(\mathbf{x},\mathbf{x}') = \mathbf{x}$ and $\pi'(\mathbf{x},\mathbf{x}') = \mathbf{x}'$.  Moreover, let $\Pi(\mathbf{x},\mathbf{x}')[\mathbf{y},\mathbf{y}'] = \mathbf{y}$ and $\Pi'(\mathbf{x},\mathbf{x'})[\mathbf{y},\mathbf{y}'] = \mathbf{y}'$, where $\mathbf{y}$ is a $d$-tuple and $\mathbf{y}'$ is a $d'$-tuple.  Then define
	\begin{multline*}
		\mathscr{J}_{a,b}^{d,d'} := \{\mathbf{J} \in C_{\tr}^\infty(\R^{*(d+d')})_{\sa}^d \}: \norm{\mathbf{J} - \pi}_{BC_{\tr}(\R^{*(d+d')})^d} \leq a, \\
		\norm{\partial \mathbf{J} - \Pi}_{BC_{\tr}(\R^{*(d+d')},\mathscr{M}(\R^{*d}))^d} \leq 1 - c\}.
	\end{multline*}
\end{definition}

\begin{definition} \label{def:Vheatsemigroup2}
	Let $\mathbf{J} \in \mathscr{J}_{a,b}^{d,d'}$.  Let $(\cA,\tau)$ be a tracial $\mathrm{W}^*$-algebra, let $(\cB,\sigma)$ be the tracial $\mathrm{W}^*$-algebra generated by a $d$-tuple of freely independent self-adjoint free Brownian motions $(\mathcal{S}_1(t),\dots,\mathcal{S}_d(t))$ for $t \in [0,\infty)$, and let $(\cA * \cB, \tau * \sigma)$ be the tracial free product of $(\cA,\tau)$ and $(\cB,\sigma)$.  For $\mathbf{X} = (X_1,\dots,X_d) \in \cA_{\sa}^d$ and $\mathbf{X}' = (X_1',\dots,X_{d'}') \in \cA_{\sa}^{d'}$, let $\mathcal{X}^{\cA,\tau}(X,X',t)$ be the solution to the integral equation
	\[
	\mathcal{X}(\mathbf{X},\mathbf{X}',t) = \mathbf{X} + \mathcal{S}(t) + \int_0^t \mathbf{J}(\mathcal{X}(\mathbf{X},\mathbf{X}',u),X')\,du
	\]
	(which we will show is well-defined in Lemma \ref{lem:SDEsolution}).  Note that $\mathcal{X}^{\cA,\tau}$ is a function $\cA_{\sa}^{d + d'} \times [0,\infty) \to (\cA * \cB)_{\sa}^d$.  For $f \in C_{\tr}^k(\R^{*(d+d')},\mathscr{M}(\R^{*d_1},\dots,\R^{*d_\ell}))^{d''}$, we define
	\[
	(e^{tL_{\mathbf{x},\mathbf{J}}}f)^{\cA,\tau}(\mathbf{X},\mathbf{X}')[\mathbf{Y}_1,\dots,\mathbf{Y}_\ell] \\
	= E_{\cA}[f^{\cA * \cB, \tau * \sigma}(\mathcal{X}(X,X',2t),X')[\mathbf{Y}_1,\dots,\mathbf{Y}_\ell]],
	\]
	where $E_{\cA}: \cA * \cB \to \cA$ is the unique trace-preserving conditional expectation.
\end{definition}

We refer to Propositions \ref{prop:kernelprojection} and \ref{prop:pseudoinverse} for the precise generalizations of Theorem \ref{thm:Vheatsemigroup1} to the conditional setting.

\subsection{The process $\mathcal{X}(\mathbf{X},\mathbf{X}',t)$} \label{subsec:stochasicprocess}

	The bulk of the technical work to prove Theorem \ref{thm:Vheatsemigroup1} lies in showing that $\mathcal{X}$ is a ``$C_{\tr}^\infty$ function of $(\mathbf{X},\mathbf{X}')$ and $\mathcal{S}$'' in a certain sense.  Once we prove that, it is relatively easy to deduce that if $f$ is a $C_{\tr}^k$ function of $(\mathbf{X},\mathbf{X}')$, then so is $e^{tL_{\mathbf{x},\mathbf{J}}} f$, as we will do in \S \ref{subsec:heatsemigroup}.  The results of this section are closely parallel to \cite[\S 3.2]{DGS2016}, except with different spaces of functions.

Recall that $\mathcal{X}^{\cA,\tau}(\mathbf{X},\mathbf{X}',t)$ depends on $\mathbf{X}$ and $\mathbf{X}'$ as well as the free Brownian motion $\mathcal{S}(t)$, and thus we want to define a similar space to $C_{\tr}^k(\R^{*(d+d')})$ which also allows dependence on a freely independent free Brownian motion.  Since of course we will need to study the space-derivatives of $\mathcal{X}^{\cA,\tau}(\mathbf{X},\mathbf{X}',t)$ of arbitrary orders, this involves defining analogs of $C_{\tr}(\R^{*(d+d')},\mathscr{M}(\R^{*d_1},\dots,\R^{*d_\ell}))^{d''}$ that also allow dependence on $\mathcal{S}(t)$.  For simplicity, we call the tuple of formal variables $\mathbf{x}$ rather than $(\mathbf{x},\mathbf{x}')$ in the definition.

\begin{definition}
	Let $\mathbf{s}$ denote a collection of formal self-adjoint variables $(s_j(t))_{t \in [0,\infty),j \in [d]}$ and let $\mathbf{x}$ denote a collection of formal self-adjoint variables $x_1$, \dots, $x_{d'}$.  We denote by $\TrP_{\mathbf{s}}(\R^{*d},\mathscr{M}(\R^{*d_1},\dots,\R^{*d_\ell}))$ the space of trace polynomials in the formal variables $x_1$, \dots, $x_d$, $\{s(t)\}_{t \in [0,\infty)}$, and $\mathbf{y}_1$, \dots, $\mathbf{y}_\ell$ (where $\mathbf{y}_j$ is a $d_j$-tuple) that are real-multilinear in $y_1$, \dots, $y_\ell$.
\end{definition}

\begin{definition}
	With $x$ and $s$ as above, suppose that $f = (f^{\cA,\tau})_{(\cA,\tau) \in \mathbb{W}}$ is a tuple of functions where
	\[
	f^{\cA,\tau}: (\cA * \cB)_{\sa}^{d'} \times (\cA * \cB)_{\sa}^{d_1} \times \dots \times (\cA * \cB)_{\sa}^{d_\ell} \to (\cA * \cB)^{d''}
	\]
	is a function which is real-multilinear in the last $\ell$ variables.  We say that $f \in C_{\tr,\mathcal{S}}(\R^{*d'},\mathscr{M}(\R^{*d_1},\dots,\R^{*d_1}))^{d''}$ if for every $R > 0$ and $\epsilon > 0$, there exists a $g \in \TrP_{\mathbf{s}}(\R^{*d},\mathscr{M}(\R^{*d_1},\dots,\R^{*d_\ell})$ such that for every $(\cA,\tau)$ we have
	\[
	\sup \{\norm{\mathbf{f}^{\cA,\tau}(\mathbf{X}) - \mathbf{g}|_{\cA * \cB, \tau * \sigma}(\mathcal{S},\mathbf{X})}_{\mathscr{M}^\ell,\tr}: \mathbf{X} \in (\cA * \cB)_{\sa}^d \text{ with } \norm{\mathbf{X}}_\infty \leq R \} < \epsilon.
	\]
	We equip $C_{\tr,\mathcal{S}}(\R^{*d'},\mathscr{M}(\R^{*d_1},\dots,\R^{*d_\ell}))^{d''}$ with the Fr\'echet topology given by the seminorms
	\[
	\norm{f}_{C_{\tr,\mathcal{S}}(\R^{*d'},\mathscr{M}^\ell),R}
	:= \sup_{(\cA,\tau) \in \mathbb{W}} \sup \{\norm{f^{\cA,\tau}(\mathbf{X})}_{\mathscr{M}^\ell,\tr}: \\ \mathbf{X} \in (\cA * \cB)_{\sa}^d \text{ with } \norm{\mathbf{X}}_\infty \leq R \}
	\]
	for $R > 0$.
\end{definition}

\begin{definition}
	Let $k \in \N_0 \cup \{\infty\}$.  Suppose that $f = (f^{\cA,\tau})_{(\cA,\tau) \in \mathbb{W}}$ is a tuple of functions where
	\[
	\mathbf{f}^{\cA,\tau}: (\cA * \cB)_{\sa}^{d'} \times (\cA * \cB)^{d_1} \times \dots \times (\cA * \cB)^{d_\ell} \to (\cA * \cB)^{d''}
	\]
	is a function which is real-multilinear in the last $\ell$ variables.  We say that $f \in C_{\tr,\mathcal{S}}^k(\R^{*d'},\mathscr{M}^\ell)^{d'}$ if for every $k' \in \N_0$ with $k' \leq k$, there exists $\mathbf{g}_{k'} \in C_{\tr,\mathcal{S}}(\R^{*d'},\mathscr{M}^{\ell+k'})^{d'}$ such that for every $(\cA,\tau) \in \mathbb{W}$,
	\[
	\partial^{k'} \mathbf{f}^{\cA,\tau} = \mathbf{g}_{k'}^{\cA,\tau}
	\]
	as functions $(\cA * \cB)_{\sa}^{d_2+\ell+k'} \to \cA * \cB$.  We equip $C_{\tr,\mathcal{S}}^k(\R^{*d'},\mathscr{M}(\R^{*d_1},\dots,\R^{*d_\ell}))^{d''}$ with the family of seminorms
	\[
	\norm{\partial^{k'} \mathbf{f}}_{C_{\tr,\mathcal{S}}(\R^{*d'},\mathscr{M}^{\ell+k'})^{d''},R}
	\]
	for $k' \leq k$ and $j_1$, \dots, $j_{k'} \in [d']$ and $R > 0$.
\end{definition}

\begin{proposition} \label{prop:chainrule2}
	Lemma \ref{lem:composition} and Theorem \ref{thm:chainrule} hold with each space $C_{\tr}^k(\R^{*d'
	},\mathscr{M}(\R^{*d_1},\dots,\R^{*d_\ell}))^{d''}$ replaced by $C_{\tr,\mathcal{S}}^k(\R^{*d'},\mathscr{M}(\R^{*d_1},\dots,\R^{*d_\ell}))^{d''}$.
\end{proposition}

The proof of this proposition is exactly the same as the original statements, and so we leave the details to the reader.  Now we are ready to define the solution to the integral equation.  We continue to use $\mathcal{S}$ to denote a $d$-tuple of free Brownian motions.

\begin{lemma} \label{lem:SDEsolution}
	For each $(\cA,\tau)$, there exists a unique function $\mathcal{X}^{\cA,\tau}: (\cA * \cB)_{\sa}^{d+d'} \times [0,\infty) \to (\cA * \cB)_{\sa}^d$ that is continuous in $t$ and satisfies
	\begin{equation} \label{eq:mainSDE}
		\mathcal{X}^{\cA,\tau}(\mathbf{X},\mathbf{X}',t) = \mathbf{X} + \mathcal{S}(t) - \frac{1}{2} \int_0^t \mathbf{J}^{\cA * \cB,\tau * \sigma}(\mathcal{X}^{\cA,\tau}(\mathbf{X},\mathbf{X}',u),\mathbf{X}')\,du.
	\end{equation}
	Moreover, $\mathcal{X}$ defines a continuous map $[0,\infty) \to C_{\tr,\mathcal{S}}(\R^{*(d+d')})_{\sa}^d$ which satisfies
	\begin{equation} \label{eq:SDEestimate}
		\norm{\mathcal{X}(\cdot,t)}_{C_{\tr,\mathcal{S}}(\R^{*(d+d')})^d,R} \leq e^{-t/2}(R + 2) + (1 - e^{-t/2})\norm{\mathbf{J} - \pi}_{BC_{\tr}(\R^{*(d+d')})^d}.
	\end{equation}
\end{lemma}

\begin{proof}
	Define Picard iterates inductively by
	\begin{align*}
		\mathcal{X}_0^{\cA,\tau}(\mathbf{X},\mathbf{X}',t) &= \mathbf{X} \\
		\mathcal{X}_{n+1}^{\cA,\tau}(\mathbf{X},\mathbf{X}',t) &= \mathcal{S}(t) - \frac{1}{2} \int_0^t \mathbf{J}^{\cA * \cB,\tau * \sigma}(\mathcal{X}_n^{\cA,\tau}(\mathbf{X},\mathbf{X}',u),\mathbf{X}')\,du.
	\end{align*}
	We will show by induction $\mathcal{X}_n^{\cA,\tau}$ is well-defined and that $t \mapsto \mathcal{X}_n(\cdot,t)$ is a continuous map $[0,\infty) \to C_{\tr}(\R^{*(d+d')})_{\sa}^d$.  The base case is immediate.  For the induction step, recall that composition is a continuous operation by Lemma \ref{lem:composition} / Proposition \ref{prop:chainrule2}, and hence $\mathbf{J}(\mathcal{X}_n(\mathbf{x},\mathbf{x}',t),\mathbf{x}')$ defines a continuous map $[0,\infty) \to C_{\tr}(\R^{*(d+d')})_{\sa}^d$.  Thus, it makes sense to integrate from $0$ to $t$ using Riemann integration for functions taking values in a Fr\'echet space, and of course the output will again be a continuous function $[0,\infty) \to C_{\tr}(\R^{*(d+d')})_{\sa}^d$ (the argument is the same as in \cite[\S 14.3]{JekelThesis}).  Thus, $\mathcal{X}_{n+1}$ defines such a continuous function as desired.
	
	Next, we prove convergence of the Picard iterates as $n \to \infty$.  Because $\partial_{\mathbf{x}} \mathbf{J} - \Pi$ is globally bounded by $c$, it follows that $\mathbf{J}^{\cA * \cB, \tau * \sigma}$ is $(1 + c)$-Lipschitz in $\mathbf{X}$ (with respect to $\norm{\cdot}_\infty$).  This implies that for $n \geq 1$,
	\[
	\norm{\mathcal{X}_{n+1}^{\cA,\tau}(\mathbf{X},\mathbf{X}',t) - \mathcal{X}_n^{\cA,\tau}(\mathbf{X},\mathbf{X}',t)}_\infty \leq \frac{1+c}{2} \int_0^t \norm{\mathcal{X}_n^{\cA,\tau}(\mathbf{X},\mathbf{X}',u) - \mathcal{X}_{n-1}^{\cA,\tau}(\mathbf{X},\mathbf{X}',u)}_\infty\,du,
	\]
	so that
	\begin{equation} \label{eq:Picarderror}
		\norm{\mathcal{X}_{n+1}(\cdot,t) - \mathcal{X}_n(\cdot,t)}_{C_{\tr}(\R^{*(d+d')})_{\sa}^d,R} \leq \frac{1+c}{2} \int_0^t \norm{\mathcal{X}_n(\cdot,u) - \mathcal{X}_{n-1}(\cdot,u)}_{C_{\tr}(\R^{*(d+d')})_{\sa}^d,R}\,du.
	\end{equation}
	Let
	\[
	C(t,R) = \sup_{u \in [0,t]} \norm{\mathcal{X}_1(\cdot,u) - \mathcal{X}_0(\cdot,u)}_{C_{\tr}(\R^{*(d+d')})_{\sa}^d,R}.
	\]
	Then a straightforward induction argument shows that
	\[
	\norm{\mathcal{X}_{n+1}(\cdot,t) - \mathcal{X}_n(\cdot,t)}_{C_{\tr}(\R^{*(d+d')})_{\sa}^d,R} \leq C(T,R) \frac{(1+c)^k t^n}{2^n n!},
	\]
	for $t \in [0,T]$.  This implies the convergence of $\mathcal{X}_n$ in $C_{\tr}(\R^{*(d+d')})_{\sa}^d$ uniformly for $t \in [0,T]$ as $n \to \infty$.  Thus, the limit $\mathcal{X}$ is a solution to the integral equation satisfying the desired continuity property.
	
	Note that we have asserted the uniqueness claim in a weaker setting than that of continuous functions $[0,\infty) \to C_{\tr}(\R^{*(d+d')})_{\sa}^d$.  Indeed, we claim that for a fixed $(\cA,\tau)$ and initial condition $\mathbf{X}$, the trajectory defined by the integral equation is unique.  This follows from the Picard-Lindel\"of theory because $\mathbf{J}^{\cA*\cB,\tau*\sigma}$ is Lipschitz in $\mathbf{X}$.
	
	Finally, to prove \eqref{eq:SDEestimate}, the idea is to ``differentiate'' $e^{t/2} \mathcal{X}^{\cA,\tau}(\mathbf{X},\mathbf{X}',t)$ with respect to $t$.  One can find a stochastic differential equation for $e^{t/2} \mathcal{X}(\mathbf{X},\mathbf{X}',t)$ using free It{\^o} calculus and then use standard SDE techniques to estimate it.  However, let us give this argument in an elementary language that does not require knowledge of free SDE.
	
	Fix $t$ and $n$, and let $t_j = jt/n$ for $j = 0$, \dots, $n$.  Then
	\[
	\mathcal{X}^{\cA,\tau}(\mathbf{X},\mathbf{X}',t_j) - \mathcal{X}^{\cA,\tau}(\mathbf{X},\mathbf{X}',t_{j-1})
	= \mathcal{S}(t_j) - \mathcal{S}(t_{j-1}) - \frac{1}{2} \int_{t_{j-1}}^{t_j} \mathbf{J}^{\cA*\cB,\tau*\sigma} (\mathcal{X}^{\cA,\tau}(\mathbf{X},\mathbf{X}',u),\mathbf{X}')\,du.
	\]
	Let $\mathbf{K} = \mathbf{J} - \pi$.  By continuity of $\mathcal{X}$ in $t$, we have
	\[
	\int_{t_{j-1}}^{t_j} \mathcal{X}^{\cA,\tau}(\mathbf{X},\mathbf{X}',u)\,du = (t/n) \mathcal{X}^{\cA,\tau}(\mathbf{X},\mathbf{X}',t_j) + o(1/n),
	\]
	where the error estimate holds uniformly for $\norm{(\mathbf{X},\mathbf{X}')} \leq R$ and is independent of $j$.  Thus,
	\begin{multline*}
		(1 + t/2n) \mathcal{X}^{\cA,\tau}(\mathbf{X},\mathbf{X}',t_j) - \mathcal{X}^{\cA,\tau}(\mathbf{X},\mathbf{X}',t_{j-1}) \\
		= \mathcal{S}(t_j) - \mathcal{S}(t_{j-1}) - \frac{1}{2} \int_{t_{j-1}}^{t_j} \mathbf{K}^{\cA*\cB,\tau*\sigma}(\mathcal{X}^{\cA,\tau}(\mathbf{X},\mathbf{X}',u),\mathbf{X}')\,du + o(1/n).
	\end{multline*}
	Note that $1 + t/n = e^{t/2n} + o(1/n)$ and hence
	\begin{multline*}
		e^{t/2n} \mathcal{X}^{\cA,\tau}(\mathbf{X},\mathbf{X}',t_j) - \mathcal{X}^{\cA,\tau}(\mathbf{X},\mathbf{X}',t_{j-1}) \\
		= \mathcal{S}(t_j) - \mathcal{S}(t_{j-1}) - \frac{1}{2} \int_{t_{j-1}}^{t_j} e^{(u-t_{j-1})/2} \mathbf{K}^{\cA*\cB,\tau*\sigma}(\mathcal{X}^{\cA,\tau}(\mathbf{X},\mathbf{X}',u),\mathbf{X}')\,du + o(1/n).
	\end{multline*}
	Now multiply by $e^{t_{j-1}/2}$ and sum from $j = 1$ to $n$ to obtain
	\begin{multline} \label{eq:Riemannapproximation}
		e^{t/2} \mathcal{X}^{\cA,\tau}(\mathbf{X},\mathbf{X}',t) - \mathbf{X} \\ = \sum_{j=1}^n e^{t_{j-1}/2} [\mathcal{S}(t_j) - \mathcal{S}(t_{j-1})] + \int_0^t e^{u/2} \mathbf{K}^{\cA*\cB,\tau*\sigma}(\mathcal{X}^{\cA,\tau}(\mathbf{X},\mathbf{X}',u),\mathbf{X}')\,du + o(1),
	\end{multline}
	where the error estimate $o(1)$ holds uniformly as $n \to \infty$ for $\norm{(\mathbf{X},\mathbf{X}')}_\infty \leq R$ (and in fact independently of $(\cA,\tau)$).  Note that
	\[
	\sum_{j=1}^n e^{t_{j-1}/2} [\mathcal{S}(t_j) - \mathcal{S}(t_{j-1})]
	\]
	is a sum of freely independent semicircular $d$-tuples of mean zero and hence it is a free semicircular $d$-tuple of mean zero, such that each coordinate has variance
	\[
	\sum_{j=1}^n e^{t_{j-1}}(t_j - t_{j-1}) \leq \int_0^t e^{u/2}\,du = e^t - 1 \leq e^t.
	\]
	Hence,
	\[
	\norm*{ \sum_{j=1}^n e^{t_{j-1}/2} [\mathcal{S}(t_j) - \mathcal{S}(t_{j-1})] }_\infty \leq 2e^{t/2}.
	\]
	We also have
	\[
	\norm*{ \int_0^t e^{u/2} \mathbf{K}^{\cA*\cB,\tau*\sigma}(\mathcal{X}^{\cA,\tau}(\mathbf{X},\mathbf{X}',u),\mathbf{X}')\,du} \leq (1 - e^{-t/2})\norm{\mathbf{K}}_{BC_{\tr}(\R^{*(d+d')})^d}.
	\]
	Thus, upon taking $n \to \infty$ in \eqref{eq:Riemannapproximation}, we obtain the desired estimate.
\end{proof}

Since $t \mapsto \mathcal{X}(\cdot,t)$ is a continuous map $[0,\infty) \to C_{\tr,\mathcal{S}}(\R^{*(d+d')})_{\sa}^d$, we can define the Riemann integral
\[
\int_0^t \mathbf{J}(\mathcal{X}(\cdot,u),\pi')\,du,
\]
where $\mathbf{J}(\mathcal{X}(\cdot,u),\pi')$ denotes the function in $C_{\tr,\mathcal{S}}(\R^{*(d+d')})_{\sa}^d$ given by composing $\mathcal{X}(\cdot,u)$ and $\pi'$ in the prescribed manner.  Relying once again on the fact that the Riemann integrals are defined for continuous functions from $[0,t]$ to a Fr\'echet space, it follows that the identity
\[
\mathcal{X}(\cdot,t) = \mathcal{S}(t) - \frac{1}{2} \int_0^t \mathbf{J}(\mathcal{X},\pi')\,du
\]
holds in $C_{\tr,\mathcal{S}}(\R^{*(d+d')})_{\sa}^d$.  Similarly, $t \mapsto \mathcal{X}(\cdot,t) - \mathcal{S}(t)$ is a continuously differentiable function $[0,\infty) \to C_{\tr,\mathcal{S}}(\R^{*(d+d')})^d$.  It will be convenient in the rest of the section to view our equations as integral / differential equations in $C_{\tr,\mathcal{S}}(\R^{*(d+d')})_{\sa}^d$ rather than equations for functions on $\cA_{\sa}^{d+d'}$ for every $(\cA,\tau)$ separately.

The next lemma will be used to construct the process $\partial \mathcal{X}(\cdot,t)$.

\begin{lemma} \label{lem:niceGronwall}
	Let $t \mapsto \mathcal{F}(\cdot,t)$ be a continuous function $[0,\infty) \to C_{\tr,\mathcal{S}}(\R^{*{d+d'}},\mathscr{M}(\R^{*d_1},\dots,\R^{*d_\ell}))^d$, and let $\mathcal{G}_0 \in C_{\tr,\mathcal{S}}(\R^{*(d+d')},\mathscr{M}(\R^{*d_1},\dots,\R^{*d_\ell}))^d$.  Then there exists a unique continuous $\mathcal{G}: [0,\infty) \to C_{\tr,\mathcal{S}}(\R^{*(d+d')},\mathscr{M}(\R^{*d_1},\dots,\R^{*d_\ell}))^d$ satisfying
	\begin{align}
		\mathcal{G}(\cdot,0) &= \mathcal{G}_0 \\
		\frac{d}{dt} \mathcal{G}(\cdot,t) &= -\frac{1}{2} \partial_{\mathbf{x}} \mathbf{J}(\mathcal{X}(\cdot,t),\pi') \# \mathcal{G}(\cdot,t) + \mathcal{F}(\cdot,t). \label{eq:mainODE}
	\end{align}
	Moreover, we have
	\begin{multline} \label{eq:niceGronwall}
		\norm{\mathcal{G}(\cdot,t)}_{C_{\tr,\mathcal{S}}(\R^{*(d+d')},\mathscr{M}^\ell)^d,R} \\
		\leq e^{-ct/2} \left(\norm{\mathcal{G}_0}_{C_{\mathcal{S}}(\R^{*(d+d')},\mathscr{M}^\ell)^d,R} + \int_0^t e^{cu/2} \norm{ \mathcal{F}(\cdot,u) }_{C_{\mathcal{S}}(\R^{*(d+d')},\mathscr{M}^\ell)^d,R} \,du \right).
	\end{multline}
\end{lemma}

\begin{proof}
	Recall our assumption that $\mathbf{J} = \pi + \mathbf{K}$ with
	\[
	\norm{\partial_{\mathbf{x}} \mathbf{K}}_{BC_{\mathcal{S}}(\R^{*(d+d')},\mathscr{M}^1)^d} \leq 1 - c.
	\]
	Hence,
	\[
	\norm{\partial_{\mathbf{x}} \mathbf{J}}_{BC_{\mathcal{S}}(\R^{*(d+d')},\mathscr{M}^1)^d} \leq 2 - c.
	\]
	It follows that for each $t$, the right-hand side of the differential equation depends in a Lipschitz manner upon $\mathcal{G}(\cdot,t)$ with respect to $\norm{\cdot}_{C_{\tr,\mathcal{S}}^k(\R^{*(d+d')},\mathscr{M}^\ell)^d,R}$ for every $R > 0$, with the Lipschitz constant being $(2 - c)/2$.  Hence, the standard Picard-Lindel\"of argument proves the existence and uniqueness of a solution.
	
	Because $\mathbf{J} = \pi + \mathbf{K}$, we also obtain
	\[
	\frac{d}{dt} \mathcal{G}(\cdot,t) + \frac{1}{2} \mathcal{G}(\cdot,t) = -\frac{1}{2} \partial_{\mathbf{x}} \mathbf{K}(\mathcal{X}(\cdot,t),\pi') \# \mathcal{G}(\cdot,t) + \mathcal{F}(\cdot,t).
	\]
	Hence, upon multiplying by $e^{t/2}$ and using the given bound for $\partial_{\mathbf{x}} \mathbf{K}$, we obtain
	\begin{multline*}
		\norm*{ \frac{d}{dt} \left[ e^{t/2} \mathcal{G}(\cdot,t) \right]}_{C_{\tr,\mathcal{S}}(\R^{*(d+d')},\mathscr{M}^\ell)^d,R} \\
		\leq  \frac{1}{2}(1 - c) \norm*{ e^{t/2} \mathcal{G}(\cdot,t) }_{C_{\tr,\mathcal{S}}(\R^{*(d+d')},\mathscr{M}^\ell)^d,R} + e^{t/2} \norm*{ \mathcal{F}(\cdot,t)}_{C_{\tr,\mathcal{S}}(\R^{*(d+d')},\mathscr{M}^\ell)^d,R} .
	\end{multline*}
	Using Gr\"onwall's inequality,
	\begin{multline*}
		\norm*{ e^{t/2} \mathcal{G}(\cdot,t) }_{C_{\tr,\mathcal{S}}(\R^{*(d+d')},\mathscr{M}^\ell)^d,R} \\
		\leq e^{(1-c)t/2} \left( \norm*{ \mathcal{G}_0 }_{C_{\tr,\mathcal{S}}(\R^{*(d+d')},\mathscr{M}^\ell)^d,R} + \int_0^t e^{-(1-c)u/2} e^{u/2} \norm*{ \mathcal{F}(\cdot,t)}_{C_{\tr,\mathcal{S}}(\R^{*(d+d')},\mathscr{M}^\ell)^d,R}\,du \right).
	\end{multline*}
	This simplifies to the desired estimate \eqref{eq:niceGronwall}.
\end{proof}

Next, we explain how to differentiate $\mathcal{G}(\cdot,t)$ with respect to $(\mathbf{x},\mathbf{x}')$ in the situation of Lemma \ref{lem:niceGronwall} when $\mathbf{F}$ is a $C_{\tr,\mathcal{S}}^1$ function.  This will allow us to show that $\mathcal{X}(\cdot,t)$ is a $C_{\tr,\mathcal{S}}^\infty$ function by induction.

\begin{lemma} \label{lem:niceGronwall2}
	Let $t \mapsto \mathcal{F}(\cdot,t)$ be a continuous function $[0,\infty) \to C_{\tr,\mathcal{S}}^1(\R^{*{d+d'}},\mathscr{M}(\R^{*d_1},\dots,\R^{*d_\ell}))^d$, and let $\mathcal{G}_0 \in C_{\mathcal{S}}^1(\R^{*(d+d')},\mathscr{M}(\R^{*d_1},\dots,\R^{*d_\ell}))^d$.  Then the solution $\mathcal{G}$ in Lemma \ref{lem:niceGronwall} is a continuous function $[0,\infty) \to C_{\tr,\mathcal{S}}^1(\R^{*(d+d')},\mathscr{M}(\R^{*d_1},\dots,\R^{*d_\ell}))^d$, and we have
	\begin{equation} \label{eq:exchangedifferentiation}
		\frac{d}{dt} \partial \mathcal{G}(\cdot,t) = -\frac{1}{2} \partial_{\mathbf{x}} \mathbf{J}(\mathcal{X}(\cdot,t),\pi') \# \partial \mathcal{G}(\cdot,t)
		- \frac{1}{2} \partial[ \partial_{\mathbf{x}} \mathbf{J}(\mathcal{X}(\cdot,t),\pi')] \# [\mathcal{G}(\cdot,t), \Pi]
		+ \partial \mathcal{F}(\cdot,t).
	\end{equation}
\end{lemma}

\begin{proof}
	We claim that for each $t$, the right hand side of \eqref{eq:mainODE} depends in a Lipschitz manner upon $\mathcal{G}(\cdot,t)$ in $C_{\tr,\mathcal{S}}^1(\R^{*(d+d')},\mathscr{M}(\R^{*d_1},\dots,\R^{*d_\ell}))^d$.  More precisely, if we subtract the right hand side of \eqref{eq:mainODE} for two different functions $\mathcal{G}$ and $\mathcal{G}'$, then $\norm{\cdot}_{C_{\tr,\mathcal{S}}(\R^{*(d+d')},\mathscr{M}^\ell)^d,R} + \norm{\partial \cdot }_{C_{\tr,\mathcal{S}}(\R^{*(d+d')},\mathscr{M}^{\ell+1})^d,R}$ of the difference is bounded by a constant times
	\[
	\norm{\mathcal{G}(\cdot,t) - \mathcal{G}'(\cdot,t)}_{C_{\mathcal{S}}(\R^{*(d+d')},\mathscr{M}^\ell)^d,R} + \norm{\partial \mathcal{G}(\cdot,t) - \partial \mathcal{G}'(\cdot,t)}_{C_{\mathcal{S}}(\R^{*(d+d')},\mathscr{M}^\ell)^d,R}.
	\]
	We already explained in the proof of Lemma \ref{lem:niceGronwall} how to estimate $\mathcal{G}(\cdot,t) - \mathcal{G}'(\cdot,t)$ with respect to $\norm{\cdot}_{C_{\tr,\mathcal{S}}(\R^{*(d+d')},\mathscr{M}^\ell)^d,R}$.  To estimate $\partial \mathcal{G}(\cdot,t) - \partial \mathcal{G}(\cdot,t)$, note that applying $\partial$ to the right-hand side of \eqref{eq:mainODE} results in the right-hand side of \eqref{eq:exchangedifferentiation}. We subtract the right-hand side of \eqref{eq:exchangedifferentiation} at $\mathcal{G}$ from the corresponding quantity in $\mathcal{G}'$, and then estimate
	\begin{multline*}
		\norm*{\frac{1}{2} \partial_{\mathbf{x}} \mathbf{J}(\mathcal{X}(\cdot,t),\pi') \# (\partial \mathcal{G}'(\cdot,t) - \mathcal{G}(\cdot,t))}_{C_{\tr,\mathcal{S}}(\R^{*(d+d')},\mathscr{M}(\R^{*d_1},\dots,\R^{*d_\ell},\R^{*d}))} \\
		\leq 
		\frac{1}{2} \norm*{\partial_{\mathbf{x}} \mathbf{J}(\mathcal{X}(\cdot,t),\pi')}_{C_{\tr,\mathcal{S}}(\R^{*d},\mathscr{M}(\R^{*d})^d)} \norm{\partial \mathcal{G}'(\cdot,t) - \mathcal{G}(\cdot,t))}_{C_{\tr,\mathcal{S}}(\R^{*(d+d')},\mathscr{M}(\R^{*d_1},\dots,\R^{*d_\ell},\R^{*d}))},
	\end{multline*}
	and in turn,
	\[
	\norm{\partial[ \partial_{\mathbf{x}} \mathbf{J}(\mathcal{X}(\mathbf{X},\mathbf{X}',t),\mathbf{X}')] }_{C_{\tr}(\R^{*(d+d')},\mathscr{M}^2),R} \leq \norm{ \partial[ \partial_{\mathbf{x}} \mathbf{J}]}_{C_{\tr}(\R^{*(d+d')},\mathscr{M}^2),R'},
	\]
	where $R' = \max(R+2,\norm{\mathbf{J} - \pi}_{BC_{\tr}(\R^{*(d+d')})^d})$ using \eqref{eq:SDEestimate}.  The second term
	\[
	-\frac{1}{2} \partial[\partial_{\mathbf{x}} \mathbf{J}(\mathcal{X}(\cdot,t),\pi')] \# [\mathcal{G}'(\cdot,t) - \mathcal{G}(\cdot,t), \Pi]
	\]
	can be estimated similarly.  This shows the desired Lipschitz property, and hence the Picard-Lindel\"of method shows that the equation \eqref{eq:mainODE} has a solution in $C_{\tr,\mathcal{S}}^1(\R^{*(d+d')}, \mathscr{M}(\R^{*d_1},\dots,\R^{*d_\ell}))^d$.  This must agree with the solution in $C_{\tr,\mathcal{S}}(\R^{*(d+d')},\mathscr{M}(\R^{*d_1},\dots,\R^{*d_\ell}))^d$ from Lemma \ref{lem:niceGronwall}.  Then by applying $\partial$ to both sides, we obtain \eqref{eq:exchangedifferentiation}.
\end{proof}

\begin{lemma} \label{lem:processCinfinity}
	The function $\mathcal{X}$ from Lemma \ref{lem:SDEsolution} is a continuous map $[0,\infty) \to C_{\tr,\mathcal{S}}^\infty(\R^{*(d+d')})_{\sa}^d$.  Moreover, there exist constants $C_{k,\mathbf{J},R}$ such that
	\begin{equation} \label{eq:derivativeestimate1}
		\norm{\partial^k \mathcal{X}(\cdot,t)}_{C_{\tr,\mathcal{S}}(\R^{*(d+d')},\mathscr{M}^k)^d,R} \leq C_{k,\mathbf{J},R}
	\end{equation}
	for $k \geq 1$ and polynomials $p_{k,\mathbf{J},R}: \R \to \R$ such that $p_{k,\mathbf{J},R}$ has degree $k$ and
	\begin{equation} \label{eq:derivativeestimate2}
		\norm{\partial_{\mathbf{x}} \partial^k \mathcal{X}(\cdot,t)}_{C_{\tr,\mathcal{S}}(\R^{*(d+d')},\mathscr{M}^{k+1})^d,R} \leq e^{-ct/2} p_{k,\mathbf{J},R}(t)
	\end{equation}
	for $k \geq 0$.
\end{lemma}

\begin{proof}
	Let $\pi(\mathbf{X},\mathbf{X}') = \mathbf{X}$ and $\pi'(\mathbf{X},\mathbf{X}') = \mathbf{X}'$.  We claim that for each $k \geq 1$, $t \mapsto \partial^k \mathcal{X}(\cdot,t)$ is a continuous function
	\[
	[0,\infty) \to BC_{\tr,\mathcal{S}}(\R^{*(d+d')},\mathscr{M}(\underbrace{\R^{*(d+d')},\dots,\R^{*(d+d')}}_{k}))^d
	\]
	and it satisfies
	\begin{multline} \label{eq:iterateddifferentiation}
		\frac{d}{dt} \partial^k \mathcal{X}(\cdot,t) = -\frac{1}{2} \sum_{k'=0}^k \sum_{j=0}^{k - k'} \binom{j + k'}{j} \sum_{\substack{(B_1, \dots, B_j) \\ \text{partition of } [j] \\ \min(B_1) < \dots < \min(B_j)}}
		\frac{1}{k!} \sum_{\sigma \in \Perm([k])} \partial_{\mathbf{x}'}^{k'} \partial_{\mathbf{x}}^j \mathbf{J}(\mathcal{X}(\cdot,t),\pi') \\ \# [\partial^{|B_1|} \mathcal{X}(\cdot,t), \dots, \partial^{|B_j|} \mathcal{X}(\cdot,t), \underbrace{\Pi',\dots,\Pi'}_{k'}]_\sigma.
	\end{multline}
	We will deduce this from Lemma \ref{lem:niceGronwall2} by induction.
	
	We make a few preliminary comments on the form of the above equation before we show the terms are well-defined.  We obtained \eqref{eq:iterateddifferentiation} by formally repeatedly differentiating the equation for $\mathcal{X}(\cdot,t)$ using the chain rule.  More precisely, we differentiated the composition of $\mathbf{J}$ with $(\mathcal{X}(\cdot,t),\pi')$, and evaluated the derivative of the inner function as $(\partial \mathcal{X}(\cdot,t), \Pi')$, and then expressed the result in terms of these two pieces.  We moved the occurrences of $\Pi'$ to the right for each term.  In order not to worry about which order to plug in the tangent vectors, we symmetrized over $\Perm(k)$, which is valid because the $k$th derivative is a symmetric $k$-linear map.
	
	On the right-hand side of \eqref{eq:iterateddifferentiation}, the term with $k' = 0$, $j = 1$, and $B_1 = [k]$ is exactly
	\[
	\partial_{\mathbf{x}} \mathbf{J}(\mathcal{X}(\cdot,t),\pi') \# \partial^k \mathcal{X}(\cdot,t),
	\]
	and all the other terms only involve lower-order derivatives of $\mathcal{X}(\cdot,t)$.  We will denote the sum of all these other terms by $\mathcal{F}^{(k)}(\cdot,t)$.
	
	Now we prove by induction on $k$ that $\mathcal{X}(\cdot,t)$ defines a continuous map
	\[
	[0,\infty) \to C_{\tr,\mathcal{S}}(\R^{*(d+d')},\mathscr{M}(\underbrace{\R^{*(d+d')},\dots,\R^{*(d+d')}}_k))^d
	\]
	(and hence $\mathcal{F}^{(k)}$ is also well-defined) and that $\mathbf{X}$ satisfies the formula \eqref{eq:iterateddifferentiation} and the estimate \eqref{eq:derivativeestimate1}.
	
	For the base case $k = 1$, let $\mathcal{G}: [0,\infty) \to C_{\tr}(\R^{*(d+d')}, \mathscr{M}(\R^{*d}))^d$ be the solution to
	\begin{align*}
		\mathcal{G}(\cdot,0) &= \pi, \\
		\mathcal{G}(\cdot,t) &=  -\frac{1}{2} \partial_{\mathbf{x}} \mathbf{J} (\mathcal{X}(\cdot,t),\pi')\# [\mathcal{G}(\cdot,t)] + \partial_{\mathbf{x}'} \mathbf{J}(\mathcal{X}(\cdot,t),\pi') \# \Pi'.
	\end{align*}
	The solution exists by applying Lemma \ref{lem:niceGronwall} with $\mathcal{F}$ given by 
	\[
	\partial_{\mathbf{x}'} \mathbf{J}(\mathcal{X}(\cdot,t),\pi') \# \Pi' = \partial_{\mathbf{x}'} \mathbf{K}(\mathcal{X}(\cdot,t),\pi') \# \Pi',
	\]
	which is bounded by a constant $C_{1,\mathbf{J}}'$ by assumption.  Thus, by \eqref{eq:niceGronwall}, we have
	\[
	\norm{\mathcal{G}(\cdot,t)}_{BC_{\tr,\mathcal{S}}(\R^{*(d+d')})^d} \leq e^{-ct/2}\left(1 + \frac{2}{c} C_{1,\mathbf{J}}'(e^{ct/2} - 1)  \right),
	\]
	which is bounded by a constant $C_{1,\mathbf{J}}$.
	
	To complete the base case, we need to show $\mathcal{G} = \partial \mathcal{X}$.  Let $\mathcal{X}_n$ be the Picard iterate as in the proof of Lemma \ref{lem:SDEsolution}.  Using continuity of the composition operation on $C_{\tr,\mathcal{S}}^1$ functions, we see that $\mathcal{X}_n$ is in $C_{\tr,\mathcal{S}}^1(\R^{*(d+d')})_{\sa}^d$, and we have
	\[
	\partial \mathcal{X}_{n+1}(\cdot,t) = \Pi - \frac{1}{2} \int_0^t \partial_{\mathbf{x}} \mathbf{J}(\mathcal{X}_n(\cdot,u),\pi')\# \partial \mathcal{X}_n(\cdot,u)
	+ \partial_{\mathbf{x}'} \mathbf{J}(\mathcal{X}_n(\cdot,u),\pi') \# \Pi'\,du.
	\]
	By the same token as \eqref{eq:Picarderror}, we have
	\[
	\norm{\mathcal{X}_{n+1}(\cdot,t) - \mathcal{X}(\cdot,t)}_{C_{\tr}(\R^{*(d+d')})_{\sa}^d,R} \leq \frac{1+c}{2} \int_0^t \norm{\mathcal{X}_n(\cdot,u) - \mathcal{X}(\cdot,u)}_{C_{\tr}(\R^{*(d+d')})_{\sa}^d,R}\,du.
	\]
	In a similar way, we have
	\begin{multline*}
		\norm*{ \partial \mathcal{X}_{n+1}(\cdot,t) - \mathcal{F}(\cdot,t) }_{C_{\tr}(\R^{*(d+d')})^d,R} \leq \int_0^t \biggl(\frac{1-c}{2} \norm{ \partial \mathcal{X}_n(\cdot,u) - \mathcal{F}(\cdot,u)}_{C_{\tr,\mathcal{S}}(\R^{*(d+d')},\mathscr{M}^1)^d,R} \\
		+ \norm{\partial_{\mathbf{x}} \partial \mathbf{J}}_{C_{\tr}(\R^{*(d+d')},\mathscr{M}^2)^d,R} \norm{\mathcal{X}_n(\cdot,u) - \mathcal{X}(\cdot,u)}_{C_{\tr,\mathcal{S}}(\R^{*(d+d')})^d,R} (C_{1,\mathbf{J}} + 1) \biggr)\,du
	\end{multline*}
	where the first error term comes from swapping out the $\partial \mathcal{X}_n$ in $\partial_{\mathbf{x}} \mathbf{J}(\mathcal{X}_n(\cdot,u),\pi') \# \partial \mathbf{X}_n(\cdot,u)$ for $\mathcal{F}$, and the second error times comes from swapping out $\mathcal{X}$ for $\mathcal{X}_n$ inside $\partial_{\mathbf{x}} \mathbf{J}$.  Altogether the function
	\[
	\phi_{n,R}(t) := \norm{\mathcal{X}_{n+1}(\cdot,t) - \mathcal{X}(\cdot,t)}_{C_{\tr}(\R^{*(d+d')})_{\sa}^d,R} + \norm{ \partial \mathcal{X}_{n+1}(\cdot,t) - \mathcal{F}(\cdot,t) }_{C_{\tr}(\R^{*(d+d')})^d,R}
	\]
	satisfies
	\[
	\phi_{n+1,R}(t) \leq K \int_0^t \phi_{n,R}(u)\,du
	\]
	for some constant $K$ that depends only on $W$,
	and this implies that $\phi_{n,R} \to 0$ uniformly on compact sets as $n \to \infty$.  Thus, $\partial \mathcal{X}_n$ converges to $\mathcal{F}$ in $C_{\tr}(\R^{*(d+d')},\mathscr{M})^d$ as $n \to \infty$.  It follows that $\mathcal{X}$ is in $C_{\tr}^1(\R^{*(d+d')})_{\sa}^d$ and $\partial \mathcal{X} = \mathcal{F}$.
	
	For the induction step, suppose the claim holds for $k-1$, so that
	\[
	\frac{d}{dt} \partial^{k-1} \mathcal{X}(\cdot,t) =
	- \frac{1}{2} \partial_{\mathbf{x}} \mathbf{J}(\mathcal{X}(\cdot,t), \pi') \# \partial^{k-1} \mathcal{X}(\cdot,t) + \mathcal{F}^{(k)}(\mathbf{X},\mathbf{X}',t).
	\]
	Then by Lemma \ref{lem:niceGronwall2}, we deduce that $\partial^{k-1} \mathcal{X}$ is in $BC_{\tr,\mathcal{S}}^1(\R^{*(d+d')},\mathscr{M}^{k-1})^d$ (and depends continuously on $t$) and that $\partial^k \mathcal{X}$ satisfies the differential equation computed by applying $\partial$ termwise to both sides. This computation of derivatives results in \eqref{eq:iterateddifferentiation}.  Next, by our induction hypothesis the spatial derivatives of $\mathcal{X}(\cdot,t)$ of order $< k$ satisfy \eqref{eq:derivativeestimate1}.  This implies that $\mathcal{F}^{(k)}$ is bounded in $BC_{\tr,\mathcal{S}}(\R^{*(d+d')},\mathscr{M}^k)^d$ by some constant $C_{k,\mathbf{J},R}'$ independent of $t$, because the derivatives of $\mathcal{X}$ of order $< k$ are bounded on each ball of radius $R$, and so are the derivatives of $\mathbf{J}(\mathcal{X},\pi')$.  Now we apply \eqref{eq:niceGronwall} with $\mathcal{G} = \partial^k \mathcal{X}$, noting that $\mathcal{G}_0 = 0$ for $k \geq 2$, and thus conclude that
	\[
	\norm{\partial^k \mathcal{X}(\cdot,t)}_{C_{\tr,\mathcal{S}}(\R^{*(d+d')},\mathscr{M}^k)^d,R} \leq e^{-ct/2} \int_0^t e^{cu/2} C_{k,\mathbf{J},,R}' \,du \leq \frac{2}{c} C_{k,\mathbf{J},R}' =: C_{k,\mathbf{J},R}.
	\]
	
	To show \eqref{eq:derivativeestimate2}, we again proceed by induction on $k$.  We can deduce a different equation for $\partial_{\mathbf{x}} \partial^k \mathcal{X}(\cdot,t)$ from \eqref{eq:iterateddifferentiation}, which has the same type of terms as \eqref{eq:iterateddifferentiation} except that each term has one multilinear argument of the form $\partial^j \mathcal{X}$ replaced by $\partial_{\mathbf{x}} \partial^j \mathcal{X}$.  As before, one of the terms is
	\[
	-\frac{1}{2} \partial_{\mathbf{x}} \mathbf{J}( \mathcal{X}(\cdot,t),\pi') \# \partial_{\mathbf{x}} \partial^k \mathcal{X}(\cdot,t),
	\]
	while all the other terms involve lower-order derivatives of $\mathcal{X}$.  We separate this first term out, and denote the sum of the remaining terms by $\mathcal{H}^{(k)}(\cdot,t)$.
	
	For the base case $k = 0$, we have $\mathcal{H}^{(0)} = 0$ and $\partial_{\mathbf{x}} \mathcal{X}(\cdot,0) = \id_d$.  Thus, using \eqref{eq:niceGronwall} with $\mathcal{F} = \mathcal{H}^{(0)}$, we get
	\[
	\norm{\partial_{\mathbf{x}} \mathcal{X}(\cdot,t)}_{BC_{\tr,\mathcal{S}}(\R^{*(d+d')},\mathscr{M}^1)^d} \leq e^{-ct/2}.
	\]
	Thus, the claim holds with $p_{0,W,R}(t) = 1$.
	
	For the induction step, let $k \geq 2$, and suppose the claim holds for $k - 1$.  Observe that $\mathcal{H}^{(k)}$ is bounded in $\norm{\cdot}_{C_{\tr,\mathcal{S}}(\R^{*(d+d')},\mathscr{M}^k)^d,R}$ by $e^{-ct/2} p_{k,\mathbf{J},R}'(t)$ for some polynomial $p_{k,\mathbf{J},R}'$ of degree $k - 1$.  This is verified by using the induction hypothesis for \eqref{eq:derivativeestimate2} on each occurrence of $\partial_{\mathbf{x}} \partial^j \mathcal{X}$ in $\mathcal{H}^{(k)}$ (there being one occurrence per summand) and applying \eqref{eq:derivativeestimate1} to all the other terms.  Then we apply \eqref{eq:niceGronwall} to $\partial_{\mathbf{x}} \partial^k \mathcal{X}$, noting that it vanishes when $t = 0$, and thus obtain
	\begin{align*}
		\norm{\partial_{\mathbf{x}} \partial^k \mathcal{X}(\cdot,t)}_{C_{\tr,\mathcal{S}}(\R^{*(d+d')},\mathscr{M}^{k+1})^d,R} &\leq e^{-ct/2} \int_0^t e^{cu/2} e^{-cu/2} p_{k,\mathbf{J},R}'(u)\,du \\
		&=: e^{-ct/2} p_{k,\mathbf{J},R}(t).
	\end{align*}
	This completes the inductive step and hence verifies \eqref{eq:derivativeestimate2}.
\end{proof}

\begin{remark} \label{rem:boundedcoefficients}
	From the proof, it is apparent that $C_{1,\mathbf{J},R}$ is independent of $R$.  Moreover, for $k > 1$, the constant $C_{k,\mathbf{J},R}$ only depends on $\norm{\partial^{k'} (\mathbf{J} - \pi)}_{C_{\tr}(\R^{*d},\mathscr{M}^{k'})^d,R'}$ for $k' \leq k$, where $R' = \max(R + 2, \norm{\mathbf{J} - \pi}_{BC_{\tr}(\R^{*(d+d')},\mathscr{M}^1)})$.  In particular, if $\mathbf{J} - \pi \in BC_{\tr}^k(\R^{*(d+d')},\mathscr{M}^1)$, then $\partial \mathcal{X} \in BC_{\tr,\mathcal{S}}(\R^{*(d+d')},\mathscr{M}^1)$.
\end{remark}

\subsection{The semigroup $e^{tL_{\mathbf{x},\mathbf{J}}}$} \label{subsec:heatsemigroup}

Next, we explain results about the heat semigroup parallel to \cite[\S 3.3]{DGS2016}.  To deduce smoothness for the heat semigroup from smoothness of the stochastic process $\mathcal{X}$, we use the following result about conditional expectations.

\begin{lemma} \label{lem:Ckconditionalexpectation}
	Let $k \in \N_0 \cup \{\infty\}$.  Let $d$, $d'$, $d'' \in \N$ and $\ell \in \N_0$ and $d_1$, \dots, $d_\ell \in \N$.  Let $\mathcal{S}$ be a $d$-variable free Brownian motion, and let $(\cB,\sigma)$ be the associated $\mathrm{W}^*$-algebra.  Let $\mathcal{F} \in C_{\tr,\mathcal{S}}^k(\R^{*d'},\mathscr{M}(\R^{*d_1},\dots,\R^{*d_\ell}))^{d''}$.  Recall that
	\[
	\mathcal{F}^{\cA,\tau}: (\cA * \cB)_{\sa}^{d'} \times (\cA * \cB)_{\sa}^{d_1} \times \dots \times (\cA * \cB)^{d_\ell} \to (\cA * \cB)^{d''},
	\]
	and let
	\[
	\mathbf{F}^{\cA,\tau}(\mathbf{X})[\mathbf{Y}_1,\dots,\mathbf{Y}_\ell] = E_{\cA} \left[ \mathcal{F}^{\cA,\tau}(\mathbf{X})[\mathbf{Y}_1,\dots,\mathbf{Y}_\ell] \right]
	\]
	for all $\mathbf{X} \in \mathcal{A}_{\sa}^{d'} \subseteq (\cA*\cB)_{\sa}^{d'}$ and $\mathbf{Y}_1$, \dots, $\mathbf{Y}_\ell$ with $\mathbf{Y}_j \in \cA_{\sa}^{d_j} \subseteq (\cA* \cB)_{\sa}^{d_j}$.  Then $\mathbf{F} = (\mathbf{F}^{\cA,\tau})_{(\cA,\tau) \in \mathbb{W}}$ is in $C_{\tr}^k(\R^{*d_1}, \mathscr{M}^\ell)^{d''}$ and for each $k' \leq k$ and $R > 0$,
	\[
	\norm{\partial^{k'} \mathbf{F}}_{C_{\tr}(\R^{*d_1}, \mathscr{M}^{\ell+k'})^{d''},R} \leq \norm{\partial^{k'} \mathcal{F}}_{C_{\tr,\mathcal{S}}(\R^{*d_1}, \mathscr{M}^{\ell+k'})^{d''},R}.
	\]
\end{lemma}

\begin{proof}
	Fix $(\cA,\tau)$.  Recall that $E_{\cA}: \cA * \cB \to \cA$ is a linear map which is bounded map with respect to $\norm{\cdot}_\infty$, the chain rule for Fr\'echet differentiation implies that $\mathbf{F}^{\cA,\tau}$ is Fr\'echet-$C^k$ and that for $k' \leq k$,
	\[
	\partial^{k'} \mathbf{F}^{\cA,\tau}(\mathbf{X})[\mathbf{Y}_1,\dots,\mathbf{Y}_{\ell+k'}] = E_{\cA}[\partial^{k'} \mathcal{F}^{\cA,\tau}(\mathbf{X})[\mathbf{Y}_1,\dots,\mathbf{Y}_{\ell+k'}]].
	\]
	Since $E_{\cA}$ is a contraction with respect to the non-commutative $L^\alpha$ norm for every $\alpha \in [1,\infty]$, we have
	\[
	\norm{\partial^{k'} \mathbf{F}}_{\mathscr{M}^{\ell+k'},\tr,R} \leq \norm{\partial^{k'} \mathcal{F}}_{\mathscr{M}^{\ell+k'},\tr,R}
	\]
	for every $R > 0$.  Note that this estimate is independent of $(\cA,\tau)$.
	
	For each $k'$, $R > 0$, and $\epsilon > 0$, there exists $\mathbf{g} \in \TrP_{\mathcal{S}}(\R^{*d'},\mathscr{M}(\R^{*d_1},\dots,\R^{*d_\ell}))^{d''}$ such that
	\[
	\norm{\partial^{k'} \mathcal{F}^{\cA,\tau} - \mathbf{g}^{(\cA,\tau)}}_{\mathscr{M}^\ell,\tr,R} \leq \epsilon \text{ for all } (\cA,\tau) \in \mathbb{W}.
	\]
	Now $\mathbf{g}$ is really a trace polynomial in the variables $\mathbf{x}$, $\mathbf{y}_1$, \dots, $\mathbf{y}_{\ell+k'}$ and $\mathcal{S}(t_1)$, \dots, $\mathcal{S}(t_m)$ for some finitely many times $0 < t_1 < \dots < t_m$.  We can rewrite this trace polynomial in terms of $\mathbf{X}$, the $\mathbf{Y}_j$'s, and the freely independent increments $\mathcal{S}(t_j) - \mathcal{S}(t_{j-1})$ for $j = 1$, \dots, $m$, where $t_0 := 0$; in other words, there exists $\widehat{\mathbf{g}} \in \TrP(\R^{*(d'+md)},\mathscr{M}(\R^{*d_1},\dots,\R^{*d_\ell})^{d''})$ such that
	\[
	\mathbf{g}^{\cA,\tau}(\mathbf{X})[\mathbf{Y}_1,\dots,\mathbf{Y}_{\ell+k'}] = \widehat{\mathbf{g}}^{\cA*\cB,\tau*\sigma}(\mathbf{X}, (t_1 - t_0)^{-1/2} (\mathcal{S}(t_1) - \mathcal{S}(t_0)), \dots, (t_m - t_{m-1})^{-1/2}(\mathcal{S}(t_m) - \mathcal{S}(t_{m-1}))[\mathbf{Y}_1,\dots,\mathbf{Y}_{\ell+k'}].
	\]
	Now $(t_1 - t_0)^{-1/2} (\mathcal{S}(t_1) - \mathcal{S}(t_0)), \dots, (t_m - t_{m-1})^{-1/2}(\mathcal{S}(t_m) - \mathcal{S}(t_{m-1})$ is a standard free semicircular $dm$-tuple.  Lemma \ref{lem:freeCE} implies that there is a trace polynomial $\mathbf{h} \in \TrP(\R^{*d'}, \mathscr{M}(\R^{*d_1},\dots,\R^{*d_\ell}))^{d''}$ such that
	\[
	E_{\cA} [\mathbf{g}^{\cA,\tau}(\mathbf{X})[\mathbf{Y}_1,\dots,\mathbf{Y}_{\ell+k'}] = \mathbf{h}^{\cA,\tau}(\mathbf{X})[\mathbf{Y}_1,\dots,\mathbf{Y}_{\ell+k'}]
	\]
	for every $(\cA,\tau) \in \mathbb{W}$, every $\mathbf{X} \in \cA_{\sa}^{d'}$, and every $\mathbf{Y}_1$, \dots, $\mathbf{Y}_{\ell+k'}$ with $\mathbf{Y}_j \in \cA_{\sa}^{d_j}$.  Then
	\[
	\norm{\partial^{k'} \mathbf{F}^{\cA,\tau} - \mathbf{h}^{\cA,\tau}}_{\mathscr{M}^{\ell+k'},\tr,R} \leq \epsilon \text{ for all } (\cA,\tau) \in \mathbb{W},
	\]
	and hence $\partial^{k'} \mathbf{F} \in C_{\tr}(\R^{*d_1},\mathscr{M}^{\ell+k'})^{d''}$.  This holds for all $k' \leq k$, hence $\mathbf{F} \in C_{\tr}^k(\R^{*d_1}, \mathscr{M}^\ell)^{d''}$.
\end{proof}

\begin{remark}
	In fact, in the above argument, one can compute $\mathbf{h}$ explicitly from $\mathbf{g}$ by studying the action on trace polynomials of the heat semigroup associated to the flat free Laplacian $L$ as in \cite[\S 2]{Cebron2013}, \cite[\S 3]{DHK2013}, \cite[\S 14.2]{JekelThesis}.  This reasoning could be applied here to those $dm$ inputs of the function $\mathbf{g}$ where the free semicircular family is located.
\end{remark}

\begin{lemma} \label{lem:heatsemigroupbounds}
	Let $k \in \N_0 \cup \{\infty\}$.  Then for $\mathbf{f} \in C_{\tr}^k(\R^{*(d+d')},\mathscr{M}(\R^{*d_1},\dots,\R^{*d_\ell}))^{d''}$, we have $e^{tL_{\mathbf{x},\mathbf{J}}} \mathbf{f} \in C_{\tr}^k(\R^{*(d+d')},\mathscr{M}(\R^{*d_1},\dots,\R^{*d_\ell}))^{d''}$.  Moreover, fix $R > 0$, and let
	\[
	R' = \max(R+2, \norm{\nabla_{\mathbf{X}} W}).
	\]
	then for $k' \leq k$,
	\begin{equation} \label{eq:heatsemigroupestimate1}
		\norm*{ \partial^{k'} [e^{tL_{\mathbf{x},\mathbf{J}}} \mathbf{f}]}_{C_{\tr}(\R^{*(d+d')},\mathscr{M}^{\ell+k'})^{d''},R} \leq C_{k',\mathbf{J},R} \sum_{j=1}^{k'} \norm{\partial^j \mathbf{f}}_{C_{\tr}(\R^{*(d+d')})^{d''},R'},
	\end{equation}
	where $C_{k',\mathbf{J},R}$ is a constant depending only on $k'$ and $W$ and $R$.  Also, if $\partial_{\mathbf{x}} \mathbf{f} \in C_{\tr}^k(\R^{*(d+d')},\mathscr{M}(\R^{*d_1},\dots,\R^{*d_\ell}))^{d''}$, then $k' \leq k$,
	\begin{equation} \label{eq:heatsemigroupestimate2}
		\norm*{ \partial_{\mathbf{x}} \partial^{k'} [e^{tL_{\mathbf{x},\mathbf{J}}} \mathbf{f}]}_{C_{\tr}(\R^{*(d+d')},\mathscr{M}^{\ell+k'+1})^{d''},R}
		\leq e^{-ct} p_{k',\mathbf{J},R}(t) \sum_{j=1}^{k'} \norm{\partial_{\mathbf{x}} \partial^j \mathbf{f}}_{C_{\tr}(\R^{*(d+d')})^{d''},R'},
	\end{equation}
	where $p_{k',\mathbf{J},R}$ is a polynomial of degree $k'$ depending only on $k'$ and $W$ and $R$.
\end{lemma}

\begin{remark}
	These are not the same constants and polynomials from Lemma \ref{lem:processCinfinity}, but they are derived from them.
\end{remark}

\begin{proof}
	Since $C_{\tr}^k(\R^{*(d+d')},\mathscr{M}(\R^{*d_1},\dots,\R^{*d_\ell}))^{d''} \subseteq C_{\tr,\mathcal{S}}^k(\R^{*(d+d')},\mathscr{M}(\R^{*d_1},\dots,\R^{*d_\ell})^{d''}$, we may view $\mathbf{f}$ as an element of the latter space.  By Lemma \ref{lem:processCinfinity}, $\mathcal{X} \in C_{\tr,\mathcal{S}}^\infty(\R^{*(d+d')})^d$ and hence by Proposition \ref{prop:chainrule2}, $\mathbf{f}(\mathcal{X}(\cdot,t),\pi')$ is a function in $C_{\tr}^k(\R^{*(d+d')},\mathscr{M}(\R^{*d_1},\dots,\R^{*d_\ell}))^{d''}$.  So by Lemma \ref{lem:Ckconditionalexpectation}, we $e^{tL_{\mathbf{x},\mathbf{J}}} \mathbf{f} \in C_{\tr}^k(\R^{*(d+d')},\mathscr{M}(\R^{*d_1},\dots,\R^{*d_\ell}))^{d''}$.
	
	To prove \eqref{eq:heatsemigroupestimate1}, observe that by similar reasoning as in \eqref{eq:iterateddifferentiation},
	\begin{multline} \label{eq:iterateddifferentiation2}
		\partial^{k'} [\mathbf{f}(\mathcal{X}(\cdot,2t),\pi')] \\
		= \sum_{k^*=0}^{k'} \sum_{j=0}^{k' - k^*} \binom{j + k^*}{j} \sum_{\substack{(B_1, \dots, B_j) \\ \text{partition of } [k'-k^*] \\ \min(B_1) < \dots < \min(B_j)}}
		\frac{1}{k'!} \sum_{\sigma \in \Perm([k'])} \partial_{\mathbf{x}'}^{k^*} \partial_{\mathbf{x}}^j  \mathbf{f}(\mathcal{X}(\cdot,2t),\pi') \\
		\# \bigl[\underbrace{\Id,\dots,\Id}_\ell, \partial^{|B_1|} \mathcal{X}(\mathbf{X},\mathbf{X}',2t), \dots, \partial^{|B_j|} \mathcal{X}(\mathbf{X},\mathbf{X}',2t), \underbrace{\Pi',\dots,\Pi'}_{k^*} \bigr]_\sigma.
	\end{multline}
	It follows from \eqref{eq:SDEestimate} that
	\[
	\norm{\mathcal{X}(\cdot,2t)}_{C_{\tr,\mathcal{S}}(\R^{*(d+d')})^d,R} \leq R',
	\]
	and the same estimate holds for $(\mathcal{X}(\cdot,2t), \pi')$ since $R' > R$.  Thus, using \eqref{eq:derivativeestimate1}, we can bound $\partial^{k'} [\mathbf{f}(\mathcal{X}(\cdot,t),\pi')]$ by the right-hand side of \eqref{eq:derivativeestimate1}, and then apply Lemma \ref{lem:Ckconditionalexpectation} to finish the proof of \eqref{eq:heatsemigroupestimate1}.  The proof of \eqref{eq:heatsemigroupestimate2} is similar using \eqref{eq:derivativeestimate2} instead of \eqref{eq:derivativeestimate1}.
\end{proof}

\begin{lemma} \label{lem:semigroup}
	For $s, t \geq 0$ and $\mathbf{f} \in C_{\tr}(\R^{*(d+d')}, \mathscr{M}(\R^{*d_1},\dots,\R^{*d_\ell}))^{d''}$, we have
	\[
	e^{sL_{\mathbf{x},\mathbf{J}}} [e^{tL_{\mathbf{x},\mathbf{J}}} \mathbf{f}] = e^{(s+t)L_{\mathbf{x},\mathbf{J}}} \mathbf{f}.
	\]
\end{lemma}

\begin{proof}
	Fix $(\cA,\tau)$, let $(\cB_1,\sigma_1)$ be a freely independent tracial von $\mathrm{W}^*$-algebra generated by a free Brownian motion $\mathcal{S}_1$, and let $(\cB_2,\sigma_2)$ be another freely independent copy of $(\cB,\sigma)$ generated by another free Brownian motion $\mathcal{S}_2$.  For each algebra $(\cA,\tau)$, and $j = 1, 2$, let $\mathcal{X}_j$ be the solution to \eqref{eq:mainSDE} with $\mathcal{S}_j$ instead of $\mathcal{S}$.  Then
	\begin{align*}
		[e^{sL_{\mathbf{x},\mathbf{J}}}&[e^{tL_{\mathbf{x},\mathbf{J}}} \mathbf{f}]]^{\cA,\tau}(\mathbf{X},\mathbf{X}') \\
		&= E_{\cA}[ [e^{tL_{\mathbf{x},\mathbf{J}}}\mathbf{f}]^{\cA*\cB_1,\tau*\sigma_1}(\mathcal{X}_1^{\cA,\tau}(\mathbf{X},\mathbf{X}',2s),\mathbf{X}')] \\
		&= E_{\cA} \circ E_{\cA * \cB_1}[\mathbf{f}^{\cA * \cB_1 * \cB_2, \tau * \sigma_1 * \sigma_2}(\mathcal{X}_2^{\cA*\cB_1,\tau*\sigma}(\mathcal{X}_1^{\cA,\tau}(\mathbf{X},\mathbf{X}',2s),\mathbf{X}',2t),\mathbf{X}')].
	\end{align*}
	Let
	\[
	\mathcal{S}_3(u) = \begin{cases} \mathcal{S}_1(u), & u \in [0,2s], \\ \mathcal{S}_1(2s) + \mathcal{S}_2(u-2s), & u \in [2s,\infty), \end{cases}
	\]
	and let $\mathcal{S}_4(u) = \mathcal{S}_1(u+2s)$.  Let $(\cB_3,\sigma_3)$ and $(\cB_4,\sigma_4)$ be the associated tracial $\mathrm{W}^*$-algebras.  Then $\cB_3$ and $\cB_4$ are subalgebras of $\cB_1 * \cB_2$, and $\cB_1 * \cB_2 = \cB_3 * \cB_4$.  Since $\mathbf{X}$ and $\mathbf{X}'$ are tuples from $\cA_{\sa}$, we have
	\[
	\mathcal{X}_2^{\cA*\cB_1,\tau*\sigma}(\mathcal{X}_1^{\cA,\tau}(\mathbf{X},\mathbf{X}',2s),\mathbf{X}',2t) = \mathcal{X}_3^{\cA,\tau}(\mathbf{X},\mathbf{X}',2(s+t)),
	\]
	because the flowing for time $2s$ along \eqref{eq:mainSDE} with $\mathcal{S}_1$ and then for time $2t$ with $\mathcal{S}_2$ is the same as flowing for time $2s + 2t$ with $\mathcal{S}_3$.  Now $E_{\cA} \circ E_{\cA * \cB_1}$ is equal to the unique trace-preserving conditional expectation $\cA * \cB_1 * \cB_3 \to \cA$.  Thus, this agrees with first taking the conditional expectation from $\cA * \cB_1 * \cB_2$ onto $\cA * \cB_3$ and then onto $\cA$.  Now $\mathcal{X}_3^{\cA,\tau}(\mathbf{X},\mathbf{X}',2(s+t))$ is in $\cA * \cB_3$ already and hence the above expression reduces to
	\[
	E_{\cA} [\mathbf{f}^{\cA * \cB_3, \tau * \sigma_3}(\mathcal{X}_3^{\cA,\tau}(\mathbf{X},\mathbf{X}',2s+2t),\mathbf{X}')] = [e^{(s+t)L_{\mathbf{x},\mathbf{J}}} \mathbf{f}]^{\cA,\tau}(\mathbf{X},\mathbf{X}').  \qedhere
	\]
\end{proof}

\begin{lemma} \label{lem:semigroupcontinuity}
	Let $\mathbf{f} \in C_{\tr}^k(\R^{*(d+d')}, \mathscr{M}^\ell)^{d''}$.  Then $t \mapsto e^{tL_{\mathbf{x},\mathbf{J}}} \mathbf{f}$ is a continuous function
	\[
	[0,\infty) \to C_{\tr}^k(\R^{*(d+d')},\mathscr{M}(\R^{*d_1},\dots,\R^{*d_\ell}))^{d''}.
	\]
\end{lemma}

\begin{proof}
	By Lemma \ref{lem:processCinfinity}, $\mathcal{X}$ is a continuous map $[0,\infty) \to C_{\tr,\mathcal{S}}^\infty(\R^{*(d+d')})_{\sa}^d$.  By continuity of composition in Theorem \ref{thm:chainrule} / Proposition \ref{prop:chainrule2}, $t \mapsto \mathcal{F}(\mathcal{X},\pi')$ defines a continuous map $[0,\infty) \to C_{\tr,\mathcal{S}}^k(\R^{*(d+d')},\mathscr{M}(\R^{*d_1},\dots,\R^{*d_\ell}))^{d''}$.  Using Lemma \ref{lem:Ckconditionalexpectation}, continuity is preserved when we apply the conditional expectation to obtain the heat semigroup.
\end{proof}

\subsection{Kernel projection and pseudo-inverse of the Laplacian}

Our next goal is to construct a ``kernel projection'' $\mathbb{E}_{\mathbf{x},\mathbf{J}}$ and pseudo-inverse $\Psi_{\mathbf{x},\mathbf{J}}$ for the Laplacian $L_{\mathbf{x},\mathbf{J}}$.  The operator $\mathbb{E}_{\mathbf{x},\mathbf{J}}$ is obtained as the limit of $e^{tL_{\mathbf{J}}}$ as $t \to \infty$.

\begin{lemma}
	Let $\mathbf{f} \in C_{\tr}^k(\R^{*(d+d')},\mathscr{M}(\R^{*d_1},\dots,\R^{*d_\ell}))^{d''}$, let $R > 0$, and let
	\[
	R' = \max(R+2, \norm{\mathbf{J} - \pi}_{C_{\tr}(\R^{*d})^d,R}).
	\]
	Then for $k' \leq k$,
	\begin{equation} \label{eq:derivativedifferenceestimate}
		\norm{\partial^{k'} \mathbf{f} - \partial^k e^{tL_{\mathbf{x},\mathbf{J}}} \mathbf{f} }_{C_{\tr}(\R^{(d+d')},\mathscr{M}^{\ell+k})^{d''},R}
		\leq C_{k',\mathbf{J},R} R' \sum_{j=0}^{k'} \norm{\partial_{\mathbf{x}} \partial^j \mathbf{f}}_{C_{\tr}(\R^{*(d+d')},\mathscr{M}^{\ell+j})^{d''},R'},
	\end{equation}
	where $C_{k,\mathbf{J},R}$ is a constant depending only on $k$ and $\mathbf{J}$ and $R$.
\end{lemma}

\begin{proof}
	Using Lemma \ref{lem:Ckconditionalexpectation}, we have
	\[
	\norm{\partial^k \mathbf{f} - \partial^k e^{tL_{\mathbf{x},\mathbf{J}}} \mathbf{f} }_{C_{\tr}(\R^{(d+d')},\mathscr{M}^{\ell+k})^{d''},R} \leq \norm{\partial^k \mathbf{f} - \partial^k [\mathbf{f} \circ (\mathcal{X},\pi') ]}_{C_{\tr,\mathcal{S}}(\R^{(d+d')},\mathscr{M}^{\ell+k})^{d''},R}
	\]
	Recall that $\partial^{k'}[\mathbf{f}(\mathcal{X}(\cdot,2t)]$ is given by \eqref{eq:iterateddifferentiation2}.  Let us first control the terms where $\partial_{\mathbf{x}'}^{k*} \partial_{\mathbf{x}}^j \mathbf{f}$ has some multilinear argument of the form $\partial^m \mathcal{X}$ with $m \geq 2$.  Of course, this can only happen if $j \geq 1$, which means $\mathbf{f}$ is differentiated with respect to $\mathbf{x}$ at least once.  Using \eqref{eq:derivativeestimate1}, we can bound the term
	\[
	\partial_{\mathbf{x}'}^{k'} \partial_{\mathbf{x}}^j \mathbf{f}(\mathcal{X}(\cdot,t),\mathbf{X}') \# [\underbrace{\Id,\dots,\Id}_\ell,\partial^{|B_1|} \mathcal{X}(\cdot,2t), \dots, \partial^{|B_j|} \mathcal{X}(\cdot,2t), \underbrace{\Pi',\dots,\Pi'}_{k'}]_\sigma
	\]
	by a constant times the sum of the norms of $\partial_{\mathbf{x}} \partial^j \mathbf{f}$ for $j \leq k' - 1$.  This produces a bound of the same form as the right-hand side of \eqref{eq:derivativedifferenceestimate} since $j \geq 1$ and since $2 \leq R'$.
	
	The remaining terms of \eqref{eq:iterateddifferentiation2} are those where $|B_i| = 1$ for all $i$.  This implies that $j + k^* = k'$, and hence these terms add up to
	\begin{equation} \label{eq:derivativecombinatorics1}
		\sum_{j=0}^{k'} \binom{k'}{j}
		\frac{1}{k'!} \sum_{\sigma \in \Perm([k'])} \partial_{\mathbf{x}'}^{k^*} \partial_{\mathbf{x}}^j  \mathbf{f}(\mathcal{X}(\cdot,2t),\pi')
		\# [\underbrace{\Id,\dots,\Id}_\ell, \underbrace{\partial \mathcal{X}(\cdot,2t), \dots, \partial \mathcal{X}(\cdot,2t)}_{j}, \underbrace{\Pi',\dots,\Pi'}_{k'-j}]_\sigma.
	\end{equation}
	When $t = 0$, this reduces to
	\begin{equation} \label{eq:derivativecombinatorics2}
		\sum_{j=0}^{k'} \binom{k'}{j}
		\frac{1}{k'!} \sum_{\sigma \in \Perm([k'])} \partial_{\mathbf{x}'}^{k^*} \partial_{\mathbf{x}}^j  \mathbf{f} 
		\# [\underbrace{\Id,\dots,\Id}_\ell,\underbrace{\Pi,\dots,\Pi}_{j}, \underbrace{\Pi',\dots,\Pi'}_{k'-j}]_\sigma = \partial^k \mathbf{f}.
	\end{equation}
	Thus, to complete the proof, it suffices to estimate the difference between \eqref{eq:derivativecombinatorics1} and \eqref{eq:derivativecombinatorics2} by the right-hand side of \eqref{eq:derivativedifferenceestimate}.  Now \eqref{eq:derivativecombinatorics2} is obtained from \eqref{eq:derivativecombinatorics1} by swapping out each $\partial \mathcal{X}$ for $\Pi$ and swapping out $\mathcal{X}$ for $\pi$ inside $\partial^k \mathbf{f}$.
	
	By \eqref{eq:derivativeestimate1}, $\partial \mathcal{X}(\cdot,2t)$ is bounded by a constant.  Hence, when swapping out each $\partial \mathcal{X}$ for $\Pi$, the error is bounded by the right-hand side of \eqref{eq:derivativedifferenceestimate} as desired.  Finally, we must replace $\partial^{k'} \mathbf{f}(\mathcal{X}(\cdot,2t),\pi')$ by $\partial^k \mathbf{f}$. Given $(\cA,\tau)$, if $\norm{(\mathbf{X},\mathbf{X}')}_\infty \leq R$, then $\norm{\mathcal{X}^{\cA,\tau}(\mathbf{X},\mathbf{X}',2t)}_\infty$ is also bounded by $R'$.  Thus, the error can be controlled in $\norm{\cdot}_{C_{\tr,\mathcal{S}}(\R^{*d},\mathscr{M}^{\ell+k'})^{d''},R}$ by
	\[
	\norm{\partial_{\mathbf{x}} \partial^{k'} \mathbf{f}}_{C_{\tr}(\R^{*(d+d')},\mathscr{M}^{\ell+k'})^{d''},R'} \norm{\mathcal{X}(\cdot,2t) - \pi}_{C_{\tr}(\R^{*(d+d')})^d,R}.
	\]
	Then using Lemma \ref{lem:SDEsolution}, we have
	\[
	\norm{\mathcal{X}(\cdot,2t) - \pi}_{C_{\tr,\mathcal{S}}(\R^{*(d+d')})^d,R} \leq 2R'.
	\]
	Thus, we can bound the error by the right-hand side of \eqref{eq:derivativedifferenceestimate} as desired.
\end{proof}

\begin{proposition} \label{prop:kernelprojection}
	There exists a unique continuous operator
	\[
	\mathbb{E}_{\mathbf{x},\mathbf{J}}: C_{\tr}(\R^{*(d+d')}\mathscr{M}(\R^{*d_1},\dots,\R^{*d_\ell}))^{d''} \to C_{\tr}(\R^{*d'},\mathscr{M}(\R^{*d_1},\dots,\R^{*d_\ell}))^{d''}
	\]
	such that
	\begin{equation} \label{eq:defineE}
		(\mathbb{E}_{\mathbf{x},\mathbf{J}} \mathbf{f}) \circ \pi' = \lim_{t \to \infty} e^{tL_{\mathbf{x},\mathbf{J}}} \mathbf{f} \text{ in } C_{\tr}(\R^{*(d+d')}\mathscr{M}(\R^{*d_1},\dots,\R^{*d_\ell}))^{d''}.
	\end{equation}
	For $k \in \N_0 \cup \{\infty\}$, the operator $\mathbb{E}_{\mathbf{x},\mathbf{J}}$ maps $C_{\tr}^k(\R^{*(d+d')},\mathscr{M}^\ell)^{d''}$ into $C_{\tr}^k(\R^{*d},\mathscr{M}^\ell)^{d''}$.  It satisfies
	\begin{equation} \label{eq:Ederivativeestimate}
		\norm{\partial^{k'} \mathbb{E}_{\mathbf{x},\mathbf{J}}\mathbf{f}}_{C_{\tr}(\R^{*d'},\mathscr{M}^\ell)^{d''},R} \leq C_{k',\mathbf{J},R} \sum_{j=1}^{k'} \norm{\partial^j \mathbf{f}}_{C_{\tr}(\R^{*d'},\mathscr{M}^\ell)^{d''},R'}
	\end{equation}
	for $k' \leq k$, where $R' = \max(R+2,\norm{\mathbf{J} - \pi}_{BC_{\tr}(\R^{*(d+d')})^d})$.  Finally, the limit \eqref{eq:defineE} holds in $C_{\tr}^k(\R^{*(d+d')},\mathscr{M}(\R^{*d_1},\dots,\R^{*d_\ell}))^{d''}$ whenever $\mathbf{f} \in C_{\tr}^{k+1}(\R^{*(d+d')},\mathscr{M}(\R^{*d_1},\dots,\R^{*d_\ell}))^{d''}$ (or more generally the closure of $C_{\tr}^{k+1}(\R^{*(d+d')},\mathscr{M}(\R^{*d_1},\dots,\R^{*d_\ell}))^{d''}$ in $C_{\tr}^k(\R^{*(d+d')},\mathscr{M}(\R^{*d_1},\dots,\R^{*d_\ell}))^{d''}$).
\end{proposition}

\begin{remark}
	Unfortunately, we have not proved that $C_{\tr}^{k+1}(\R^{*(d+d')},\mathscr{M}(\R^{*d_1},\dots,\R^{*d_\ell}))^{d''}$ is dense in $C_{\tr}^k(\R^{*(d+d')},\mathscr{M}(\R^{*d_1},\dots,\R^{*d_\ell}))^{d''}$.
\end{remark}

\begin{proof}
	First, suppose that $\mathbf{f} \in C_{\tr}^\infty(\R^{*(d+d')},\mathscr{M}(\R^{*d_1},\dots,\R^{*d_\ell}))^{d''}$.  Let
	\begin{align*}
		R' &= \max(R+2,\norm{\mathbf{J} - \pi}_{BC_{\tr}(\R^{*(d+d')})^d}), \\
		R'' &= \max(R'+2,\norm{\mathbf{J} - \pi}_{BC_{\tr}(\R^{*(d+d')})^d}).
	\end{align*}
	Then for $t \geq s$,
	\begin{align}
		& \quad \sum_{j=1}^k \norm{\partial^j [e^{tL_{\mathbf{x},\mathbf{J}}}\mathbf{f}] - \partial^j[e^{sL_{\mathbf{x},\mathbf{J}}} \mathbf{f}]}_{C_{\tr}(\R^{*(d+d')},\mathscr{M}^{\ell_j})^{d''},R} \nonumber \\
		&\leq
		C_{k,\mathbf{J},R} R' \sum_{j=1}^k \norm{\partial_{\mathbf{x}} \partial^j[e^{sL_{\mathbf{x},\mathbf{J}}} \mathbf{f}]}_{C_{\tr}(\R^{*(d+d')},\mathscr{M}^{\ell+j})^{d''},R'} \nonumber \\
		&\leq
		e^{-cs} p_{k,\mathbf{J}}(s) R' \sum_{j=1}^k \norm{\partial_{\mathbf{x}} \partial^j\mathbf{f}}_{C_{\tr}(\R^{*(d+d')},\mathscr{M}^{\ell+j})^{d''},R''}, \label{eq:heatsemigroupconvergence}
	\end{align}
	where the first inequality for some constant $C_{k,\mathbf{J},R}$ follows from Lemma \ref{lem:semigroup} and \eqref{eq:derivativedifferenceestimate}, and the second inequality for some polynomial $p_{k,\mathbf{J}}$ follows from \eqref{eq:heatsemigroupestimate2}.  (As before, the constants and polynomials here are not the same ones as in the previous lemmas.)  Because of the $e^{-cs}$ term, the difference goes to zero as $s, t \to \infty$, and thus $e^{tL_{\mathbf{x},\mathbf{J}}} \mathbf{f}$ is Cauchy with respect to each of the seminorms in $C_{\tr}^\infty(\R^{*(d+d')},\mathscr{M}(\R^{*d_1},\dots,\R^{*d_\ell}))^{d''}$.  So the limit
	\[
	T\mathbf{f} := \lim_{t \to \infty} e^{tL_{\mathbf{x},\mathbf{J}}} \mathbf{f}
	\]
	exists in $C_{\tr}^\infty(\R^{*(d+d')},\mathscr{M}(\R^{*d_1},\dots,\R^{*d_\ell}))^{d''}$.  Let
	\[
	[\mathbb{E}_{\mathbf{x},\mathbf{J}}\mathbf{f}]^{\cA,\tau}(\mathbf{X}') = [T\mathbf{f}]^{\cA,\tau}(0,\mathbf{X}').
	\]
	Note that $\mathbb{E}_{\mathbf{x},\mathbf{J}} \mathbf{f} \in C_{\tr}^\infty(\R^{*d'},\mathscr{M}(\R^{*d_1},\dots,\R^{*d_\ell}))^{d''}$.  Because of \eqref{eq:heatsemigroupestimate2}, we see that $\partial_{\mathbf{x}} T\mathbf{f} = 0$, and therefore,
	\[
	T\mathbf{f} = T\mathbf{f}(0,\pi') = \mathbb{E}_{\mathbf{x},\mathbf{J}} \mathbf{f}(\pi').
	\]
	
	So we have proved existence of the limit for $\mathbf{f} \in C_{\tr}^\infty(\R^{*(d+d')}\mathscr{M}(\R^{*d_1},\dots,\R^{*d_\ell}))^{d''}$.  Next, note that $\TrP(\R^{*(d+d')},\mathscr{M}(\R^{*d_1},\dots,\R^{*d_\ell}))^{d''} \subseteq C_{\tr}^\infty(\R^{*(d+d')}\mathscr{M}(\R^{*d_1},\dots,\R^{*d_\ell}))^{d''}$ is dense in $C_{\tr}(\R^{*(d+d')}\mathscr{M}(\R^{*d_1},\dots,\R^{*d_\ell}))^{d''}$.  By \eqref{eq:heatsemigroupestimate1}, the operators $e^{tL_{\mathbf{x},\mathbf{J}}}$ for $t \in [0,\infty)$ are equicontinuous on $C_{\tr}(\R^{*(d+d')}\mathscr{M}(\R^{*d_1},\dots,\R^{*d_\ell}))^{d''}$.  Thus, since the limit as $t \to \infty$ exists on a dense subset, it exists everywhere.  Thus, $\mathbf{E}_{V,\mathbf{X}}$ is a well-defined continuous operator on $C_{\tr}(\R^{*(d+d')}\mathscr{M}(\R^{*d_1},\dots,\R^{*d_\ell}))^{d''}$.
	
	Similarly, \eqref{eq:heatsemigroupestimate1} shows that the operators $e^{tL_{\mathbf{x},\mathbf{J}}}$ for $t \in [0,\infty)$ are equicontinuous on $C_{\tr}^k(\R^{*(d+d')}\mathscr{M}(\R^{*d_1},\dots,\R^{*d_\ell}))^{d''}$.  Using \eqref{eq:heatsemigroupconvergence}, if $\mathbf{f} \in C_{\tr}^{k+1}(\R^{*(d+d')}\mathscr{M}(\R^{*d_1},\dots,\R^{*d_\ell}))^{d''}$, then the limit of $e^{tL_{\mathbf{x},\mathbf{J}}} \mathbf{f}$ exists in $C_{\tr}^{k+1}(\R^{*(d+d')}\mathscr{M}(\R^{*d_1},\dots,\R^{*d_\ell}))^{d''}$ as $t \to \infty$, and hence the same holds in the closure by equicontinuity.
\end{proof}

\begin{proposition} \label{prop:expectationmultiplicative}
	Let $\mathbf{J} \in \mathscr{J}_{a,c}^d$ for some $c \in (0,1)$ and $a \in \R$.  Then $e^{tL_{\mathbf{x},\mathbf{J}}}$ and $\mathbb{E}_{\mathbf{x},\mathbf{J}}: C_{\tr}(\R^{*(d+d')}) \to C_{\tr}(\R^{*d'})$ are multiplicative over $\tr(C_{\tr}(\R^{*(d+d')}))$, they are positive, and they satisfy $e^{tL_{\mathbf{x},\mathbf{J}}} \circ \tr = \tr \circ e^{tL_{\mathbf{x},\mathbf{J}}}$ and $\mathbb{E}_{\mathbf{J}} \circ \tr = \tr \circ \mathbb{E}_{\mathbf{J}}$.
\end{proposition}

\begin{remark}
	In particular, in the case $d' = 0$, we see that $\mathbb{E}_{\mathbf{J}}$ defines a non-commutative law by Lemma \ref{lem:Claw}.  This turns out to be one method to obtain the law $\mu_V$ associated to a potential $V$ when $\nabla V \in \mathcal{J}_{a,c}^d$, as we will explain in \S \ref{subsec:constructtransport}.
\end{remark}

\begin{proof}
	To prove multiplicativity for the heat semigroup, let $\phi \in \tr(C_{\tr}(\R^{*d}))$ and $f \in C_{\tr}(\R^{*d})$.  Then
	\begin{align*}
		e^{tL_{\mathbf{x},\mathbf{J}}} & [\phi f]^{\cA,\tau}(\mathbf{X},\mathbf{X}') \\
		&= E_{\cA}[ \phi^{\cA*\cB,\tau*\sigma}(\mathcal{X}^{\cA,\tau}(\mathbf{X},\mathbf{X}',2t),\mathbf{X}') f^{\cA*\cB,\tau*\sigma}(\mathcal{X}^{\cA,\tau}(\mathbf{X},\mathbf{X}',2t),\mathbf{X}')] \\
		&= \phi^{\cA*\cB,\tau*\sigma}(\mathcal{X}^{\cA,\tau}(\mathbf{X},\mathbf{X}',2t),\mathbf{X}')  E_{\mathbf{A}}[f^{\cA*\cB,\tau*\sigma}(\mathcal{X}^{\cA,\tau}(\mathbf{X},\mathbf{X}',2t),\mathbf{X}')] \\
		&= e^{tL_{\mathbf{x},\mathbf{J}}}[\phi]^{\cA,\tau}(\mathbf{X},\mathbf{X}') e^{tL_{\mathbf{x},\mathbf{J}}}[f]^{\cA,\tau}(\mathbf{X},\mathbf{X}'),
	\end{align*}
	which follows because $\phi^{\cA*\cB,\tau*\sigma}(\mathcal{X}^{\cA,\tau}(\mathbf{X},\mathbf{X}',2t),\mathbf{X}')$ is scalar-valued and thus can be pulled out of the conditional expectation onto $\cA$.  The multiplicativity property for $\mathbb{E}_{\mathbf{x},\mathbf{J}}$ follows by taking $t \to \infty$.
	
	The positivity property is immediate because $e^{tL_{\mathbf{x},\mathbf{J}}} f$ is obtained by evaluating $f$ on some operator and then applying a conditional expectation.
	
	The trace-preserving property follows by similar reasoning.  Indeed,
	\begin{align*}
		[\tr(e^{tL_{\mathbf{x},\mathbf{J}}} f)]^{\cA,\tau}(\mathbf{X},\mathbf{X}') &= \tau[E_{\cA} f^{\cA*\cB,\tau*\sigma}(\mathcal{X}(\mathbf{X},\mathbf{X}',2t),\mathbf{X}')] \\
		&= \tau[f^{\cA*\cB,\tau*\sigma}(\mathcal{X}(\mathbf{X},\mathbf{X}',2t),\mathbf{X}')] \\
		&= E_{\cA}[[\tr(f)]^{\cA*\cB,\tau*\sigma}(\mathcal{X}(\mathbf{X},\mathbf{X}',2t),\mathbf{X}')] \\
		&= [e^{tL_{\mathbf{x},\mathbf{J}}}[\tr(f)]]^{\cA,\tau}(\mathbf{X},\mathbf{X}').
	\end{align*}
	The trace-preserving property for $\mathbb{E}_{\mathbf{x},\mathbf{J}}$ follows by taking $t \to \infty$.
\end{proof}

\begin{proposition} \label{prop:pseudoinverse}
	Let $R' = \max(2 + R, \norm{\mathbf{J} - \Pi}_{BC_{\tr}(\R^{*(d+d')})^d})$.  Let $k \geq 0$.
	\begin{enumerate}[(1)]
		\item For $\mathbf{f} \in C_{\tr}^1(\R^{*(d+d')},\mathscr{M}(\R^{*d_1},\dots,\R^{*d_\ell}))^{d''}$, the integral
		\[
		\Psi_{\mathbf{x},\mathbf{J}}\mathbf{f} := \int_0^\infty e^{tL_{\mathbf{x},\mathbf{J}}} (\mathbf{f} - \mathbb{E}_{\mathbf{x},\mathbf{J}} \mathbf{f} \circ \pi') \,dt := \lim_{T \to \infty} \int_0^T e^{tL_{\mathbf{x},\mathbf{J}}} (\mathbf{f} - \mathbb{E}_{\mathbf{x},\mathbf{J}} \mathbf{f} \circ \pi')\,dt
		\]
		exists as an improper Riemann integral in $C_{\tr}(\R^{*(d+d')},\mathscr{M}(\R^{*d_1},\dots,\R^{*d_\ell}))^{d''}$.
		\item $\Psi_{\mathbf{x},\mathbf{J}}$ maps $C_{\tr}^{k+1}(\R^{*(d+d')},\mathscr{M}(\R^{*d_1},\dots,\R^{*d_\ell}))^{d''}$ into $C_{\tr}^k(\R^{*(d+d')},\mathscr{M}(\R^{*d_1},\dots,\R^{*d_\ell}))^{d''}$ and satisfies
		\[
		\sum_{j=0}^k \norm{\partial^j \Psi_{\mathbf{x},\mathbf{J}}}_{C_{\tr}(\R^{*(d+d')},\mathscr{M}^{\ell+j})^{d''},R} \leq C_{k,\mathbf{J},R} \sum_{j=0}^k \norm{\partial_{\mathbf{x}} \partial^j \mathbf{f}}_{C_{\tr}(\R^{*(d+d')},\mathscr{M}^{\ell+j})^{d''},R'}
		\]
		for some constants $C_{k,\mathbf{J},R}$.
		\item Furthermore, if $\partial_{\mathbf{x}} \mathbf{f}$ is in $C_{\tr}^k(\R^{*(d+d')},\mathscr{M}(\R^{*d_1},\dots,\R^{*d_\ell},\R^{*d}))^{d''}$, then 
		\[
		\sum_{j=0}^k \norm{\partial_{\mathbf{x}} \partial^j \Psi_{\mathbf{J}}}_{C_{\tr}(\R^{*d},\mathscr{M}^{\ell+j+1})^{d''},R} \leq C_{k,\mathbf{J},R}' \sum_{j=0}^k \norm{ \partial_{\mathbf{x}} \partial^j \mathbf{f}}_{C_{\tr}(\R^{*d},\mathscr{M}^{\ell+j+1})^{d''},R'}
		\]
		for some constants $C_{k,\mathbf{J},R}'$.  In particular, in the case $d' = 0$ where there is no $\mathbf{x}'$, the operator, which we will denote $\Psi_{\mathbf{J}}$, maps $C_{\tr}^k(\R^{*d},\mathscr{M}^\ell)^{d''}$ into itself.
	\end{enumerate}
\end{proposition}

\begin{proof}
	We shall prove (1) and (2) at the same time.  Let $k \geq 0$ and let $\mathbf{f} \in C_{\tr}^{k+1}(\R^{*(d+d')}, \mathscr{M}^\ell)^{d''}$.  Then by Proposition \ref{prop:kernelprojection}, $E_{V,\mathbf{X}} \mathbf{f}$ is in $C_{\tr}^k(\R^{*(d+d')},\mathscr{M}^\ell)^{d''}$.  Because $t \mapsto e^{tL_{\mathbf{x},\mathbf{J}}} \mathbf{f}$ is a continuous function $[0,\infty) \to C_{\tr}^{k+1}(\R^{*(d+d')},\mathscr{M}(\R^{*d_1},\dots,\R^{*d_\ell}))^{d''}$, the Riemann integral
	\[
	\int_0^T e^{tL_{\mathbf{x},\mathbf{J}}} (\mathbf{f} - \mathbb{E}_{\mathbf{x},\mathbf{J}} \mathbf{f} \circ \pi')\,dt
	\]
	is well-defined in $C_{\tr}^k(\R^{*(d+d')},\mathscr{M}(\R^{*d_1},\dots,\R^{*d_\ell}))^{d''}$.  Then using \eqref{eq:heatsemigroupconvergence} and taking $t \to \infty$, we see that 
	\begin{align*}
		& \quad \sum_{j=1}^k \norm{\partial^j [e^{sL_{\mathbf{x},\mathbf{J}}}\mathbf{f}] - \partial^j[\mathbb{E}_{\mathbf{x},\mathbf{J}} \mathbf{f} \circ \pi']}_{C_{\tr}(\R^{*(d+d')},\mathscr{M}^{\ell_j})^{d''},R} \\
		&\leq
		e^{-cs} p_{k,\mathbf{J}}(s) R' \sum_{j=1}^k \norm{\partial_{\mathbf{x}} \partial^j\mathbf{f}}_{C_{\tr}(\R^{*(d+d')},\mathscr{M}^{\ell+j})^{d''},R''},
	\end{align*}
	which implies convergence of the integral in $C_{\tr}^k(\R^{*(d+d')},\mathscr{M}(\R^{*d_1},\dots,\R^{*d_\ell}))^{d''}$ as $T \to \infty$ with the bounds asserted in (2).  In particular, by taking $k = 0$, we obtain (1).
	
	(3) Using \eqref{eq:heatsemigroupestimate2}, the improper integral $\int_0^\infty \partial_{\mathbf{x}} \partial^j e^{tL_{\mathbf{J}}}\mathbf{f}\,dt$ converges in
	\[
	C_{\tr}(\R^{*(d+d')},\mathscr{M}(\R^{*d_1},\dots,\R^{*d_\ell},\underbrace{\R^{*(d+d')},\dots,\R^{*(d+d')}}_{j},\R^{*d}))^{d''}
	\]
	for $j = 1$, \dots, $k$, and we have
	\[
	\norm*{ \int_0^\infty \partial^j \partial_{\mathbf{x}} e^{tL_{\mathbf{x},\mathbf{J}}}\mathbf{f}\,dt }_{C_{\tr}(\R^{*d},\mathscr{M}^{\ell+j+1})^{d''},R} \leq \int_0^\infty e^{-ct} p_{k,\mathbf{J}}(t)\,dt \sum_{j'=0}^j \norm{ \partial^{j'} \partial_{\mathbf{x}} \mathbf{f}}_{C_{\tr}(\R^{*d},\mathscr{M}^{\ell+j'+1})^{d''},R'},
	\]
	where $R'$ is as above.  Convergence of the integral in this space implies that for a fixed $(\cA,\tau)$, the integral
	\[
	\int_0^\infty \partial_{\mathbf{x}} \partial^j [e^{tL_{\mathbf{x},\mathbf{J}}}\mathbf{f}]^{\cA,\tau}(\mathbf{X},\mathbf{X}')\,dt
	\]
	converges uniformly for $\mathbf{X} \in \cA_{\sa}$ with $\norm{\mathbf{X}}_\infty \leq R$, for each $j = 1$, \dots, $k$.  Uniform convergence implies that we can exchange integration with Fr\'echet-differentiation.  This shows that
	\[
	\partial^j \partial_{\mathbf{x}} [\Psi_{\mathbf{J}} \mathbf{f}]^{\cA,\tau} = \int_0^\infty \partial^j \partial_{\mathbf{x}} [e^{tL_{\mathbf{J}}}\mathbf{f}]^{\cA,\tau}\,dt.
	\]
	Since this holds for all $(\cA,\tau)$, we have
	\[
	\partial^j \partial_{\mathbf{x}} [\Psi_{\mathbf{J}} \mathbf{f}] = \int_0^\infty \partial^j \partial_{\mathbf{x}} [e^{tL_{\mathbf{J}}}\mathbf{f}] \,dt
	\]
	for $j = 0$, \dots, $k$.  This proves the desired estimate.
\end{proof}

\begin{remark} \label{rem:boundedcoefficients2}
	In (2), the constants $C_{k,\mathbf{J},R}$ only depend on $R$ and on the norms of the derivatives up to order $k + 1$ of $\mathbf{J}$ on the ball of radius $R'$.  In (3), the constants $C_{k,\mathbf{J},R}$ only depend on the norms of the derivatives of $\mathbf{J} - \pi$ up to order $k+1$ of $\mathbf{J}$ on the ball of radius $R'$, and there is no direct dependence on $R$, i.e.\ no dependence on $R$ other than through these norms.  In particular, if $\mathbf{J} \in BC_{\tr}^{k+1}(\R^{*d})$, then $\sup_R C_{k,\mathbf{J},R} < \infty$.
\end{remark}

\subsection{Differential equation and continuity properties}

\begin{proposition} \label{prop:heatequation}
	Let $k \in \N_0 \cup \{\infty\}$, and let $\mathbf{f} \in C_{\tr}^{k+2}(\R^{*(d+d')},\mathscr{M}(\R^{*d_1},\dots,\R^{*d_\ell}))^{d''}$.  Let $\mathbf{F}(\mathbf{X},t) = e^{tL_{\mathbf{x},\mathbf{J}}} \mathbf{f}(\mathbf{X})$.  Then $\mathbf{F}$ defines a differentiable map $[0,\infty) \to C_{\tr}^k(\R^{*(d+d')},\mathscr{M}(\R^{*d_1},\dots,\R^{*d_\ell}))^{d''}$, and
	\[
	\frac{d}{dt} \mathbf{F} = L_{\mathbf{x},\mathbf{J}} \mathbf{F} = L_{\mathbf{x}} \mathbf{F} - \partial_{\mathbf{x}} \mathbf{F} \# \mathbf{J}.
	\]
\end{proposition}

\begin{proof}
	By considering each coordinate of $\mathbf{f}$ separately, it suffices to consider the case $d'' = 1$.  We will first prove differentiability in a weak sense and then deduce the stronger statement by general tricks.
	
	We claim that for $\mathbf{f} \in C_{\tr}(\R^{*(d+d')},\mathscr{M}(\R^{*d_1},\dots,\R^{*d_\ell}))$ and $(\cA,\tau) \in \mathbb{W}$ and for $(\mathbf{X},\mathbf{X}')$ and $\mathbf{Y}_1$, \dots, $\mathbf{Y}_\ell$ in $\cA_{\sa}^{d+d'}$, we have
	\begin{equation} \label{eq:heatequation0}
		\lim_{\delta \to 0} \frac{(e^{\delta L_{\mathbf{x},\mathbf{J}}} \mathbf{f})^{\cA,\tau}(\mathbf{X},\mathbf{X}')[\mathbf{Y}_1,\dots,\mathbf{Y}_\ell] - \mathbf{f}(\mathbf{X},\mathbf{X}')[\mathbf{Y}_1,\dots,\mathbf{Y}_\ell]}{\delta}
		= [L_{\mathbf{x},\mathbf{J}} \mathbf{f}]^{\cA,\tau}(\mathbf{X},\mathbf{X}')[\mathbf{Y}_1,\dots,\mathbf{Y}_\ell]
	\end{equation}
	with respect to $\norm{\cdot}_{\infty}$.  By \eqref{eq:mainSDE}, we have
	\[
	\mathcal{X}^{\cA,\tau}(\mathbf{X},\mathbf{X}',2\delta) = \mathbf{X} + \mathcal{S}(2\delta) - \frac{1}{2} \int_0^{2\delta} \mathbf{J}^{\cA*\cB,\tau*\sigma}(\mathcal{X}^{\cA,\tau}(\mathbf{X},\mathbf{X}',u),\mathbf{X}')\,du.
	\]
	From the continuity of $\mathcal{X}^{\cA,\tau}$ in $t$, it follows that
	\[
	\mathcal{X}^{\cA,\tau}(\mathbf{X},\mathbf{X}',2\delta) = \mathbf{X} + \mathcal{S}(2\delta) - \delta \mathbf{J}^{\cA*\cB,\tau*\sigma}(\mathbf{X},\mathbf{X}') + o(\delta).
	\]
	Since $\mathbf{f}$ is a Fr\'echet-$C^2$ function and $\mathcal{S}(2\delta)$ is $O(\delta^{1/2})$, we have the Taylor expansion
	\begin{align*}
		\mathbf{f}^{\cA * \cB,\tau * \sigma} & (\mathcal{X}^{\cA,\tau}(\mathbf{X},\mathbf{X}',2\delta),\mathbf{X}')[\mathbf{Y}_1,\dots,\mathbf{Y}_\ell] \\
		&= - \mathbf{f}^{\cA*\cB,\tau*\sigma}(\mathbf{X},\mathbf{X}') \\
		&\quad - \delta \, \partial_{\mathbf{x}} \mathbf{f}^{\cA*\cB,\tau*\sigma}(\mathbf{X},\mathbf{X}') \# [\mathbf{Y}_1,\dots,\mathbf{Y}_\ell, \nabla_{\mathbf{X}} V^{\cA*\cB,\tau*\sigma}(\mathbf{X},\mathbf{X}')] \\
		& \quad + \partial_{\mathbf{x}} \mathbf{f}^{\cA*\cB,\tau*\sigma}(\mathbf{X},\mathbf{X}') \# [\mathbf{Y}_1,\dots,\mathbf{Y}_\ell, \mathcal{S}(2\delta)] \\
		&\quad + \frac{1}{2} \partial_{\mathbf{x}}^2 \mathbf{f}^{\cA*\cB,\tau*\sigma}(\mathbf{X},\mathbf{X}') \# [\mathbf{Y}_1,\dots,\mathbf{Y}_\ell, \mathcal{S}(2\delta),\mathcal{S}(2\delta)] + o(\delta).
	\end{align*}
	The first term on the right-hand side is already in $\cA_{\sa}^d$.  When we apply the expectation $E_{\cA}$, the second term on the right-hand side vanishes using free independence, while the third term (by our very definition of $L_{\mathbf{x}}$ in Definitions \ref{def:freedivergence} and \ref{def:freeLaplacian}) produces
	\[
	\delta (L_{\mathbf{x},\mathbf{J}} \mathbf{f})^{\cA,\tau}(\mathbf{X},\mathbf{X}')[\mathbf{Y}_1,\dots,\mathbf{Y}_\ell].
	\]
	This establishes \eqref{eq:heatequation0}.
	
	Now we begin the main argument.  By Lemma \ref{lem:semigroupcontinuity}, $t \mapsto \mathbf{F}(\cdot,t)$ is a continuous function from $[0,\infty)$ to  $C_{\tr}^{k+2}(\R^{*(d+d')},\mathscr{M}(\R^{*d_1},\dots,\R^{*d_\ell}))$, and hence $t \mapsto L_{\mathbf{x},\mathbf{J}} \mathbf{F}(\cdot,t)$ is a continuous function from $[0,\infty)$ to $C_{\tr}^k(\R^{*(d+d')},\mathscr{M}(\R^{*d_1},\dots,\R^{*d_\ell}))$.  This follows by continuity of
	\[
	L_{\mathbf{x}}: C_{\tr}^{k+2}(\R^{*(d+d')},\mathscr{M}(\R^{*d_1},\dots,\R^{*d_\ell})) \to C_{\tr}^k(\R^{*(d+d')},\mathscr{M}(\R^{*d_1},\dots,\R^{*d_\ell})),
	\]
	which in turn implies continuity of $\mathbf{f} \mapsto \partial_{\mathbf{x}} \mathbf{f} \# \nabla_{\mathbf{X}} V$ using continuity of composition.  Therefore, we may define
	\[
	\mathbf{G}(\cdot,t) = \mathbf{f} + \int_0^t L_{\mathbf{x},\mathbf{J}}\mathbf{F}(\cdot,u)\,du
	\]
	as a Riemann integral with values in $C_{\tr}^k(\R^{*(d+d')},\mathscr{M}(\R^{(d_1},\dots,\R^{*d_\ell}))$.  By the fundamental theorem of calculus, $\mathbf{G}$ is differentiable as a function $[0,\infty) \to C_{\tr}^k(\R^{*(d+d')},\mathscr{M}(\R^{*d_1},\dots,\R^{*d_\ell}))$ with derivative equal to $\mathbf{F}(\cdot,t)$.  Therefore, it suffices to show that $\mathbf{G} = \mathcal{F}$.
	
	Fix $(\mathcal{A},\tau)$, let $(\mathbf{X},\mathbf{X}') \in \cA_{\sa}^{d+d'}$, let $t \in [0,\infty)$, and let $\phi$ be a state on $\cA$, and we will prove that
	\begin{equation} \label{eq:heatequationstate}
		\phi \circ (\mathbf{F} - \mathbf{G})^{\cA,\tau}(\mathbf{X},\mathbf{X}',t) = 0.
	\end{equation}
	As in the proof of the mean value theorem, consider the function $\beta: [0,t] \to \R$ given by
	\[
	\beta(u) = \phi\left( (\mathbf{X},\mathbf{X}',u) - \frac{u}{t} (\mathbf{F} - \mathbf{G})^{\cA,\tau})(\mathbf{X},\mathbf{X}',t) \right).
	\]
	Note that $\beta(0) = \beta(t) = 0$ and $\beta$ is continuous.  Moreover, by \eqref{eq:heatequation0} applied to $e^{uL_{\mathbf{x},\mathbf{J}}} \mathbf{f}$, we have
	\begin{multline*}
		\lim_{\delta \to 0^+} \frac{1}{\delta} \left( (\mathbf{F} - \mathbf{G})^{\cA,\tau}(\mathbf{X},\mathbf{X}',u+\delta) - (\mathbf{F} - \mathbf{G})^{\cA,\tau}(\mathbf{X},\mathbf{X}',u) \right) \\ = (L_{\mathbf{x},\mathbf{J}}\mathbf{F}(\cdot,u) - L_{\mathbf{x},\mathbf{J}} \mathbf{F}(\cdot,u))^{\cA,\tau} = 0.
	\end{multline*}
	This implies (by the product rule) that $\beta$ is right-differentiable in $u$ with right-derivative given by
	\[
	\beta_+'(u) = \frac{1}{t} \phi \circ (\mathbf{F} - \mathbf{G})^{\cA,\tau}(\mathbf{X},\mathbf{X}',t).
	\]
	Since $\beta(0) = \beta(t) = 0$ and $\beta$ is continuous, it must achieve a maximum at some point in $u_0 \in (0,t)$, and at this maximum
	\[
	\frac{1}{t} \phi \circ (\mathbf{F} - \mathbf{G})^{\cA,\tau}(\mathbf{X},\mathbf{X}',t) = \beta_+'(u_0) \leq 0.
	\]
	By the same token, it has a local minimum, so the opposite inequality holds as well, which proves \eqref{eq:heatequationstate}.
\end{proof}

\begin{proposition} \label{prop:Laplacianrelations}
	Let $\mathbf{J} \in \mathscr{J}_{a,b}^d$.  Then the operators $\{e^{tL_{\mathbf{x},\mathbf{J}}}\}_{t \in [0,\infty)}$, $L_{\mathbf{x},\mathbf{J}}$, $\mathbb{E}_{\mathbf{x},\mathbf{J}}[-] \circ \pi'$, and $\Psi_{\mathbf{x},\mathbf{J}}$ all commute as operators on $C_{\tr}^\infty(\R^{*(d+d')},\mathscr{M}(\R^{*d_1},\dots,\R^{*d_\ell}))^{d''}$.  Moreover,
	\begin{equation} \label{eq:Laplacianrelation1}
		L_{\mathbf{x},\mathbf{J}} [\mathbb{E}_{\mathbf{x},\mathbf{J}} \mathbf{f} \circ \pi'] = 0
	\end{equation}
	and
	\begin{equation} \label{eq:Laplacianrelation2}
		(-L_{\mathbf{x},\mathbf{J}} \Psi_{\mathbf{x},\mathbf{J}} + \mathbb{E}_{\mathbf{x},\mathbf{J}})\mathbf{f} =  \mathbf{f}.
	\end{equation}
\end{proposition}

\begin{proof}
	By Lemma \ref{lem:semigroup}, the operators $\{e^{tL_{\mathbf{x},\mathbf{J}}}\}_{t \in [0,\infty)}$ form a semigroup, and hence they all commute with each other.  This implies that
	\[
	e^{tL_{\mathbf{x},\mathbf{J}}} \frac{e^{sL_{\mathbf{x},\mathbf{J}}} - 1}{s} \mathbf{f} = \frac{e^{sL_{\mathbf{x},\mathbf{J}}} - 1}{s} e^{tL_{\mathbf{x},\mathbf{J}}} \mathbf{f}.
	\]
	When we take $s \to 0^+$, by Proposition \ref{prop:heatequation} and the continuity of $e^{tL_{\mathbf{x},\mathbf{J}}}$ as an operator on $C_{\tr}^\infty(\R^{*(d+d')},\mathscr{M}^\ell)^{d''}$, we obtain that $e^{tL_{\mathbf{x},\mathbf{J}}}$ and $ L_{\mathbf{x},\mathbf{J}}$ commute.
	
	Similarly, since $e^{tL_{\mathbf{x},\mathbf{J}}} \mathbf{f} \to \mathbb{E}_{\mathbf{x},\mathbf{J}} \mathbf{f} \circ \pi'$ as $t \to \infty$, we see that the operators $e^{sL_{\mathbf{x},\mathbf{J}}}$ and $L_{\mathbf{x},\mathbf{J}}$ commute with $\mathbb{E}_{\mathbf{x},\mathbf{J}}[-] \circ \pi'$.
	
	Next, for each $T \in [0,\infty)$, the operator
	\[
	\mathbf{f} \mapsto \int_0^T [e^{tL_{\mathbf{x},\mathbf{J}}} \mathbf{f} - \mathbb{E}_{\mathbf{x},\mathbf{J}} \mathbf{f} \circ 
	\pi]\,dt
	\]
	commutes with $e^{sL_{\mathbf{x},\mathbf{J}}}$, $L_{\mathbf{x},\mathbf{J}}$, and $\mathbb{E}_{\mathbf{x},\mathbf{J}}[-] \circ \pi$, because the Riemann sum approximations of this integral commute with them.  Then taking $T \to \infty$, we see that $\Psi_{\mathbf{x},\mathbf{J}}$ commutes with all these operators.
	
	To prove \eqref{eq:Laplacianrelation1}, observe that $\mathbb{E}_{\mathbf{x},\mathbf{J}} \mathbf{f} \circ \pi'$ is a function that only depends on $\mathbf{X}'$, and hence the output will be in the kernel of $\nabla_{\mathbf{x}}$ and $\partial_{\mathbf{x}}^2$, and hence in the kernel of $L_{\mathbf{x},\mathbf{J}}$.
	
	To prove \eqref{eq:Laplacianrelation2}, observe that using \eqref{eq:Laplacianrelation1} and the previous proposition, we have
	\begin{align*}
		-L_{\mathbf{x},\mathbf{J}} \int_0^T [e^{tL_{\mathbf{x},\mathbf{J}}} \mathbf{f} - \mathbb{E}_{\mathbf{x},\mathbf{J}} \mathbf{f}]\,dt &= -\int_0^T L_{\mathbf{x},\mathbf{J}} [e^{tL_{\mathbf{x},\mathbf{J}}}\mathbf{f}]\,dt \\
		&= -\int_0^T \frac{d}{dt}[e^{tL_{\mathbf{x},\mathbf{J}}} \mathbf{f}]\,dt \\
		&= \mathbf{f} - e^{TL_{\mathbf{x},\mathbf{J}}} \mathbf{f}.
	\end{align*}
	As we take $T \to \infty$, the right-hand side approaches $\mathbf{f} - \mathbb{E}_{\mathbf{x},\mathbf{J}} \mathbf{f}$ in $C_{\tr}^\infty(\R^{*(d+d')},\mathscr{M}(\R^{*d_1},\dots,\R^{*d_\ell}))^{d''}$ by Proposition \ref{prop:kernelprojection}.  Moreover, as in the proof of Proposition \ref{prop:pseudoinverse}, $\int_0^T e^{tL_{\mathbf{x},\mathbf{J}}}\,dt$ converges in $C_{\tr}^\infty(\R^{*(d+d')},\mathscr{M}(\R^{*d_1},\dots,\R^{*d_\ell}))^{d''}$ as $T \to \infty$ to $\Psi_{\mathbf{x},\mathbf{J}} \mathbf{f}$, and hence
	\[
	-L_{\mathbf{x},\mathbf{J}} \Psi_{\mathbf{x},\mathbf{J}} \mathbf{f} = \mathbf{f} - \mathbb{E}_{\mathbf{x},\mathbf{J}} \mathbf{f},
	\]
	which rearranges to \eqref{eq:Laplacianrelation2}.
\end{proof}

\begin{proposition} \label{prop:expectationbimodule}
	Let $T$ be any one of the operators $\{e^{tL_{\mathbf{x},\mathbf{J}}}\}_{t \in [0,\infty)}$, $L_{\mathbf{x},\mathbf{J}}$, $\mathbb{E}_{\mathbf{x},\mathbf{J}}[-] \circ \pi'$, and $\Psi_{\mathbf{x},\mathbf{J}}$.  Then for $\mathbf{f} \in C_{\tr}^\infty(\R^{*(d+d')},\mathscr{M}(\R^{*d_1},\dots,\R^{*d_\ell}))^{d''}$ and $g \in C_{\tr}^\infty(\R^{*d'})$, we have
	\begin{align} \label{eq:bimodulemaps}
		T[\mathbf{f} \cdot (g \circ \pi')] &= T[\mathbf{f}] \cdot (g \circ \pi'), & T[(g \circ \pi') \cdot \mathbf{f}] &= (g \circ \pi') \cdot T[\mathbf{f}].
	\end{align}
\end{proposition}

\begin{proof}
	Note that for $(\cA,\tau) \in \mathbb{W}$ and $(\mathbf{X},\mathbf{X}') \in \cA_{\sa}^{d+d'}$,
	\begin{align*}
		e^{L_{\mathbf{x},\mathbf{J}}}[\mathbf{f} \cdot (g \circ \pi')]^{\cA,\tau}(\mathbf{X},\mathbf{X}')
		&= E_{\cA}[\mathbf{f}^{\cA*\cB,\tau*\sigma}(\mathcal{X}^{\cA,\tau}(\mathbf{X},\mathbf{X}',2t),\mathbf{X}') g^{\cA*\cB,\tau*\sigma}(\mathbf{X}')] \\
		&= E_{\cA}[\mathbf{f}^{\cA*\cB,\tau*\sigma}(\mathcal{X}^{\cA,\tau}(\mathbf{X},\mathbf{X}',2t),\mathbf{X}')] g^{\cA,\tau}(\mathbf{X}') \\
		&= e^{L_{\mathbf{x},\mathbf{J}}}[\mathbf{f}]^{\cA,\tau}(\mathbf{X},\mathbf{X}') g^{\cA,\tau}(\mathbf{X}')
	\end{align*}
	since $g^{\cA*\cB,\tau*\sigma}(\mathbf{X}') =g^{\cA,\tau}(\mathbf{X}') \in \cA$.  The same reasoning holds when $g$ is on the left side of $\mathbf{f}$, which proves the first case of \eqref{eq:bimodulemaps}.  In other words, $e^{tL_{\mathbf{x},\mathbf{J}}}$ is a bimodule map over $C_{\tr}^\infty(\R^{*d'})$.  Since the identity is a bimodule map, and bimodule maps are closed under linear combinations and limits (hence also derivatives and integrals with respect to $t$), we see that $L_{\mathbf{x},\mathbf{J}}$, $\mathbb{E}_{\mathbf{x},\mathbf{J}}[-] \circ \pi'$, and $\Psi_{\mathbf{x},\mathbf{J}}$ are also bimodule maps over $C_{\tr}^\infty(\R^{*d'})$.  This proves \eqref{eq:bimodulemaps}.
\end{proof}

We close with the following observation about continuous dependence of $\Psi_{\mathbf{x},\mathbf{J}}$ on $\mathbf{J}$, which has a similar purpose in this paper to \cite[Lemma 44]{DGS2016}.

\begin{proposition} \label{prop:Laplaciancontinuity}
	Fix $c \in (0,1)$ and $a \in (0,\infty)$. Let $T_{\mathbf{J}}$ be one of the operators $e^{tL_{\mathbf{x},\mathbf{J}}}$, $\mathbb{E}_{\mathbf{x},\mathbf{J}}[-] \circ \pi'$, $L_{\mathbf{x},\mathbf{J}}$, or $\Psi_{\mathbf{x},\mathbf{J}}$.  Then $(\mathbf{J},\mathbf{f}) \mapsto T_{\mathbf{J}} \mathbf{f}$ defines a continuous map
	\[
	\mathscr{J}_{a,c}^{d,d'} \times C_{\tr}^\infty(\R^{*(d+d')},\mathscr{M}(\R^{*d_1},\dots,\R^{*d_\ell}))^{d''} \to C_{\tr}^\infty(\R^{*d'},\mathscr{M}(\R^{*d_1},\dot,\R^{*d_\ell}))^{d''},
	\]
	where $\mathscr{J}_{a,c}^{d,d'}$ is equipped with the subspace topology from $C_{\tr}^\infty(\R^{*(d+d')})_{\sa}^d$.
\end{proposition}

\begin{proof}
	First, let us prove that $\mathcal{X}$ depends continuously on $\mathbf{J}$ in $\mathscr{J}_{a,c}^{d,d'}$.  Specifically, we will show that for $\mathbf{J}_1 \in \mathscr{J}_{a,c}^{d,d'}$ and $T > 0$, and for every $k$ and $\epsilon > 0$ and $R > 0$, there is a neighborhood $\mathcal{U}$ of $\mathbf{J}_1$ in $\mathscr{J}_{a,c}^{d,d'}$ such that $\mathbf{J}_2 \in \mathcal{U}$ implies that
	\[
	\sup_{t \in [0,T]} \norm{\partial^k \mathcal{X}_1(\cdot,t) - \partial^k \mathcal{X}_2(\cdot,t)}_{C_{\tr,\mathcal{S}}(\R^{*(d+d')},\mathscr{M}^k)^d,R} < \epsilon,
	\]
	where $\mathcal{X}_1$ and $\mathcal{X}_2$ are the processes corresponding to $\mathbf{J}_1$ and $\mathbf{J}_2$ respectively.
	
	As one might expect, the argument proceeds by induction on $k$ using Gr\"onwall's inequality with the differential equations for $\partial^k \mathcal{X}$.  For $k = 0$, by \eqref{eq:mainSDE}, we obtain
	\[
	\mathcal{X}_1(\cdot,t) - \mathcal{X}_2(\cdot,t)
	= - \frac{1}{2} \int_0^t \mathbf{J}_1(\mathcal{X}_1(\cdot,u),\pi') - \mathbf{J}_2(\mathcal{X}_2(\cdot,u),\pi')\,du
	-\frac{1}{2} \int_0^t (\mathbf{J}_1 - \mathbf{J}_2)(\mathcal{X}_2(\cdot,u),\pi')\,du.
	\]
	In the second term on the right-hand side, the integrand is bounded in $\norm{\cdot}_{C_{\tr,\mathcal{S}}(\R^{*(d+d')})^d,R}$ by $(1/2) \norm{\mathbf{J}_1 - \mathbf{J}_2}_{C_{\tr}(\R^{*(d+d')})^d,R'}$ where $R' = \max(R+2,a)$ using \eqref{eq:SDEestimate}.  In the first term on the right-hand side, the integrand is bounded in $\norm{\cdot}_{C_{\tr,\mathcal{S}}(\R^{*(d+d')})^d,R}$ by $2 - c$ times $\norm{\mathcal{X}_1(\cdot,u) - \mathcal{X}_2(\cdot,u)}_{C_{\tr}(\R^{*(d+d')})^d,}$.  Thus, using Gr\"onwall's inequality, we get a bound of the desired form for $k = 0$.
	
	For the induction step, the argument uses \eqref{eq:iterateddifferentiation} instead of \eqref{eq:mainSDE}.  As in the proof of Lemma \ref{lem:processCinfinity}, we separate out the terms $\partial_{\mathbf{x}} \mathbf{J}_j(\mathcal{X},\pi') \# \partial^k \mathbf{X}_j$.  By induction hypothesis, we can arrange that each of the other terms have approximately the same value in $C_{\tr}(\R^{*(d+d')},\mathscr{M}^k)^d$ when $\mathbf{J}_1$ and $\mathbf{J}_2$ are sufficiently close (using an argument where we swap out each $\mathcal{X}_1$ in the product for an $\mathcal{X}_2$ iteratively).  Then we use Gr\"onwall's inequality.  The details are left as an exercise.
	
	Now that we proved our claim about continuous dependence of $\mathcal{X}$ on $\mathbf{J}$, observe that by continuity of composition, $\mathbf{f}(\mathcal{X},\pi')$ in $C_{\tr,\mathcal{S}}^\infty(\R^{*(d+d')},\mathscr{M}^\ell)^{d''}$ depends continuously on $(\mathbf{J},\mathbf{f})$.  Then by Lemma \ref{lem:Ckconditionalexpectation}, we obtain the continuity of $e^{tL_{\mathbf{x},\mathbf{J}}} \mathbf{f}$ asserted in the proposition.
	
	Next, we prove continuity of $(\mathbf{J},\mathbf{f}) \mapsto \mathbb{E}_{\mathbf{J},\mathbf{X}} \mathbf{f} \circ \pi'$.  From our argument about the continuous dependence of $\mathcal{X}$ on $\mathbf{J}$, we can deduce that for each $\mathbf{J}_0$ and $k$ and $R$, there is a neighborhood $\mathcal{U} \subseteq \mathscr{J}_{a,c}^{d,d'}$ such that the constants $C_{k,\mathbf{J},R}$ in Lemma \ref{lem:processCinfinity} are uniformly bounded for $\mathbf{J} \in \mathcal{U}$.  Tracing through our previous arguments, it follows that the constants in Proposition \ref{prop:kernelprojection} are also uniformly bounded for $\mathbf{J}$ in a neighborhood of $\mathbf{J}_0$.  Therefore, we can conclude from Proposition \ref{prop:kernelprojection} the following:  For each $\mathbf{J}_0 \in \mathscr{J}_{a,c}^{d,d'}$ and $\mathbf{f}_0 \in C_{\tr}^\infty(\R^{*(d+d')},\mathscr{M}^\ell)^{d''}$ and $R > 0$, there exists neighborhoods $\mathcal{U} \subseteq \mathcal{W}_{a,c}$ and $\mathcal{V} \subseteq C_{\tr}^\infty(\R^{*(d+d')},\mathscr{M}(\R^{*d_1},\dots,\R^{*d_\ell}))^{d''}$ such that the convergence of $e^{tL_{\mathbf{x},\mathbf{J}}} \mathbf{f}$ in $\norm{\cdot}_{C_{\tr}^k}$ as $t \to \infty$ is uniform for $(\mathbf{J},\mathbf{f}) \in \mathcal{U} \times \mathbf{V}$.  Since continuity is preserved under locally uniform limits, we have that $(\mathbf{J},\mathbf{f}) \mapsto \mathbb{E}_{\mathbf{x},\mathbf{J}} \mathbf{f} \circ \pi'$ is continuous in the sense asserted by this proposition.
	
	In a similar way, using the continuity of $(\mathbf{J},\mathbf{f}) \mapsto e^{tL_{\mathbf{x},\mathbf{J}}} \mathbf{f}$ (which is uniform for $t \in [0,T]$) and $(\mathbf{J},\mathbf{f}) \mapsto \mathbb{E}_{\mathbf{x},\mathbf{J}} \mathbf{f} \circ \pi$, we obtain the continuity of $(\mathbf{J},\mathbf{f}) \mapsto \Psi_{\mathbf{x},\mathbf{J}}$.  Finally, the continuity of $(\mathbf{J},\mathbf{f}) \mapsto L_{\mathbf{x},\mathbf{J}} \mathbf{f}$ can be checked directly from the definition since $L_{\mathbf{x},\mathbf{J}}\mathbf{f}$ is obtained by differentiation and multiplication.
\end{proof}

\section{Free Gibbs laws} \label{sec:freeGibbslaws}

The last section described one method of associating a non-commutative law to a potential $V$.  Namely, if $V \in \tr(C_{\tr}^\infty(\R^{*d}))$ such that $\nabla V \in \mathscr{J}_{a,c}^d$, the non-commutative law is obtained from the expectation functional $\mathbb{E}_V := \mathbb{E}_{\nabla V}: \tr(C_{\tr}(\R^{*d})) \to \C$.

In this section, we describe another approach based on free entropy, which works in greater generality.  For certain potentials $V$, we show the existence of \emph{free Gibbs laws}, that is, non-commutative laws maximizing $\chi^\omega(\nu) - \tilde{\nu}(V)$, where $\chi^\omega$ is a variant of Voiculescu's free entropy depending on a free ultrafilter $\omega$ on $\N$ (Proposition \ref{prop:entropymaximizer}).  This idea was suggested by the results and comments in \cite[\S 3.7]{Voiculescu2002}, \cite{BCG2003}, and \cite{Hiai2005}, but these papers were not able to directly show the existence of maximizers for technical reasons.  We generalize Voiculescu's change of variables formula for entropy to the setting of non-commutative smooth functions (Proposition \ref{prop:changeofvariables}).  We show that any free Gibbs law for $V$ satisfies a certain integration-by-parts relation (Proposition \ref{prop:integrationbyparts}) and we deduce an exponential bound for $\nu$ directly from this equation (Theorem \ref{thm:magicnormbound}).  Finally, we show in Proposition \ref{prop:generic} that (for a fixed $\omega$) ``most'' potentials $V$ with bounded first and second derivative have a unique free Gibbs law.

\subsection{Microstates free entropy and free Gibbs laws}

Free Gibbs laws for a potential $V$ will be defined as the maximizers of a certain entropy functional $\chi_V^\omega$.  This is a variant of Voiculescu's microstates free entropy $\chi$ that uses limits along an ultrafilter.  We also slightly modify Voiculescu's framework.  Rather than assuming a priori that the non-commutative laws arise from bounded operators, we allow ourselves to work with something like measures of finite variance, or more precisely, linear functionals defined on a space $\mathcal{C}$ of test functions with quadratic growth at $\infty$.  Thus, we will work with matricial microstate spaces that do not have any operator-norm cutoff.

In the end, we will show that for $V$ satisfying certain bounds on the first and second derivative, the free Gibbs laws are automatically given as the non-commutative laws of bounded operators.  Thus, the space $\mathcal{C}$ is mostly a technical artifice.  We will therefore allow ourselves an ad hoc definition of $\mathcal{C}$ for the sake of making the statements and proofs cleaner.

Let $V_0 \in \tr(C_{\tr}(\R^{*d}))$ be given by
\[
V_0^{\cA,\tau}(\mathbf{X}) = \frac{1}{2} \sum_{j=1}^d \tau(X_j^*X_j) =  \frac{1}{2} \norm{\mathbf{X}}_2^2.
\]
Note that if $g \in C_{\tr}^1(\R^{*d})$ has bounded first derivatives, then $g$ is $\norm{\cdot}_2$-Lipschitz; more precisely, for all $(\cA,\tau) \in \mathbb{W}$ and $\mathbf{X}, \mathbf{Y} \in \cA_{\sa}^d$, we have
\[
\norm{g^{\cA,\tau}(\mathbf{X}) - g^{\cA,\tau}(\mathbf{Y})}_2 \leq \norm{\partial g}_{BC_{\tr}(\R^{*d},\mathscr{M}^1)} \norm{\mathbf{X} - \mathbf{Y}}_2.
\]
In particular, $\tr(g^*g)$ is bounded by a constant times $1 + V_0$.  Hence, if $g$ and $h$ are in $C_{\tr}^1(\R^{*d})$ and have bounded first derivative, then $\tr(gh) / (1 + V_0)$ is bounded.

We define $\mathcal{C}$ to be the set of $f \in \tr(C_{\tr}(\R^{*d}))$ such that $f / (1 + V_0) \in \tr(BC_{\tr}(\R^{*d}))$ and such that $f / (1 + V_0)$ is the limit in $\tr(BC_{\tr}(\R^{*d}))$ of a sequence $f_n / (1 + V_0)$, where each $f_n$ is a linear combination of functions of the form $\tr(gh)$, where $g$ and $h \in C_{\tr}^1(\R^{*d})$ have bounded first derivatives.  We equip $\mathcal{C}$ with the norm
\[
\norm{f}_{\mathcal{C}} = \norm{f/(1 + V_0)}_{BC_{\tr}(\R^{*d})},
\]
which makes $\mathcal{C}$ into a Banach space.  Note that $V_0 \in \mathcal{C}$, since $V_0 = (1/2) \sum_{j=1}^d \tr(x_j^2)$ and $x_j$ has bounded first derivative.  Clearly, $\mathcal{C}$ also contains $\tr(g) = \tr(1g)$ for any $g \in C_{\tr}(\R^{*d})$ with bounded first derivative.

\begin{remark}
	In fact, the property that elements of the form $\tr(gh)$, where $g$ and $h$ have bounded first derivatives, span a dense subspace of $\mathcal{C}$ is only needed at the end of the proof of Theorem \ref{thm:magicnormbound}.  The rest of the results of this section would hold with $\mathcal{C}$ replaced with the larger space of functions $f \in \tr(C_{\tr}(\R^{*d}))$ such that $f / (1 + V_0)$ is bounded.
\end{remark}

The next lemma describes how non-commutative laws give rise to linear functionals on $\mathcal{C}$.

\begin{lemma}
	Let $\mathcal{C}^\star$ denote the Banach-space dual of $\mathcal{C}$.  There is an injective map $I: \Sigma_d \to \mathcal{C}^\star$ given by
	\[
	I(\lambda)(f) = f^{\cA,\tau}(\mathbf{X}),
	\]
	where $\mathbf{X}$ is a $d$-tuple of operators in $(\cA,\tau)$ which realizes the law $\lambda$.  We also have
	\begin{equation} \label{eq:dualnormequality}
		\norm{I(\lambda)}_{\mathcal{C}^\star} = 1 + \sum_{j=1}^d \lambda(x_j^2).
	\end{equation}
	For each $R > 0$, $I|_{\Sigma_{d,R}}$ is a homeomorphism onto its image with respect to the weak-$\star$ topologies on $\Sigma_{d,R}$ and $\mathcal{C}^\star$
\end{lemma}

\begin{proof}
	To see that $I$ is injective, suppose that $\lambda$, $\mu \in \Sigma_{d,R}$ for some $R$ and $I(\lambda) = I(\mu)$.  Let $\phi \in C_c^\infty(\R;\R)$ with $\phi(t) = t$ for $|t| \leq R$.  If $p$ is a non-commutative polynomial in $d$ variables, then $f(x) := \tr(p(\phi(x_1),\dots,\phi(x_d)))$ is in $\tr(BC_{\tr}(\R^{*d}))$, hence $f \in \mathcal{C}$.  Since $f = \tr(p)$ on the ball of radius $R$, we have
	\[
	\lambda(p) = I(\lambda)(f) = I(\mu)(f) = \mu(p).
	\]
	
	Next, to show \eqref{eq:dualnormequality}, note that if $f \in \mathcal{C}$ with $\norm{f}_{\mathcal{C}} \leq 1$, then $|f| \leq 1 + V_0$ and hence
	\[
	|I(\lambda)(f)| \leq I(\lambda)(1 + V_0) = 1 + \frac{1}{2} \sum_{j=1}^d \lambda(x_j^2),
	\]
	while on the other hand equality is clearly achieved for $f = 1 + V_0$.
	
	Finally, we show that $I|_{\Sigma_{d,R}}$ is a weak-$\star$ homeomorphism onto its image.  Consider a net $\lambda_i$ and a potential limit point $\lambda$. Let $\nu_i$ and $\nu$ be the corresponding homomorphisms $\tr(C_{\tr}(\R^{*d})) \to \C$.  If $\lambda_i \to \lambda$ in the weak-$\star$ topology, then $\nu_i(f) \to \nu(f)$ for every scalar-valued trace polynomial $f$ and hence for every $f \in \tr(C_{\tr}(\R^{*d}))$ by density.  Since $I(\lambda_i) = \nu_i|_{\mathcal{C}}$ and $I(\lambda) = \nu|_{\mathcal{C}}$, we have $I(\lambda_i) \to I(\lambda)$ in the weak-$\star$ topology.  Conversely, if $I(\lambda_i) \to I(\lambda)$ in the weak-$\star$ topology, then $\lambda_i \to \lambda$ in the weak-$\star$ topology because we can compute $\lambda_i(p)$ as $I(\lambda)(\tr(p(\phi,\dots,\phi))$, where $\phi$ is a cut-off function as in the first part of the proof.
\end{proof}

We will denote the weak-$\star$ closure of $I(\Sigma_d)$ in $\mathcal{C}^\star$ by $\mathcal{E}$.  By the Banach-Alaoglu theorem, closed and bounded subsets of $\mathcal{C}^{\star}$ (and in particular of $\mathcal{E}$) are compact, which will become important later for proving the existence of maximizers of $\chi_V$.  Indeed, using Voiculescu's original definition of $\chi$, it is possible to find a maximizer of $\Sigma_{d,R}$ (laws where the operator norm is bounded by $R$) because it is compact, but it not clear whether we obtain a global maximum over $\Sigma_d$ (without using external information).  On the other hand, compactness of the space of laws in $\mathcal{E}$ with ``second moment'' bounded by $R$ is enough to obtain a global maximizer in Proposition \ref{prop:entropymaximizer} below.

\begin{remark}
	Unfortunately, the price we pay for such compactness is that there exist ``spurious'' laws in $\mathcal{E}$ that do not arise from any $d$-tuple of operators in $L^2(\cA,\tau)$ for any $(\cA,\tau) \in \mathbb{W}$.  Examples can be constructed as follows.  Let $\mathbf{X}^{(n)}$ be some $d$-tuple of operators with such that $X_j^{(n)}$ has spectral measure $\frac{1}{2n}(\delta_n + \delta_{-n}) + (1 - \frac{1}{2n}) \delta_0$.  Note that the second moment of $X_j^{(n)}$ is $1$.  By compactness, the sequence $(I(\lambda_{\mathbf{X}^{(n)}}))_{n \in \N}$ has a  weak-$\star$ limit point $\nu \in \mathcal{E}$.  Then $\nu(\tr(x_j^2)) = 1$ but $\nu(\tr(\phi(x_j))) = \phi(0)$ for every $\phi \in C_c(\R)$, which is impossible if $\nu$ arose from a $d$-tuple in $L^2$ of a tracial $\mathrm{W}^*$-algebra.
\end{remark}

Free entropy will be defined as the exponential growth rate of microstate spaces.  When studying such exponential growth rates, we do not know whether the limits in question exist; see \cite[\S 2.3, Remark a]{Voiculescu2002} or \cite[\S 7]{BCG2003}.  This stands in contrast with other more classical notions of entropy where subadditivity guarantees the existence of limits.  This problem may seem technical on the surface, but it relates to deep model-theoretic questions about the asymptotic behavior of the matrix algebras $M_N(\C)$ as $N \to \infty$; see \cite[\S 6.4]{FHS2014} and \cite[\S 13.7]{JekelThesis}.  Thus, free entropy has $\limsup$ and $\liminf$ variants as well as a version where we take the limit along a free ultrafilter \cite{Voiculescu2002}.

The ultrafilter approach will be convenient for our purposes.  Let $\beta \N$ denote the Stone-\v{C}ech compactification of $\N$.  Recall that $\beta \N$ is a compact space containing $\N$ as an open dense subset, and any function from $\N$ into a compact Hausdorff space $\Omega$ extends uniquely to a continuous function $\beta \N \to \Omega$.  In particular, if $(a^{(N)})_{N \in \N}$ is a bounded sequence of complex numbers, and if $\omega \in \beta \N$, then $\lim_{N \to \omega} a^{(N)}$ exists.  Similarly, for any sequence in $[-\infty,\infty]$, the limit as $N \to \omega$ exists in $[-\infty,\infty]$.

\begin{definition}
	For $\mathcal{U} \subseteq \mathcal{C}^\star$, we define the microstate space
	\[
	\Gamma^{(N)}(\mathcal{U}) = \{\mathbf{X} \in M_N(\C)_{\sa}^d: I(\lambda_{\mathbf{X}}) \in \mathcal{U}\}.
	\]
\end{definition}

\begin{definition}
	Let $V \in \mathcal{C}$ such that $V^{\cA,\tau}(\mathbf{X}) \geq a V_0 + b$ for some $a > 0$ and $b \in \R$.  Then we define a probability measure $\mu_V^{(N)}$ on $M_N(\C)_{\sa}^d$ by
	\[
	d\mu_V^{(N)}(\mathbf{X}) =  \frac{1}{Z_V^{(N)}} e^{-N^2 V^{M_N(\C),\tr_N}(\mathbf{X})}\,d\mathbf{X},
	\]
	where
	\[
	Z_V^{(N)} = \int_{M_N(\C)_{\sa}^d} e^{-N^2 V^{M_N(\C),\tr_N}(\mathbf{X})}\,d\mathbf{X}.
	\]
	Here $d\mathbf{X}$ denotes Lebesgue measure on $M_N(\C)_{\sa}^d$, which is a real inner product space of dimension $dN^2$ with respect to $\ip{\cdot,\cdot}_2$ and hence has a canonical Lebesgue measure obtained by mapping it onto $\R^{dN^2}$ by a linear isometry.  Note that the lower bound for $V$ implies that $e^{-N^2 V}$ is integrable on $M_N(\C)_{\sa}^d$.
\end{definition}

\begin{definition} \label{def:entropy}
	Let $V$ be as above, let $\nu \in \mathcal{C}^\star$, and let $\omega \in \beta \N \setminus \N$.  We define
	\[
	\chi_V^\omega(\nu) = \inf_{\text{open } \mathcal{U} \ni \nu} \limsup_{N \to \omega} \frac{1}{N^2} \log \mu_V^{(N)}(\Gamma^{(N)}(\mathcal{U})),
	\]
	where the infimum is taken over all weak-$\star$ neighborhoods $\mathcal{U}$ of $\nu$ in $\mathcal{C}^\star$.
\end{definition}

\begin{observation} \label{obs:neighborhoodlimit}
	If $\mathcal{U} \subseteq \mathcal{V}$, then $\mu_V^{(N)}(\Gamma^{(N)}(\mathcal{U})) \leq \mu_V^{(N)}(\Gamma^{(N)}(\mathcal{V}))$.  Hence, $\chi_V^\omega(\nu)$ is the limit of the net $\limsup_{N \to \omega} \frac{1}{N^2} \log \mu_V^{(N)}(\Gamma^{(N)}(\mathcal{U}))$ as $\mathcal{U}$ tends to $\{\nu\}$, that is, the limit of the net over the directed system of neighborhoods of $\mathcal{U}$ ordered by reverse inclusion.
\end{observation}

\begin{definition}
	We say that $\nu \in \mathcal{C}^{\star}$ is a \emph{free Gibbs law for $V$ with respect to $\omega$} if it maximizes $\chi_V^\omega$.
\end{definition}

\begin{proposition}~
	Let $V \in \mathcal{C}$ with $V \geq a V_0 + b$ for some $a > 0$ and $b \in \R$.  Let $\omega \in \beta \N \setminus \N$.
	\begin{enumerate}[(1)]
		\item We have $\chi_V^\omega(\nu) \leq 0$.
		\item $\chi_V^\omega$ is upper semi-continuous on $\mathcal{C}^\star$ with respect to the weak-$\star$ topology.
		\item If $\chi_V^\omega(\nu) > -\infty$, then $\nu$ must be in $\mathcal{E}$, that is, the weak-$\star$ closure of $I(\Sigma_d)$.  In particular, we have $\nu(1) = 1$, $\nu(f) \geq 0$ for every nonnegative $f \in \mathcal{C}$, and $\nu(fg) = \nu(f) \nu(g)$ whenever $f$, $g$, and $fg$ are in $\mathcal{C}$.
	\end{enumerate}
\end{proposition}

\begin{proof}
	(1) This is immediate since $\mu_V^{(N)}$ is a probability measure.
	
	(2) For each weak-$\star$ open set $\mathcal{U} \subseteq \mathcal{C}^\star$, define
	\[
	\chi_{V,\mathcal{U}}^\omega(\nu) = \begin{cases} \lim_{N \to \omega} \frac{1}{N^2} \log \mu_V^{(N)}(\Gamma^{(N)}(\mathcal{U})), & \nu \in \mathcal{U}, \\ \infty, & \nu \not \in \mathcal{U}. \end{cases}
	\]
	Thus, $\chi_{V,\mathcal{U}}^\omega$ only takes two values, one of which is $\infty$.  Since $\mathcal{U}$ is open, $\chi_{V,\mathcal{U}}^\omega$ is upper semi-continuous.  Observe that $\chi_V = \inf_{\text{open } \mathcal{U}} (\chi_{V,\mathcal{U}}^\omega)$, hence $\chi_V^\omega$ is upper semi-continuous as the infimum of a family of upper semi-continuous functions.
	
	(3) Let $\mathcal{E}$ be the weak-$\star$ closure of $I(\Sigma_d)$.  Then $\mathcal{C}^\star \setminus \mathcal{E}$ is an open set.  Since $I(\lambda_{\mathbf{X}}) \in \mathcal{E}$ for every matrix tuple $\mathbf{X}$, we have $\Gamma^{(N)}(\mathcal{C}^\star \setminus \mathcal{E}) = \varnothing$.  Hence, if $\nu \in \mathcal{C}^\star \setminus \mathcal{E}$, we have
	\[
	\chi_V^\omega(\nu) \leq \lim_{N \to \omega} \frac{1}{N^2} \log \mu_V^{(N)}(\Gamma^{(N)}(\mathcal{C}^\star \setminus \mathcal{E})) = -\infty.
	\]
	Thus, by contrapositive, if $\chi_V(\nu) > -\infty$, then $\nu \in \mathcal{E}$.
	
	Clearly, if $\nu \in I(\Sigma_d)$, then $\nu(1) = 1$, $\nu(f) \geq 0$ for $f \geq 0$, and $\nu(fg) = \nu(f) \nu(g)$ whenever $f$, $g$, and $fg$ are in $\mathcal{C}$.  Since these conditions are given by equalities or non-strict inequalities of quantities that are weak-$\star$ continuous functions in $\nu$, they also hold for $\nu$ in the closure of $I(\Sigma_d)$.
\end{proof}

\begin{proposition}
	Suppose that $V \in \mathcal{C}$ and $V \geq a V_0 + b$ for some $a > 0$ and $b \in \R$ and let $\omega \in \beta \N \setminus \N$.  Then
	\[
	\frac{1}{N^2} \log Z_V^{(N)} + d \log N
	\]
	is bounded as $N \to \infty$.  Moreover, the quantity
	\begin{equation} \label{eq:chi}
		\chi_V^\omega(\nu) + \nu(V) + \lim_{N \to \omega} \left( \frac{1}{N^2} \log Z_V^{(N)} + d \log N \right)
	\end{equation}
	is independent of $V$, so long as $V \geq a V_0 + b$ for some $a > 0$ and $b \in \R$.  Denoting this quantity by $\chi^\omega(\nu)$, we have
	\begin{equation} \label{eq:chiestimate}
		\chi^\omega(\nu) \leq \frac{d}{2} \log \frac{2\nu(V_0)}{d} + \frac{d}{2} \log 2\pi e.
	\end{equation}
\end{proposition}

\begin{proof}
	Let $\sigma_{d,a}^{(N)}$ be the Gaussian measure on $M_N(\C)_{\sa}^d$ given by
	\[
	d\sigma_{d,a}^{(N)}(\mathbf{X}) = \frac{1}{Z_{aV_0}^{(N)}} e^{-N^2 aV_0^{M_N(\C),\tr_N}(\mathbf{X})}\,d\mathbf{X} = \frac{1}{Z_{aV_0}^{(N)}} e^{-N^2 a \norm{\mathbf{X}}_2^2/2}\,d\mathbf{X},
	\]
	where
	\[
	Z_{aV_0}^{(N)} = \int e^{-N^2 a V_0(\mathbf{X})} \,d\mathbf{X}.
	\]
	Since $M_N(\C)_{\sa}^d$ is a real inner product space of dimension $dN^2$, we have from a well-known computation that
	\[
	Z_{aV_0}^{(N)} = \left( \sqrt{2\pi / N^2 a} \right)^{dN^2} = \frac{(2\pi)^{dN^2/2}}{a^{dN^2/2} N^{dN^2}},
	\]
	hence
	\[
	\frac{1}{N^2} \log Z_{aV_0}^{(N)} + d \log N = \frac{d}{2} \log \frac{2\pi}{a}.
	\]
	We assumed that $V \in \mathcal{C}$ and $V \geq a V_0 + b$.  Since $V \in \mathcal{C}$, we also have $V \leq AV_0 + B$ for some $A > 0$ and $B \in \R$.  Thus,
	\[
	e^{-N^2AV_0} e^{-N^2B} \leq e^{-N^2 V} \leq e^{-N^2aV_0} e^{-N^2 b}.
	\]
	Hence,
	\[
	Z_{AV_0}^{(N)} e^{-N^2B} \leq Z_V^{(N)} \leq Z_{aV_0}^{(N)} e^{-N^2b}
	\]
	and
	\[
	-B + \frac{d}{2} \log \frac{2\pi}{A} \leq \log Z_V^{(N)} + d \log N \leq -a + \frac{d}{2} \log \frac{2\pi}{a},
	\]
	which proves the first claim about boundedness.
	
	Next, to show that \eqref{eq:chi} is independent of $V$, consider two potentials $V_1$ and $V_2$ satisfying the given assumptions.  Let $\mathcal{U}$ be a weak-$\star$ neighborhood of $\nu$ in $\mathcal{C}^\star$ such that $\psi(V_1 - V_2)$ is bounded for $\psi \in \mathcal{U}$.  Then
	\begin{align*}
		\mu_{V_1}^{(N)}(\Gamma^{(N)}(\mathcal{U}))
		&= \frac{1}{Z_{V_1}^{(N)}} \int_{\Gamma^{(N)}(\mathcal{U})} e^{-N^2 V_1(\mathbf{X})}\,d\mathbf{X} \\
		&\leq \frac{1}{Z_{V_1}^{(N)}} e^{N^2 \sup_{\psi \in \mathcal{U}} \psi(V_2 - V_1)} \int_{\Gamma^{(N)}(\mathcal{U})} e^{-N^2 V_2(\mathbf{X})}\,\mathbf{X} \\
		&= \frac{Z_{V_2}^{(N)}}{Z_{V_1}^{(N)}} e^{N^2 \sup_{\psi \in \mathcal{U}} \psi(V_2 - V_1)} \mu_{V_2}^{(N)}(\Gamma^{(N)}(\mathcal{U})).
	\end{align*}
	Thus,
	\begin{multline*}
		\frac{1}{N^2} \log \mu_{V_1}^{(N)}(\Gamma^{(N)}(\mathcal{U})) + \frac{1}{N^2} \log Z_{V_1}^{(N)} + d \log N \\
		\leq \frac{1}{N^2} \log \mu_{V_2}^{(N)}(\Gamma^{(N)}(\mathcal{U})) + \frac{1}{N^2} \log Z_{V_2}^{(N)} + d \log N + \sup_{\psi \in \mathcal{U}} \psi(V_2 - V_1).
	\end{multline*}
	Taking the limit $N \to \omega$ and then the limit as $\mathcal{U}$ shrinks to $\nu$ (see Observation \ref{obs:neighborhoodlimit}), we have
	\[
	\chi_{V_1}^\omega(\nu) + \lim_{N \to \omega} \left( \frac{1}{N^2} \log Z_{V_1}^{(N)} + d \log N \right)
	\leq \chi_{V_2}^\omega(\nu) + \lim_{N \to \omega} \left( \frac{1}{N^2} \log Z_{V_2}^{(N)} + d \log N \right) + \nu(V_2 - V_1).
	\]
	Now we add $\nu(V_1)$ to both sides and observe that the same result holds with $V_1$ and $V_2$ switched, which proves that \eqref{eq:chi} yields the same value for $V_1$ and $V_2$.
	
	To prove \eqref{eq:chiestimate}, we will use the potential $V_0$ for the computation of $\chi^\omega$.  The associated measure $\mu_{V_0}^{(N)}$ gives a Gaussian random variable $\mathbf{S}^{(N)}$ in $M_N(\C)_{\sa}^d$ with mean zero and covariance matrix $N^{-2} I$.
	Now for $R > 1$,
	\begin{align*}
		\int_{\norm{\mathbf{X}}_2 > d^{1/2} R} e^{-N^2 \norm{\mathbf{X}}_2^2/2}\,d\mathbf{X} &= \int_{\norm{\mathbf{Y}}_2 > d^{1/2}} R^{dN^2} e^{-N^2 R^2 \norm{\mathbf{Y}}_2^2/2}\,d\mathbf{Y} \\
		&= R^{dN^2} \int_{\norm{\mathbf{Y}}_2 > d^{1/2}} e^{-N^2 \norm{\mathbf{Y}}_2^2/2} e^{-N^2(R^2 - 1)\norm{\mathbf{Y}}_2^2/2} \,d\mathbf{Y} \\
		&\leq R^{dN^2} e^{-dN^2(R^2 - 1)} \int_{M_N(\C)_{\sa}^d} e^{-N^2 \norm{\mathbf{X}}_2^2/2}\,d\mathbf{X}. 
	\end{align*}
	so that
	\[
	\mu_{V_0}^{(N)}( \{ \mathbf{X}: \norm{\mathbf{X}}_2 \geq d^{1/2} R\} ) \leq (Re^{-(R^2-1)/2})^{-dN^2}.
	\]
	(This can also be deduced from the Chernoff bound for the chi-squared distribution.)  Hence, for $R > 1$,
	\[
	\frac{1}{N^2} \log \mu_{V_0}^{(N)}( \{ \mathbf{X}: \norm{\mathbf{X}}_2 \geq d^{1/2} R\} ) \leq d\left(\log R - \frac{1}{2} (R^2 - 1) \right).
	\]
	Let $\nu \in \mathcal{C}^\star$ and assume that $2 \nu(V_0) > d$.  Let $1 < R < \sqrt{2 \nu(V_0)/d}$.  Then let $\mathcal{U} = \{ \psi \in \mathcal{C}^\star: 2 \psi(V_0) / d > R^2 \}$.  Thus,
	\[
	\Gamma^{(N)}(\mathcal{U}) = \{ \mathbf{X} \in M_N(\C)_{\sa}^d: \norm{\mathbf{X}}_2 > d^{1/2} R \}.
	\]
	Hence,
	\begin{equation} \label{eq:measurebound}
		\chi_{V_0}^\omega(\nu) \leq \lim_{N \to \omega} \frac{1}{N^2} \log \mu_{V_0}^{(N)}(\Gamma^{(N)}(\mathcal{U})) \leq d\left(\log R - \frac{1}{2} (R^2 - 1) \right).
	\end{equation}
	Taking $R \to \sqrt{2 \nu(V_0)/d}$, we obtain
	\begin{equation} \label{eq:chigestimate}
		\chi_{V_0}^\omega(\nu) \leq \frac{d}{2} \log_+ \frac{2\nu(V_0)}{d} - \nu(V_0) + \frac{d}{2}.
	\end{equation}
	In the case where $\nu(V_0) = d/2$, the right-hand side is zero and hence \eqref{eq:chigestimate} holds automatically.  In the case where $\nu(V_0) < d/2$, we can verify \eqref{eq:chigestimate} with symmetrical reasoning to the $\nu(V_0) > d/2$ case; we use the estimate that
	\[
	\mu_{V_0}^{(N)}(\{\mathbf{X}: \norm{\mathbf{X}}_2 < d^{1/2} R \}) \leq R^{dN^2} e^{-dN^2(R^2 - 1)} \text{ for } R < 1,
	\]
	which is obtained in the same way except that now $\norm{\mathbf{Y}}_2^2 < d$ and $R^2 - 1 < 0$.  Now \eqref{eq:chiestimate} follows easily from \eqref{eq:chigestimate} because
	\[
	\lim_{N \to \omega} \frac{1}{N^2} \left( \log Z_{V_0}^{(N)} + d \log N \right) = \frac{d}{2} \log 2\pi. \qedhere
	\]
\end{proof}

\begin{proposition} \label{prop:entropymaximizer}
	Let $V \in \mathcal{C}$ with $V \geq aV_0 + b$, and let $\omega \in \beta \N \setminus \N$.  If $\mathcal{F}$ is a weak-$\star$ closed subset of $\mathcal{C}^\star$, then $\chi_V^\omega$ achieves a maximum on $\mathcal{F}$, and
	\begin{equation} \label{eq:LDP}
		\inf_{\mathcal{U} \supseteq \mathcal{F} \text{ open}} \lim_{N \to \omega} \frac{1}{N^2} \log \mu_V^{(N)}(\Gamma^{(N)}(\mathcal{U})) = \max_{\nu \in \mathcal{F}} \chi_V^\omega(\nu).
	\end{equation}
	In particular, the maximum of $\chi_V^\omega$ over $\mathcal{C}^\star$ is achieved and the maximum is zero.  Thus, a free Gibbs law for $V$ with respect to $\omega$ exists.
\end{proposition}

\begin{proof}
	Let $\mathcal{F}$ be a given closed set, and let us prove that the maximum is achieved in $\mathcal{F}$.  If $\chi_V^\omega$ is identically $-\infty$ on $\mathcal{F}$, then there is nothing to prove, so assume that $\nu_0 \in \mathcal{F}$ with $\chi_V^\omega(\nu_0) > - \infty$.
	
	In order to restrict our attention to a compact set, we first exclude a neighborhood of $\infty$ from achieving the maximum.  Since $V \geq aV_0 + b$, similar reasoning as in the previous proposition shows that
	\begin{align*}
		\mu_V^{(N)}(\Gamma^{(N)}(\nu(V_0) > dR^2))
		&\leq \frac{Z_{aV_0+b}^{(N)}}{Z_V^{(N)}} \mu_{aV_0+b}^{(N)}(\Gamma^{(N)}(\{\nu: \nu(V_0) > dR^2\})) \\
		&\leq \frac{Z_{aV_0+b}^{(N)}}{Z_V^{(N)}} \mu_{V_0}^{(N)}(\Gamma^{(N)}(\{\nu: \nu(V_0) > adR^2\}))
	\end{align*}
	and hence for $R > a^{-1/2}$,
	\begin{multline*}
		\lim_{N \to \omega} \frac{1}{N^2} \log \mu_V^{(N)}\left(\Gamma^{(N)}\left(\nu(V_0) > daR^2/2 \right) \right) \\
		\leq \lim_{N \to \omega} \frac{1}{N^2} \log \frac{Z_{aV_0+b}^{(N)}}{Z_V^{(N)}} + d \left( \log a^{1/2}R - \frac{1}{2} (aR^2 - 1) \right).
	\end{multline*}
	Let
	\[
	C = \lim_{N \to \omega} \frac{1}{N^2} \log \frac{Z_{aV_0+b}^{(N)}}{Z_V^{(N)}},
	\]
	which is finite by the previous proposition.  Fix $R$ sufficiently large that $C + d (\log a^{1/2}R - (1/2)(aR^2 - 1)) < \chi_V^\omega(\nu_0)$.
	
	Let $\mathcal{E}$ be the weak-$\star$ closure of $I(\Sigma_d)$, and let
	\begin{equation} \label{eq:defineK}
		\mathcal{K} = \mathcal{F} \cap \mathcal{E} \cap \{\nu: \nu(V_0) \leq daR^2/2 \}.
	\end{equation}
	Then $\mathcal{K}$ is weak-$\star$ closed.  Moreover, $\mathcal{K}$ is contained in the ball of radius $1 + M$ in $\mathcal{C}^\star$.  Indeed, if $\norm{f}_{\mathcal{C}} \leq 1$, then $-(1 + V_0) \leq \re f \leq (1 + V_0)$.  Since $\nu \in \mathcal{E}$, it is unital and positive and hence
	\[
	-(1 + \nu(V_0)) \leq \re \nu(f) \leq 1 + \nu(V_0).
	\]
	Since the same holds for $\alpha f$ for all $\alpha$ in the unit circle, we have $|\nu(f)| \leq 1 + M$.  By Banach-Alaoglu, the ball of radius $1 + M$ is weak-$\star$ compact, hence $\mathcal{K}$ is weak-$\star$ compact.
	
	Since $\chi_V^\omega$ is weak-$\star$ upper semi-continuous, it achieves a maximum on $\mathcal{K}$.  In fact, this is the maximum over all of $\mathcal{F}$.  Indeed, if $\nu$ is not in $\mathcal{E}$, then $\chi_V^\omega(\nu) = -\infty$.  Moreover, if $\nu(V_0) > daR^2 / 2$, then by our choice of $R$,
	\[
	\chi_V^\omega(\nu) \leq C + \log a^{1/2}R - \frac{1}{2} (aR^2 - 1) < \chi_V^\omega(\nu_0) \leq \max_{\mathcal{K}} \chi_V^\omega.
	\]
	Thus, the maximum over $\mathcal{K}$ is the maximum over $\mathcal{F}$.
	
	Next, we prove \eqref{eq:LDP}.  The inequality $\geq$ is immediate because every neighborhood $\mathcal{U}$ of $\mathcal{F}$ is also a neighborhood of each $\nu \in \mathcal{F}$.  To prove the opposite inequality, fix $M > \max_{\mathcal{F}} \chi_V^\omega$.  (Here the maximum of $\chi_V^\omega$ on $\mathcal{F}$ is allowed to be $-\infty$.)   Choose $R$ sufficiently large that $C + \log a^{1/2}R - \frac{1}{2}(aR^2 - 1) < M$, and let $\mathcal{K}$ be given again by \eqref{eq:defineK}.  For each $\nu \in \mathcal{K}$, there is a neighborhood $\mathcal{U}_\nu$ such that
	\[
	\lim_{N \to \omega} \frac{1}{N^2} \log \mu_V^{(N)}(\Gamma^{(N)}(\mathcal{U}_\nu)) < M.
	\]
	By compactness, we may choose finitely many $\nu_1$, \dots, $\nu_k$ such that the neighborhoods $\mathcal{U}_j = \mathcal{U}_{\nu_j}$ cover $\mathcal{K}$.  Let \[
	\mathcal{U}_0 = \{\nu: \nu(V_0) > daR^2/2\}, \qquad  \mathcal{U} = \mathcal{E}^c \cup \bigcup_{j=0}^k \mathcal{U}_j.
	\]
	Since $\Gamma^{(N)}(\mathcal{E}^c) = \varnothing$, we have
	\[
	\Gamma^{(N)}(\mathcal{U}) = \bigcup_{j=0}^k \Gamma^{(N)}(\mathcal{U}_j).
	\]
	For each $j = 0$, \dots, $k$, we have $\lim_{N \to \omega} (1/N^2) \log \mu_V^{(N)}(\Gamma^{(N)}(\mathcal{U}_j)) < M$, so for $N$ sufficiently close to $\omega$,
	\[
	\mu_V^{(N)}(\Gamma^{(N)}(\mathcal{U}_j)) < e^{-N^2 M}.
	\]
	Thus,
	\[
	\mu_V^{(N)}(\Gamma^{(N)}(\mathcal{U})) < (k+1) e^{-N^2M}.
	\]
	This implies that $\lim_{N \to \omega} \frac{1}{N^2} \log \mu_V^{(N)}(\Gamma^{(N)}(\mathcal{U})) \leq M$.  Since $M > \max_{\mathcal{F}} \chi_V^\omega$ was arbitrary, \eqref{eq:LDP} holds.
	
	By considering $\mathcal{F} = \mathcal{C}^\star$, we see that $\chi_V^\omega$ achieves a maximum.  Moreover,
	\[
	0 = \lim_{N \to \omega} \mu_V^{(N)}(\Gamma^{(N)}(\mathcal{C}^\star)) \leq \max \chi_V^\omega \leq 0. \qedhere
	\]
\end{proof}

\begin{corollary}
	If there is a unique free Gibbs law $\nu$ for $V$ with respect to $\omega$, then for every weak-$\star$ neighborhood $\mathcal{U}$ of $\nu$, we have
	\[
	\lim_{N \to \infty} \frac{1}{N^2} \log \mu_V^{(N)}(\Gamma^{(N)}(\mathcal{U})^c) < 0.
	\]
\end{corollary}

\begin{proof}
	Note that $\mathcal{U}^c$ is closed and so $\chi_V^\omega$ achieves a maximum on this set, which must be strictly less than $\chi_V^\omega(\nu) = 0$ because we assumed $\nu$ is the unique maximizer.  Hence, the claim follows from the previous proposition.
\end{proof}

\subsection{Change of variables for free entropy}

Next, we will prove a change-of-variables formula for free entropy for $\nu \in \mathcal{C}^{\star}$, a generalization of Voiculescu's result in \cite[\S 3]{VoiculescuFE5}.  Since $\nu$ is only in $\mathcal{C}^{\star}$ rather than $\Sigma_d$, we will assume that the transport function $\mathbf{f}$ and its inverse have bounded derivatives.  We begin by describing the action of diffeomorphisms on $\mathcal{C}$ and $\mathcal{C}^{\star}$, along the same lines as Lemma \ref{lem:groupaction}.

\begin{lemma}~
	\begin{enumerate}[(1)]
		\item There is a right group action $\mathcal{C} \times \BDiff_{\tr}^1(\R^{*d}) \to \mathcal{C}$ given by $(h,\mathbf{f}) \mapsto h \circ \mathbf{f}$.  Each element of $\BDiff_{\tr}^1(\R^{*d})$ induces a Banach-space automorphism of $\mathcal{C}$.
		\item There is a left group action of $\BDiff_{\tr}^1(\R^{*d})$ on $\mathcal{C}^{\star}$ by weak-$\star$ homeomorphisms given by $(\mathbf{f}_* \nu)(h) = \nu(h \circ \mathbf{f})$.
		\item There is a left group action of $\BDiff_{\tr}^2(\R^{*d})$ on the set of potentials $V \in \mathcal{C}$ satisfying $V \geq a V_0 + b$ for some $a > 0$ and $b \in \R$, given by
		\[
		\mathbf{f}_* V = V \circ \mathbf{f}^{-1} - \log \Delta_\#(\partial \mathbf{f}^{-1}).
		\]
	\end{enumerate}
\end{lemma}

\begin{proof}
	(1) Let $\mathbf{f} \in \BDiff_{\tr}^1(\R^{*d})$.  If $g, h \in C_{\tr}(\R^{*d})$ have bounded first derivatives, then so do $g \circ \mathbf{f}$ and $h \circ \mathbf{f}$.  Thus, $\tr(gh) \circ \mathbf{f} \in \mathcal{C}$.  Recall that linear combinations of functions of the form $\tr(gh)$ are dense in $\mathcal{C}$ by definition.  Thus, to show that precomposition with $\mathbf{f}$ maps $\mathcal{C}$ into $\mathcal{C}$, it suffices to show that $\norm{(u \circ \mathbf{f}) / (1 + V_0)}_{BC_{\tr}(\R^{*d})} \leq C \norm{u/(1 + V_0)}_{BC_{\tr}(\R^{*d})}$ for some constant $C$.  However, because $\mathbf{f}$ is $\norm{\cdot}_2$-Lipschitz by Remark \ref{rem:Lipschitz}, we obtain $\norm{\mathbf{f}(\mathbf{X})}_2 \leq a' \norm{\mathbf{X}}_2 + b'$ for some constants $a'$ and $b'$.  It follows that $1 + V_0 \circ \mathbf{f} \leq (1/C)(1 + V_0)$ for some $C > 0$ and hence $1 / (1 + V_0) \leq C / (1 + V_0 \circ \mathbf{f})$, which implies the desired bound.  The linearity and associativity properties of this action are clear.  It follows that the action of $\mathbf{f}$ defines a Banach-space automorphism of $\mathcal{C}$.
	
	(2) The map $\mathbf{f}_*: \mathcal{C}^{\star} \to \mathcal{C}^{\star}$ is simply the adjoint of the map $h \mapsto h \circ \mathbf{f}$ and thus it is weak-$\star$ continuous.  Since the same considerations apply to $\mathbf{f}^{-1}$, the inverse map $h \mapsto h \circ \mathbf{f}^{-1}$ is also weak-$\star$ continuous.
	
	(3) This follows by similar reasoning as Lemma \ref{lem:groupaction}.  Note that $\log \Delta_{\#}(\partial \mathbf{f}^{-1})$ has bounded first derivative and therefore is in $\mathcal{C}$.
\end{proof}

\begin{proposition} \label{prop:changeofvariables}
	Let $V \in \mathcal{C}$ with $V \geq aV_0 + b$ for some $a > 0$ and $b \in \R$, let $\nu \in \mathcal{C}^{\star}$, and let $\mathbf{f} \in \BDiff_{\tr}^2(\R^{*d})$.  Then we have the following relations:
	\begin{equation} \label{eq:COV1}
		\lim_{N \to \omega} \frac{1}{N^2} \log \frac{Z_{\mathbf{f}_*V}^{(N)}}{Z_V^{(N)}} = 0,
	\end{equation}
	\begin{equation} \label{eq:COV2}
		\chi_{\mathbf{f}_* V}^\omega(\mathbf{f}_* \nu) = \chi_{V}^\omega(\nu),
	\end{equation}
	\begin{equation} \label{eq:COV3}
		\chi^\omega(\mathbf{f}_* \nu) = \chi^\omega(\nu) + \nu [\log \Delta_\#(\partial \mathbf{f})].
	\end{equation}
	In particular, $\nu$ is a free Gibbs law for $V$ if and only if $\mathbf{f}_* \nu$ is a free Gibbs law for $\mathbf{f}_* V$ (both with respect to the given $\omega$), and hence $V$ has a unique free Gibbs law if and only if $\mathbf{f}_* V$ has a unique free Gibbs law.
\end{proposition}

\begin{proof}
	As an intermediate step to proving \eqref{eq:COV1} and \eqref{eq:COV2}, we will show that for $\nu \in \mathcal{C}^{\star}$, we have
	\begin{equation} \label{eq:COV0}
		\chi_{V}^\omega(\nu) = \chi_{\mathbf{f}_* V}^\omega(\mathbf{f}_* \nu) + \lim_{N \to \omega} \frac{1}{N^2} \log \frac{Z_{\mathbf{f}_*V}^{(N)}}{Z_V^{(N)}}.
	\end{equation}
	Let $\mathcal{U}$ be a neighborhood of $\mathbf{f}_* \nu$ in $\mathcal{C}^{\star}$ and let $\mathcal{V} = (\mathbf{f}_*)^{-1}(\mathcal{U})$, which is a neighborhood of $\nu$.  Let $\mathbf{g} = \mathbf{f}^{-1}$.  Observe that by change of variables,
	\begin{align*}
		& \quad \int_{\Gamma^{(N)}(\mathcal{V})} e^{-N^2 V^{M_N(\C),\tr_N}(\mathbf{X})}\,d\mathbf{X} \\
		&= \int_{\Gamma^{(N)}(\mathcal{U})} e^{-N^2 (V \circ \mathbf{g})^{M_N(\C),\tr_N}(\mathbf{X})} |\det [\partial \mathbf{g}]^{M_N(\C),\tr_N}(\mathbf{X})|\,d\mathbf{X} \\
		&= \int_{\Gamma^{(N)}(\mathcal{U})} \exp\left( -N^2 \left( (V \circ \mathbf{g})^{M_N(\C),\tr_N}(\mathbf{X}) - \frac{1}{N^2} \log |\det [\partial \mathbf{g}]^{M_N(\C),\tr_N}(\mathbf{X})| \right) \right)\,d\mathbf{X}.
	\end{align*}
	By choosing $\mathcal{U}$ small enough, we may guarantee that $\norm{\mathbf{X}}_2^2$ is uniformly bounded on $\Gamma^{(N)}(\mathcal{U})$ independently of $N$.  Hence, since $\partial^2 \mathbf{g} \in BC_{\tr}(\R^{*d},\mathscr{M}^2)^d$, by Lemma \ref{lem:classicaltrace2}, we have
	\[
	\lim_{N \to \omega} \sup_{\mathbf{X} \in \Gamma^{(N)}(\mathcal{U})} \left| \frac{1}{N^2} \log |\det [\partial \mathbf{g}]^{M_N(\C),\tr_N}(\mathbf{X})| - (\log \Delta_\#(\mathbf{g}))^{M_N(\C),\tr_N}(\mathbf{X}) \right| = 0.
	\]
	Therefore,
	\[
	\lim_{N \to \omega} \frac{1}{N^2} \left( \log \int_{\Gamma^{(N)}(\mathcal{V})} e^{-N^2 V^{M_N(\C),\tr_N}(\mathbf{X})}\,d\mathbf{X} - \log \int_{\Gamma^{(N)}(\mathcal{U})} e^{-N^2 (\mathbf{f}_*V)^{M_N(\C),\tr_N}(\mathbf{X})}\,d\mathbf{X} \right) = 0.
	\]
	This implies
	\[
	\lim_{N \to \omega} \frac{1}{N^2} \log \mu_V^{(N)}(\Gamma^{(N)}(\mathcal{V})) = \lim_{N \to \omega} \frac{1}{N^2} \log \mu_{\mathbf{f}_*V}(\Gamma^{(N)}(\mathcal{U})) + \lim_{N \to \omega} \frac{1}{N^2} \log \frac{Z_{\mathbf{f}_*V}^{(N)}}{Z_V^{(N)}}.
	\]
	Then we take the limit as $\mathcal{U}$ shrinks to $\mathbf{f}_* \nu$, which is equivalent to $\mathcal{V}$ shrinking to $\nu$, since $\mathbf{f}_*$ is a weak-$\star$ homeomorphism.  This yields \eqref{eq:COV0}.
	
	By Proposition \ref{prop:entropymaximizer}, the maximum of $\chi_V^\omega$ and the maximum of $\chi_{\mathbf{f}_*V}^\omega$ are both equal to zero.  This fact, together with \eqref{eq:COV0} and that the fact that $\mathbf{f}_*$ is a bijection on $\mathcal{C}^{\star}$, implies \eqref{eq:COV1}.  Then substituting \eqref{eq:COV1} back into \eqref{eq:COV0} produces \eqref{eq:COV2}.  Next, from the definition of $\chi^\omega$ and \eqref{eq:COV1}, we have
	\begin{align*}
		\chi^\omega(\mathbf{f}_* \nu) &= \chi^\omega(\nu) - \nu(V) + (\mathbf{f}_* \nu)(\mathbf{f}_*V) \\
		&= \chi^\omega(\nu) - \nu(V) + (\mathbf{f}_* \nu)(V \circ \mathbf{g}) - (\mathbf{f}_* \nu)(\log \Delta_\# \partial \mathbf{g}) \\
		&= \chi^\omega(\nu) - \nu( \log \Delta_\# \partial \mathbf{g} \circ \mathbf{f} ) \\
		&= \chi^\omega(\nu) + \nu( \log \Delta_\# \partial \mathbf{f}),
	\end{align*}
	since $\partial \mathbf{g} \circ \mathbf{f}$ is the $\#$-inverse of $\partial \mathbf{f}$, and this proves \eqref{eq:COV3}.  Then from \eqref{eq:COV2}, it follows immediately that $\nu$ is a free Gibbs law for $V$ if and only if $\mathbf{f}_* \nu$ is a free Gibbs law for $\mathbf{f}_* V$.
\end{proof}

Next, by applying the change-of-variables formula to diffeomorphisms obtained from flows along vector fields, we will show that any maximizer of $\chi_V^\omega$ must satisfy a certain ``integration-by-parts'' relation.

\begin{proposition} \label{prop:integrationbyparts}
	Let $V \in \mathcal{C} \cap \tr(C_{\tr}^2(\R^{*d}))_{\sa}$ satisfies
	\begin{align*}
		|\partial V^{\cA,\tau}(\mathbf{X})[\mathbf{Y}]| &\leq (a_1 + b_1 \norm{\mathbf{X}}_2^2) \norm{\mathbf{Y}}_\infty \\
		|\partial^2 V^{\cA,\tau}(\mathbf{X})[\mathbf{Y}_1,\mathbf{Y}_2]| &\leq (a_2 + b_2 \norm{\mathbf{X}}_2^2) \norm{\mathbf{Y}_1}_\infty \norm{\mathbf{Y}_2}_\infty
	\end{align*}
	for some constants $a_1$, $b_1$, $a_2$, $b_2 > 0$.  Suppose that $\nu$ is a free Gibbs law for $V$ with respect to $\omega$.  Then for all $\mathbf{h} \in C_{\tr}^2(\R^{*d})^d$ with $\partial \mathbf{h} \in BC_{\tr}^1(\R^{*d},\mathscr{M}^1)^d$, we have
	\begin{equation} \label{eq:DSE}
		\nu \left( \partial V \# \mathbf{h} - \Tr_\#(\partial \mathbf{h}) \right) = 0.
	\end{equation}
\end{proposition}

\begin{remark} \label{rem:IBPhypotheses}
	The hypotheses are chosen so that if $V$ satisfies the hypotheses and $\mathbf{g} \in \BDiff^3(\R^{*d})$, then $V \circ \mathbf{g}$ also satisfies the hypotheses.  This is straightforward to verify from the fact that $\log \Delta_\# \mathbf{g}^{-1}$ has bounded first and second derivatives, while
	\[
	\partial(V \circ \mathbf{g}^{-1}) = \partial V(\mathbf{g}^{-1}) \# \partial \mathbf{g}^{-1}
	\]
	and
	\[
	\partial^2(V \circ \mathbf{g}^{-1}) = \partial^2 V(\mathbf{g}^{-1}) \# [\partial \mathbf{g}^{-1}, \partial \mathbf{g}^{-1}] + \partial V(\mathbf{g}^{-1}) \# \partial^2 \mathbf{g}^{-1}.
	\]
	Furthermore, the hypotheses are satisfied in the case where $\nabla V - \id$ is bounded and $\partial^2 V$ is bounded, which is the case we usually focus on in this paper.
\end{remark}

\begin{proof}[Proof of Proposition \ref{prop:integrationbyparts}]
	By linearity, it suffices to prove \eqref{eq:DSE} in the case where $\mathbf{h}$ is self-adjoint.
	
	Let $\mathbf{f}_t$ and $\mathbf{g}_t$ be the functions constructed by Lemma \ref{lem:flow} by taking $\mathbf{h}_t \equiv \mathbf{h}$, and note that $\mathbf{f}_t \in \BDiff^2(\R^{*d})$.  Hence, by \eqref{eq:COV3},
	\[
	\chi^\omega((\mathbf{f}_t)_*\nu) = \chi^\omega(\nu) + \nu(\log \Delta_\#(\partial \mathbf{f}_t))
	\]
	Since $\nu$ is a free Gibbs law for $V$, we have $\chi_V^\omega((\mathbf{f}_t)_* \nu) \leq \chi_V^\omega(\nu)$.  Since $\chi_V^\omega(\nu)$ is equal to $\chi^\omega(\nu) - \nu(V)$ plus a constant, this amounts to
	\[
	0 \leq (\mathbf{f}_t)_*\nu(V) - \nu(V) - \nu(\log \Delta_\#(\partial \mathbf{f}_t)) = \nu(V \circ \mathbf{f}_t - V - \log \Delta_\#(\partial \mathbf{f}_t)).
	\]
	We claim that
	\begin{equation} \label{eq:COVdiff}
		\lim_{t \to 0^+} \left( V \circ \mathbf{f}_t - V - \log \Delta_\#(\partial \mathbf{f}_t) \right) = \partial V \# \mathbf{h} - \Tr_\#(\partial \mathbf{h}) \text{ in } \mathcal{C}.
	\end{equation}
	
	To prove this, let us first derive error bounds for the Taylor expansion of $t \mapsto \mathbf{f}_t$ as $t \to 0^+$.  Note that
	\[
	\norm*{ \mathbf{f}_t - \id }_{BC_{\tr}(\R^{*d})^d} \leq \int_0^t \norm{ \mathbf{h} \circ \mathbf{f}_u }_{BC_{\tr}(\R^{*d})^d}\,du \leq t \norm{\mathbf{h}}_{BC_{\tr}(\R^{*d})^d}.
	\]
	This implies that
	\begin{align*}
		\norm{\mathbf{h} \circ \mathbf{f}_t - \mathbf{h}}_{BC_{\tr}(\R^{*d})^d} &\leq \norm{\partial \mathbf{h}}_{BC_{\tr}(\R^{*d},\mathscr{M}^1)^d} \norm{\mathbf{f}_t - \id}_{BC_{\tr}(\R^{*d})^d} \\
		&\leq t \norm{\partial \mathbf{h}}_{BC_{\tr}(\R^{*d},\mathscr{M}^1)^d} \norm{\mathbf{h}}_{BC_{\tr}(\R^{*d})^d}.
	\end{align*}
	Hence,
	\begin{align*}
		\norm*{ \mathbf{f}_t - \id - t \mathbf{h} }_{BC_{\tr}(\R^{*d})^d} &\leq \int_0^t \norm*{ \mathbf{h} \circ \mathbf{f}_u - \mathbf{h} }_{BC_{\tr}(\R^{*d})^d}\,du \\
		&\leq \frac{t^2}{2} \norm{\partial \mathbf{h}}_{BC_{\tr}(\R^{*d},\mathscr{M}^1)^d} \norm{\mathbf{h}}_{BC_{\tr}(\R^{*d})^d}.
	\end{align*}
	
	By Taylor expansion, we have
	\[
	V \circ \mathbf{f}_t - V = \partial V \# [\mathbf{f}_t - \id] + \frac{1}{2} \int_0^1 \partial^2 V*((1-s)\id + s\mathbf{f}_t) \# [\mathbf{f}_t - \id, \mathbf{f}_t - \id]\,ds.
	\]
	Since $\partial \mathbf{f}_t$ is bounded, we have $\norm{\mathbf{f}_t^{\cA,\tau}(\mathbf{X})}_2 \leq a_3 + b_3 \norm{\mathbf{X}}_2$ for some constants $a_3$ and $b_3$.  Hence,
	\begin{multline*}
		\left| \int_0^1 \partial^2 V*((1-s)\id + s\mathbf{f}_t)^{\cA,\tau}(\mathbf{X}) \# [\mathbf{f}_t^{\cA,\tau}(\mathbf{X}) - \mathbf{X}, \mathbf{f}_t^{\cA,\tau}(\mathbf{X}) - \mathbf{X}]\,ds \right| \\
		\leq (a_2 + b_2(1 + a_3 + b_3 \norm{\mathbf{X}}_2)^2) t^2  \norm{\mathbf{h}}_{BC_{\tr}(\R^{*d})_{\sa}^d}^2.
	\end{multline*}
	Therefore, this term is $O(t^2)$ in $\mathcal{C}$.  So computing the limit of $(1/t)(V \circ \mathbf{f}_t - V)$ in $\mathcal{C}$ is equivalent to computing the limit of $(1/t)\partial V \# (\mathbf{f}_t - \id)$.  Our earlier estimates show that
	\[
	\frac{\mathbf{f}_t - \id}{t} \to \mathbf{h} \text{ in } BC_{\tr}(\R^{*d})^d.
	\]
	Combining this with our hypothesis on $\partial V$, we get that
	\[
	\lim_{t \to 0^+} \frac{1}{t} \left( V \circ \mathbf{f}_t - V \right) = \lim_{t \to 0^+} \frac{1}{t} \ip{\nabla V, \mathbf{f}_t - \id}_{\tr} = \ip{\nabla V, \mathbf{h}}_{\tr} = \partial V \# \mathbf{h} \text{ in } \mathcal{C}.
	\]
	Next, we deal with the second term on the right-hand side of \eqref{eq:COVdiff}.  Note that
	\[
	\partial \mathbf{f}_t - \Id = \int_0^t (\partial \mathbf{h} \circ \mathbf{f}_u) \# \partial \mathbf{f}_u \,du.
	\]
	Recall that (similar to Gr\"onwall's formula)
	\[
	\norm{\partial \mathbf{f}_t}_{BC_{\tr}(\R^{*d},\mathscr{M}^1)^d} \leq \exp(t \norm{\mathbf{h}}_{BC_{\tr}(\R^{*d},\mathscr{M}^1)^d}).
	\]
	Plugging this into the integral, we obtain
	\[
	\partial \mathbf{f}_t = \Id + O(t) \text{ in } BC_{\tr}(\R^{*d},\mathscr{M}^1)^d.
	\]
	Then because $\partial^2 \mathbf{h}$ is bounded, we get
	\[
	\partial \mathbf{h} \circ \mathbf{f}_u \# \partial \mathbf{f}_u = \partial \mathbf{h} \# \id + O(u)
	\]
	and thus
	\[
	\partial \mathbf{f}_t - I = \int_0^t (\partial \mathbf{h} + O(u))\,du = t \partial \mathbf{h} + O(t^2)
	\]
	in $BC_{\tr}(\R^{*d},\mathscr{M}(\R^{*d}))^d$.  If the right-hand side is strictly smaller than $1$, then we may evaluate
	\[
	\log \Delta_\#(\partial \mathbf{f}_t) = \frac{1}{2} \Tr_{\#} \left( \sum_{m=1}^\infty \frac{(-1)^{m+1}}{m} ((\partial \mathbf{f}_t)^{\varstar} \# \partial \mathbf{f}_t - I)^{\# m} \right) = \frac{t}{2} \Tr_{\#} (\partial \mathbf{h} + \partial \mathbf{h}^{\varstar} ) + O(t^2).
	\]
	Therefore, by the same reasoning as in Lemma \ref{lem:groupactiontangent}
	\[
	\lim_{t \to 0^+} \frac{1}{t} \log \Delta_\#(\partial \mathbf{f}_t) = \Tr_\#(\partial \mathbf{h}) \text{ in } BC_{\tr}(\R^{*d},\mathscr{M}^1)^d,
	\]
	and hence the same limit also holds in $\mathcal{C}$.  This completes the proof of \eqref{eq:COVdiff}.
	
	It follows from \eqref{eq:COVdiff} that
	\[
	\nu \left( \partial V \# \mathbf{h} - \Tr_{\#}(\partial \mathbf{h}) \right) \geq 0.
	\]
	But the same argument applies with $-\mathbf{h}$ instead of $\mathbf{h}$, so that \eqref{eq:DSE} holds.
\end{proof}

\subsection{Consequences of the Dyson-Schwinger equation}

The equation \eqref{eq:DSE} is sometimes called the \emph{Dyson-Schwinger equation},  In the classical setting, this relation can be proved directly using integration-by-parts.  The Dyson-Schwinger equation and the considerations of the previous section lead to the following result.

\begin{corollary} \label{cor:randommatrixconvergence}
	Let $\mathcal{E}$ be the weak-$\star$ closure of $I(\Sigma_d)$ in $\mathcal{C}^{\star}$.  Suppose that there is a unique $\nu \in \mathcal{E}$ satisfying \eqref{eq:DSE}.  Then for every neighborhood $\mathcal{U}$ of $\nu$ in $\mathcal{C}^{\star}$, we have
	\[
	\limsup_{N \to \infty} \frac{1}{N^2} \log \mu_V^{(N)}(\Gamma^{(N)}(\mathcal{U})^c) < 0.
	\]
	More generally, if $f \in \mathcal{C}$ and $\nu(f) = c$ for every $\nu$ satisfying \eqref{eq:DSE}, then for every $\epsilon > 0$, we have
	\[
	\limsup_{N \to \infty} \frac{1}{N^2} \log \mu_V^{(N)}(\{\mathbf{X}: |f(\mathbf{X}) - c| \geq \epsilon\}) < 0.
	\]
\end{corollary}

\begin{proof}
	For each $\omega \in \beta \N \setminus \N$, a free Gibbs law must satisfy \eqref{eq:DSE}.  Thus, $\nu$ is the unique free Gibbs law with respect to $\omega$, so that for each neighborhood $\mathcal{U}$ of $\nu$, we have
	\begin{equation} \label{eq:limit}
		\lim_{N \to \omega} \frac{1}{N^2} \log \mu_V^{(N)}(\Gamma^{(N)}(\mathcal{U})^c) < 0.
	\end{equation}
	But since this holds for every $\omega$, it must also hold for the $\limsup$ as $N \to \infty$.  For the second claim, let $\mathcal{U} = \{\nu: |\nu(f) - c| < \epsilon\}$.  For each $\omega$, the entropy $\chi_V^\omega$ achieves a maximum on $\mathcal{U}^c$ that is strictly less than zero.  Thus, \eqref{eq:limit} also holds, and we conclude as before.
\end{proof}

Amazingly, for a potential $V_0 + W$ with $\partial W$ and $\partial^2 W$ bounded, the Dyson-Schwinger equation is enough to guarantee that an element of $\mathcal{E}$ actually agrees with a law in $\Sigma_d$ with an explicit bound on the ``support radius.''

\begin{theorem} \label{thm:magicnormbound}
	Let $k \geq 2$.  Let $V = V_0 + W \in \tr(C_{\tr}^2(\R^{*d}))$ with $\partial W \in BC_{\tr}^1(\R^{*d},\mathscr{M}(\R^{*d}))$.  Suppose that $\nu \in \mathcal{E}$ satisfies
	\begin{equation} \label{eq:DSE2}
		\nu( \partial V \# \mathbf{h} -  \Tr_{\#}(\partial \mathbf{h} )) = 0 \text{ for } \mathbf{h} \in C_{\tr}^k(\R^{*d})^d \text{ with } \partial \mathbf{h} \in BC_{\tr}^k(\R^{*d},\mathscr{M}(\R^{*d}))^d.
	\end{equation}
	Then there exists $(\cA,\tau) \in \mathbb{W}$ and $\mathbf{X} \in \cA_{\sa}^d$ such that $\nu = I(\lambda_{\mathbf{X}})$ and
	\begin{equation} \label{eq:DSEnormestimate}
		\norm{\mathbf{X}}_\infty \leq C \left( \norm{\partial W}_{BC_{\tr}(\R^{*d},\mathscr{M}^1)} + \sqrt{\norm{\partial W}_{BC_{\tr}(\R^{*d},\mathscr{M}^1)}^2 + 4} \right),
	\end{equation}
	where $C$ is a universal constant.  Moreover, \eqref{eq:DSE2} holds for all $\mathbf{h} \in C_{\tr}^k(\R^{*d})$.
\end{theorem}

\begin{proof}
	{\bf GNS Construction:} Let $\cB$ be the set of functions $f \in BC_{\tr}(\R^{*d})$ such that $f$ is uniformly $\norm{\cdot}_2$-continuous on each $\norm{\cdot}_2$-ball.  Note that $\cB$ is a $\mathrm{C}^*$-subalgebra of $BC_{\tr}(\R^{*d})$.  Moreover, we may define a trace $\tau$ on $\cB$ by
	\[
	\tau(f) = \nu[\tr(f)],
	\]
	which makes sense because $\tr(f) \in \mathcal{C}$.  Let $\mathcal{H}_\tau$ be the GNS Hilbert space associated to $\cB$ and $\tau$, that is, the separation-completion of $\cB$ with respect to $\ip{\cdot,\cdot}_\tau$.  Let $\pi_\tau: \cB \to B(\mathcal{H}_\tau)$ be the GNS representation.  Recall $\tau$ passes to a well-defined faithful trace on $\pi_\tau(\cB)$, and $\pi_\tau(\cB)$ can be completed to a $\mathrm{W}^*$-algebra $\cA \subseteq B(\mathcal{H}_\tau)$, and we will denote the associated trace also by $\tau$ by a slight abuse of notation.
	
	{\bf Bump functions:} Let $\rho \in C_c^\infty(\R)$ be a nonnegative symmetric function supported in $[-1,1]$ which integrates to $2$.  Then let $\psi(t) = \int_0^t \rho$, so that $\widehat{\psi}(s) = \widehat{\rho}(s) / 2\pi i s$.  As in \S \ref{subsec:smoothfunctionalcalculus}, let $\psi(x_j)$ denote the function in $C_{\tr}(\R^{*d})$ given by $[\psi(x_j)]^{\cA_0,\tau_0}(\mathbf{X}) = \psi(X_j)$ for $\mathbf{X} \in (\cA_0)_{\sa}^d$ for $(\cA_0,\tau_0) \in \mathbb{W}$; here $x_j$ denotes a formal self-adjoint variable while $X_j$ denotes an operator from $(\cA_0,\tau_0)$ as in our notation for trace polynomials.  It follows from Lemma \ref{lem:Fouriertransformmisc} that $\psi(x_j) \in C_{\tr}^1(\R^{*d})$, and we have
	\[
	\norm{\partial \psi(x_j)}_{BC_{\tr}(\R,\mathscr{M}^1)} \leq \int_{\R} |\widehat{\rho}(s)|\,ds.
	\]
	In particular, $\psi$ is uniformly $\norm{\cdot}_2$-Lipschitz and hence $\psi(x_j)$ is in $\mathcal{C}$ for $j = 1$, \dots, $d$.  Let
	\[
	\phi_R(t) = R (\psi(t / R + 1) + 1)
	\]
	Note that $\phi_R \geq 0$.  Since $\phi_R$ is defined by scaling and translation of $\psi$, we obtain that
	\[
	\partial[\phi_R(x_j)] = \partial \psi(x_j / R + 1),
	\]
	and hence
	\[
	\norm{\partial \phi_R(x_j)}_{BC_{\tr}(\R^{*1},\mathscr{M}^1)} \leq \int_{\R} |\widehat{\rho}(s)|\,ds.
	\]
	So $\phi_R(x_j) \in \cB$.  In fact, since $\rho \in C_c^\infty(\R)$, we have $\phi_R(x_j) \in BC_{\tr}^\infty(\R^{*d})$.
	
	{\bf Application of Dyson-Schwinger equation:} Recall that $V = V_0 + W$, hence $\nabla V = \nabla V_0 + \nabla W = \id + \nabla W$, and thus for $n \in \N_0$, we have
	\begin{equation} \label{eq:DSEapplied}
		\nu(\ip{x_j, \phi_R(x_j)^n }_{\tr}) = \nu(\ip{\nabla_{x_j} W, \phi_R(x_j)^n }_{\tr}) + \nu(\Tr_{\#}(\partial(\phi_R(x_j)^n))).
	\end{equation}
	Note that $x_j \phi_R(x_j)^n$ is obtained by applying a $C_c^\infty(\R)$ function to $x_j$ and hence is in $\mathcal{C}$.  Thus,
	\[
	\nu(\ip{x_j, \phi_R(x_j)^n }_{\tr}) = \nu(\tr(x_j \phi_R(x_j)^n)) = \tau( x_j \phi_R(x_j)^n)).
	\]
	Also, $\phi_R(t) \leq t \norm{\rho}_{L^\infty(\R)} \leq t \norm{\widehat{\rho}}_{L^1(\R)}$, so that $\phi_R(t)^{n+1} \leq \norm{\widehat{\rho}}_{L^1(\R)} t \phi_R(t)^n$, which implies that
	\[
	\tau( \phi_R(x_j)^{n+1} ) \leq \norm{\widehat{\rho}}_{L^1(\R)} \tau (x_j \phi_R(x_j)^{n+1} ) = \nu(\ip{x_j, \phi_R(x_j)^n }_{\tr}).
	\]
	Meanwhile, the first term on the right-hand side of \eqref{eq:DSEapplied} gives
	\[
	\nu(\ip{\nabla_{x_j} W, \phi_R(x_j)^n }_{\tr}) = \tau(\nabla_{x_j} W \cdot \phi_R(x_j)^n) \leq \norm{\nabla_{x_j} W}_{\mathcal{B}} \tau( \phi_R(x_j)^n) \leq \norm{\partial W}_{BC_{\tr}(\R^{*d},\mathscr{M}^1)},
	\]
	where we have used the fact that $\phi_R(x_j)^n \geq 0$ in the $\mathrm{C}^*$-algebra $\cB$.  Finally, for the second term on the right-hand side of \eqref{eq:DSEapplied}, observe that by the product rule (which follows from the chain rule Theorem \ref{thm:chainrule}),
	\[
	\partial(\phi_R(x_j)^n)) = \sum_{i=0}^{n-1} \phi_R(x_j)^i \partial[\phi_R(x_j)] \phi_R(x_j)^{n-1-i}.
	\]
	For $f, g \in BC_{\tr}(\R^{*d})$ and $i,j = 1,\dots,d$, we may define an element $E_{i,j} \otimes f \otimes g \in BC_{\tr}(\R^{*d},\mathscr{M}^1)^d$
	\[
	[E_{i,j} \otimes f \otimes g]^{\cA_0,\tau_0}(\mathbf{X})[\mathbf{Y}]_{i'} = \delta_{i=i'} f(\mathbf{X}) Y_j g(\mathbf{X}).
	\]
	Note that $(E_{i,j} \otimes f \otimes g)^{\varstar} = E_{j,i} \otimes f^* \otimes g^*$ and
	\[
	[E_{i,j} \otimes f_1 \otimes g_1] \# [E_{i',j'} \otimes f_2 \otimes g_2] = \delta_{j=i'} E_{i,j'} \otimes f_1 f_2 \otimes g_2 g_1,
	\]
	and
	\[
	\Tr_\#(E_{i,j} \otimes f \otimes g) = \delta_{i=j} \tr(f) \tr(g),
	\]
	which follows from a straightforward computation with free independence.  In particular, since $\phi_R(x_j)$ is positive in $BC_{\tr}(\R^{*d})$, we can write
	\[
	E_{j,j} \otimes \phi_R(x_j)^i \otimes \phi_R(x_j)^{n-1-i} = [E_{j,j} \otimes \phi_R(x_j)^{i/2} \otimes \phi_R(x_j)^{(n-1-i)/2}]^{\# 2},
	\]
	which is positive in $BC_{\tr}(\R^{*d},\mathscr{M}^1)$.  Since this is positive and $(1/d) \Tr_{\#}$ defines a $\tr(BC_{\tr}(\R^{*d}))$-valued trace on $BC_{\tr}(\R^{*d},\mathscr{M}^1)^d$, we obtain
	\begin{align*}
		\Tr_{\#}\bigl( \phi_R(x_j)^i & \partial[\phi_R(x_j)] \phi_R(x_j)^{n-1-i} \bigr) \\
		&= \Tr_{\#} \bigl( [E_{j,j} \otimes \phi_R(x_j)^i \otimes \phi_R(x_j)^{n-1-i}] \# \partial [\phi_R(x_j)] \bigr) \\
		&\leq \norm{\partial[\phi_R(x_j)]}_{BC_{\tr}(\R^{*d},\mathscr{M}^1)^d} \Tr_{\#}\bigl( E_{j,j} \otimes \phi_R(x_j)^i \otimes \phi_R(x_j)^{n-1-i} \bigr) \\
		&\leq \norm{\widehat{\rho}}_{L^1(\R)} \tr(\phi_R(x_j)^i) \tr(\phi_R(x_j)^{n-1-i}),
	\end{align*}
	where the inequality holds in $\tr(BC_{\tr}(\R^{*d}))$.  Then using positivity of $\nu$, we have
	\begin{align*}
		\nu(\Tr_{\#}(\partial(\phi_R(x_j)^n))) &= \nu \left( \Tr_{\#}\left( \sum_{i=0}^{n-1} \phi_R(x_j)^i \partial[\phi_R(x_j)] \phi_R(x_j)^{n-1-i} \right) \right) \\
		&\leq \norm{\widehat{\rho}}_{L^1(\R)} \nu\left( \sum_{i=0}^{n-1} \tr(\phi_R(x_j)^i) \tr(\phi_R(x_j)^{n-1-i} \right) \\
		&= \norm{\widehat{\rho}}_{L^1(\R)} \sum_{i=0}^{n-1} \tau(\phi_R(x_j)^i) \tau(\phi_R(x_j)^{n-1-i}.
	\end{align*}
	Putting all these inequalities together, \eqref{eq:DSEapplied} implies
	\begin{multline} \label{eq:DSEapplied2}
		\tau(\phi_R(x_j)^{n+1}) \\
		\leq \norm{\widehat{\rho}}_{L^1(\R)} \left( \norm{\partial W}_{BC_{\tr}(\R^{*d},\mathscr{M}^1)} \tau(\phi_R(x_j)^n) + \norm{\widehat{\rho}}_{L^1(\R)} \sum_{i=0}^{n-1} \tau(\phi_R(x_j)^i) \tau(\phi_R(x_j)^{n-1-i} \right).
	\end{multline}
	
	{\bf Combinatorial estimate:} We use a similar trick as in \cite[proof of Theorem 3.2.1]{BS1998}.  Recall that the \emph{Catalan numbers} are given by
	\[
	C_n = \frac{1}{n+1} \binom{2n}{n}.
	\]
	The Catalan numbers are increasing in $n$, and they satisfy the recursive formula
	\[
	C_{n+1} = \sum_{j=0}^n C_j C_{n-j}. 
	\]
	Moreover, $C_n$ is the $2n$th moment of the semicircular measure $\frac{1}{\pi} \sqrt{4 - x^2} \mathbf{1}_{[-2,2]}(x) \,dx$, so that in particular $C_n \leq 4^n$.
	
	Let $M = \norm{\partial W}_{BC_{\tr}(\R^{*d},\mathscr{M}^1)}$, and let
	\[
	R_0 = \norm{\widehat{\rho}}_{L^1(\R)} \frac{M + \sqrt{M^2 + 4}}{2}.
	\]
	so that
	\[
	R_0^2 = \norm{\widehat{\rho}}_{L^1(\R)} M R_0 + \norm{\widehat{\rho}}_{L^1(\R)}^2.
	\]
	We claim that for $n \in \N_0$, we have
	\[
	\tau(\phi_R(x_j)^n) \leq R_0^n C_n.
	\]
	The base case $n = 0$ is trivial.  For the induction step, using \eqref{eq:DSEapplied2}, we get
	\begin{align*}
		& \quad \tau(\phi_R(x_j)^{n+1}) \\
		&\leq \norm{\rho}_{L^\infty(\R)} \left( M \tau(\phi_R(x_j)^n) + \norm{\widehat{\rho}}_{L^1(\R)} \sum_{i=0}^{n-1} \tau(\phi_R(x_j)^i) \tau(\phi_R(x_j)^{n-1-i} \right) \\
		&\leq \norm{\rho}_{L^\infty(\R)} \left( M R_0^n C_n + \norm{\widehat{\rho}}_{L^1(\R)} \sum_{i=0}^{n-1} R_0^{n-1} C_i C_{n-1-i} \right) \\
		&= \norm{\rho}_{L^\infty(\R)} \left( M R_0 + \norm{\widehat{\rho}}_{L^1(\R)} \right) R_0^{n-1} C_n \\
		&= R_0^{n+1} C_n \leq R_0^{n+1} C_{n+1}.
	\end{align*}
	This completes the induction step.  This implies that
	\[
	\tau(\phi_R(x_j)^n) \leq (4R_0)^n
	\]
	for all $n$, and hence $\norm{\pi_\tau(\phi_R(x_j))} \leq 4R_0$.
	
	{\bf Choice of operators:}  We claim that if $\zeta \in C_c(\R)$ with $\supp(\zeta) \subseteq (4R_0,\infty)$, then $\pi_\tau(\zeta(x_j)) = 0$.  To see this, let $R = \inf \supp(\zeta) > 4R_0$.  Note that for $n \in \N_0$,
	\[
	|\zeta|^2 \leq \norm{\zeta}_{C_0(\R)}^2 1_{[R,\infty)} \leq \norm{\zeta}_{C_0(\R)}^2 \frac{\phi_R^n}{R^n}.
	\]
	Hence,
	\[
	\tau(|\zeta(x_j)|^2) \leq \norm{\zeta}_{C_0(\R)}^2 \frac{\tau(\phi_R(x_j)^n)}{R^n} \leq \norm{\zeta}_{C_0(\R)}^2 \left( \frac{4R_0}{R} \right)^n.
	\]
	Taking $n \to \infty$, we see that $\tau(\zeta(x_j)^*\zeta(x_j)) = 0$ and hence $\pi_\tau(\zeta(x_j)) = 0$.
	
	The same reasoning can be applied with $-\mathbf{x}$ substituted for $\mathbf{x}$ since the $(-\id)_* \phi$ will satisfy the Dyson-Schwinger equation with $-\partial W \circ (-\id)$.  Thus, we also have $\pi_\tau(\zeta(x_j)) = 0$ when $\supp(\zeta) \subseteq (-\infty,-4R_0)$.
	
	Let $X_j = \pi_\tau[\eta(x_j)]$ where $\eta \in C_c^\infty(\R;\R)$ is some function with $\eta(t) = t$ for $|t| < 4R_0 + \epsilon$, for some $\epsilon > 0$.  The preceding argument implies that the resulting operator $X_j$ is independent of the particular choice of $\eta$.  Moreover, for any $\epsilon > 0$, we can arrange that $\norm{\eta}_{C_0(\R)} \leq 4R_0 + \epsilon$, hence $\norm{X_j} \leq \norm{\zeta(x_j)}_{\cB} \leq 4R_0 + \epsilon$.  Since $\epsilon$ was arbitrary, we have $\norm{X_j} \leq 4R_0$, which proves \eqref{eq:DSEnormestimate} with $C = 2 \norm{\widehat{\rho}}_{L^1(\R)}$.
	
	{\bf Agreement of $\nu$ and $I(\lambda_{\mathbf{X}})$ on functions with bounded derivative:}  We claim that $\nu[f] = f^{\cA,\tau}(\mathbf{X})$ for $f \in \tr(C_{\tr}^1(\R^{*d}))$ with $\partial f$ bounded.
	
	Let $\eta \in C_c^\infty(\R;\R)$ with $\eta(t) = t$ for $t$ in a neighborhood of $[-4R_0,4R_0]$.  Since $\pi_\tau$ is a $*$-homomorphism, we have for any $p \in \C\ip{x_1,\dots,x_d}$ that
	\[
	\pi_\tau(p(\eta(x_1),\dots,\eta(x_d))) = p(X_1,\dots,X_d),
	\]
	and hence
	\[
	\nu[\tr(p(\eta(x_1),\dots,\eta(x_d)))] = \tau(p(\eta(x_1),\dots,\eta(x_d))) = \tau(p(X_1,\dots,X_d)).
	\]
	Since $\nu$ is multiplicative on $\cB$ and $\lambda_{\mathbf{X}}$ is also multiplicative, it follows that $\nu[f(\eta(x_1), \dots, \eta(x_d))] = f(X_1,\dots,X_d)$ whenever $f \in \tr(\TrP(\R^{*d}))$.
	
	Next, consider $f(\eta(x_1),\dots,\eta(x_d))$ where $f \in \tr(C_{\tr}^1(\R^{*d}))$ with $\partial f$ bounded.  If we choose $R > \norm{\eta}_{C_0(\R)}$, then we can approximate $f$ uniformly on the $\norm{\cdot}_\infty$-ball of radius $R$ by trace polynomials $(f_n)_{n \in \N}$.  Since $\norm{\eta(x_j)}_{BC_{\tr}(\R^{*d})} < R$, this implies that $f_n(\eta(x_1),\dots,\eta(x_d))$ approximates $f(\eta(x_1),\dots,\eta(x_n))$ in $\tr(BC_{\tr}(\R^{*d}))$ (and hence in $\mathcal{C}$), and therefore in this case we still have the identity
	\[
	\nu[f(\eta(x_1), \dots, \eta(x_d))] = f(X_1,\dots,X_d) = f(\eta(X_1),\dots,\eta(X_d)).
	\]
	
	Keeping $f$ fixed, we use a sequence of functions $\eta_R$ to approximate the identity.  We can arrange that $\eta(t)$ is between $0$ and $t$ for all $t \in \R$ and $\eta(t) = t$ for $|t| \leq 1$.  Then let $\eta_R(t) = R \eta(t/R)$.  Note that for any self-adjoint operator $Y$ from $(\cA_0,\tau_0)$, we have
	\[
	\norm{\eta_R(Y) - Y}_1 \leq \frac{\norm{Y}_2^2}{R}.
	\]
	Since $\partial f$ is bounded, we know that $f$ is uniformly $\norm{\cdot}_1$-continuous, and hence as $R \to \infty$, we have
	\[
	f(\eta_R(x_1),\dots,\eta_R(x_d)) \to f(x_1,\dots,x_d)
	\]
	uniformly on $\norm{\cdot}_2$-balls.  Also $f$ is uniformly $\norm{\cdot}_2$-continuous and hence for all $(\cA_0,\tau_0) \in \mathbb{W}$, we have  $\norm{f^{\cA_0,\tau_0}(\mathbf{Y})}_2 \leq A(1 + V_0^{\cA_0,\tau_0}(\mathbf{Y}))^{1/2}$ for some constant $A$.  We also have $\norm{f^{\cA_0,\tau_0}(\eta(Y_1),\dots,\eta(Y_d))}_2 \leq A(1 + V_0^{\cA_0,\tau_0}(\mathbf{Y}))^{1/2}$ since $\norm{\eta(Y_j)}_2 \leq \norm{Y_j}_2$.  Thus, $|f^{\cA_0,\tau_0}(\eta_R(Y_1),\dots,\eta_R(X_d)) - f^{\cA_0,\tau_0}(Y_1,\dots,Y_d)| / (1 + V_0^{\cA_0,\tau_0}(\mathbf{Y}))$ is bounded by $2 A(1 + V_0^{\cA_0,\tau_0}(\mathbf{Y}))^{-1/2}$, which can be made arbitrarily small outside of $\norm{\cdot}_2$-ball (independently of $(\cA_0,\tau_0)$).  Therefore,
	\[
	\frac{1}{1 + V_0(\mathbf{x})} | f(\eta_R(x_1),\dots,\eta_R(x_d)) - f(x_1,\dots,x_d)| \to 0
	\]
	in $\tr(BC_{\tr}(\R^{*d}))$.  This means that $f(\eta_R(x_1),\dots,\eta_R(x_d)) \to f(x_1,\dots,x_d)$ in $\mathcal{C}$, and therefore,
	\[
	\nu(f(x_1,\dots,x_d)) = \lim_{R \to \infty} \nu(f(\eta(x_1),\dots,\eta(x_d))) = f^{\cA,\tau}(X_1,\dots,X_d).
	\]
	
	{\bf $I(\lambda_{\mathbf{X}})$ satisfies the Dyson-Schwinger equation \eqref{eq:DSE2}:}  Let $R > 4R_0$.  Let $\zeta \in C_c^\infty(\R,[0,1])$ be a function which equals $1$ on $[-R,R]$.  Suppose that $\mathbf{h} \in BC_{\tr}^k(\R^{*d})^d$.  Then
	\[
	\sum_{j=1}^d \nu \circ \tr[x_j h_j + h_j \nabla_{x_j} W] - \nu(\Tr_{\#}(\partial \mathbf{h}) = 0.
	\]
	Because $\nu \circ \tr$ agrees with $I(\lambda_{\mathbf{X}})$ on $BC_{\tr}^1(\R^{*d})$ and because $\nabla_{x_j} W h_j$ and $\Tr_{\#}(\partial \mathbf{h})$ are in $BC_{\tr}^1(\R^{*d})$, we have
	\begin{align*}
		\nu \circ \tr[h_j \nabla_{x_j} W] &= (h_j \nabla_{x_j} W)^{\cA,\tau}(\mathbf{X}) &
		\nu(\Tr_{\#}(\partial \mathbf{h})) = \Tr_{\#}(\partial \mathbf{h})^{\cA,\tau}(\mathbf{X}).
	\end{align*}
	The only term that remains to substitute is $\tr[x_j h_j]$.  But note that
	\begin{align*}
		\nu \circ \tr[x_j(1 - \zeta(x_j))h_j]
		&\leq (\nu \circ \tr[x_j^2])^{1/2} (\nu \circ \tr[(1 - \zeta(x_j))h_j^*h_j(1 - \zeta(x_j))])^{1/2} \\
		&= (\nu \circ \tr(x_j^2))^{1/2} (\tau((1 - \zeta(X_j))(h_j^*h_j)^{\cA,\tau}(\mathbf{X})(1 - \zeta(X_j))))^{1/2} \\
		&= 0
	\end{align*}
	because $(1 - \zeta(x_j))h_j^*h_j(1 - \zeta(x_j))$ has bounded first derivative since $h_j$ and $\partial h_j$ are bounded.  Therefore,
	\begin{equation} \label{eq:Gibbslawagreement}
		\nu \circ \tr[x_j h_j] = \nu \circ \tr[\zeta(x_j) x_j h_j] = \tau[\zeta(X_j) X_j h_j^{\cA,\tau}(\mathbf{X})] = \tau[X_j h_j^{\cA,\tau}(\mathbf{X})],
	\end{equation}
	where we have used the fact that $\zeta(x_j) x_j h_j \in BC_{\tr}^1(\R^{*d})$ and $\zeta(X_j) X_j = X_j$.  This establishes \eqref{eq:DSE2} when $\mathbf{h} \in BC_{\tr}^k(\R^{*d})$.
	
	However, using smooth cut-off functions, every $\mathbf{h} \in C_{\tr}^k(\R^{*d})$ agrees on the ball of radius $R$ with some function $\mathbf{g}$ in $BC_{\tr}^k(\R^{*d})$.  It follows from the definition of Fr\'echet differentiation that $\partial \mathbf{h} = \partial \mathbf{g}$ on the open ball of radius $R$.  Hence, both sides of \eqref{eq:DSE2} are the same for $\mathbf{h}$ and for $\mathbf{g}$.  So $I(\lambda_{\mathbf{X}})$ satisfies \eqref{eq:DSE2} for all $\mathbf{h} \in C_{\tr}(\R^{*d})^d$ as desired.  In particular, the last claim of the theorem will be proved as soon as we know that $\nu = I(\lambda_{\mathbf{X}})$.
	
	{\bf Agreement of $\nu$ and $I(\lambda_{\mathbf{X}})$ on $\mathcal{C}$:}  Let $\zeta$ be as above.  Using \eqref{eq:DSE2} for $\nu$, we have
	\[
	\nu \circ \tr[x_j^2] = \nu \circ \tr[x_j \nabla_{x_j} V(\mathbf{x})] - \nu \circ \tr[x_j \nabla_{x_j} W(\mathbf{x})] = 1 - \nu \circ \tr[x_j \nabla_{x_j} W(\mathbf{x})].
	\]
	since $\partial(\mathbf{x}) = \Id \in BC_{\tr}^\infty(\R^{*d},\mathscr{M}^1(\R^{*d}))$.  The same holds for $I(\lambda_{\mathbf{X}})$ because it also satisfies \eqref{eq:DSE2}.  Hence,
	\[
	\nu\circ \tr[x_j^2] - \nu \circ \tr[\zeta(x_j)^2 x_j^2] = \nu(x_j)^2 - \tau(X_j^2) = - \nu(x_j \nabla_{x_j} W(\mathbf{x})) + \tau[X_j \nabla_{x_j} W(\mathbf{X})].
	\]
	Because the function $h_j = \nabla_{x_j} W$ is bounded and has bounded first derivative, \eqref{eq:Gibbslawagreement} applies and shows that
	\[
	\tau[X_j \nabla_{x_j} W(\mathbf{X})] = \nu \circ \tr[x_j \nabla_{x_j} W(\mathbf{X})].
	\]
	Therefore, $\nu \circ \tr[x_j^2] = \nu \circ \tr[\zeta(x_j)^2 x_j^2]$.  Now $\tr[\zeta(x_j)^2 x_j^2] \leq \tr[\zeta(x_j) x_j^2] \leq \tr[x_j^2]$, hence $\nu \circ \tr[\zeta(x_j)x_j^2]$ is equal to the common value of $\nu \circ \tr[x_j^2]$ and $\nu \circ \tr[\zeta(x_j)^2 x_j^2]$.  This implies that
	\[
	\nu \circ \tr[(x_j - \zeta(x_j)x_j)^2] = \nu \circ \tr[x_j^2 - 2 \zeta(x_j)x_j^2 + \zeta(x_j)^2 x_j^2] = 0.
	\]
	Now suppose that $g, h \in C_{\tr}^1(\R^{*d})$ have bounded first derivative.  Then writing $\mathbf{z}(\mathbf{x}) =  (x_1 \zeta(x_1),\dots,x_2\zeta(x_d))$, we have
	\[
	\nu \circ \tr[(g(\mathbf{x}) - g(\mathbf{z}(\mathbf{x})))^2] \leq \norm{\partial g}_{BC_{\tr}(\R^{*d},\mathscr{M}(\R^{*d}))}^2 \sum_{j=1}^d \nu \circ \tr[(x_j - \zeta(x_j)x_j)^2] = 0.
	\]
	The same holds for $h$.  Hence, because of the Cauchy-Schwarz inequality,
	\[
	\nu \circ \tr[gh] = \nu \circ \tr[(g \circ \mathbf{z})(h \circ \mathbf{z})] = \tau(g(\mathbf{X}) h(\mathbf{X})).
	\]
	Because linear combinations of functions like $\tr[gh]$ are dense in $\mathcal{C}$ by definition, it follows that $\nu$ and $I(\lambda_{\mathbf{X}})$ agree on all of $\mathcal{C}$.
\end{proof}

\subsection{Existence of potentials with unique free Gibbs laws}

We shall show in the next section that for perturbations of $V_0$, there is a unique law satisfying the Dyson-Schwinger equation, and hence in particular a unique free Gibbs law for every ultrafilter $\omega$.  But we pause here to first establish a more general result that for each $\omega$, generic potentials $V$ with bounded first and second derivatives have a unique free Gibbs law with respect to $\omega$.

\begin{proposition} \label{prop:generic}
	Fix $\omega \in \beta \N \setminus \N$ and $k \geq 2$ and $C_1$, $C_2 > 0$.  Consider the space
	\[
	\mathscr{V}_{C_1,C_2}^k := \{V_0 + W: W \in \tr(C_{\tr}^k(\R^{*d}))_{\sa} \text{ with } \norm{\partial^{j-1} \nabla W}_{BC_{\tr}(\R^{*d},\mathscr{M}^{j-1})^d} \leq C_j \text{ for } j = 1, 2\},
	\]
	equipped with the subspace topology inherited from $\tr(C_{\tr}^k(\R^{*d}))$.  Then the set of $V \in \mathscr{V}_{C_1,C_2}^k$ which have a unique free Gibbs law with respect to $\omega$ is a dense $G_\delta$-set.
\end{proposition}

Recall that a \emph{$G_\delta$ set} in a topological space is a countable intersection of open sets.  Moreover, the Baire category theorem states that in a complete metric space, a countable intersection of dense open sets is dense.  Such a set is often called \emph{generic}.  Also, note that $\mathscr{V}_{C_1,C_2}^k$ is a complete metric space.  Indeed, since the topology of $\tr(C_{\tr}^k(\R^{*d}))$ is defined by a countable family of seminorms, it is metrizable.  It is straightforward to check that $\mathscr{V}_{C_1,C_2}^k$ is a closed subset of $\tr(C_{\tr}^k(\R^{*d}))$, hence complete.

\begin{remark}
	As far as we know, $\chi_V^\omega$ may depend in general on $\omega$, and hence so does the dense $G_\delta$ set in the proposition.  The proof would apply equally well to the entropy $\chi_V$ defined by using the $\limsup$ rather than limit as $N \to \omega$ in the definition.  However, then the condition of being a free Gibbs law (maximizer of $\chi_V$) only implies convergence of the random matrix models along a subsequence of $\mu_V^{(N)}$.
\end{remark}

To prove the proposition, we do not in fact need to use the Baire category theorem.  Rather, if a potential does not have a unique free Gibbs law, we will perturb it using the following lemma.

\begin{lemma} \label{lem:testfunction}
	Let $\lambda \in \Sigma_d$.  Then there exists $f \in \tr(BC_{\tr}^\infty(\R^{*d}))_{\sa}$ such that $f^{\cA,\tau}(\mathbf{X}) \geq 0$ for all $(\cA,\tau) \in \mathbb{W}$ and $\mathbf{X} \in \cA_{\sa}^d$, and $f^{\cA,\tau}(\mathbf{X}) = 0$ if and only if $\lambda_{\mathbf{X}} = \lambda$.
\end{lemma}

\begin{proof}
	Let $R$ be an exponential bound for $\lambda$, so that $\lambda \in \Sigma_{d,R}$.  Let $R' > R$.  Let $\phi \in C^\infty(\R)$ be a function such that $\phi(t) = t$ on $[-R,R]$ and $\phi'$ is nonnegative, symmetric, and supported in $[-R',R']$.  Similar to the bump function construction in the proof of Theorem \ref{thm:magicnormbound}, Lemma \ref{lem:Fouriertransformmisc} implies that $\phi(x_j) \in BC_{\tr}^\infty(\R^{*d})$.  We claim that the sum
	\[
	f(\mathbf{x}) = \sum_{m \geq 1} \frac{1}{m!} \sum_{i_1,\dots,i_m \in \{1,\dots,d\}} |\tr(\phi(x_{i_1}) \dots \phi(x_{i_m})) - \lambda(x_{i_1} \dots x_{i_m})|^2
	\]
	converges in $\tr(BC_{\tr}^\infty(\R^{*d}))$.  For each $k > 0$ and $g \in BC_{\tr}^\infty(\R^{*d})$,
	\[
	\norm{g}_{BC_{\tr}^k(\R^{*d})} = \sum_{j=0}^k \frac{1}{j!} \norm{\partial^j g}_{BC_{\tr}^k(\R^{*d})}.
	\]
	By the same reasoning as in Lemma \ref{lem:submultiplicativenorm}, we have
	\[
	\norm{g_1g_2}_{BC_{\tr}^k(\R^{*d})} \leq \norm{g_1}_{BC_{\tr}^k(\R^{*d})} \norm{g_1}_{BC_{\tr}^k(\R^{*d})}.
	\]
	In particular,
	\begin{align*}
		\lVert(\tr(\phi(x_{i_1}) \dots \phi(x_{i_m}))& - \lambda(x_{i_1} \dots x_{i_m}))^2 \rVert_{BC_{\tr}^k(\R^{*d})} \\
		&\leq \norm{\tr(\phi(x_{i_1}) \dots \phi(x_{i_m})) - \lambda(x_{i_1} \dots x_{i_m})}_{BC_{\tr}^k(\R^{*d})}^2 \\
		&\leq \left(\norm{\tr(\phi(x_{i_1}) \dots \phi(x_{i_m})}_{BC_{\tr}^k(\R^{*d})} + |\lambda(x_{i_1} \dots x_{i_m})| \right)^2 \\
		&\leq \left( \norm{\phi(x_1)}_{BC_{\tr}^k(\R^{*d})}^m + R^m \right)^2.
	\end{align*}
	Note that
	\[
	\sum_{m \geq 1} \frac{1}{m!} \sum_{i_1,\dots,i_m \in \{1,\dots,d\}} \left( \norm{\phi(x_1)}_{BC_{\tr}^k(\R^{*d})}^m + R^m \right)^2 = \sum_{m \geq 1} \frac{1}{m!} d^m \left( \norm{\phi(x_1)}_{BC_{\tr}^k(\R^{*d})}^m + R^m \right)^2 < \infty.
	\]
	Therefore, the sum defining $f$ converges in $\tr(BC_{\tr}^k(\R^{*d}))$ for every $k$, which means it converges in $\tr(BC_{\tr}^\infty(\R^{*d}))$.
	
	Clearly, $f \geq 0$.  If $f^{\cA,\tau}(\mathbf{X}) = 0$, then $\tau(\phi(X_{i_1}) \dots \phi(X_{i_m})) = \lambda(x_{i_1} \dots x_{i_m})$ for all $m$ and $i_1$, \dots, $i_m \in \{1,\dots,d\}$.  Thus, the tuple $\mathbf{Y} = (\phi(X_1),\dots,\phi(X_d))$ satisfies $\lambda_{\mathbf{Y}} = \lambda$.  In particular, $\norm{Y_j} \leq R$.  Recall $\phi$ is an increasing function and $\phi' = 1$ on $[-R,R]$, and therefore, $|\phi(t)| > R$ whenever $|t| > R$.  By the spectral mapping theorem, the only way that $\norm{\phi(X_j)}$ can be less than or equal to $R$ is if $\norm{X_j} \leq R$.  Hence, $\phi(X_j) = X_j$, and so $\lambda_{\mathbf{X}} = \lambda$.
\end{proof}

\begin{proof}[Proof of Proposition \ref{prop:generic}]
	By Theorem \ref{thm:magicnormbound}, there exists $R > 0$ depending only on $C_1$ such that every free Gibbs law for any $V \in \mathscr{V}_{C_1,C_2}^k$ is in $I(\Sigma_{d,R})$.  
	
	We claim that any open subset $\mathscr{U}$ of $\mathscr{V}_{C_1,C_2}^k$ contains some potential which has a unique free Gibbs law with respect to $\omega$.  Let $V_0 + W \in \mathscr{U}$.  Fix $t \in (0,1)$ sufficiently close to $1$ that $V_0 + t W \in \mathscr{U}$, and note that $\norm{t \partial^j W}_{BC_{\tr}(\R^{*d},\mathscr{M}^j)} < C_j$ for $j = 1$, $2$.  Let $I(\lambda)$ be some free Gibbs law for $V_0 + t W$.  Let $f$ be as in Lemma \ref{lem:testfunction} for $\lambda$.  By choosing $\epsilon > 0$ small enough, we can guarantee that $V_0 + t W + \epsilon f$ is in $\mathscr{U}$.
	
	We claim that $I(\lambda)$ is the unique free Gibbs law for $V = V_0 + t W + \epsilon f$.  Recall that
	\[
	\chi_V^\omega(\nu) = \chi^\omega(\nu) - \nu(V) + K
	\]
	for some constant $K$.  Any free Gibbs law has the form $I(\mu)$ for some $\mu \in \Sigma_{d,R}$.  Now
	\[
	\chi^\omega(I(\mu)) - I(\mu)[V_0 + tW] \leq \chi^\omega(I(\lambda)) - I(\lambda)[V_0 + tW]
	\]
	By our choice of $f$,
	\[
	I(\mu)[-f] \leq 0 = I(\lambda)[-f]
	\]
	with equality if and only if $\mu = \lambda$.  It follows that $I(\lambda)$ is the unique maximizer of $\chi_V^\omega$.
	
	It remains to show that the set of $V$ which have a unique free Gibbs law is a $G_\delta$ set.  Recall that $\Sigma_{d,R}$ is compact and metrizable, so let $\rho$ be a metric.  Let $V \in \mathscr{V}_{C_1,C_2}^k$, let $\mathcal{G}(V) \subseteq \Sigma_{d,R}$ be the set of $\lambda$ such that $I(\lambda)$ is a free Gibbs law for $V$ with respect to $\omega$.  By upper semi-continuity of $\chi_V^\omega$, the space of free Gibbs laws for $V$ is closed in $\mathcal{C}^{\star}$, hence in light of Lemma \ref{lem:Claw}, $\mathcal{G}(V)$ is closed in $\Sigma_{d,R}$.  Let
	\[
	\mathscr{U}_n = \{ V \in \mathscr{V}_{C_1,C_2}^k: \mathcal{G}(V) \subseteq B_{1/n}(\mu) \text{ for some } \mu \in \Sigma_{d,R} \},
	\]
	where $B_{1/n}(\mu)$ is the open ball of radius $n$ in $\Sigma_{d,R}$ with respect to the metric $\rho$.  Observe that $V \in \bigcap_{n=1}^\infty \mathscr{U}_n$ if and only if the set $\mathcal{G}(V)$ has diameter zero if and only if $V$ has a unique free Gibbs law.
	
	We claim that $\mathscr{U}_n$ is open.  Fix $V \in \mathscr{U}_n$.  Let $\mu \in \Sigma_{d,R}$ such that $\mathcal{G}(V) \subseteq B_{1/n}(\mu)$.  Note that $\Sigma_{d,R} \setminus B_{1/n}(\mu)$ is compact, hence its image in $\mathcal{C}^{\star}$ is a closed set, so $\chi_V^\omega$ achieves a maximum, which must be strictly less than zero since all the free Gibbs laws for $V$ are in $B_{1/n}(\mu)$.  Call the maximum $-\epsilon$.  Let $I(\lambda)$ be a free Gibbs law for $V$.  Then
	\[
	\sup_{\nu \in \Sigma_{d,R} \setminus B_{1/n}(\mu)} \left( \chi^\omega(I(\nu)) - I(\nu)[V] \right) \leq \chi^\omega(I(\lambda)) - I(\lambda)[V] - \epsilon.
	\]
	If $V' \in \mathscr{V}_{C_1,C_2}^k$ such that $\norm{V' - V}_{C_{\tr}(\R^{*d}),R} \leq \epsilon / 3$, then
	\begin{align*}
		\sup_{\nu \in \Sigma_{d,R} \setminus B_{1/n}(\mu)} \chi^\omega(I(\nu)) - I(\nu)[V']
		&\leq \sup_{\nu \in \Sigma_{d,R} \setminus B_{1/n}(\mu)} \chi^\omega(I(\nu)) - I(\nu)[V] + \frac{\epsilon}{3} \\
		&\leq \chi^\omega(I(\lambda)) - I(\lambda)[V] - \frac{2 \epsilon}{3} \\
		&\leq \chi^\omega(I(\lambda)) - I(\lambda)[V'] - \frac{\epsilon}{3}.
	\end{align*}
	Hence, for $V'$ in a neighborhood of $V$, the elements of $\Sigma_{d,R} \setminus B_{1/n}(\mu)$ are not free Gibbs laws, which implies that $\mathcal{G}(V') \subseteq B_{1/n}(\mu)$, so $V' \in \mathscr{U}_n$.  Thus, $\mathscr{U}_n$ is open as desired.
\end{proof}

\section{Rigorous transport results in the perturbative setting} \label{sec:rigoroustransport}

In this section, we will combine the results of \S \ref{sec:pseudoinverse} and \S \ref{sec:freeGibbslaws} to study free transport for potentials $V$ sufficiently close to $(1/2) \sum_j \tr(x_j^2)$.  If $V$ satisfies $\nabla V \in \mathscr{J}_{a,c}^d$ (see Definition \ref{def:J}).  In \S \ref{sec:pseudoinverse}, we constructed an expectation map $\mathbb{E}_V := \mathbb{E}_{\nabla V}$.  We will also use the notation $L_V$, $e^{tL_V}$, and $\Psi_V$ rather than $L_{\nabla V}$, $e^{tL_{\nabla V}}$, and $\Psi_{\nabla V}$.  We will show in Proposition \ref{prop:convexDSE} that $\mathbb{E}_V$ describes the unique free Gibbs law for $V$.  Then Theorem \ref{thm:transport} will complete the strategy of \ref{subsec:discussion} to construct transport.

We use the same strategy to prove a more refined result (Theorem \ref{thm:triangulartransport}), which produces triangular smooth transport  which produces a triangular smooth transport, and hence triangular isomorphisms of $\mathrm{C}^*$- and $\mathrm{W}^*$-algebras.  Several of the necessary ingredients, such as a conditional version of the Dyson-Schwinger equation, cannot be deduced directly from the results of \S \ref{sec:freeGibbslaws}.  We rely instead upon the relationship between $\mathbb{E}_{\mathbf{x},V}$ to conditional expectations from random matrix theory and operator algebras, which is also of interest in its own right.

\subsection{Existence of transport} \label{subsec:constructtransport}

\begin{proposition} \label{prop:convexDSE}
	Let $V$ satisfy $\nabla V \in \mathscr{J}_{a,c}^d$ for some $a \in \R$ and $c \in (0,1)$.  Then $\mathbb{E}_V|_{\mathcal{C}}$ is the unique element of $\mathcal{C}^{\star}$ satisfying \eqref{eq:DSE}.  In particular, for any $\omega \in \beta \N \setminus \N$, it is the unique free Gibbs law for $V$ with respect to $\omega$.
\end{proposition}

\begin{proof}
	Let $\nu \in \mathcal{C}^{\star}$ satisfy \eqref{eq:DSE2}.  By Theorem \ref{thm:magicnormbound}, $\nu = I(\lambda)$ for some $\lambda \in \Sigma_{d,R}$ for some $R > 0$, and the corresponding homomorphism $\tilde{\lambda}: \tr(C_{\tr}(\R^{*d})) \to \C$ satisfies the Dyson-Schwinger equation for all smooth test functions.  If $f \in \tr(C_{\tr}^\infty(\R^{*d}))$, then Proposition \ref{prop:pseudoinverse} we have $\Psi_V f \in \tr(C_{\tr}^\infty(\R^{*d}))$ and hence $\nabla (\Psi_V f) \in C_{\tr}^\infty(\R^{*d})^d$.  Thus, by \eqref{eq:DSE2},
	\[
	0 = \tilde{\lambda}[\nabla_V^* \nabla (\Psi_V f)] = \tilde{\lambda}[f - \mathbb{E}_V[f]] = \tilde{\lambda}(f) - \mathbb{E}_V(f).
	\]
	Therefore, $\tilde{\lambda}[f] = \mathbb{E}_V[f]$ for all smooth $f$.  By density, this extends to all of $\tr(C_{\tr}(\R^{*d}))$.  Hence, $\tilde{\lambda} = \mathbb{E}_V$ and $\nu = \mathbb{E}_V|_{\mathcal{C}}$.
\end{proof}

\begin{corollary}
	If $V$ satisfies $\nabla V \in \mathcal{J}_{a,c}^d$ for some $c > 0$ and $a \in \R$, then for every $f \in \tr(C_{\tr}^1(\R^{*d}))$ with $\partial f$ bounded and for every $\epsilon > 0$, we have
	\[
	\limsup_{N \to \infty} \frac{1}{N^2} \log \mu_V^{(N)}(\{\mathbf{X}: |f(\mathbf{X}) - \mathbb{E}_V(f)| \geq \epsilon \}) < 0.
	\]
\end{corollary}

As a consequence of \eqref{eq:DSE0} and Proposition \ref{prop:Laplacianrelations}, any such $V$ satisfies Assumptions \ref{ass:freeGibbs} and \ref{ass:Laplacian}.  Hence, all the properties of Propositions \ref{prop:Laplacianrelations2} and \ref{prop:Laplacianrelations3} hold.  Now we give a rigorous proof of transport for log-densities close to the quadratic, and in fact ``infinitesimally optimal'' transport.

\begin{theorem} \label{thm:transport}
	Let $V_t = (1/2) \ip{\mathbf{x},\mathbf{x}}_{\tr} + W_t$, where $t \mapsto W_t$ be a continuously differentiable path $[0,T] \to \tr(C_{\tr}^\infty(\R^{*d}))_{\sa}$.  Suppose that
	\begin{align*}
		\norm{\partial^{k-1} \nabla W}_{BC_{\tr}(\R^{*d},\mathscr{M}^{k-1})^d} \leq C_k \text{ for } k = 1,2,3, \\
		\norm{\partial^{k-1} \nabla \dot{W}}_{BC_{\tr}(\R^{*d},\mathscr{M}^{k-1})^d} \leq C_k' \text{ for } k = 1,2, 
	\end{align*}
	for constants $C_1$, $C_2$, $C_3$, $C_1'$, $C_2' \in [0,\infty)$ such that $C_2 < 1$.  Let $V_t = \norm{\mathbf{x}}_{2,\tr}^2 + W_t$.  Let
	\[
	\mathbf{h}_t = -\nabla \Psi_{V_t} \dot{V}_t.
	\]
	Then the solution $\mathbf{f}_t$ to \eqref{eq:flow} satisfies $(\mathbf{f}_t)_* V_0 = V_t$ modulo constants for all $t$.  Moreover, this choice of $\mathbf{h}_t$ minimizes
	\[
	\int_0^T \mathbb{E}_{V_t} \norm{\mathbf{h}_t}_{2,\tr}^2\,dt
	\]
	among all maps $t \mapsto \mathbf{h}_t$ satisfying the hypotheses of Lemma \ref{lem:transport0} with $(\mathbf{f}_t)_* V_0 = V_t$ for all $t$.
\end{theorem}

\begin{remark}
	The last condition says that the transport is ``infinitesimally optimal.''
\end{remark}

\begin{proof}
	Note that $\nabla V_t \in \mathscr{J}_{C_1,1-C_2}^d$, and thus Proposition \ref{prop:pseudoinverse} constructs a pseudo-inverse $\Psi_{V_t}$ for $-L_{V_t}$.  Let
	\[
	\mathbf{h}_t = -\nabla \Psi_{V_t} \dot{W}_t.
	\]
	We apply Proposition \ref{prop:pseudoinverse} (3) and Remark \ref{rem:boundedcoefficients2}, observing that $\partial_{\mathbf{x}}$ reduces to $\partial$ since there is no $\mathbf{x}'$.  Because $\nabla W_t \in BC_{\tr}^3(\R^{*d})_{\sa}^d$, we have
	\[
	\norm{\partial^2 \Psi_{V_t} \dot{W}_t}_{BC_{\tr}(\R^{*d},\mathscr{M}^2),R}
	\leq \text{constant} \sum_{k=0}^2 \norm{\partial^k \dot{W}_t}_{BC_{\tr}(\R^{*d},\mathscr{M}^2),R'},
	\]
	which is bounded by a constant, and similarly $\partial^3 \Psi_{V_t} \dot{W}_t$ is bounded by a constant.  Therefore, $\partial \mathbf{h}_t$ and $\partial^2 \mathbf{h}_t$ are bounded by constants.  By Lemma \ref{lem:flow}, there is a family of diffeomorphisms $\mathbf{f}_t$ satisfying $\dot{\mathbf{f}}_t = \mathbf{h}_t \circ \mathbf{f}_t$ and $\mathbf{f}_0 = \id$.  Note that $-\nabla_{V_t}^* \mathbf{h}_t = \nabla_{V_t}^* \nabla \Psi_{V_t} \dot{W}_t = \dot{W}_t$ modulo constants.  Therefore, by Lemma \ref{lem:transport0}, we have $(\mathbf{f}_t)_* V_0 = V_t$ modulo constants.
	
	Finally, consider another possible choice of functions $\tilde{\mathbf{h}}_t$.  If the flow generated by $\tilde{\mathbf{h}}_t$ transports $V_0$ to $V_t$ modulo constants, then by the previous proposition, we must have $\nabla_{V_t}^* \tilde{\mathbf{h}}_t = \dot{W}_t = \nabla_{V_t}^* \mathbf{h}_t$ modulo constants.  Since $\tilde{\mathbf{h}}_t - \mathbf{h}_t$ is in the kernel of $\nabla_V^*$, it is orthogonal with respect to $\mathbb{E}_{V_t}$ to any gradient by Proposition \ref{prop:Laplacianrelations3} (4), and in particular orthogonal to $\mathbf{h}_t$.  Hence,
	\[
	\mathbb{E}_{V_t} \norm*{\mathbf{h}_t}_{2,\tr}^2 \leq \mathbb{E}_{V_t} \norm*{\tilde{\mathbf{h}}_t}_{2,\tr}^2,
	\]
	which shows the desired optimality condition.
\end{proof}

In the situation of Theorem \ref{thm:transport}, the law $\mu_{V_t}$ is the unique free Gibbs law associated to $V_t$ by Proposition \ref{prop:convexDSE}.  Therefore, $(\mathbf{f}_t)_* V_0 = V_t$ implies that $(\mathbf{f}_t)_* \mu_{V_0} = \mu_{V_t}$ by Proposition \ref{prop:changeofvariables}.  This directly implies isomorphism of $\mathrm{W}^*$- and $\mathrm{C}^*$-algebras associated to $\mu_{V_0}$ and $\mu_{V_t}$.  This result is closely related to those of \cite{GS2014,DGS2016,JekelExpectation,JekelThesis}, and can be stated precisely as follows.

\begin{observation} \label{obs:transportapplication}
	Suppose that $V_0$ and $V_1 \in \tr(C_{\tr}^\infty(\R^{*d})))$ such that $V_j = (1/2)\norm{\mathbf{x}}_2^2 + W_j$ with
	\[
	\norm{\partial^k W_j}_{BC_{\tr}(\R^{*d},\mathscr{M}^k)} < C_k \text{ for } k = 1, 2, 3,
	\]
	with $C_2 < 1$.  Then the path $W_t = (1 - t) W_0 + t W_1$ satisfies the assumptions of Theorem \ref{thm:transport}.  Hence, by the theorem, there exists some $\mathbf{f} \in C_{\tr}^\infty(\R^{*d})_{\sa}^d$ such that $\mathbf{f}_* \mu_{V_0} = \mu_{V_1}$.  Because $\mathbf{f}$ is given by solving the ODE \eqref{eq:flow}, the function $\mathbf{f}$ also has an inverse $\mathbf{g} \in C_{\tr}^\infty(\R^{*d})_{\sa}^d$.  In particular, by Observation \ref{obs:pushforwardisomorphism2}, there is a tracial $\mathrm{W}^*$ isomorphism between the GNS representations of $\mu_{V_0}$ and $\mu_{V_1}$ which also restricts to an isomorphism of the associated $\mathrm{C}^*$-algebras.
\end{observation}

\begin{corollary} \label{cor:transportapplication}
	Suppose that $V = (1/2)\norm{\mathbf{x}}_2^2 + W$ where $\partial W \in \tr(BC_{\tr}(\R^{*d},\mathscr{M}(\R^{*d})))$ and $\norm{\partial^2 W}_{BC_{\tr}(\R^{*d},\mathscr{M}^2)} < 1$.  Then the GNS representation of $\mu_V$ is isomorphic to the tracial $\mathrm{W}^*$-algebra generated by a standard semicircular family $\mathbf{S} = (S_1,\dots,S_d)$, and the isomorphism restricts to an isomorphism of the $\mathrm{C}^*$-algebras.
\end{corollary}

\subsection{Matrix approximation and non-commutative functions} \label{subsec:matrixapproximation}

Although the construction of $\mathbb{E}_{\mathbf{x},V}$ nowhere used matrix approximations, we will use the matrix approximations to prove various relations among different conditional expectation maps.  Even in the previous section, we could only prove the properties of Proposition \ref{prop:Laplacianrelations2} after knowing the Dyson-Schwinger equation $\mathbf{E}_V \nabla_V^* \mathbf{h} = 0$ for $\mathbf{h} \in BC_{\tr}^2(\R^{*d})^d$.  The Dyson-Schwinger equation in turn was deduced from the fact that the free Gibbs law maximized the free entropy $\chi_V^\omega$.  But free entropy is defined in terms of matricial microstates.  Hence, even our previous results depended on matrix approximation.

As we do not yet know a good definition for conditional microstate entropy, our results in the conditional setting will rely on the random matrix models in a more explicit fashion.  As in \cite{Jekel2018,JekelExpectation,JekelThesis}, we will view the functions in $C_{\tr}(\R^{*d})$ as large-$N$ asymptotic descriptions of certain sequences of functions on $M_N(\C)_{\sa}^d$.  For this reason, we desire a function $f$ to be uniquely determined by knowing its restrictions $f^{M_N(\C),\tr_N}$ for all $N$.  Thus, we must restrict our attention to tracial $\mathrm{W}^*$-algebras that can be approximated by matrices in a certain sense.

We say that $(\cA,\tau)$ is \emph{Connes-approximable} or \emph{Connes-embeddable} if for every $d$ and every $\mathbf{X} \in \cA_{\sa}^d$, there exists a sequence of $d$-tuples $\mathbf{X}^{(N)} \in M_N(\C)^d$ that converges in non-commutative law to $\mathbf{X}$.  It is well-known in von Neumann algebras that this is equivalent to the embeddability of $(\cA,\tau)$ into the ultrapower $(\mathcal{R},\tau_{\mathcal{R}})^\omega$ for some $\omega \in \beta \N \setminus \N$.  However, recent work has shown that not every tracial von Neumann algebra has this property \cite{JNVWY2020}.

The space $C_{\tr}(\R^{*d})$ by definition consists of tuples of functions on $d$-tuples for any separable tracial $\mathrm{W}^*$-algebra, since we used a set of isomorphism class representative of such tracial $\mathrm{W}^*$-algebras to define the norm.  However, the same constructions can be performed using some subclass of tracial $\mathrm{W}^*$-algebras.  When we replace the set of representatives $\mathbb{W}$ with a set of representatives $\mathbb{W}_{\app}$ for Connes-approximable tracial $\mathrm{W}^*$-algebras, we obtain analogous spaces to $C_{\tr}^k(\R^{*d}, \mathscr{M}(\R^{*d_1}, \dots, \R^{*d_\ell}))$ which we will denote $C_{\tr,\app}^k(\R^{*d}, \mathscr{M}(\R^{*d_1}, \dots, \R^{*d_\ell}))$, where the subscript $\app$ stands for ``approximable.''

All the results in the paper work with $C_{\tr}^k$ replaced with $C_{\tr,\app}^k$.  For \S \ref{sec:pseudoinverse}, one of course has to define the Connes-approximable versions of $C_{\tr,\mathcal{S}}$ where $\mathcal{S}$ is a Brownian motion.  It is well-known that if $(\cA,\tau)$ is Connes-embeddable and if $(\cB,\sigma)$ is the tracial $\mathrm{W}^*$-algebra generated by the free Browian motion $\mathcal{S}$, then $(\cA,\tau) * (\cB,\sigma)$ is Connes-embeddable \cite[Proposition 3.3]{Voiculescu1998}.
%also found in \cite[Corollary 4.5]{BDJ2008}

The next lemma shows that functions in $C_{\tr,\app}^k(\R^{*d}, \mathscr{M}(\R^{*d_1}, \dots, \R^{*d_\ell}))$ are uniquely determined by their values on matrix tuples.  The proof may be obvious to those familiar with folklore about Connes-approximability, but nonetheless we will explain the argument here for the sake of completeness.

\begin{lemma} \label{lem:matrixnorm}
	Given $\mathbf{f}^{(N)}: M_N(\C)_{sa}^d \times M_N(\C)^{d_1} \times \dots \times M_N(\C)^{d_\ell} \to M_N(\C)^{d'}$ that is multilinear in the last $\ell$ arguments, define as in \ref{def:basicnorm}
	\[
	\norm{\mathbf{f}^{(N)}}_{\mathscr{M}^\ell,\tr,R} = \sup \{ \norm{\mathbf{f}^{(N)}(\mathbf{X})}_{\mathscr{M}^\ell,\tr}: \mathbf{X} \in M_N(\C)_{\sa}^d, \norm{\mathbf{X}} \leq R\}.
	\]
	Let $\mathbf{f} \in \C_{\tr,\app}(\R^{*d},\mathscr{M}(\R^{*d_1},\dots,\R^{*d_\ell})^{d'}$.  Then
	\[
	\norm{\mathbf{f}}_{C_{\tr,\app}(\R^{*d})^{d'},R} = \sup_N \norm{\mathbf{f}^{M_N(\C),\tr_N}}_{\mathscr{M}^\ell,\tr,R} = \lim_{N \to \infty} \norm{\mathbf{f}^{M_N(\C),\tr_N}}_{\mathscr{M}^\ell,\tr,R}.
	\]
\end{lemma}

\begin{proof}
	Note that it suffices to prove both equalities when $\mathbf{f}$ is a trace polynomial, since any $\mathbf{f} \in C_{\tr,\app}(\R^{*d},\mathscr{M}(\R^{*d_1},\dots,\R^{*d_\ell}))^{d'}$ can be approximated in $\norm{\cdot}_{C_{\tr,\app}(\R^{*d},\mathscr{M}^\ell)^{d'},R}$ by trace polynomials, and this norm clearly dominates the matrix version on the right-hand side. Now given some Connes-approximable $\cA$, some $\alpha$, $\alpha_1$, \dots, $\alpha_\ell$ with $1/\alpha = 1/\alpha_1 + \dots + 1/\alpha_\ell$,  and some $\mathbf{X}$, $\mathbf{Y}_1$, \dots, $\mathbf{Y}_\ell \in \cA_{\sa}^d$, we may choose some matrix tuples $\mathbf{X}^{(N)} \in M_N(\C)_{\sa}^d$, $\mathbf{Y}_1^{(N)} \in M_N(\C)^{d_1}$, \dots, $\mathbf{Y}_\ell \in M_N(\C)_{\sa}^{d_\ell}$ such that $\mathbf{X}^{(N)}$ and the real and imaginary parts of $\mathbf{Y}_j^{(N)}$ converge in joint non-commutative law to $\mathbf{X}$ and the real and imaginary parts of $\mathbf{Y}_1$, \dots, $\mathbf{Y}_\ell$.  By applying a cut-off function to $\mathbf{X}^{(N)}$ and the real and imaginary parts of $\mathbf{Y}_j^{(N)}$, we may also assume that $\norm{\mathbf{X}^{(N)}} \leq R$ and $\norm{\mathbf{Y}_j^{(N)}} \leq 2 \norm{\mathbf{Y}_j}$.  Convergence in law also implies convergence of the $L^\beta$ norms of $\mathbf{X}^{(N)}$, $\mathbf{Y}_1^{(N)}$, \dots, $\mathbf{Y}_\ell^{(N)}$ to those of the corresponding operators for $\beta \in [1,\infty)$.  Using convergence in law again, we also have
	\[
	\lim_{N \to \infty} \norm{f^{(N)}(\mathbf{X}^{(N)})[\mathbf{Y}_1^{(N)},\dots,\mathbf{Y}_\ell^{(N)}]}_\beta = \norm{f(\mathbf{X})[\mathbf{Y}_1,\dots,\mathbf{Y}_\ell]}_\beta \text{ for } \beta \in [1,\infty)
	\]
	and because the $\infty$-norm can be recovered as the limit of the $\beta$-norms as $\beta \to \infty$, we have
	\[
	\norm{f(\mathbf{X})[\mathbf{Y}_1,\dots,\mathbf{Y}_\ell]}_\infty \leq \liminf_{N \to \infty} \norm{f^{(N)}(\mathbf{X}^{(N)})[\mathbf{Y}_1^{(N)},\dots,\mathbf{Y}_\ell^{(N)}]}_\infty.
	\]
	This implies that
	\begin{align*}
		\norm{\mathbf{f}}_{C_{\tr,\app}(\R^{*d},\mathscr{M}^\ell)^{d'},R} &\leq \liminf_{N \to \infty} \norm{\mathbf{f}^{M_N(\C),\tr_N}}_{\mathscr{M}^\ell,\tr,R}^{(N)} \\
		&\leq \limsup_{N \to \infty} \norm{\mathbf{f}^{M_N(\C),\tr_N}}_{\mathscr{M}^\ell,\tr,R}^{(N)} \\
		&\leq \sup_N \norm{\mathbf{f}^{M_N(\C),\tr_N}}_{\mathscr{M}^\ell,\tr,R}^{(N)} \\
		&\leq \norm{\mathbf{f}}_{C_{\tr,\app}(\R^{*d},\mathscr{M}^\ell)^{d'},R} \qedhere
	\end{align*}
\end{proof}

Next, we define a precise notion of an element of $C_{\tr,\app}(\R^{*d},\mathscr{M}^\ell)$ describing the large $N$ limit of a sequence of functions on $M_N(\C)_{\sa}^d$.

\begin{definition} \label{def:asymptotic}
	Let
	\[
	\mathbf{f}^{(N)}: M_N(\C)_{\sa}^d \times M_N(\C)_{\sa}^{d_1} \times \dots \times M_N(\C)_{\sa}^{d_\ell} \to M_N(\C)^{d'}
	\]
	and let $\mathbf{f} \in C_{\tr,\app}(\R^{*d}, \mathscr{M}(\R^{*d_1}, \dots, \R^{*d_\ell}))^{d'}$.  We say that $(\mathbf{f}^{(N)})_{N \in \N}$ is \emph{asymptotic to} $\mathbf{f}$, or $\mathbf{f}^{(N)} \rightsquigarrow \mathbf{f}$ if
	\[
	\lim_{N \to \infty} \norm{\mathbf{f}^{(N)} - \mathbf{f}^{M_N(\C),\tr_N}}_{\tr,R} = 0.
	\]
\end{definition}

\begin{remark}
	In the case $\ell = 0$, the error is measured in $\norm{\cdot}_\infty$ uniformly on operator norm balls.  This condition is stronger than the one in \cite{JekelExpectation} and \cite{JekelThesis}, which measured the error in $\norm{\cdot}_2$.
\end{remark}

\begin{remark} \label{rem:uniquelimit}
	It follows from Lemma \ref{lem:matrixnorm} that the condition $\mathbf{f}^{(N)} \rightsquigarrow \mathbf{f}$ uniquely determines $\mathbf{f}$.
\end{remark}

\begin{lemma} \label{lem:asymptoticcomposition}
	Let $\mathbf{f} \in C_{\tr}(\R^{*d'},\mathscr{M}(\R^{*d_1},\dots,\R^{*d_n}))^{d''}$ for some $n, d' \in \N_0$ and $d''$, $d_1$, \dots, $d_n \in \N$.  Let $\mathbf{g} \in C_{\tr}(\R^{*d})_{\sa}^{d'}$ for some $d \in \N_0$.  For each $m = 1$, \dots, $n$, let $\mathbf{h}_m \in C_{\tr}(\R^{*d}, \mathscr{M}(\R^{*d_{m,1}},\dots,\R^{*d_{m,\ell_m}})^{d_m}$ for some $\ell_m \in \N_0$ and $d_{m,1}$, \dots, $d_{m,\ell_m}$.  Similarly, let
	\begin{align*}
		\mathbf{f}^{(N)}: & M_N(\C)_{\sa}^{*d'} \times M_N(\C)^{d_1} \times \dots \times M_N(\C)^{d_n} \to M_N(\C)^{d''} \\
		\mathbf{g}^{(N)}: & M_N(\C)_{\sa}^d \to M_N(\C)_{\sa}^{d'} \\
		\mathbf{h}_m^{(N)}: & M_N(\C)_{\sa}^{d'} \times M_N(\C)^{d_{m,1}} \times M_N(\C)^{d_{m,\ell_m}} \to M_N(\C)^{d''},
	\end{align*}
	where $\mathbf{f}^{(N)}$ and $\mathbf{h}_m^{(N)}$ are multilinear in the last $n$ and $\ell_m$ arguments respectively.  If $\mathbf{f}^{(N)} \rightsquigarrow \mathbf{f}$, $\mathbf{g}^{(N)} \rightsquigarrow \mathbf{g}$, and $\mathbf{h}_m^{(N)} \rightsquigarrow \mathbf{h}_m$ for each $m$, then
	\[
	\mathbf{f}^{(N)}(\mathbf{g}^{(N)})[\mathbf{h}_1^{(N)},\dots,\mathbf{h}_n^{(N)}] \rightsquigarrow \mathbf{f}(\mathbf{g})[\mathbf{h}_1,\dots,\mathbf{h}_n].
	\]
\end{lemma}

The proof is essentially the same as the proof of continuity of composition in Lemma \ref{lem:composition}, hence we leave the details to the reader.

\subsection{$\mathbb{E}_{\mathbf{x},V}$ and conditional expectations} \label{subsec:conditional}

\begin{definition}
	For each choice of $C_1$, $C_2$, $C_3 > 0$, let $\mathscr{V}_{d,C_1,C_2,C_3}$ be the set of functions $V = \frac{1}{2} \norm{\mathbf{x}}_2^2 + W \in \tr(C_{\tr}^\infty(\R^{*d}))_{\sa}$ satisfying
	\[
	\norm{ \partial^{k-1} \nabla W }_{BC_{\tr}(\R^{*d},\mathscr{M}^k)^d} \leq C_k
	\]
	for $k = 1$, $2$, $3$.
\end{definition}

For $V \in \mathscr{V}_{d,C_1,C_2,C_3}$, we will denote the expectation $\mathbb{E}_{\mathbf{x},\nabla_{\mathbf{x} V}}$ from \S \ref{sec:pseudoinverse} simply by $\mathbb{E}_{\mathbf{x},V}$.  In this subsection, we will show that the expectation map $\mathbb{E}_{\mathbf{x},V}$ describes the large $N$ limit of classical conditional expectations associated to the measures $\mu_V^{(N)}$.

Given a potential $V^{(N)}: M_N(\C)_{\sa}^{d+d'} \to \R$ such that $e^{-N^2 V^{(N)}}$ is integrable, we define
\[
d\mu_{V^{(N)}}(\mathbf{X},\mathbf{X}') = \frac{e^{-N^2} V^{(N)}(\mathbf{X},\mathbf{X}')\,d\mathbf{X}\,d\mathbf{X}'}{\int_{M_N(\C)_{\sa}^{d+d'}} e^{-N^2 V^{(N)}(\mathbf{X},\mathbf{X}')}\,d\mathbf{X}\,d\mathbf{X}' }.
\]
Moreover, we define the conditional distribution
\[
d\mu_{V^{(N)}}(\mathbf{X} | \mathbf{X}') = \frac{ e^{-N^2 V^{(N)}(\mathbf{X},\mathbf{X}')}\,d\mathbf{X}}{\int_{M_N(\C)_{\sa}^d} e^{-N^2 V^{(N)}(\mathbf{X},\mathbf{X}')}\,d\mathbf{X}}.
\]
If $\mathbf{f}^{(N)}: M_N(\C)_{\sa}^{d+d'} \times M_N(\C)_{\sa}^{d_1} \times \dots \times M_N(\C)_{\sa}^{d_\ell} \to M_N(\C)^{d_2}$ is real-multilinear in the last $\ell$ arguments, we set
\[
\mathbf{E}_{\mathbf{x},V^{(N)}}[\mathbf{f}^{(N)}](\mathbf{X}')[\mathbf{Y}_1,\dots,\mathbf{Y}_\ell] = \frac{\int_{M_N(\C)_{\sa}^d} \mathbf{f}^{(N)}(\mathbf{X})[\mathbf{Y}_1',\dots,\mathbf{Y}_\ell] e^{-N^2 V^{(N)}(\mathbf{X},\mathbf{X}')}\,d\mathbf{X}}{\int_{M_N(\C)_{\sa}^d} e^{-N^2 V^{(N)}(\mathbf{X},\mathbf{X}')}\,d\mathbf{X}}.
\]
This describes the conditional expectation of $\mathbf{f}^{(N)}(\mathbf{X},\mathbf{X}')$ given $\mathbf{X}'$, when $(\mathbf{X},\mathbf{X}')$ is a random variable with the distribution $\mu^{(N)}$.  Note that the subscript $\mathbf{x}$ denotes integration with respect to $\mathbf{x}$, hence conditioning on $\mathbf{x}'$.

\begin{theorem} \label{thm:classicalconditionalexpectation}
	Let $V \in \mathscr{V}_{C_1,C_2,C_3}$ for some $C_2 < 1$.  Let $V^{(N)}: M_N(\C)_{\sa}^{d+d'} \to \R$ such that
	\begin{enumerate}[(1)]
		\item $V^{(N)}$ is invariant under conjugation of $X_1$, \dots, $X_{d+d'}$ by a fixed unitary $U$.
		\item $V^{(N)}$ is a $C^1$ function and $\nabla V^{(N)} \rightsquigarrow \nabla V$.
		\item $V^{(N)}(\mathbf{X}) - \frac{1}{2} c \norm{\mathbf{X}}_2^2$ is convex and $V^{(N)}(\mathbf{X}) - \frac{1}{2} C \norm{\mathbf{X}}_2^2$ is concave for some $0 < c < C$.
	\end{enumerate}
	Let $\mathbf{f} \in C_{\tr}(\R^{*(d+d')}, \mathscr{M}(\R^{*d_1}, \dots, \R^{*d_\ell})^{d''}$, and let $\mathbf{f}^{(N)}: M_N(\C)_{\sa}^{(d+d')}  \times M_N(\C)_{\sa}^{d_1} \times \dots \times M_N(\C)_{\sa}^{d_\ell} \to M_N(\C)_{\sa}^{d''}$ with $\mathbf{f}^{(N)} \rightsquigarrow \mathbf{f}$ and
	\[
	\norm{\mathbf{f}^{(N)}(\mathbf{X},\mathbf{X}')}_{\mathscr{M}^\ell,\tr} \leq K_1 e^{K_2 \norm{(\mathbf{X},\mathbf{X}')}_\infty}
	\]
	for some constants $K_1$ and $K_2$.  Then
	\[
	\mathbb{E}_{\mathbf{x},V^{(N)}}[\mathbf{f}^{(N)}] \rightsquigarrow \mathbb{E}_{\mathbf{x},V}[\mathbf{f}].
	\]
\end{theorem}

\begin{remark} \label{rem:canonicalapproximant}
	If we take $V^{(N)} = V^{M_N(\C),\tr_N}$, then the hypotheses (1), (2), (3) are automatically satisfied.  For the condition (3), we set $c = 1 - C_2$ and $C = 1 + C_2$ where $C_2 = \norm{\partial \nabla V - \Id}_{BC_{\tr}(\R^{*(d+d')},\mathscr{M}^1)}$.
\end{remark}

Since the asymptotic approximation relation $\rightsquigarrow$ relies on approximation for each operator norm ball, we will have to truncate the conditional distribution $\mu_{V^{(N)}}(\mathbf{X} | \mathbf{X}')$ to operator-norm balls.  The following lemma from \cite{JekelThesis} relies on concentration of measure (see e.g.\ \cite{Gross1975}, \cite{Ledoux1992}, \cite{BL2000}, \cite[\S 2.3.3 and 4.4.2]{AGZ2009}) and its application to random matrices (see \cite{GZ2000}) through an $\epsilon$-net argument (see \cite[\S 2.3.1]{Tao2012}) as well as the fact that the conditional expectation of a Lipschitz function is Lipschitz when $V^{(N)}$ satisfies (3).  For the proof, refer to \cite[p.\ 277]{JekelThesis}.  The constant $R_3$ on p.\ 277 is the $R'$ in the lemma statement here.

\begin{lemma} \label{lem:operatornormtailbound}
	Suppose that $V^{(N)}: M_N(\C)_{\sa}^{d+d'}$ satisfies assumptions (1), (2), and (3) of the theorem, and let $K > 0$ and $R > 0$.  Then there is some constant $R'$ such that
	\[
	\lim_{N \to \infty} \sup_{\norm{\mathbf{X}'}_\infty \leq R} \int_{\norm{\mathbf{x}}_\infty \geq R'} e^{K \norm{\mathbf{X}}_\infty}\,d\mu^{(N)}(\mathbf{X} | \mathbf{X}') = 0.
	\]
\end{lemma}

\begin{proof}[Proof of Theorem \ref{thm:classicalconditionalexpectation}]
	First, consider the case where $\mathbf{f}^{(N)}$ is exactly equal to $\mathbf{f}^{M_N(\C),\tr_N}$ and
	\[
	\mathbf{f} \in BC_{\tr,\app}^2(\R^{*(d+d')},\mathscr{M}(\R^{*d_1},\dots,\R^{*d_\ell})^{d''} \cap C_{\tr,\app}^\infty(\R^{*(d+d')},\mathscr{M}(\R^{*d_1},\dots,\R^{*d_\ell})^{d''}.
	\]
	Let $\mathbf{g} = \Psi_{\mathbf{x},V} \mathbf{f}$.  Recall that
	\[
	\mathbf{f} = \mathbb{E}_{\mathbf{x},V}[\mathbf{f}] \circ \pi' - L_{\mathbf{x},V} \mathbf{g},
	\]
	and hence
	\[
	\mathbb{E}_{\mathbf{x},V^{(N)}}[\mathbf{f}^{M_N(\C),\tr_N}] - \mathbb{E}_{\mathbf{x},V}[\mathbf{f}]^{M_N(\C),\tr_N} = \mathbb{E}_{\mathbf{x},V^{(N)}}[L_{\mathbf{x},V} \mathbf{g}^{M_N(\C),\tr_N}].
	\]
	
	For a function $\mathbf{h}$ on $M_N(\C)_{\sa}^{d+d'} \times (M_N(\C)_{\sa}^{d+d'})^{\ell}$, let
	\[
	L_{\mathbf{x},V^{(N)}} \mathbf{h} = \frac{1}{N^2} \Delta_{\mathbf{x}} \mathbf{h} - \partial_{\mathbf{x}} \mathbf{h} \# \nabla_{\mathbf{x}} V^{(N)}.
	\]
	Because of Lemma \ref{lem:classicaltrace}, we have
	\[
	\frac{1}{N^2} \Delta_{\mathbf{x}} [\mathbf{g}^{M_N(\C),\tr_N}] \rightsquigarrow L_{\mathbf{x}} \mathbf{g}.
	\]
	Similarly, using Lemma \ref{lem:asymptoticcomposition}, we have
	\[
	\partial_{\mathbf{x}} \mathbf{g}^{M_N(\C),\tr_N} \# \nabla_{\mathbf{x}} V^{(N)} \rightsquigarrow \partial_{\mathbf{x}} \mathbf{g} \# \nabla_{\mathbf{x}} V.
	\]
	Thus,
	\[
	L_{\mathbf{x},V^{(N)}} \mathbf{g}^{M_N(\C),\tr_N} \rightsquigarrow L_{\mathbf{x},V} \mathbf{g}.
	\]
	Note that because of integration by parts
	\begin{equation} \label{eq:CEproof2}
		\int_{M_N(\C)_{\sa}^d} L_{\mathbf{x},V^{(N)}} \mathbf{g}^{M_N(\C),\tr_N}(\mathbf{X},\mathbf{X}')\,d\mu^{(N)}(\mathbf{X} | \mathbf{X}') = 0.
	\end{equation}
	
	Fix $R > 0$, and let $R'$ be the radius associated to $R$ as in Lemma \ref{lem:operatornormtailbound}, and let $M = \max(R,R')$. Because of assumption (3), $\nabla V^{(N)}$ is $C$-Lipschitz with respect to $\norm{\cdot}_2$.  Since $\norm{\nabla V^{(N)}(0)}_2$ is bounded as $N \to \infty$, we have
	\[
	\norm{\nabla_{x_j} V^{(N)}(\mathbf{X},\mathbf{X}')}_2 \leq A + B \norm{(\mathbf{X},\mathbf{X}')}_2
	\]
	for some constants $A$ and $B$.  But it follows from \cite[Lemma 11.5.4]{JekelThesis} that
	\[
	\norm{\nabla_{x_j} V^{(N)}(\mathbf{X},\mathbf{X}') - \tr_N(\nabla_{x_j} V^{(N)}(\mathbf{X},\mathbf{X}'))}_\infty \leq B' \norm{(\mathbf{X},\mathbf{X}')}_\infty
	\]
	for some constant $B'$.  Thus, overall,
	\[
	\norm{\nabla_{x_j} V^{(N)}(\mathbf{X},\mathbf{X}') - \tr_N(\nabla_{x_j} V^{(N)}(\mathbf{X},\mathbf{X}'))}_\infty \leq A + (Bd + B') \norm{(\mathbf{X},\mathbf{X}')}_\infty.
	\]
	Moreover, note that $\partial \mathbf{f}_t^{M_N(\C),\tr_N}$ and $(1/N^2) \Delta f^{M_N(\C),\tr_N}$ are uniformly bounded for every $N$ and $(\mathbf{X},\mathbf{X}')$ and $t$ since $\partial \mathbf{f}_t$ and $\partial^2 \mathbf{f}_t$ is uniformly bounded.  Therefore, using Lemma \ref{lem:operatornormtailbound}, we see that
	\[
	\lim_{N \to \infty} \sup_{\norm{\mathbf{X}'} \leq R} \int_{\norm{\mathbf{X}}_\infty \geq M} \norm{(L_{\mathbf{x},V^{(N)}} \mathbf{g}^{M_N(\C),\tr_N}(\mathbf{X},\mathbf{X}') - [L_{\mathbf{x},V} \mathbf{g}]^{M_N(\C),\tr_N}(\mathbf{X},\mathbf{X}')}_{\mathscr{M}^\ell,\tr} \,d\mu_{V^{(N)}}(\mathbf{X} | \mathbf{X}') = 0.
	\]
	Meanwhile, we can estimate the same integral over $\norm{\mathbf{X}}_\infty \leq M$ by using the condition that $L_{\mathbf{x},V^{(N)}} \mathbf{g}^{M_N(\C),\tr_N} \rightsquigarrow L_{\mathbf{x},V} \mathbf{f}_t$, and thus putting the two pieces together,
	\[
	\lim_{N \to \infty} \sup_{\norm{\mathbf{X}'} \leq R} \int \norm{(L_{\mathbf{x},V^{(N)}} \mathbf{g}^{M_N(\C),\tr_N}(\mathbf{X},\mathbf{X}') - [L_{\mathbf{x},V} \mathbf{g}]^{M_N(\C),\tr_N}(\mathbf{X},\mathbf{X}')}_{\mathscr{M}^\ell,\tr} \,d\mu_{V^{(N)}}(\mathbf{X} | \mathbf{X}') = 0.
	\]
	Since $R$ was arbitrary, it follows that
	\[
	\mathbb{E}_{\mathbf{x},V^{(N)}}[L_{\mathbf{x},V^{(N)}} [\mathbf{g}^{M_N(\C),\tr_N}] - [L_{\mathbf{x},V} \mathbf{g}]^{M_N(\C),\tr_N}] \rightsquigarrow 0
	\]
	and thus in light of \eqref{eq:CEproof2}, we have
	\[
	\mathbb{E}_{\mathbf{x},V^{(N)}}[\mathbf{f}^{M_N(\C),\tr_N}] \rightsquigarrow \mathbb{E}_{\mathbf{x},V}[\mathbf{f}].
	\]
	
	For the more general case, suppose that $\mathbf{f}^{(N)} \rightsquigarrow \mathbf{f}$ and that $\mathbf{f}^{(N)}$ satisfies the given operator norm bounds.  Fix $R$ and let $M$ be as above and also let $M' = \max(M, R + 2, C_1)$.  If $\epsilon > 0$, then we may choose some
	\[
	\mathbf{g} \in C_{\tr,\app}^\infty(\R^{*(d+d')}, \mathscr{M}(\R^{*d_1},\dots,\R^{*d_\ell})^{d''} \cap BC_{\tr,\app}(\R^{*(d+d')}, \mathscr{M}(\R^{*d_1},\dots,\R^{*d_\ell})^{d''}
	\]
	with $\norm{\mathbf{g} - \mathbf{f}}_{C_{\tr}(\R^{*(d+d')})^{d_2},M} < \epsilon$ (here $\mathbf{g}$ can be taken to be a trace polynomial composed with a smooth cut-off function in $(\mathbf{X},\mathbf{X}')$).  Then observe that $\norm{\mathbb{E}_{\mathbf{x},V} \mathbf{f} - \mathbb{E}_{\mathbf{x},V} \mathbf{g}}_{C_{\tr}(\R^{*d'},\mathscr{M}^\ell)^{d_2},R} \leq \epsilon$ using Proposition \ref{prop:kernelprojection} and the definition of $M'$.  Moreover,
	\[
	\limsup_{N \to \infty} \sup_{\norm{\mathbf{x}'}_\infty \leq R} \int_{\norm{\mathbf{x}}_\infty \leq M} \norm{\mathbf{f}^{(N)}(\mathbf{X},\mathbf{X}') - \mathbf{g}^{M_N(\C),\tr_N}(\mathbf{X},\mathbf{X}')}_{\mathscr{M}^\ell,\tr} \,d\mu_{V^{(N)}}(\mathbf{X} | \mathbf{X}') \leq \epsilon,
	\]
	while the integral over $\norm{\mathbf{X}}_{\tr_N,\infty} > M$ can be estimated using Lemma \ref{lem:operatornormtailbound}.  Hence,
	\[
	\limsup_{N \to \infty} \norm{\mathbb{E}_{\mathbf{x},V^{(N)}}[\mathbf{f}^{(N)}] - \mathbb{E}_{\mathbf{x},V}[\mathbf{f}] }_{\mathscr{M}^\ell,\tr} \leq 2 \epsilon,
	\]
	and since $R$ and $\epsilon$ were arbitrary, we are done.
\end{proof}

Next, given a potential $V(\mathbf{x},\mathbf{x}')$ in $\mathscr{V}_{C_1,C_2,C_3}^{d+d'}$, we want to describe the ``marginal potential'' $\widehat{V}(\mathbf{x}')$ for the distribution of $\mathbf{x}'$, that is, the function describing the large $N$ limit of the log of the marginal density of $\mu_V^{(N)}$ for $\mathbf{x}'$.  Choose $V^{(N)}$ as in Theorem \ref{thm:classicalconditionalexpectation}.  We can define the marginal potential
\[
\widehat{V}^{(N)}(\mathbf{X}') = - \frac{1}{N^2} \log \int e^{-N^2 V^{(N)}(\mathbf{X},\mathbf{X}')}\,d\mathbf{X}
\]
A straightforward computation shows that
\[
\nabla \widehat{V}^{(N)}(\mathbf{X}') = \mathbb{E}_{\mathbf{x},V^{(N)}}[\nabla_{\mathbf{x}'} V^{(N)}].
\]
Now it follows from the previous theorem that
\[
\nabla \widehat{V}^{(N)} \rightsquigarrow \mathbb{E}_{\mathbf{x},V}[\nabla_{\mathbf{x}'} V].
\]
Our next goal is to show that $\mathbb{E}_{\mathbf{x},V}[\nabla_{\mathbf{x}'} V]$ is the gradient of some function $\widehat{V} \in \tr(C_{\tr,\app}^\infty(\R^{*d}))$.  To this end, we use the following lemma.

\begin{lemma} \label{lem:asymptoticgradient}
	Let $\mathbf{g} \in C_{\tr}^k(\R^{*d})_{\sa}^d$.  If there exist $C^1$ functions $f^{(N)}: M_N(\C)_{\sa}^d \to \R$ such that $\nabla f^{(N)} \rightsquigarrow \mathbf{g}$, then there exists $f \in \tr(C_{\tr,\app}^{k+1}(\R^{*d}))_{\sa}$ such that $\nabla f = \mathbf{g}$.  This $f$ is unique up to an additive constant.  It also satisfies $f^{(N)} - f^{(N)}(0) \rightsquigarrow f - f(0)$.
\end{lemma}

\begin{proof}
	We may define a function $\mathbf{h}(\mathbf{x}_1,\mathbf{x}_2,\mathbf{x}_3)$ in $\tr(C_{\tr,\app}(\R^{*3d}))$ by
	\[
	\mathbf{h}(\mathbf{x}_1,\mathbf{x}_2,\mathbf{x}_3) = \sum_{j=1}^3 \int_0^1 \ip{\mathbf{g}(t \mathbf{x}_j + (1 - t)\mathbf{x}_{j+1}), \mathbf{x}_j - \mathbf{x}_{j+1}}_{\tr} \,dt,
	\]
	where the index $j + 1$ is reduced modulo $3$.  The function $\mathbf{h}$ is intuitively the path integral of $\mathbf{g}$ around a triangle with vertices $\mathbf{x}_1$, $\mathbf{x}_2$, $\mathbf{x}_3$.  Here $\mathbf{x}_1$, $\mathbf{x}_2$, $\mathbf{x}_3$ are formal variables, and thus $\ip{\mathbf{g}(t \mathbf{x}_j + (1 - t)\mathbf{x}_{j+1}), \mathbf{x}_j - \mathbf{x}_{j+1}}$ is an element of $\tr(C_{\tr,\app}(\R^{*3d}))$.  Moreover, it depends continuously on $t$ in this space by continuity of composition.  It follows that the Riemann integral of these functions is defined.
	
	Next, let
	\[
	\mathbf{h}^{(N)}(\mathbf{X}_1,\mathbf{X}_2,\mathbf{X}_3) = \sum_{j=1}^3 \int_0^1 \ip{\nabla f^{(N)} (t \mathbf{X}_j + (1 - t)\mathbf{X}_{j+1}), \mathbf{X}_j - \mathbf{X}_{j+1}}_{\tr_N} \,dt,
	\]
	where $\mathbf{X}_1$, $\mathbf{X}_2$, $\mathbf{X}_3$ represent elements of $M_N(\C)_{\sa}^d$.  It is straightforward to show that since $\nabla f^{(N)} \rightsquigarrow \mathbf{g}$, we have $\mathbf{h}^{(N)} \rightsquigarrow \mathbf{h}$.  But because $\nabla f^{(N)}$ is a gradient, we have $\mathbf{h}^{(N)} = 0$.  Therefore, $\mathbf{h} = 0$.
	
	Define
	\[
	f(\mathbf{x}) = \int_0^1 \ip{\mathbf{g}(t\mathbf{x}), \mathbf{x}}_{\tr} \,dt.
	\]
	Given that $\mathbf{h} = 0$, we have for any $(\cA,\tau)$ and any $\mathbf{X}_1$, $\mathbf{X}_2$, $\mathbf{X}_3 \in \cA_{\sa}^d$ that
	\[
	0 = \mathbf{h}^{\cA,\tau}(0,\mathbf{X}_1,\mathbf{X}_2) = f^{\cA,\tau}(\mathbf{X}_2) - f^{\cA,\tau}(\mathbf{X}_1) + \int_0^1 \ip{\mathbf{g}^{\cA,\tau}(t \mathbf{X}_1 + (1 - t)\mathbf{X}_2), \mathbf{X}_1 - \mathbf{X}_2}_{\tau} \,dt.
	\]
	It follows easily that $\nabla f = \mathbf{g}$.
	
	Moreover, $f$ is unique up to an additive constant because $f^{\cA,\tau}(\mathbf{X}) - f^{\cA,\tau}(0)$ can be evaluated by integrating the $\nabla f^{\cA,\tau}$ along the path from $0$ to $\mathbf{X}$.  Similarly, since $f^{(N)}(\mathbf{X}) - f^{(N)}(0) = \int_0^1 \ip{\nabla f^{(N)}(t\mathbf{X}), \mathbf{X}}_{\tr_N}\,dt$, we obtain $f^{(N)} - f^{(N)}(0) \rightsquigarrow f - f(0)$.
	
	Finally, observe that if $\mathbf{g} = \nabla f \in C_{\tr,\app}^k(\R^{*d})^d$, then $f \in \tr(C_{\tr,\app}^{k+1}(\R^{*d}))$
\end{proof}

\begin{proposition} \label{prop:marginalpotential}
	Let $V \in \mathscr{V}_{d+d',C_1,C_2,C_3}$ for some $C_2 < 1$.  Then there exists $\widehat{V} \in \tr(C_{\tr,\app}^\infty(\R^{*d'}))_{\sa}$, unique up to an additive constant, such that
	\[
	\nabla \widehat{V} = \mathbb{E}_{\mathbf{x},V}[\nabla_{\mathbf{x}'} V].
	\]
	Furthermore, we have $\widehat{V} \in \mathscr{V}_{d',C_1',C_2',C_3'}$ for some constants $C_1'$, $C_2'$, and $C_3'$ depending only on $C_1$, $C_2$, and $C_3$, where specifically
	\begin{align*}
		C_1' &= C_1 & 
		C_2' &= \frac{C_2(1 + C_2)}{1 - C_2}
	\end{align*}
\end{proposition}

\begin{proof}
	Let $V^{(N)} = V^{M_N(\C),\tr_N}$, so that $\nabla V^{(N)} \rightsquigarrow \nabla V$.  By Theorem \ref{thm:classicalconditionalexpectation} and Remark \ref{rem:canonicalapproximant}, we have
	\[
	\mathbb{E}_{\mathbf{x},V^{(N)}}[\nabla_{\mathbf{x}'} V^{(N)}] \rightsquigarrow \mathbb{E}_{\mathbf{x},V}[\nabla_{\mathbf{x}'} V]
	\]
	We know that $\mathbb{E}_{\mathbf{x},V^{(N)}}[\nabla_{\mathbf{x}'} V^{(N)}] = \nabla \widehat{V}^{(N)}$ for the function $\widehat{V}^{(N)}$ discussed above.  Hence, by Lemma \ref{lem:asymptoticgradient}, there exists $\widehat{V} \in C_{\tr}^\infty$ with $\nabla \widehat{V} = \mathbb{E}_{\mathbf{x},V}[\nabla_{\mathbf{x}'} V]$, which is unique up to an additive constant.
	
	Next, we must show that $\widehat{V} \in \mathscr{V}_{d',C_1',C_2',C_3'}$.  Let $W = V - (1/2) \ip{\mathbf{x},\mathbf{x}}_{\tr} - (1/2) \ip{\mathbf{x}',\mathbf{x}'}_{\tr}$ and $\widehat{W} = V - (1/2) \ip{\mathbf{x}',\mathbf{x}'}_{\tr}$.  Note that $\nabla_{\mathbf{x}'} V(\mathbf{x},\mathbf{x}') = \mathbf{x}' + \nabla_{\mathbf{x}'} W(\mathbf{x},\mathbf{x}')$ and $\nabla \widehat{V}(\mathbf{x}') = \mathbf{x}' + \nabla \widehat{W}(\mathbf{x}')$.  Thus, since $\mathbb{E}_{\mathbf{x},V}[\mathbf{x}'] = \mathbf{x}'$, we have $\nabla \widehat{W} = \mathbb{E}_{\mathbf{x},V}[\nabla_{\mathbf{x}'} W]$.
	
	Now recall that $e^{tL_{\mathbf{x},V}} \mathbf{f}$ is obtained as a conditional expectation of the function $\mathbf{f}(\mathcal{X},\pi')$, and hence
	\[
	\norm{e^{tL_{\mathbf{x},V}} \nabla_{\mathbf{x}'} W}_{BC_{\tr,\app}(\R^{*(d+d')})^{d'}} \leq \norm{\nabla_{\mathbf{x}'} W}_{BC_{\tr,\app}(\R^{*(d+d')})^{d'}}.
	\]
	Taking $t \to \infty$, we get $\norm{\nabla \widehat{W}}_{BC_{\tr,\app}(\R^{*(d+d')})^{d'}} \leq C_1$.
	
	Next, recall that the process $\mathcal{X}$ from \S \ref{sec:pseudoinverse} satisfies
	\[
	\partial_{\mathbf{x}'} \mathcal{X}(\cdot,t) = \int_0^t [\partial_{\mathbf{x}} \nabla_{\mathbf{x}} V(\mathcal{X}(\cdot,u),\pi') \# \partial_{\mathbf{x}'} \mathcal{X} + \partial_{\mathbf{x}'} \nabla_{\mathbf{x}} V(\mathcal{X}(\cdot,u),\pi')] \,du.
	\]
	In the proof of the base case of Lemma \ref{lem:processCinfinity}, we applied Lemma \ref{lem:niceGronwall} to get a bound for this function.  The $c$ from that proof is here $1 - C_2$ and the constant $C_{1,\mathbf{J}}' = C_{1,\nabla_{\mathbf{x}} W}'$ is here $C_2$.  Thus,
	\[
	\norm{\partial_{\mathbf{x}'} \mathcal{X}(\cdot,t)}_{BC_{\tr,\mathcal{S}}(\R^{*(d+d')},\mathscr{M}^1)} \leq e^{-(1 - C_2)t} \left( 1 + \frac{2C_2}{1 - C_2} (e^{(1 - C_2)t} - 1)\right).
	\]
	It follows as in the proof of Lemma \ref{lem:heatsemigroupbounds} that
	\begin{multline*}
		\norm{\partial_{\mathbf{x}'} e^{tL_{\mathbf{x},V}}\nabla_{\mathbf{x}'} W}_{BC_{\tr}(\R^{*(d+d')},\mathscr{M}^1)} \\
		\leq e^{-(1 - C_2)t/2} \left( 1 + \frac{2C_2}{1 - C_2} (e^{(1 - C_2)t/2} - 1)\right) \norm{\partial_{\mathbf{x}} \nabla_{\mathbf{x}'} W}_{BC_{\tr}(\R^{*(d+d')},\mathscr{M}^1)} + \norm{\partial_{\mathbf{x}'} \nabla_{\mathbf{x}'} W}_{BC_{\tr}(\R^{*(d+d')},\mathscr{M}^1)}.
	\end{multline*}
	Taking $t \to \infty$, we obtain
	\[
	\norm{\partial_{\mathbf{x}'} \mathbb{E}_{\mathbf{x},V} \nabla_{\mathbf{x}'} W}_{BC_{\tr}(\R^{*d'},\mathscr{M}^1)} \leq \frac{2C_2}{1 - C_2} \cdot C_2 + C_2 = \frac{C_2(1 + C_2)}{1 - C_2}.
	\]
	The existence of $C_3'$ follows by similar reasoning, which we leave as an exercise.
\end{proof}

\begin{proposition} \label{prop:iteratedCE}
	Consider variables $\mathbf{x}$, $\mathbf{x}'$, $\mathbf{x}''$ which are $d$, $d'$, and $d''$-tuples respectively.  Let $V \in \mathscr{V}_{d+d'+d'',C_1,C_2,C_3}$ for some $C_2 < \sqrt{2} - 1$.  Let $\widehat{V}$ be the marginal potential for $(\mathbf{x}',\mathbf{x}'')$.  Then
	\[
	\mathbb{E}_{\mathbf{x}',\widehat{V}} \circ  \mathbb{E}_{\mathbf{x},V}[\mathbf{f}] = \mathbb{E}_{(\mathbf{x},\mathbf{x}'),V}[\mathbf{f}]
	\]
	for $\mathbf{f} \in C_{\tr,\app}(\R^{*(d+d'+d'')},\mathscr{M}(\R^{*d_1},\dots,\R^{*d_\ell}))^{d^*}$.
\end{proposition}

	\begin{proof}
	By Proposition \ref{prop:Laplaciancontinuity}, it suffices to prove the relation for $\mathbf{f}$ in a dense subset of
	\[
	C_{\tr,\app}(\R^{*(d+d'+d'')}, \mathscr{M}(\R^{*d_1}, \dots, \R^{*d_\ell}))^{d^*}.
	\]
	In particular, we may restrict our attention to bounded $\mathbf{f}$.
	
	Let $V^{(N)} = V^{M_N(\C),\tr_N}$ which satisfies the assumptions of Theorem \ref{thm:classicalconditionalexpectation} with $c = 1 - C_2$ and $C = 1 + C_2$.  Let $\widehat{V}^{(N)}$ be the marginal potential for $(\mathbf{X}',\mathbf{X}'')$, which satisfies
	\[
	\nabla \widehat{V}^{(N)} = \mathbb{E}_{\mathbf{x},V^{(N)}}[\nabla_{\mathbf{x}',\mathbf{x}''} V^{(N)}].
	\]
	By Theorem \ref{thm:classicalconditionalexpectation},
	\[
	\nabla \widehat{V}^{(N)} \rightsquigarrow \nabla \widehat{V}.
	\]
	By Proposition \ref{prop:marginalpotential}, $\widehat{V} \in \mathscr{V}_{d'+d'',C_1',C_2',C_3'}$ with $C_2' = C_2(1 + C_2) / (1 - C_2)$.  Note that $C_2' < 1$ provided that $C_2 < \sqrt{2} - 1$.
	
	It follows from the work of Brascamp and Lieb \cite[Theorem 4.3]{BL1976} that $\widehat{V}^{(N)}(\mathbf{X}) - (c/2) \norm{\mathbf{X}}_2^2$ is convex and $\widehat{V}^{(N)}(\mathbf{X}) - (C/2) \norm{\mathbf{X}}_2^2$ is concave for the same constants $c$ and $C$ that worked for $V^{(N)}$.  (Note equation (4.18) of \cite{BL1976} should read $D = A - BC^{-1}B^*$.  Of course, if the block $2 \times 2$ matrix is a constant multiple of the identity, then the Schur complement matrix $D$ is the same scalar multiple of the appropriately sized identity matrix.)  Overall, we conclude that $\widehat{V}^{(N)}$ and $\widehat{V}$ also satisfy the hypotheses of Theorem \ref{thm:classicalconditionalexpectation}.
	
	Now let $\mathbf{f}^{(N)} = \mathbf{f}^{M_N(\C),\tr_N}$.  Then by Theorem \ref{thm:classicalconditionalexpectation} applied to $V$ and $V^{(N)}$, we have
	\[
	\mathbb{E}_{\mathbf{x},V^{(N)}}[\mathbf{f}^{(N)}] \rightsquigarrow \mathbb{E}_{\mathbf{x},V}[\mathbf{f}].
	\]
	Note that these functions are uniformly bounded because we assumed $\mathbf{f}$ was bounded.  By Theorem \ref{thm:classicalconditionalexpectation} applied to $\widehat{V}$ and $\widehat{V}^{(N)}$, we have
	\[
	\mathbb{E}_{\mathbf{x},\widehat{V}^{(N)}} \circ \mathbb{E}_{\mathbf{x}',V^{(N)}}[\mathbf{f}^{(N)}] \rightsquigarrow \mathbb{E}_{\mathbf{x},\widehat{V}} \circ \mathbb{E}_{\mathbf{x},V}[\mathbf{f}].
	\]
	From the well-known properties of classical conditional expectations,
	\[
	\mathbb{E}_{\mathbf{x},\widehat{V}^{(N)}} \circ \mathbb{E}_{\mathbf{x}',V^{(N)}}[\mathbf{f}^{(N)}] = \mathbb{E}_{(\mathbf{x},\mathbf{x}'),V^{(N)}}[\mathbf{f}^{(N)}].
	\]
	By another application of Theorem \ref{thm:classicalconditionalexpectation},
	\[
	\mathbb{E}_{(\mathbf{x},\mathbf{x}'),V^{(N)}}[\mathbf{f}^{(N)}] \rightsquigarrow \mathbb{E}_{(\mathbf{x},\mathbf{x}'),V}[\mathbf{f}].
	\]
	Therefore, $\mathbb{E}_{\mathbf{x},\widehat{V}} \circ \mathbb{E}_{\mathbf{x},V}[\mathbf{f}] = \mathbb{E}_{(\mathbf{x},\mathbf{x}'),V}[\mathbf{f}]$ as desired.
\end{proof}

As a corollary, in the situation of the previous proposition, we will get the same answer for the marginal potential for $\mathbf{x}''$ whether we compute it from $V$ or from $\widehat{V}$.  There is a variant of the previous proposition that does not explicitly refer to $\widehat{V}$ and hence works whenever $C_2 < 1$.

\begin{proposition} \label{prop:iteratedCE2}
	Consider variables $\mathbf{x}$, $\mathbf{x}'$, $\mathbf{x}''$ which are $d$, $d'$, and $d''$-tuples respectively.  Fix $\ell \geq 0$ and $d_1$, \dots, $d_\ell \in \N$.  Let $\iota$ be the canonical inclusion map
	\[
	\iota: C_{\tr,\app}(\R^{*(d'+d'')},\mathscr{M}(\R^{*d_1},\dots,\R^{*d_\ell})) \to C_{\tr,\app}(\R^{*(d'+d'')},\mathscr{M}(\R^{*d_1},\dots,\R^{*d_\ell})),
	\]
	obtained by viewing a function of $(\mathbf{x}',\mathbf{x}'')$ as a function of $(\mathbf{x},\mathbf{x}',\mathbf{x}'')$.  Let $V \in \mathscr{V}_{d+d'+d'',C_1,C_2,C_3}$ for some $C_2 < 1$. Then
	\[
	\mathbb{E}_{(\mathbf{x},\mathbf{x}'),V} \circ \iota \circ \mathbb{E}_{\mathbf{x},V}[\mathbf{f}] = \mathbb{E}_{(\mathbf{x},\mathbf{x}'),V}[\mathbf{f}]
	\]
	for $\mathbf{f} \in C_{\tr,\app}(\R^{*(d+d'+d'')},\mathscr{M}(\R^{(*d_1},\dots,\R^{*d_\ell}))^{d_2}$.
\end{proposition}

The proof of the proposition is similar to the previous one.  Use the fact that the analogous result holds for the classical conditional expectation maps associated to $V^{(N)}$ and then take the large $N$ limit using Theorem \ref{thm:classicalconditionalexpectation}.  We leave the details to the reader.

The next proposition relates the map $\mathbb{E}_{\mathbf{x},V}$ to $\mathrm{W}^*$-algebraic conditional expectations.  This result is similar to \cite[Theorem 15.1.7]{JekelThesis}.  The only difference is that we have a smaller space of non-commutative functions, and hence we are able to make conclusions about the $\mathrm{C}^*$-algebras, not only the $\mathrm{W}^*$-algebras.

\begin{proposition}
	Let $V \in \mathscr{V}_{d+d',C_1,C_2,C_3}$ where $C_2 < \sqrt{2} - 1$.  Let $(\cA,\tau)$ be a tracial $\mathrm{W}^*$-algebra with self-adjoint generators $(\mathbf{X},\mathbf{X}')$ satisfying
	\[
	\tau(f^{\cA,\tau}(\mathbf{X},\mathbf{X}')) = \mathbb{E}_V[\tr(f(\mathbf{X},\mathbf{X}'))] \text{ for } f \in C_{\tr,\app}(\R^{*(d+d')}).
	\]
	Then we have
	\[
	E_{\mathrm{W}^*(\mathbf{X}')}[f^{\cA,\tau}(\mathbf{X},\mathbf{X}')] = (\mathbb{E}_{\mathbf{x},V}[f])^{\cA,\tau}(\mathbf{X}') \text{ for } f \in C_{\tr,\app}(\R^{*(d+d')}),
	\]
	where $\mathrm{W}^*(\mathbf{X}')$ is the $\mathrm{W}^*$-subalgebra of $\cA$ generated by $\mathbf{X}'$ and $E_{\mathrm{W}*(\mathbf{X}')}: \cA \to \mathrm{W}^*(\mathbf{X}')$ is the unique trace-preserving conditional expectation.  Furthermore, $E_{\mathrm{W}^*(\mathbf{X}')}$ maps $\mathrm{C}^*(\mathbf{X},\mathbf{X}')$ into $\mathrm{C}^*(\mathbf{X}')$.
\end{proposition}

\begin{proof}
	Let $f \in C_{\tr}(\R^{*(d+d')})$ and $g \in C_{\tr}(\R^{*d'})$.  Then using \ref{eq:bimodulemaps} and Proposition \ref{prop:iteratedCE2},
	\begin{align*}
		\tau[(\mathbb{E}_{\mathbf{x},V}[f])^{\cA,\tau}(\mathbf{X}') g^{\cA,\tau}(\mathbf{X}')] &= \tau[(\mathbb{E}_{\mathbf{x},V}[f] g)^{\cA,\tau}(\mathbf{X}')] \\
		&= \tau[\mathbb{E}_{\mathbf{x},V}[f \cdot (g \circ \pi')]^{\cA,\tau}(\mathbf{X}')] \\
		&= \mathbb{E}_V[\mathbb{E}_{\mathbf{x},V}[f \cdot (g \circ \pi')] \circ \pi'] \\
		&= \mathbb{E}_V[f \cdot (g \circ \pi')] \\
		&= \tau(f^{\cA,\tau}(\mathbf{X},\mathbf{X}') g^{\cA,\tau}(\mathbf{X}')).
	\end{align*}
	Since this holds for all $g$, it holds in particular for non-commutative polynomials.  Non-commutative polynomials in $\mathbf{X}'$ are dense in $\mathrm{W}^*(\mathbf{X}')$ with respect to the weak operator topology.  Thus, the above relation shows that $(\mathbb{E}_{\mathbf{x},V}[f])^{\cA,\tau}(\mathbf{X}')$ equals the conditional expectation of $f^{\cA,\tau}(\mathbf{X},\mathbf{X}')$ onto $\mathrm{W}^*(\mathbf{X}')$.
	
	Because $\mathbf{E}_{\mathbf{x},V}[f] \in C_{\tr}(\R^{*d'})$, the operator $\mathbf{E}_{\mathbf{x},V}[f]^{\cA,\tau}(\mathbf{X}')$ is in $\mathrm{C}^*(\mathbf{X}')$.  Hence, $E_{\mathrm{W}^*(\mathbf{X}')}[f^{\cA,\tau}(\mathbf{X},\mathbf{X}')] \in \mathrm{C}^*(\mathbf{X}')$ whenever $f \in C_{\tr}(\R^{*(d+d')})$.  But elements of the form $f^{\cA,\tau}(\mathbf{X},\mathbf{X}')$ are dense in $\mathrm{C}^*(\mathbf{X},\mathbf{X}')$, and therefore, $E_{\mathrm{W}^*(\mathbf{X}')}$ maps $\mathrm{C}^*(\mathbf{X},\mathbf{X}')$ into $\mathrm{C}^*(\mathbf{X}')$.
\end{proof}

\subsection{Triangular transport} \label{subsec:triangulartransport}

In this section, we will prove a triangular transport result similar to \cite[Theorem 8.11]{JekelExpectation}.  However, in both the hypotheses and conclusion we will use $C_{\tr}^\infty$ functions rather than $\norm{\cdot}_2$-Lipschitz functions, and thus our new result yields triangular isomorphisms of the $\mathrm{C}^*$-algebras generated by our non-commutative random variables, not only the $\mathrm{W}^*$-algebras.  Moreover, our current result constructs triangular transport at the infinitesimal level and thus allows us to construct a family of transport maps along any path of potentials $V_t$ that are sufficiently close to the quadratic, whereas \cite{JekelExpectation} performed the transport one variable at a time and at each stage only used a path obtained by freely convolving the distribution with a freely independent semicircular family.

\begin{definition}
	For $j \leq d$, let $\iota_{j,d}: C_{\tr,\app}(\R^{*j}) \to C_{\tr,\app}(\R^{*d})$ be the canonical inclusion $\iota_{j,d}(f)(x_1,\dots,x_d) = f(x_1,\dots,x_j)$.  A function $\mathbf{f} = (f_1,\dots,f_d) \in C_{\tr,\app}(\R^{*d})_{\sa}^d$ is said to be \emph{lower-triangular} if $f_j \in \iota_{j,d}(C_{\tr,\app}(\R^{*d}))$ for every $j = 1$, \dots, $d$, or in other words $f_j$ is a function of $x_1$, \dots, $x_j$ alone.
\end{definition}

\begin{theorem} \label{thm:triangulartransport}
	Fix $C_1$, $C_2$, $C_3$ with $C_2 < \sqrt{2} - 1$, and let $t \mapsto V_t$ be a continuously differentiable path $[0,T] \to \mathscr{V}_{C_1,C_2,C_3}^d$ (where differentiation again occurs with respect to the topology on $\tr(C_{\tr}^\infty(\R^{*d}))_{\sa}$), and assume that
	\[
	\norm{ \partial^k \dot{V} }_{BC_{\tr}(\R^{*d},\mathscr{M}^k)} \leq C_k' \text{ for } k = 1, 2.
	\]
	Then there exists a family of triangular functions $(\mathbf{f}_{t,s})_{s, t \in [0,T]}$ in $C_{\tr}^\infty(\R^{*d})_{\sa}^d$ such that $\mathbf{f}_{u,t} \circ \mathbf{f}_{t,s} = \mathbf{f}_{u,s}$ for $s, t, u \in [0,T]$ and $(\mathbf{f}_{t,s})_* V_s = V_t$ for $s, t \in [0,T]$.
\end{theorem}

Similar to the proof of Theorem \ref{thm:transport}, we rely on Lemma \ref{lem:transport0}, and thus we will first construct a triangular function $\mathbf{h}$ satisfying $L_V^* \mathbf{h} = \phi$ for a given $V$ and $\phi$.

\begin{lemma} \label{lem:triangularsolution}
	Fix $C_1$, $C_2$, $C_3$ with $C_2 < \sqrt{2} - 1$.  Then for $V \in \mathscr{V}_{C_1,C_2,C_3}^d$, there exists a linear operator $T_V: \tr(C_{\tr,\app}^\infty(\R^{*d}))_{\sa} \to C_{\tr,\app}^\infty(\R^{*d})_{\sa}^d$ such that the following conditions hold:
	\begin{enumerate}[(1)]
		\item $T_V \phi$ is a lower-triangular for every $\phi$.
		\item $\nabla_V^* T_V \phi = \phi - \mathbb{E}_V(\phi)$.
		\item We have
		\begin{multline*}
			\norm{T_V \phi}_{BC_{\tr,\app}(\R^{*d})^d} + \norm{\partial T_V \phi}_{BC_{\tr,\app}(\R^{*d},\mathscr{M}^1)^d} \\
			\leq \text{\rm constant}(C_1,C_2,C_3,d)\left( \norm{\partial \phi}_{BC_{\tr,\app}(\R^{*d},\mathscr{M}^1)} + \norm{\partial^2 \phi}_{BC_{\tr,\app}(\R^{*d},\mathscr{M}^1)} \right).
		\end{multline*}
		\item We have continuity of the map
		\[
		\mathscr{V}_{C_1,C_2,C_3}^d \times \tr(C_{\tr,\app}^\infty(\R^{*d}))_{\sa} \to  C_{\tr,\app}^\infty(\R^{*d})_{\sa}^d : (V,\phi) \mapsto T_V \phi.
		\]
	\end{enumerate}
\end{lemma}

\begin{proof}
	Let $V_j$ be the marginal potential on the variables $x_1$, \dots, $x_j$ obtained from $V$ given by
	\[
	\nabla V_j = \mathbb{E}_{x_{j+1},\dots,x_d,V}[\nabla_{x_1,\dots,x_j} V],
	\]
	with the normalization $V_j(0) = 0$.  Note that
	\[
	V_j \in \mathscr{V}_{C_1',C_2',C_3'}^j
	\]
	with $C_2' = 2C_2^2 / (1 - C_2) < 1$ since $C_2 < 1/2$.  Therefore, the pseudoinverse operators $\Psi_{V_j}$ are well-defined by Proposition \ref{prop:pseudoinverse}.
	
	To simplify notation, we will view $C_{\tr}(\R^{*j})$ as a subset of $C_{\tr}(\R^{*d})$ using the canonical inclusion $\iota_{j,d}$.  Given $\phi \in \tr( C_{\tr}(\R^{*d}))_{\sa}$, we define functions $h_j \in C_{\tr}(\R^{*j})_{\sa}$ inductively by
	\begin{equation} \label{eq:triangularsolution}
		h_j = \nabla_{x_j} \Psi_{x_j,V_j} \left( \mathbb{E}_{x_{j+1},\dots,x_d,V}(\phi) - \sum_{i=1}^{j-1} \partial_{x_i} V_j \# h_i \right).
	\end{equation}
	It makes sense to apply $\Psi_{x_j,V_j}$ to $\mathbb{E}_{x_{j+1},\dots,x_d,V}(\phi) - \sum_{i=1}^{j-1} \nabla_{x_i,V_j}^* h_i$ since the latter is a function of $x_1$, \dots, $x_j$.  We set $T_V \phi = (h_1,\dots,h_d)$.  Clearly, $T_V$ is a linear operator and satisfies (1) by construction, and now we shall check that it has the other desired properties.
	
	(2) Observe that
	\begin{align*}
		\nabla_{x_j,V}^* h_j &= \partial_{x_j} V \# h_j - \Div_{x_j} h_j \\
		&= \partial_{x_j}(V - V_j) \# h_j + \partial_{x_j} V_j \# h_j - \Div_{x_j} h_j \\
		&= \partial_{x_j}(V - V_j) \# h_j + \nabla_{x_j,V_j}^* h_j.
	\end{align*}
	Meanwhile,
	\begin{align*}
		\nabla_{x_j,V}^* h_j &= \nabla_{x_j,V}^* \nabla_{x_j} \Psi_{x_j,V_j} \left( \mathbb{E}_{x_{j+1},\dots,x_d,V}(\phi) - \sum_{i=1}^{j-1} \partial_{x_i} V_j \# h_i \right) \\
		&= (1 - \E_{x_j,V_j}) \left[ \mathbb{E}_{x_{j+1},\dots,x_d,V}(\phi) - \sum_{i=1}^{j-1} \partial_{x_i} V_j \# h_i \right] \\
		&= \mathbb{E}_{x_{j+1},\dots,x_d,V}(\phi) - \mathbb{E}_{x_j,\dots,x_d,V}(\phi) - \sum_{i=1}^{j-1} \partial_{x_i} (V_j - V_{j-1}) \# h_i,
	\end{align*}
	where we have observed that for $i \leq j - 1$,
	\[
	\mathbb{E}_{x_j,V_j}[\partial_{x_i}V_j \# h_i] = \mathbb{E}_{x_j,V_j}[\partial_{x_i}V_j] \# h_i = \partial_{x_i} V_{j-1} \# h_i,
	\]
	since $h_i$ does not depend on $x_j$.  Therefore,
	\begin{align*}
		\nabla_V^* \mathbf{h} &= \sum_{j=1}^d \nabla_{x_j,V}^* h_j \\
		&= \sum_{j=1}^d \partial_{x_j}(V - V_j) \# h_j - \sum_{j=1}^d \sum_{i=1}^{j-1} \partial_{x_i}(V_j - V_{j-1}) \# h_i + \sum_{j=1}^d \left( \mathbb{E}_{x_{j+1},\dots,x_d,V}(\phi) - \mathbb{E}_{x_j,\dots,x_d,V}(\phi) \right) \\
		&= \sum_{j=1}^d \partial_{x_j}(V - V_j) \# h_j - \sum_{i=1}^{d-1} \sum_{j=i+1}^d \partial_{x_i}(V_j - V_{j-1}) \# h_i + \bigl( \phi - \mathbb{E}_V(\phi) \bigr) \\
		&= \sum_{j=1}^d \partial_{x_j}(V - V_j) \# h_j - \sum_{i=1}^{d-1} \partial_{x_i}(V - V_i) \# h_i + \bigl( \phi - \mathbb{E}_V(\phi) \bigr) \\
		&= \phi - \mathbb{E}_V(\phi).
	\end{align*}
	
	(3) Because $V_j \in \mathscr{V}_{C_1',C_2',C_3'}^j$, it follows from Proposition \ref{prop:pseudoinverse} and Remark \ref{rem:boundedcoefficients2} that for some constants $K_1$, $K_2$, $K_3$ depending only on $C_1$, $C_2$, $C_3$, we have
	\begin{align*}
		& \quad \sum_{k=0}^1 \norm{\partial^k h_j}_{BC_{\tr,\app}(\R^{*d},\mathscr{M}^k)} \\
		&= \sum_{k=0}^1 \norm*{ \partial^k \partial_{x_j} \Psi_{x_j,V_j} \left( \mathbb{E}_{x_{j+1},\dots,x_d,V}(\phi) - \sum_{i=1}^{j-1} \partial_{x_i} V_j \# h_i \right) }_{BC_{\tr,\app}(\R^{*d},\mathscr{M}^k)} \\
		&\leq K_1 \sum_{k=0}^1 \norm*{ \partial^k \partial_{x_j} \mathbb{E}_{x_{j+1},\dots,x_d,V}(\phi)}_{BC_{\tr,\app}(\R^{*d},\mathscr{M}^{1+k})} + K_1 \sum_{i=1}^{j-1} \sum_{k=0}^1 \norm*{ \partial \partial_{x_j}[\partial_{x_i} V_j \# h_i]}_{BC_{\tr,\app}(\R^{*d},\mathscr{M}^{1+k})} \\
		&\leq K_2 \sum_{k=0}^1 \norm*{ \partial^{1+k} \phi}_{BC_{\tr,\app}(\R^{*d},\mathscr{M}^{1+k})} + K_1 \sum_{i=1}^{j-1} \sum_{k=0}^1 \norm*{ \partial [\partial_{x_j}\partial_{x_i} V_j \# h_i]}_{BC_{\tr,\app}(\R^{*d},\mathscr{M}^{1+k})} \\
		&\leq K_2 \sum_{k=0}^1 \norm*{ \partial^{1+k} \phi}_{BC_{\tr,\app}(\R^{*d},\mathscr{M}^{1+k})} + K_3 \sum_{i=1}^{j-1} \sum_{k=0}^1 \norm*{ \partial^k h_i}_{BC_{\tr,\app}(\R^{*d},\mathscr{M}^{1+k})}
	\end{align*}
	where we have used Proposition \ref{prop:kernelprojection} and the fact that $\partial_{x_j} h_i = 0$.  Based on this inequality, it is easy to check by induction that each $h_j$ satisfies the desired bounds.
	
	(4) By Proposition \ref{prop:Laplaciancontinuity}, $V_j$ depends continuously on $V$.  Similarly, by applying Proposition \ref{prop:Laplaciancontinuity} to each part of \eqref{eq:triangularsolution}, we see by induction that $h_j$ depends continuously on $(V,\phi)$.
\end{proof}

\begin{proof}[Proof of Theorem \ref{thm:triangulartransport}]
	Let $\mathbf{h}_t = -T_{V_t} \dot{V}_t$, where $T_{V_t}$ is as in the previous lemma.  The lemma implies that $t \mapsto \mathbf{h}_t$ is continuous and $\partial \mathbf{h}_t$ is bounded.  Thus, Lemma \ref{lem:flow} shows that there are functions $\mathbf{f}_{t,s}$ satisfying
	\[
	\mathbf{f}_{t,s} = \id + \int_s^t \mathbf{h}_u \circ \mathbf{f}_{u,s}\,du.
	\]
	Because $\mathbf{h}_t$ is lower-triangular, so is $\mathbf{f}_{t,s}$ (for instance because the Picard iterates are lower-triangular).  From basic results on ODE, the functions satisfy the asserted properties under composition.  Finally, by Lemma \ref{lem:transport0}, since $-\nabla_{V_t}^* \mathbf{h}_t = \dot{V}_t$ modulo constants, we have $(\mathbf{f}_{t,s})_* V_s = V_t$ modulo constants for every $s, t \in [0,T]$.
\end{proof}

The operator algebraic consequences of this theorem are similar to Observation \ref{obs:transportapplication} and Corollary \ref{cor:transportapplication}.

\begin{corollary} \label{cor:triangulartransport}
	Let $V \in \mathscr{V}_{C_1,C_2,C_3}^d$ with $C_2 < \sqrt{2} - 1$, and let $\mathbf{X}$ be a $d$-tuple of non-commutative random variables that generate a tracial $\mathrm{W}^*$-algebra $(\cA,\tau)$ such that
	\[
	\mathbb{E}_V[f] = \tau(f(\mathbf{X})).
	\]
	Let $\mathbf{S}$ be a standard free semicircular $d$-tuple that generates the tracial $\mathrm{W}^*$-algebra $(\cB,\sigma) \cong L(\mathbb{F}_d)$.  Then there exists a tracial $\mathrm{W}^*$-isomorphism $\phi: (\cA,\tau) \to (\cB,\sigma)$ such that
	\[
	\phi(\mathrm{C}^*(X_1,\dots,X_j)) = \mathrm{C}^*(S_1,\dots,S_j)) \text{ for } j = 1, \dots, d.
	\]
	In particular, for each $j = 1$, \dots, $d$, $\mathrm{C}^*(X_1,\dots,X_d)$ is the internal reduced free product of $\mathrm{C}^*(X_1,\dots,X_j)$ and $\mathrm{C}^*(\phi^{-1}(S_{j+1}),\dots,\phi^{-1}(S_d))$.
\end{corollary}

\section{Equations on the free Wasserstein manifold} \label{sec:applications}

In this section, we compute the derivatives of certain functions on $\mathscr{W}(\R^{*d})$.

\subsection{Differentiation of the expectation map}

If $\mathscr{F}$ is a function from $\mathscr{W}(\R^{*d})$ to some topological vector space, then we will denote the $k$th iterated directional derivative with respect to $V$ in tangent directions $\dot{V}_1$, \dots, $\dot{V}_k$ by
\[
\delta^k \mathscr{F}(V)[\dot{V}_1,\dots,\dot{V}_k],
\]
whenever such a derivative makes sense.  If $\mathscr{F}: \mathscr{W}(\R^{*d}) \to \C$ and there is a function $\mathscr{G}$ mapping a potential $V$ to some element $\mathscr{G}(V) \in T_V \mathscr{W}(\R^{*d})$ that satisfies
\[
\delta \mathscr{F}(V)[\dot{V}] = \ip{\dot{V},\mathscr{G}(V)}_{T_V \mathscr{W}(\R^{*d})},
\]
then it is natural we say that $\mathscr{G}(V)$ is \emph{a gradient} for $\mathscr{F}$.  Due to the degeneracy of $\ip{\cdot,\cdot}_{T_V\mathscr{W}(\R^{*d})}$ we do not expect gradients to be unique.  However, in some circumstances there may turn out to be a canonical choice of gradient that the describes the large $N$ limit of the gradients associated to the random matrix models in the sense of \S \ref{subsec:matrixapproximation}.

The most basic functional we can try to differentiate is $V \mapsto \tilde{\mu}_V(g)$ for a  fixed $g \in \tr(C_{\tr}^\infty(\R^{*d}))$.  The next lemma is a precise version of the statement that
\[
\delta[\tilde{\mu}_V(g)][\dot{V}] = \ip{\dot{V}, L_V g}_{T_V \mathscr{W}(\R^{*d})},
\]
or that $V \mapsto L_V g$ is a gradient for the expectation functional of $g$.

\begin{proposition} \label{prop:differentiateexpectation}
	Suppose $t \mapsto V_t$ is a tangent vector to $V = V_0$ in $\mathscr{W}(\R^{*d})$.  Assume that each $V_t$ satisfies Assumption \ref{ass:freeGibbs} and that $V$ satisfies Assumption \ref{ass:Laplacian} and that for some fixed $R$, we have $\mu_{V_t} \in \Sigma_{d,R}$ for all $t$.  Then
	\[
	\frac{d}{dt} \Bigr|_{t=0} \tilde{\mu}_{V_t}(g) = -\tilde{\mu}_V[\ip{\nabla \dot{V}_0, \nabla \Psi_V g}_{\tr}] = \ip{\dot{V}_0, L_V g}_{T_V \mathscr{W}(\R^{*d})}.
	\]
\end{proposition}

\begin{remark}
	Our previous results show that all the assumptions of the lemma are satisfied if $\nabla V_t$ is uniformly bounded and $\partial \nabla V_t$ is uniformly bounded by a constant strictly less than $1$.
\end{remark}

\begin{proof}[Proof of Proposition \ref{prop:differentiateexpectation}]
	Since $V$ satisfies Assumption \ref{ass:Laplacian}, we have
	\[
	g - \tilde{\mu}_V[g] = \nabla_V^* \nabla \Psi_V g = \nabla_{V_t}^* \nabla \Psi_V g - \ip{\nabla V_t - \nabla V, \nabla \Psi_V g}_{\tr}.
	\]
	When we apply $\tilde{\mu}_t$, the term $\nabla_{V_t}^* \nabla \Psi_V g$ will vanish, and thus,
	\[
	\tilde{\mu}_{V_t}[g] - \tilde{\mu}_V[g] = - \tilde{\mu}_{V_t}[\ip{\nabla V_t - \nabla V, \nabla \Psi_V g}_{\tr}].
	\]
	Now $\nabla V_t - \nabla V \to 0$ in $C_{\tr}^\infty(\R^{*d})^d$ as $t \to 0$.  Since we assumed $\mu_{V_t} \in \Sigma_{d,R}$ for all $t$, this implies that $\tilde{\mu}_{V_t}[\ip{\nabla V_t - \nabla V, \nabla \Psi_V g}_{\tr}] \to 0$.  Since $f$ was an arbitrary smooth scalar-valued function, we therefore have $\mu_{V_t} \to \mu_V$ as $t \to 0$.  Thus,
	\[
	\lim_{t \to 0} \frac{\tilde{\mu}_{V_t}[g] - \tilde{\mu}_V[g]}{t} = -\lim_{t \to 0} \tilde{\mu}_{V_t}\left[\ip*{\frac{\nabla V_t - \nabla V}{t}, \nabla \Psi_V g}_{\tr}\right] = -\tilde{\mu}_V[\ip{\nabla \dot{V}_0,\nabla \Psi_V g}_{\tr}].
	\]
	It follows from Proposition \ref{prop:Laplacianrelations3} (3) that
	\[
	-\ip{\nabla \dot{V}_0,\nabla \Psi_V g}_{\tr} = -\ip{\nabla \Psi_V \dot{V}_0,\nabla g}_{\tr} = \ip{\dot{V}_0, L_V g}_V.
	\]
\end{proof}

\begin{remark}
	There is another heuristic in terms of infinitesimal transport for why this identity is true.  Suppose that $V_t = (\mathbf{f}_t)_*V$.  Then we expect that $\mu_{V_t} = (\mathbf{f}_t)_* \mu_V$.  Hence,
	\[
	\frac{d}{dt} \Bigr|_{t=0} \tilde{\mu}_{V_t}(g) = \frac{d}{dt} \Bigr|_{t=0} \tilde{\mu}_V(g \circ \mathbf{f}_t) = \tilde{\mu}_V[\ip{\dot{\mathbf{f}}_t,\nabla g}_{\tr}] = \tilde{\mu}_V[\ip{\mathbb{P}_V \dot{\mathbf{f}}_0,\nabla g}_{\tr}] = -\tilde{\mu}_V[\ip{\nabla \Psi_V \dot{V}_0,\nabla g}_{\tr}].
	\]
\end{remark}

\subsection{Heat flow and entropy dissipation}

\begin{definition}
	The \emph{heat flow for non-commutative log-densities} is the equation $\dot{V}_t = L_{V_t} V_t = L V_t - \ip{\nabla V_t, \nabla V_t}_{\tr}$ for some smooth map $t \mapsto V_t: [0,\infty) \to \mathscr{W}(\R^{*d})$, where $\dot{V}_t$ denotes the time-derivative.
\end{definition}

As in \cite{Jekel2018}, this equation describes the large $N$ limit of the equation that a function $V_t^{(N)}$ on $M_N(\C)_{\sa}^d$ satisfies when $\partial_t[e^{-N^2 V_t^{(N)}}] = (1/N^2) \Delta[e^{-N^2 V_t^{(N)}}]$.  Following the classical works of \cite{Otto2001} and \cite{OV2000}, we will explain why the heat equation can be viewed as the gradient flow on $\mathscr{W}(\R^{*d})$ of the entropy functional.  We remark that past work on the single variable case has studied the gradient flow for free entropy as a functional on $\mathcal{P}(\R)$ with the Wasserstein metric, which leads to a free Fokker-Planck equation or McKean-Vlasov equation \cite{LLX2020}.

Fix $\omega \in \beta \N \setminus \N$.  For $V$ satisfying Assumption \ref{ass:freeGibbs}, we can consider the functional $\mathscr{X}(V) := \chi^\omega(\mu_V)$.  More properly in the notation of \S \ref{sec:freeGibbslaws}, we should write $I(\mu_V)$ rather than $\mu_V$, but since the meaning is clear, we will simplify the notation hereafter.  The functional $\mathscr{X}$ is the analog of the classical entropy of the free Gibbs law associated to a potential $V$; for a precise relation between the free entropy and classical entropy of random matrix models, see \cite{Jekel2018} or \cite[\S 16.1]{JekelThesis}.   Based on the classical case, the natural guess for the derivative of $\mathscr{X}$ is
\[
\delta \mathscr{X}(V)[\dot{V}] = \ip{L_V V, \dot{V}}_{T_V\mathscr{W}(\R^{*d})},
\]
that is to say, $V \mapsto L_V V$ is a gradient for $\mathscr{X}$.  We will only prove this in the case where the tangent vector $t \mapsto V_t$ is given by transport.

\begin{proposition} \label{prop:differentiateentropy}
	Let $V \in \mathscr{W}(\R^{*d})$ with bounded first and second derivatives and let $V_t = (\mathbf{f}_t)_* V$, where $t \mapsto \mathbf{f}_t$ is a tangent vector to $\id$.  Suppose that $V_t$ satisfies Assumption \ref{ass:freeGibbs} for all $t$, and assume that $\partial^2 \mathbf{f}_t$ and $\partial^2 \mathbf{f}_t^{-1}$ are bounded.  Then
	\[
	\frac{d}{dt} \Bigr|_{t=0} \mathscr{X}(V_t) = \ip{L_V V, \dot{V}_0}_{T_V \mathscr{W}(\R^{*d})}.
	\]
\end{proposition}

\begin{proof}
	By Theorem \ref{thm:magicnormbound}, any free Gibbs law for $V$ is actually a non-commutative law (it is exponentially bounded).  Since $\mu_V$ is the unique non-commutative law satisfying the Dyson-Schwinger equation by assumption, it is the unique free Gibbs law for $V$.  Hence, by Proposition \ref{prop:changeofvariables}, $(\mathbf{f}_t)_* \mu_V$ is the unique free Gibbs law for $V_t$, so it satisfies the Dyson-Schwinger equation and thus $(\mathbf{f}_t)_* \mu_V = \mu_{V_t}$.  By Proposition \ref{prop:changeofvariables} again,
	\[
	\chi^\omega((\mathbf{f}_t)_* \mu_V) = \chi^\omega(\mu_V) + \tilde{\mu}_V[\log \Delta_{\#}(\partial \mathbf{f}_t)].
	\]
	Hence, using the Dyson-Schwinger equation and Proposition \ref{prop:Laplacianrelations3},
	\begin{align*}
		\frac{d}{dt} \Bigr|_{t=0} \chi^\omega((\mathbf{f}_t)_* \mu_V) &= \tilde{\mu}_V \circ \Tr_{\#}(\partial \dot{\mathbf{f}}_0) \\
		&= \tilde{\mu}_V[\ip{\nabla V, \dot{\mathbf{f}}_0}_{\tr}] \\
		&= \tilde{\mu}_V[\ip{\nabla V, \mathbb{P}_V \dot{\mathbf{f}}_0}_{\tr}] \\
		&= -\tilde{\mu}_V[\ip{\nabla V,\nabla \Psi_V \dot{V}_0}_{\tr}] \\
		&= \ip{L_V V, \dot{V}_0}_{T_V \mathscr{W}(\R^{*d})}.  \qedhere
	\end{align*}
\end{proof}

Using $L_V V$ as a (conjectural) gradient of the entropy functional $\mathscr{X}(V)$, the (upward) gradient flow of $\mathscr{X}(V)$ is given by the heat equation $\dot{V}_t = L_{V_t} V_t$.  Solutions to the corresponding equation on $M_N(\C)_{\sa}^d$ were studied in the large $N$ limit by \cite{Jekel2018} under the assumption that $V_0$ was uniformly convex and semi-concave.  In the paper, the equation was viewed as a ``mixture'' of the flat heat equation $\dot{V}_t = L V_t$, which can be solved explicitly using free Brownian motion, and the Hamilton-Jacobi equation $\dot{V}_t = -\ip{\nabla V_t, \nabla V_t}_{\tr}$, which can be solved using the Hopf-Lax inf-convolution semigroup.  The earlier approach of Dabrowski \cite{Dabrowski2017} applied the Clark-Ocone formula to study the solution on matrices through a stochastic optimization problem.  In the non-commutative setting, there are subtle technical questions about which stochastic processes to optimize over (and in particular, in what von Neumann algebra these stochastic processes live in).

The derivative of entropy along this gradient flow is computed in the same way as for the classical Wasserstein manifold, namely,
\[
\frac{d}{dt} \mathscr{X}(V_t) = \ip{L_{V_t} V_t, \dot{V}_t}_{T_{V_t} \mathscr{W}(\R^{*d})} = \ip{L_{V_t} V_t, L_{V_t} V_t}_{T_{V_t} \mathscr{W}(\R^{*d})} = \tilde{\mu}_{V_t}[\ip{\nabla V_t,\nabla V_t}_{\tr}].
\]
The right-hand side (under suitable assumptions) is the \emph{free Fisher information} of $V_t$; see \cite[16.2]{JekelThesis}.  This is the motivation for Voiculescu's definition of the free Fisher information and free entropy $\chi^*$ in \cite{VoiculescuFE5}.  Of course, it is challenging to make this computation rigorous for general $V$; for further discussion, see \cite{BCG2003}, \cite{Dabrowski2017}, \cite{Jekel2018}, \cite{JekelThesis}.

Since $\dot{V}_t = L_{V_t} V_t = -\nabla_{V_t}^* \nabla V_t$, in light of Lemma \ref{lem:transport0}, there is a natural family of transport maps $\mathbf{f}_t$ associated to the path $t \mapsto V_t$ given by
\[
\mathbf{f}_t = \id + \int_0^t \nabla V_u \circ \mathbf{f}_u\,du.
\]
These equations were used in \cite{JekelExpectation} and \cite[\S 17]{JekelThesis} to construct transport in the non-commutative setting.  Of course, the classical analog of these equations has been well-studied, since it comes naturally out of Lafferty's insight that the transport provides local coordinates for the Wasserstein manifold \cite[\S 3]{Lafferty1988} and Otto's result that the heat equation is the gradient flow of the entropy functional \cite{Otto2001}.  The transport maps arising from the gradient flow were also used by Otto and Villani in their proof of the Talagrand inequality \cite[Theorem 1]{OV2000}.

More generally, one can write down the gradient flow of the relative entropy functional
\[
\mathscr{X}_W(V) := \chi_W^\omega(\mu_V) = \chi^\omega(\mu_V) - \tilde{\mu}_V(W).
\]
Using Proposition \ref{prop:differentiateexpectation}, the natural guess is that
\[
\delta \mathscr{X}_W(V)[\dot{V}] = \ip{L_V V, \dot{V}}_{T_V\mathscr{W}(\R^{*d})} - \ip{\dot{V}, L_V W}_{T_V\mathscr{W}(\R^{*d})},
\]
that is, that $V \mapsto L_V[V - W]$ is a gradient for $\mathscr{X}_W$.  The gradient flow thus becomes
\[
\dot{V}_t = L_{V_t}[V_t - W] = LV_t - LW - \ip{\nabla V_t, \nabla V_t}_{\tr} + \ip{\nabla W, \nabla V_t}_{\tr},
\]
and the vector field for constructing transport is $\nabla [V_t - W]$.  It would be very interesting to study this equation when $V_0 \in \mathscr{W}(\R^{*d})$ is arbitrary and $W$ is close to $(1/2) \ip{\mathbf{x},\mathbf{x}}_{\tr}$ in order to obtain a ``transport'' proof that $W$ satisfies the non-commutative Talagrand inequality, parallel to \cite{OV2000}; for an SDE proof of the free Talagrand inequality, see \cite{HU2006}.

The case where $W = (1/2) \ip{\mathbf{x},\mathbf{x}}_{\tr}$ was studied in \cite{JekelExpectation,JekelThesis}, and in fact the conditional version of the equation was used to construct triangular transport to the Gaussian case.  That paper was able to show $\mathrm{W}^*$ triangular transport using functions that were only approximated in uniform $\norm{\cdot}_2$ by trace polynomials rather than in uniform $\norm{\cdot}_\infty$.  However, since many of the ingredients for that argument have been proved here with the new function spaces $C_{\tr}^k(\R^{*d})$, it is likely that the same argument would work to produce $\mathrm{C}^*$ triangular transport under the assumption that $\norm{\partial \nabla V - \Id}_{BC_{\tr}(\R^{*d},\mathscr{M}^1)}$ is bounded by some universal constant smaller than $1$.  That is, it is likely unnecessary to assume bounds on the third derivatives to obtain the result of Corollary \ref{cor:triangulartransport}.

\subsection{Geodesic equation and optimal transport} \label{subsec:geodesic}

\begin{definition}[Geodesic equation]
	The \emph{geodesic equation on $\mathscr{W}(\R^{*d})$} is the pair of equations
	\[
	\left\{ \begin{aligned}
		\dot{V}_t &= L_{V_t} \phi_t \\
		\dot{\phi}_t &= -\frac{1}{2} \ip{\nabla \phi_t, \nabla \phi_t}_{\tr}.
	\end{aligned} \right.
	\]
	The first equation is called the \emph{continuity equation} and the second one is called the \emph{Hamilton-Jacobi equation}.
\end{definition}

This equation arises as the large $N$ limit of the geodesic equation for densities $e^{-N^2V^{(N)}}$ on $M_N(\C)_{\sa}^d$ after expressing it in log-density coordinates and using the normalized Laplacian $(1/N^2) \Delta$ and renormalization of time.  Moreover, we could formally derive it as a Hamiltonian flow as in the classical case (Lemma \ref{lem:Hamiltonian}), relying on Proposition \ref{prop:differentiateexpectation} to differentiate $\tilde{\mu}_V[\ip{\nabla \phi, \nabla \phi}_{\tr}]$ with respect to $V$.  At present, in order to highlight the connections with optimal transport, we will give a heuristic derivation based on minimizing length, which is closely parallel to the classical case (and also related to the Hamiltonian formulation).

Consider a smooth path $[0,T] \to \mathscr{W}(\R^{*d}): t \mapsto V_t$ such that $V_t$ satisfies Assumptions \ref{ass:freeGibbs} and \ref{ass:Laplacian}.  With appropriate continuity assumptions, it makes sense to write down
\[
\int_0^T \ip{\dot{V}_t,\dot{V}_t}_{T_V\mathscr{W}(\R^{*d})}\,dt.
\]
If the curve $t \mapsto V_t$ is a geodesic, then it should minimize this quantity over all paths with the start and end points $V_0$ and $V_T$.  Assume that $\tilde{\mu}_{V_t}[\dot{V}_t] = 0$, and let $\phi_t = -\Psi_{V_t} \dot{V}_t$ (plus an arbitrary constant), so that $L_{V_t} \phi_t = \dot{V}_t$.  Then
\[
\int_0^T \ip{\dot{V}_t,\dot{V}_t}_{T_V\mathscr{W}(\R^{*d})}\,dt = \int_0^T \tilde{\mu}_{V_t}[\ip{\nabla \phi_t,\nabla \phi_t}_{\tr}]\,dt.
\]
Assume we can solve the equation $\dot{\mathbf{f}}_t = \nabla \phi_t \circ \mathbf{f}_t$ to obtain a path of diffeomorphisms $\mathbf{f}_t$ satisfying $V_t = (\mathbf{f}_t)_* V_0$ as in Lemma \ref{lem:transport0}.  This implies under appropriate assumptions that $(\mathbf{f}_t)_* \mu_{V_0} = \mu_{V_t}$ by the same reasoning as in Proposition \ref{prop:heatequation}.
Then note that
\begin{align*}
	\int_0^T \tilde{\mu}_{V_t}[\ip{\nabla \phi_t, \nabla \phi_t}_{\tr}] \,dt
	&= \int_0^T ((\mathbf{f}_t)_* \tilde{\mu}_{V_0})[\ip{\nabla \phi_t, \nabla \phi_t}_{\tr}]\,dt \\
	&= \int_0^T \tilde{\mu}_{V_0}[\ip{\nabla \phi_t \circ \mathbf{f}_t, \nabla \phi_t \circ \mathbf{f}_t}_{\tr}]\,dt \\
	&= \int_0^T \tilde{\mu}_{V_0}[\ip{ \dot{\mathbf{f}}_t, \dot{\mathbf{f}}_t}_{\tr}]\,dt.
\end{align*}
Now we could have replaced $\nabla \phi_t$ by an arbitrary vector field $\mathbf{h}_t$ satisfying $-\nabla_{V_t}^* \mathbf{h}_t = 0$, and then the diffeomorphisms $\mathbf{g}_t$ generated as the flow along $\mathbf{h}_t$ would also satisfy $(\mathbf{g}_t)_* V_0 = V_t$.  However, since $\ker(\nabla_{V_t}^*)$ and $\im(\nabla)$ are orthogonal with respect to $\mu_{V_t}$, we would have
\[
\int_0^T \tilde{\mu}_{V_0}[\ip{ \dot{\mathbf{g}}_t, \dot{\mathbf{g}}_t}_{\tr}]\,dt = \int_0^T \tilde{\mu}_{V_t}[\ip{\mathbf{h}_t, \mathbf{h}_t}_{\tr}] \,dt \geq \int_0^T \tilde{\mu}_{V_t}[\ip{\nabla \phi_t, \nabla \phi_t}_{\tr}] \,dt.
\]
Thus, we expect that $\mathbf{f}_t$ minimizes $\int_0^T \tilde{\mu}_{V_0}[\ip{ \dot{\mathbf{f}}_t, \dot{\mathbf{f}}_t}_{\tr}]\,dt$ among all paths $\mathbf{f}_t$ of diffeomorphisms satisfying $\mathbf{f}_0 = \id$ and $(\mathbf{f}_T)_* V_0 = V_T$.

Next, we use minimality to show that $\ddot{\mathbf{f}}_t = 0$ in $L^2(\mu_{V_0})$.  Let $t \mapsto \mathbf{h}_t$ be a smooth  map $[0,T] \to C_{\tr}^\infty(\R^{*d})_{\sa}^d$ such that $\partial \mathbf{h}_t$ and $\partial^2 \mathbf{h}_t$ are uniformly bounded, $\mathbf{h}_0 = \mathbf{h}_T = 0$.  Let $\mathbf{g}_{t,\epsilon}$ be diffeomorphisms given by
\[
\frac{d}{d\epsilon} \mathbf{g}_{t,\epsilon} = \mathbf{h}_t \circ \mathbf{g}_{t,\epsilon}, \qquad \mathbf{g}_{t,0} = \id,
\]
or in other words $\mathbf{g}_{t,\epsilon} = \exp(\epsilon \mathbf{h}_t)$.  Note that $\mathbf{g}_{0,\epsilon} = \mathbf{g}_{T,\epsilon} = \id$. Using e.g.\ the integral equation for $\mathbf{g}_{t,\epsilon}$, one can show that $(t,\epsilon) \mapsto \mathbf{g}_{t,\epsilon}$ and $(t,\epsilon) \mapsto \mathbf{g}_{t,\epsilon} \circ \mathbf{f}_t$ are continuously differentiable maps into $C_{\tr}(\R^{*d})_{\sa}^d$, similar to classical ODE results on smooth dependence.  Therefore, by minimality
\begin{align*}
	0 &= \frac{d}{d\epsilon} \Bigr|_{\epsilon = 0} \int_0^T \tilde{\mu}_{V_0} \left[\ip*{\frac{d}{dt}[\mathbf{g}_{t,\epsilon} \circ \mathbf{f}_t],\frac{d}{dt}[\mathbf{g}_{t,\epsilon} \circ \mathbf{f}_t]}_{\tr}\right]\,dt \\
	&= 2 \int_0^T \tilde{\mu}_{V_0} \left[\ip*{\frac{d}{dt} \frac{d}{d\epsilon} \Bigr|_{\epsilon = 0}[\mathbf{g}_{t,\epsilon} \circ \mathbf{f}_t],\dot{\mathbf{f}}_t}_{\tr}\right]\,dt \\
	&= 2 \int_0^T \tilde{\mu}_{V_0} \left[\ip*{\frac{d}{dt}[\mathbf{h}_t \circ \mathbf{f}_t],\dot{\mathbf{f}}_t}_{\tr}\right]\,dt \\
	&= -2 \int_0^T \tilde{\mu}_{V_0} \left[\ip*{\mathbf{h}_t \circ \mathbf{f}_t,\ddot{\mathbf{f}}_t}_{\tr}\right]\,dt
\end{align*}
using integration by parts.  Since $\mathbf{h}_t$ is arbitrary except for its values at the endpoints and since $\mathbf{f}_t$ is invertible, we get that $\ddot{\mathbf{f}}_t = 0$ in $L^2(\mu_{V_0})$ for $t \in (0,T)$.

Due to degeneracy of the metric, this does not imply that $\ddot{\mathbf{f}}_t = 0$ in $C_{\tr}(\R^{*d})^d$.  Nonetheless, let us proceed to impose the condition $\ddot{\mathbf{f}}_t = 0$; although this is a leap of faith, it is plausible because the same equations would hold in the random matrix setting.  By computation
\begin{align*}
	\ddot{\mathbf{f}}_t &= \frac{d}{dt} [\nabla \phi_t \circ \mathbf{f}_t] \\
	&= \nabla \dot{\phi}_t \circ \mathbf{f}_t + [\partial \nabla \phi_t \circ \mathbf{f}_t] \# \dot{\mathbf{f}}_t \\
	&= \nabla \dot{\phi}_t \circ \mathbf{f}_t + [\partial \nabla \phi_t \circ \mathbf{f}_t] \# [\nabla \phi_t \circ \mathbf{f}_t] \\
	&= [\nabla \dot{\phi}_t + \partial \nabla \phi_t \# \nabla \phi_t] \circ \mathbf{f}_t.
\end{align*}
Hence,
\begin{equation} \label{eq:phiequation}
	\nabla \dot{\phi}_t + \partial \nabla \phi_t \# \nabla \phi_t = 0.
\end{equation}
But note that $\nabla \ip{\nabla \phi_t, \nabla \phi_t}_{\tr} = 2\partial \nabla \phi_t \# \nabla \phi_t$, which follows from the computation
\begin{align*}
	\ip{\nabla \ip{\nabla \phi_t^{\cA,\tau}(\mathbf{X}), \nabla \phi_t^{\cA,\tau}(\mathbf{X})}_\tau, \mathbf{Y} }_\tau
	&= \partial[\ip{\nabla \phi_t^{\cA,\tau}(\mathbf{X}), \nabla \phi_t^{\cA,\tau}(\mathbf{X})}_\tau][\mathbf{Y}] \\
	&= \ip{\nabla \phi_t^{\cA,\tau}(\mathbf{X}),\partial \nabla \phi_t^{\cA,\tau}(\mathbf{X})[\mathbf{Y}]}_\tau \\
	&= \ip{\partial \nabla \phi_t^{\cA,\tau}(\mathbf{X})[\nabla \phi_t^{\cA,\tau}(\mathbf{X})],\mathbf{Y}}_\tau,
\end{align*}
where we use the fact that $(\nabla \partial \phi)^{\varstar} = \nabla \partial \phi$ since $\phi$ is real-valued.  Therefore, \eqref{eq:phiequation} becomes
\[
\nabla \left[\dot{\phi}_t + \frac{1}{2} \ip{\nabla \phi_t, \nabla \phi_t}_{\tr} \right] = 0.
\]
Thus, we can modify $\phi_t$ by an additive constant (depending on $t$) to achieve that $\dot{\phi}_t = -(1/2) \ip{\nabla \phi_t,\nabla \phi_t}_{\tr}$.  This is exactly the Hamilton-Jacobi equation, so our derivation is complete.

If $\phi_t$ satisfies the Hamilton-Jacobi equation, the same computations show that $\ddot{\mathbf{f}}_t = 0$, and hence $\mathbf{f}_t = \id + t \dot{\mathbf{f}}_0 = \id + t \nabla \phi_0$.  Thus, $V_t = (\id + t \nabla \phi_t)_* V_0$ for some $\phi$, or in other words, a path in $\mathscr{W}(\R^{*d})$ that solves the geodesic equation is a displacement interpolation just as in the classical case.

However, does such a displacement interpolation actually minimize the Riemannian distance?  If $\mathbf{f}_t$ is any family of transport maps with $\mathbf{f}_0 = \id$ and $(\mathbf{f}_T)_* V_0 = V_T$, then (still assuming the validity of $(\mathbf{f}_T)_* \mu_{V_0} = \mu_{V_T}$)
\begin{align*}
	\tilde{\mu}_{V_0}[\ip{\mathbf{f}_T - \id, \mathbf{f}_T - \id}_{\tr}]^{1/2} &\leq \int_0^T \tilde{\mu}_{V_0}[\ip{\dot{\mathbf{f}}_t, \dot{\mathbf{f}}_t}_{\tr}]^{1/2}\,dt \\
	&\leq T^{1/2} \left( \int_0^T \tilde{\mu}_{V_0}[\ip{\dot{\mathbf{f}}_t, \dot{\mathbf{f}}_t}_{\tr}] \,dt \right)^{1/2} \\
	&= T^{1/2} \left( \int_0^T \ip{\dot{V}_t, \dot{V}_t}_{T_V \mathscr{W}(\R^{*d})} \,dt \right)^{1/2},
\end{align*}
and equality is achieved when $\dot{\mathbf{f}}_t$ is constant.  Hence, to show that a family of transport maps $\mathbf{f}_t$ is minimal, it suffices to show that $\mathbf{f}_T$ minimizes $\tilde{\mu}_{V_0}[\ip{\mathbf{f} - \id, \mathbf{f} - \id}_{\tr}]^{1/2}$ among all $\mathbf{f}$ with $\mathbf{f}_* \mu_{V_0} = \mu_{V_T}$.  And this is a much stronger condition since we could easily have $\mathbf{f}_* \mu_{V_0} = \mu_{V_1}$ without $\mathbf{f}_* V_0 = V_1$ due to the degeneracy of the Riemannian metric.

The quantity $\tilde{\mu}_{V_0}[\ip{\mathbf{f}_1 - \id, \mathbf{f}_1 - \id}_{\tr}]^{1/2}$ is related to the non-commutative $L^2$ Wasserstein distance of \cite{BV2001} defined as follows.

\begin{definition}[Non-commutative $L^2$ coupling distance] \label{def:couplingdistance}
	As in \cite{BV2001}, for $\mu$ and $\nu \in \Sigma_d$, we define
	\[
	d_{W,2}(\mu,\nu) = \inf \{\norm{\mathbf{X} - \mathbf{Y}}_2: \mathbf{X}, \mathbf{Y} \in \cA_{\sa}^d, (\cA,\tau) \in \mathbb{W}, \lambda_{\mathbf{X}} = \mu, \lambda_{\mathbf{Y}} = \nu \}.
	\]
	If $(\cA,\tau)$ and $\mathbf{X}$, $\mathbf{Y} \in \cA_{\sa}^d$ achieve the infimum above, then they are called an \emph{optimal coupling} of $\mu$ and $\nu$.
\end{definition}

\begin{remark}
	The existence of optimal couplings is immediate from compactness \cite[Proposition 1.4]{BV2001}.  Indeed, let $\Pi(\mu,\nu)$ be the set of $\pi \in \Sigma_{2d}$ such that the marginals on the first and last $d$ cooordinates are $\mu$ and $\nu$ respectively.  Then $\Pi(\mu,\nu)$ is contained in $\Sigma_{2d,R}$ and is compact.  Because $\pi \mapsto \ip{\mathbf{x} - \mathbf{y},\mathbf{x} - \mathbf{y}}_\pi^{1/2}$ is a continuous function on $\Pi(\mu,\nu)$, it achieves a minimum.  However, it is challenging in the non-commutative case to establish any regularity for the optimal coupling, and indeed we know that there are many non-isomorphic diffuse tracial $\mathrm{W}^*$-algebras \cite{Ozawa2004}, so we do not expect optimal couplings to be given by transport functions in general.
\end{remark}

Returning to our geodesic $V_t = (\id + t \nabla \phi)_* V_0$, we want to show that $\id + t \nabla \phi$ provides an optimal coupling between $\mu_{V_0}$ and $\mu_{V_t}$ where $V_t = (\id + t \nabla \phi)_* \mu_{V_0}$.  In fact, since the potential $V_t$ and the interpolation $\id + t \nabla \phi$ are no longer important for the proof, let us proceed more generally.  Forgetting about $V_t$ and renaming $(1/2) \ip{\mathbf{x},\mathbf{x}}_{\tr} + t \phi$ as $\phi$, it suffices to show that if $\partial \nabla \phi$ is close enough to $\Id$, then $\nabla \phi$ provides an optimal coupling between $\mu$ and $(\nabla \phi)_* \mu$ for every non-commutative law $\mu$.  That is the content of the next proposition.  This is a non-commutative version of one of the easier implications of the Monge-Kantorovich characterization of transport, and it holds without any assumption that $\mu$ is a free Gibbs law or even Connes-approximable.

\begin{proposition}[Optimality of certain transport maps] \label{prop:optimaltransport}
	Let $\phi \in \tr(C_{\tr}^k(\R^{*d}))_{\sa}$ for some $k \geq 2$.  Suppose that for some $K > 0$, we have $\norm{\partial \nabla \phi - K \Id }_{BC_{\tr}(\R^{*d},\mathscr{M}^1)^d} < K$.  Then for every $\mu \in \Sigma_d$, we have
	\[
	d_{W,2}(\mu, (\nabla \phi)_* \mu) = \tilde{\mu}[\ip{\nabla \phi - \id, \nabla \phi - \id}_{\tr}].
	\]
	In other words, if $\mathbf{X}$ is a self-adjoint $d$-tuple from $(\cA,\tau) \in \mathbb{W}$, then $\mathbf{X}$ and $\nabla \phi^{\cA,\tau}(\mathbf{X})$ are an optimal coupling of $\lambda_{\mathbf{X}}$ and $(\nabla \phi)_* \lambda_{\mathbf{X}}$.
\end{proposition}

In the proof, we ``reverse-engineer'' the Monge-Kantorovich duality.  We must first construct the Legendre transform of $\psi$ of $\phi$.  The Legendre transform in the classical setting is a convex function given by
\[
\psi(x) = \sup[\ip{x,y} - \phi(y)].
\]
If $\phi$ is smooth and strictly convex, then the infimum for $\psi(x)$ is achieved at $y = (\nabla \phi)^{-1}(x)$.  Hence, the cheapest way to obtain a smooth Legendre transform for a smooth non-commutative function $\phi$ is to invert $\nabla \phi$.

\begin{lemma}[Smooth non-commutative Legendre transform] \label{lem:Legendretransform}
	Let $\phi \in \tr(C_{\tr}^k(\R^{*d}))_{\sa}$ for some $k \geq 2$.  Suppose that for some $K > 0$, we have $\norm{\partial \nabla \phi - K \Id }_{BC_{\tr}(\R^{*d},\mathscr{M}^1)^d} < K$.  Let $\mathbf{g}$ be the inverse of $\nabla \phi$ as in Proposition \ref{prop:IFT}, and let $\psi$ be given by
	\begin{equation} \label{eq:Legendre}
		\psi^{\cA,\tau}(\mathbf{Y}) := \ip{\mathbf{Y},\mathbf{g}^{\cA,\tau}(\mathbf{Y})}_{\tau} - \phi^{\cA,\tau}(\mathbf{g}^{\cA,\tau}(\mathbf{Y})).
	\end{equation}
	Then for all $(\cA,\tau) \in \mathbb{W}$ and $\mathbf{X} \in \cA_{\sa}^d$, we have
	\begin{equation} \label{eq:Legendremaximizer}
		\psi^{\cA,\tau}(\mathbf{Y}) = \sup_{\mathbf{X} \in \cA_{\sa}^d} \left[ \ip{\mathbf{Y},\mathbf{X}}_\tau - \phi^{\cA,\tau}(\mathbf{X}) \right].
	\end{equation}
	Moreover, $\nabla \psi = \mathbf{g}$ and hence $\psi \in \tr(C_{\tr}^k(\R^{*d}))$.
\end{lemma}

\begin{proof}
	Fix $(\cA,\tau)$ and $\mathbf{Y}, \mathbf{Z} \in \cA_{\sa}^d$.  Let $h: \R \to \R$ be given by
	\[
	h(t) = \ip{\mathbf{Y},\mathbf{g}^{\cA,\tau}(\mathbf{Y}) + t \mathbf{Z}}_\tau - \phi^{\cA,\tau}(\mathbf{g}^{\cA,\tau}(\mathbf{Y}) + t \mathbf{Z}).
	\]
	Then
	\[
	h'(t) = \ip{\mathbf{Y}, \mathbf{Z}}_\tau - \ip{\nabla \phi^{\cA,\tau}(\mathbf{g}^{\cA,\tau}(\mathbf{Y}) + t \mathbf{Z}),\mathbf{Z}}_\tau
	\]
	and
	\[
	h''(t) = -\ip{\partial \nabla \phi^{\cA,\tau}(\mathbf{g}^{\cA,\tau}(\mathbf{Y}) + t \mathbf{Z})[\mathbf{Z}],\mathbf{Z}}_\tau.
	\]
	Because $\norm{\partial \nabla \phi - K \Id}_{\tr} < K$, we obtain $h''(t) > 0$, so $h$ is concave.  Also, since $\nabla \phi \circ \mathbf{g} = \id$, we have $h'(0) = 0$.  Therefore, $h$ is maximized at $t = 0$, so that
	\[
	\ip{\mathbf{Y},\mathbf{g}^{\cA,\tau}(\mathbf{Y}) + \mathbf{Z}}_\tau - \phi^{\cA,\tau}(\mathbf{g}^{\cA,\tau}(\mathbf{Y}) + \mathbf{Z}) \leq \ip{\mathbf{Y},\mathbf{g}^{\cA,\tau}(\mathbf{Y})}_\tau - \phi^{\cA,\tau}(\mathbf{g}^{\cA,\tau}(\mathbf{Y})) = \psi^{\cA,\tau}(\mathbf{Y}).
	\]
	By substituting $\mathbf{X} - \mathbf{g}^{\cA,\tau}(\mathbf{Y})$ for $\mathbf{Z}$, we obtain \eqref{eq:Legendremaximizer}.
	
	Next, by direct computation,
	\begin{align*}
		\ip{\nabla \psi^{\cA,\tau}(\mathbf{Y}),\mathbf{Z}}_\tau
		&= \partial[\ip{\mathbf{Y},\mathbf{g}^{\cA,\tau}(\mathbf{Y})}_\tau - \phi^{\cA,\tau}(\mathbf{g}^{\cA,\tau}(\mathbf{Y}))][\mathbf{Z}] \\
		&= \ip{\mathbf{Z},\mathbf{g}^{\cA,\tau}(\mathbf{Y})}_\tau + \ip{\mathbf{Y},\partial \mathbf{g}^{\cA,\tau}(\mathbf{Y})[\mathbf{Z}]}_\tau - \ip{\nabla \phi^{\cA,\tau}(\mathbf{g}^{\cA,\tau}(\mathbf{Y})), \partial \mathbf{g}^{\cA,\tau}(\mathbf{Y})[\mathbf{Z}]}_\tau \\
		&= \ip{\mathbf{Z},\mathbf{g}^{\cA,\tau}(\mathbf{Y})}_\tau.
	\end{align*}
	Hence, $\nabla \psi = \mathbf{g}$, and so $\psi$ is $C_{\tr}^k$ by the chain rule.
\end{proof}

\begin{proof}[Proof of Proposition \ref{prop:optimaltransport}]
	Let $\mathbf{X}$ be a self-adjoint $d$-tuple from $(\cA,\tau)$ with non-commutative law $\mu$, and let $\nu = (\nabla \phi)_* \mu$.  As in the previous lemma, let $\mathbf{g} = (\nabla \phi)^{-1}$ and let $\psi$ be the Legendre transform of $\phi$.  Writing $\mathbf{Y} = (\nabla \phi)^{-1}(\mathbf{X})$, we have
	\[
	\ip{\mathbf{Y},\mathbf{X}}_\tau
	= \ip{\mathbf{Y},\mathbf{g}^{\cA,\tau}(\mathbf{Y})}_\tau
	= \psi^{\cA,\tau}(\mathbf{Y}) + \phi^{\cA,\tau}(\mathbf{g}(\mathbf{Y})
	= \psi^{\cA,\tau}(\mathbf{Y}) + \phi^{\cA,\tau}(\mathbf{X}).
	\]
	If $\mathbf{X}'$ and $\mathbf{Y}'$ are any other $d$-tuples from some $(\cB,\sigma)$ with the same law as $\mathbf{X}$ and $\mathbf{Y}$, then by \eqref{eq:Legendremaximizer}
	\[
	\ip{\mathbf{Y}',\mathbf{X}'}_\sigma \leq \psi^{\cB,\sigma}(\mathbf{Y}') + \phi^{\cB,\sigma}(\mathbf{X}')
	= \psi^{\cA,\tau}(\mathbf{Y}) + \phi^{\cA,\tau}(\mathbf{X})
	= \ip{\mathbf{Y},\mathbf{X}}_\tau,
	\]
	where we have used the fact that evaluation of $\phi$ and $\psi$ only depends on the non-commutative law of the argument.  Therefore, the coupling $\mathbf{X}$, $\mathbf{Y}$ maximizes the inner product and therefore minimizes the $L^2$-distance (since $\norm{\mathbf{X}}_2^2$ and $\norm{\mathbf{Y}}_2^2$ are uniquely determined by the fixed laws $\mu$ and $\nu$).  Hence, we have an optimal coupling.
\end{proof}

\begin{remark}
	Proposition \ref{prop:optimaltransport} partially answers a question of \cite[\S 5]{GS2014}.  That paper considered the free Gibbs law $\mu_V$ with $V = \tr(f)$ for some non-commutative power series $f$ on an operator-norm ball of some radius $R$, and showed the existence of another power series $g$ such that $(\id + \nabla \tr(g))_* \sigma = \mu_V$ where $\sigma$ is the law of a semicircular family.  Moreover, $\tr(g)$ goes to zero in a certain power-series norm as $\tr(f)$ goes to zero.  The paper did not settle whether the transport map constructed there was optimal, but we can prove this with Proposition \ref{prop:optimaltransport} if $\tr(g)$ is small enough.  Let $\gamma: \R \to [-R,R]$ be a smooth compactly supported function with $\gamma(t) = t$ for $t \in [-R,R]$.  If $\tr(g)$ is sufficiently small, then $\phi = \tr(g) \circ (\gamma(x_1),\dots,\gamma(x_d))$ will satisfy $\norm{\partial \nabla \phi}_{BC_{\tr}(\R^{*d},\mathscr{M}^1)^d} < 1$.  Hence, Proposition \ref{prop:optimaltransport} shows that $\id + \nabla \phi$ defines an optimal coupling between $\sigma$ and $\mu$.
\end{remark}

\subsection{Incompressible Euler equation and inviscid Burgers' equation} \label{subsec:Euler}

\begin{definition}
	Let $V$ satisfy Assumptions \ref{ass:freeGibbs} and \ref{ass:Laplacian}.  Let $\mathbb{P}_V = \nabla \Psi \nabla_V^*$, and let $\Pi_V = 1 - \mathbb{P}_V$ be the Leray projection.  The \emph{(tracial non-commutative) incompressible Euler equation} is the equation
	\[
	\left\{ \begin{aligned}
		\dot{\mathbf{u}}_t &= -\Pi_V[\partial \mathbf{u}_t \# \mathbf{u}_t] \\
		\nabla_V^* \mathbf{u}_t &= 0.
	\end{aligned} \right.
	\]
\end{definition}

This equation was formulated in the framework of non-commutative polynomials (and from there a certain completion of the space) in \cite{VoiculescuEuler}.  Here Voiculescu imitated the approach of Arnold in the classical setting.  Arnold related the incompressible Euler equation to the geodesic equation on the group of diffeomorphisms on some Riemannian manifold that preserve a given measure; more precisely, if $t \mapsto \mathbf{f}_t$ is the geodesic, then $\mathbf{u}_t = \dot{\mathbf{f}}_t \circ \mathbf{f}_t^{-1}$, that is, the right-shift of the $\dot{\mathbf{f}}_t$ to a tangent vector at $\id$.

The non-commutative incompressible Euler equation could be derived by normalizing the classical incompressible Euler equation on $M_N(\C)_{\sa}^d$, but we will give a direct heuristic based on geodesics minimizing length, similar to the earlier derivation of the geodesic equation on $\mathscr{W}(\R^{*d})$.  Recall that $\mathscr{D}(\R^{*d},V)$ is the group of non-commutative diffeomorphisms $\mathbf{f}$ with $\mathbf{f}_* V = V$.  A semi-inner product can be defined on $T_{\id} \mathcal{D}(\R^{*d},V)$ by
\[
\ip{\mathbf{h}_1,\mathbf{h}_2}_{T_{\id} \mathcal{D}(\R^{*d},V)} = \tilde{\mu}_V[\ip{\mathbf{h}_1,\mathbf{h}_2}_{\tr}].
\]
We extend this to a right-invariant formal Riemannian metric on $\mathscr{D}(\R^{*d},V)$.  Since the diffeomorphisms are elements of the vector space $C_{\tr}^d(\R^{*d})_{\sa}^d$, we can view tangent vectors at $\mathbf{f}$ concretely as elements of $C_{\tr}(\R^{*d})_{\sa}^d$, and the right-shift by $\mathbf{f}^{-1}$ of a tangent vector $\mathbf{h}$ at $\mathbf{f}$ produces the tangent vector $\mathbf{h} \circ \mathbf{f}^{-1}$ at $\id$.  Since $\mathbf{f}$ preserves $V$ and hence (again under some reasonable assumptions) $\mu_V$, the Riemannian metric at an arbitrary point is given by the same formula as at $\id$.

Suppose that $[0,T] \to \mathscr{D}(\R^{*d},V): t \mapsto \mathbf{f}_t$ minimizes the integral
\[
\int_0^T \tilde{\mu}_{V} [\ip{\dot{\mathbf{f}}_t,\dot{\mathbf{f}}_t}_{\tr}]\,dt
\]
over all paths with the same start and end points.  Let $\mathbf{u}_t = \dot{\mathbf{f}}_t \circ \mathbf{f}_t^{-1}$, so that $\dot{\mathbf{f}}_t = \mathbf{u}_t \circ \mathbf{f}_t$ and $\nabla_V^* \mathbf{u}_t = 0$ by Lemma \ref{lem:transport0} since $\mathbf{f}_t$ preserves $V$.  Let $\mathbf{h}_t$ be another time-dependent vector field with bounded first derivative such that $\mathbf{h}_0 = \mathbf{h}_T = 0$ and $\nabla_V^* \mathbf{h}_t = 0$.  Let $\mathbf{g}_{t,\epsilon} = \exp(\epsilon \mathbf{h}_t)$, and note that $\mathbf{g}_{t,\epsilon}$ is in $\mathscr{D}(\R^{*d},V)$ by Corollary \ref{cor:measurepreservinggroup}, hence $\mathbf{g}_{t,\epsilon} \circ \mathbf{f}_t$ is another candidate for the minimizer.  Thus, as in the previous section,
\begin{align*}
	0 &= \frac{d}{d\epsilon} \Bigr|_{\epsilon = 0} \int_0^T \tilde{\mu}_{V} \left[\ip*{\frac{d}{dt} [\mathbf{g}_{t,\epsilon} \mathbf{f}_t], \frac{d}{dt} [\mathbf{g}_{t,\epsilon} \mathbf{f}_t]}_{\tr} \right]\,dt \\
	&= 2 \int_0^T \tilde{\mu}_{V} \left[\ip*{\frac{d}{dt} \frac{d}{d\epsilon} \Bigr|_{\epsilon = 0}  [\mathbf{g}_{t,\epsilon} \mathbf{f}_t], \dot{\mathbf{f}}_t]}_{\tr} \right]\,dt \\
	&= -2 \int_0^T \tilde{\mu}_{V} \left[\ip*{\mathbf{h}_t \circ \mathbf{f}_t, \ddot{\mathbf{f}}_t}_{\tr} \right]\,dt \\
	&= -2 \int_0^T \tilde{\mu}_{V} \left[\ip*{\mathbf{h}_t, \ddot{\mathbf{f}}_t \circ \mathbf{f}_t^{-1}}_{\tr} \right]\,dt
\end{align*}
Now $\mathbf{h}_t$ was arbitrary with $\nabla_V^* \mathbf{h}_t = 0$ and $\mathbf{h}_0 = \mathbf{h}_T = 0$.  Although we have not proved that elements of $\ker(\nabla_V^*)$ with bounded derivative are dense in $\ker(\nabla_V^*)$, we proceed under the assumption that $\ddot{\mathbf{f}}_t \circ \mathbf{f}_t$ is orthogonal to $\ker(\nabla_V^*)$.  Then, despite the degeneracy of the Riemannian metric, we posit that $\ddot{\mathbf{f}}_t \circ \mathbf{f}_t$ is a gradient, or that $\Pi_V[\ddot{\mathbf{f}}_t \circ \mathbf{f}_t] = 0$.  But note that
\[
\ddot{\mathbf{f}}_t = \frac{d}{dt} [\mathbf{u}_t \circ \mathbf{f}_t] = \dot{\mathbf{u}}_t \circ \mathbf{f}_t + (\partial \mathbf{u}_t \circ \mathbf{f}_t) \# (\mathbf{u}_t \circ \mathbf{f}_t),
\]
hence
\[
\Pi_V[\dot{\mathbf{u}}_t + \partial \mathbf{u}_t \# \mathbf{u}_t] = 0.
\]
Now $\Pi_V \dot{\mathbf{u}}_t = \dot{\mathbf{u}}_t$, so this is the incompressible Euler equation.

One can also proceed using Arnold's framework for geodesics on Lie groups with a right-invariant Riemannian metric.  He showed that the angular velocity $\mathbf{u}_t$ of a geodesic must satisfy $\dot{\mathbf{u}}_t = -B(\mathbf{u}_t,\mathbf{u}_t)$, where $B$ is the bilinear form on the Lie algebra defined by $\ip{[\mathbf{h}_1,\mathbf{h}_2], \mathbf{h}_3} = \ip{B(\mathbf{h}_3,\mathbf{h}_1),\mathbf{h}_2}$.  This was the approach followed by Voiculescu \cite{VoiculescuEuler} in the non-commutative setting.   We present here a version of \cite[Lemma 1]{VoiculescuEuler} for tracial non-commutative smooth functions.

\begin{lemma}
	Let $V$ satisfy Assumptions \ref{ass:freeGibbs} and \ref{ass:Laplacian}.  For $\mathbf{h}_1$, $\mathbf{h}_2 \in \ker(\nabla_V^*)$, let
	\[
	B(\mathbf{h}_1,\mathbf{h}_2) := \Pi_V[\partial \mathbf{h}_1 \# \mathbf{h}_2 + (\partial \mathbf{h}_2)^{\varstar} \# \mathbf{h}_1].
	\]
	Then for $\mathbf{h}_1$, $\mathbf{h}_2$, $\mathbf{h}_3 \in \ker(\nabla_V^*) \cap C_{\tr}^\infty(\R^{*d})_{\sa}^d$, we have
	\[
	\tilde{\mu}_V[\ip{[\mathbf{h}_1,\mathbf{h}_2], \mathbf{h}_3}_{\tr}] = \tilde{\mu}_V[\ip{B(\mathbf{h}_3,\mathbf{h}_1), \mathbf{h}_2}_{\tr}].
	\]
	Moreover,
	\[
	B(\mathbf{h},\mathbf{h}) = \Pi_V[\partial \mathbf{h} \# \mathbf{h}].
	\]
\end{lemma}

\begin{proof}
	Note that
	\[
	\nabla \ip{\mathbf{h}_2,\mathbf{h}_3}_{\tr} = (\partial \mathbf{h}_2)^{\varstar} \# \mathbf{h}_3 + (\partial \mathbf{h}_3)^{\varstar} \# \mathbf{h}_2,
	\]
	which we can see from evaluating at some $\mathbf{X} \in \cA_{\sa}^d$ and pairing with a tangent vector $\mathbf{Y} \in \cA^d$.  By Proposition \ref{prop:Laplacianrelations3}, $\mathbf{h}_1$ is orthogonal to gradients.  Thus,
	\[
	\tilde{\mu}_V[\ip{\mathbf{h}_1, (\partial \mathbf{h}_2)^{\varstar} \# \mathbf{h}_3}_{\tr} + \ip{\mathbf{h}_1, (\partial \mathbf{h}_3)^{\varstar} \# \mathbf{h}_2}_{\tr}] = 0.
	\]
	Therefore,
	\begin{align*}
		\tilde{\mu}_V[\ip{[\mathbf{h}_1,\mathbf{h}_2], \mathbf{h}_3}_{\tr}] &= \tilde{\mu}_V[\ip{\partial \mathbf{h}_1 \# \mathbf{h}_2 - \partial \mathbf{h}_2 \# \mathbf{h}_1, \mathbf{h}_3}_{\tr}] \\
		&= \tilde{\mu}_V[\ip{\mathbf{h}_2, (\partial \mathbf{h}_1)^{\varstar} \# \mathbf{h}_3}] - \tilde{\mu}_V[\ip{\mathbf{h}_1, (\partial \mathbf{h}_2)^{\varstar} \# \mathbf{h}_3}_{\tr}] \\
		&= \tilde{\mu}_V[\ip{(\partial \mathbf{h}_1)^{\varstar} \# \mathbf{h}_3,\mathbf{h}_2}] + \tilde{\mu}_V[\ip{\mathbf{h}_1, (\partial \mathbf{h}_3)^{\varstar} \# \mathbf{h}_2}_{\tr}] \\
		&= \tilde{\mu}_V[\ip{(\partial \mathbf{h}_1)^{\varstar} \# \mathbf{h}_3 + \partial \mathbf{h}_3 \# \mathbf{h}_1, \mathbf{h}_2}_{\tr}].
	\end{align*}
	Since $\mathbf{h}_2$ is in the kernel of $\nabla_V^*$, we have $\mathbf{h}_2 = \Pi_V \mathbf{h}_2$.  After inserting the $\Pi_V$ into the equation, we can move it to the other side of the inner product by Proposition \ref{prop:Laplacianrelations3} (5) to obtain $\tilde{\mu}_V[\ip{B(\mathbf{h}_3,\mathbf{h}_1), \mathbf{h}_2}_{\tr}]$.
	
	For the second claim, note that $(\partial \mathbf{h})^{\varstar} \# \mathbf{h} = \nabla \ip{\mathbf{h},\mathbf{h}}_{\tr}$, and thus it is killed by $\Pi_V$.  The only remaining term is $\Pi_V[\partial \mathbf{h} \# \mathbf{h}]$.
\end{proof}

The formula $\dot{\mathbf{u}}_t = -B(\mathbf{u}_t,\mathbf{u}_t)$ clearly gives the same incompressible Euler equation.

We remark that the geodesic equation on $\mathscr{D}(\R^{*d})$ can be heuristically derived in a similar way.  Fixing $V$, we can define a right-invariant Riemannian metric by $\ip{\mathbf{h}_1,\mathbf{h}_2}_{T_{\mathbf{f}} \mathscr{D}(\R^{*d})} = \tilde{\mu}_V[\ip{\mathbf{h}_1 \circ \mathbf{f}^{-1},\mathbf{h}_2 \circ \mathbf{f}^{-1}}_{\tr}]$.  The minimality condition results in $\ddot{\mathbf{f}}_t \circ \mathbf{f}_t^{-1}$ being zero in $L^2(\mu_V)^d$.  We posit that $\ddot{\mathbf{f}}_t$ is actually zero, which results in the equation
\[
\dot{\mathbf{u}}_t = \partial \mathbf{u}_t \# \mathbf{u}_t,
\]
where $\mathbf{u}_t = \dot{\mathbf{f}}_t \circ \mathbf{f}_t^{-1}$.  This is the tracial non-commutative \emph{inviscid Burgers' equation}.  The case where $\mathbf{u}_t = \nabla \phi_t$ gives exactly the Wasserstein geodesics.

\end{document}